\documentclass[a4paper,11pt,twoside, eqno]{article}
\usepackage{mathrsfs}
\usepackage{amssymb}
\usepackage{amsmath}
\usepackage{mathrsfs}
\usepackage{amsthm}
\usepackage{amsfonts}
\usepackage{color}
\usepackage{latexsym}
\usepackage{indentfirst}

\usepackage{amscd}
\usepackage{bbm}
\usepackage{graphicx}
\usepackage[all]{xy}
\usepackage[mathscr]{eucal}
\usepackage{subfigure}

\newcommand{\N}{{\mathbb N}}
\newcommand{\C}{{\mathbb C}}
\newcommand{\R}{{\mathbb R}}

\newcommand{\Z}{{\mathbb Z}}

\allowdisplaybreaks
\newtheorem{theorem}{Theorem}[section]

\newtheorem{corollary}[theorem]{Corollary}

\newtheorem{definition}[theorem]{Definition}
\newtheorem{example}[theorem]{Example}
\newtheorem{remark}[theorem]{Remark}
\newtheorem{hypothesis}[theorem]{Hypothesis}
\newtheorem{lemma}[theorem]{Lemma}
\newtheorem{question}[theorem]{Question}

\newtheorem{proposition}[theorem]{Proposition}
\newtheorem{claim}[theorem]{Claim}

\newtheorem{assumption}[theorem]{Assumption}

\numberwithin{equation}{section}

\begin{document}

\title{\bf\Large Bifurcations for  Hamiltonian  systems
\footnotetext{\hspace{-0.35cm} 2020
{\it Mathematics Subject Classification}.
Primary 37J20, 34C23.
\endgraf
{\it Key words and phrases.}
Bifurcation;  Hamiltonian system; dual variational principle; Lyapunov-Schmidt reduction; Maslov-type index
}}
\date{}
\author{Guangcun Lu\footnote{
E-mail: \texttt{gclu@bnu.edu.cn}/{December 20, 2021(first), 
May 14, 2026(published online)}.}}

\maketitle

\vspace{-0.7cm}

\begin{center}
\begin{minipage}{13cm}
{\small {\bf Abstract}\quad
 Our assumptions and results include the following special versions:  for a topological space $\Lambda$ and $H\in C(\Lambda\times S^1\times{\R}^{2n})$, $S^1=\mathbb{R}/\mathbb{Z}$,
   suppose that each $H(\lambda,t,\cdot)$ is $C^2$ and all its partial derivatives depend continuously on
 $(\lambda, t, z)$ and that for each $\lambda\in\Lambda$ there exists $1$-periodic solution $v_\lambda:\R\to\mathbb{R}^{2n}$
 of $\dot{v}(t)=J\nabla_z H(\lambda,t, v(t))$ such that
  $\Lambda\times \R\ni (\lambda,t)\mapsto v_\lambda(t)\in\mathbb{R}^{2n}$
  and $\Lambda\times \R\ni (\lambda,t)\mapsto \dot{v}_\lambda(t)\in\mathbb{R}^{2n}$
  are also continuous.
 Let $i(\gamma_\lambda)$ be the Maslov-type index of the fundamental matrix solution $\gamma_\lambda$
of $\dot{Z}(t)=J\nabla^2_zH(\lambda,t, v_\lambda(t))Z(t)$,
 which is equal to the Conley-Zehnder index $\mu_{\rm CZ}(\gamma_\lambda)$ if
 $\nu(\gamma_\lambda):=\dim{\rm Ker}(\gamma_\lambda(1)-I_{2n})$ is zero. Then
  \begin{enumerate}
\item[\rm (i)] For some $\mu\in\Lambda$,
if there exist a sequence $(\lambda_k)\subset\Lambda$ converging to $\mu$
 and $1$-periodic solutions of $\dot{v}(t)=J\nabla_z H(\lambda_k,t, v(t))$,
 $\bar{v}_k\ne v_{\lambda_k}$, $k=1,2,\cdots$, such that $\bar{v}_k\to v_{\mu}$
 in $W^{1,2}(S^1;\R^{2n})$, then $\nu(\gamma_{\mu})>0$.

\item[\rm (ii)] If $\Lambda$ is path-connected and there exist two distinct points $\lambda^+, \lambda^-\in\Lambda$ such that
  $[i(\gamma_{\lambda^-}), i(\gamma_{\lambda^-})+\nu(\gamma_{\lambda^-})]\cap[i(\gamma_{\lambda^+}), i(\gamma_{\lambda^+})+\nu(\gamma_{\lambda^+})]=\emptyset$
 and either $\nu(\gamma_{\lambda^+})=0$ or $\nu(\gamma_{\lambda^-})=0$,
 then there exists $\mu\in \Lambda$ for which the conditions in (i) hold and $\bar{v}_k\to v_{\mu}$
 in $C^1(S^1;\R^{2n})$. If $\Lambda$ is first countable, and  for some $\mu\in\Lambda$
there exist $\lambda^+\ne\lambda^-$ in any deleted neighborhood of $\mu$ satisfying just conditions
then the conditions in (i) hold and $\bar{v}_k\to v_{\mu}$  in $C^1(S^1;\R^{2n})$.

\item[\rm (iii)] Suppose that $\Lambda=(-\epsilon, \epsilon)$, $\mu=0$ and $\gamma_\lambda$ is as above. If
$\nu(\gamma_\lambda)=0$  for each $\lambda\ne 0$, and $\nu(\gamma_0)\ne 0$, and
  $i(\gamma_\lambda)$  takes, respectively, values $i(\gamma_0)$ and $i(\gamma_0)+ \nu(\gamma_0)$
 as $\lambda\in(-\epsilon, \epsilon)$ varies in  two deleted half neighborhoods  of $0$,
then at least one of the following holds:
\begin{enumerate}
\item[$\bullet$] $\dot{v}(t)=J\nabla_z H(\mu,t, v(t))$ has a sequence of solutions, $v_\mu^k\ne v_\mu$, $k=1,2,\cdots$,
such that  $v_\mu^k\to v_\mu$ in $C^1(S^1;\R^{2n})$;

\item[$\bullet$]  for every $\lambda\in\Lambda\setminus\{\mu\}$ near $\mu$ there is a  solution $\bar{v}_\lambda\ne v_\lambda$ of
$\dot{v}(t)=J\nabla_z H(\lambda,t, v(t))$
such that  $\bar{v}_\lambda\to v_\mu$ in $C^1(S^1;\R^{2n})$  as $\lambda\to \mu$;

\item[$\bullet$] for a given neighborhood $\mathcal{W}$ of $v_\mu$ in $C^1(S^1;\R^{2n})$
there is an one-sided  neighborhood $\Lambda^0$ of $\mu$ such that
for any $\lambda\in\Lambda^0\setminus\{\mu\}$, $\dot{v}(t)=J\nabla_z H(\lambda,t, v(t))$
has at least two distinct solutions $v_\lambda^1\ne v_\lambda$ and $v_\lambda^2\ne v_\lambda$  in $\mathcal{W}$,
which can also be required to have different Hamiltonian actions
provided that $\nu(\gamma_0)>1$ and $\dot{v}(t)=J\nabla_z H(\lambda,t, v(t))$
has only finitely many  solutions in $\mathcal{W}$.
\end{enumerate}

\item[\rm (iv)] Under the assumptions of (iii), if $H$ is independent of $t$ and that each $v_\lambda$ is constant (therefore
 $\gamma_\lambda(t)=\exp(tJ\nabla_z^2H({\lambda}, v_\lambda))$) and
  ${\rm Ker}(\nabla_z^2H({0},v_0))=\{0\}$,
 then either $\dot{v}(t)=J\nabla_z H(0, v(t))$
 has a sequence of geometrically distinct solutions, $\bar{v}_k\notin\mathbb{R}\cdot v_0$, $k=1,2,\cdots$,
which  converges to $v_0$ in $C^1(S^1;\mathbb{R}^{2n})$, or
there exist left and right  neighborhoods $\Lambda^-$ and $\Lambda^+$ of $0$ in $\Lambda$
and integers $n^+, n^-\ge 0$, such that $n^++n^-\ge \nu(\gamma_0)/2$,
and for $\lambda\in\Lambda^-\setminus\{0\}$ (resp. $\lambda\in\Lambda^+\setminus\{0\}$),
$\dot{v}(t)=J\nabla_z H(\lambda, v(t))$  has at least $n^-$ (resp. $n^+$) $S^1$-distinct
 solutions, $v_\lambda^i\notin S^1\cdot v_\lambda$, $i=1,\cdots,n^-$ (resp. $n^+$),
which  converge to  $v_0$ in $C^1(S^1;\mathbb{R}^{2n})$.
\item[\rm (v)] If $H$ in (i)-(iii) is independent of $t$ and all $v_\lambda$ are equal to a nonconstant $v_\mu$,
 corresponding bifurcation results about $S^1$-distinct solutions with (i)-(iii) are given.
\item[\rm (vi)] If $H$ in (i)-(iii) also satisfies $H(\lambda,-t, (-p,q))=H(\lambda,t, (p,q))$ for
all $(\lambda,t)\in\Lambda\times\mathbb{R}$ and $(p,q)\in\mathbb{R}^{2n}$,
corresponding results about brake orbits of $\dot{v}(t)=J\nabla_z H(\lambda, t, v(t))$ with (i)-(iii) are given.
  \end{enumerate}
In addition, similar results are also proved for bifurcations for Hamiltonian paths connecting affine Lagrangian subspaces. }
\end{minipage}
\end{center}

\vspace{0.2cm}
	
\tableofcontents

\vspace{0.2cm}

\section{Introduction and main results}\label{sec:Intro0}
\setcounter{equation}{0}

This work continues our program on variational bifurcations beginning at \cite{Lu8}, \cite{Lu9}.
The current manuscript studies bifurcations of the following Hamiltonian boundary value problem
\begin{equation}\label{e:Hboundary-}
\dot{u}(t)=J\nabla_z H({\lambda,t}, u(t))\;\forall t\in [0,\tau]\quad\hbox{and}\quad (u(0), u(\tau))\in{\bf N}
\end{equation}
with respect to a continuous family $\{u_\lambda\,|\,\lambda\in\Lambda\}$ of solutions of this problem  parameterized by a topological space $\Lambda$,
where ${\bf N}\subset\mathbb{R}^{2n}\times\mathbb{R}^{2n}$ is a submanifold, $J$ is given by (\ref{e:standcompl})
and $H:\Lambda\times [0,\tau]\times{\R}^{2n}\to\R$ is as in Assumption~\ref{ass:BasiAss1}. We answer the following questions:
\begin{description}
\item[$\bullet$] Under what conditions there exists a point $\mu\in\Lambda$ such that every neighborhood of
$(\mu, u_\mu)$ in  $\Lambda\times C^1([0,\tau];\R^{2n})$ contains a point $(\lambda, v_\lambda)\notin\{(\lambda,u_\lambda)\,|\,\lambda\in\Lambda\}$
satisfying (\ref{e:Hboundary-})?
\item[$\bullet$] What are the necessary (resp. sufficient) condition for a point $(\mu, u_\mu)$ to satisfy the above properties ?
\item[$\bullet$] How are the solutions of (\ref{e:Hboundary-}) distributed near the point $(\mu, u_\mu)$ as above ?
\end{description}

There is a vast amount of literature concerning bifurcations for periodic solutions  of Hamiltonian systems,
i.e., the case that ${\bf N}$ is the diagonal in $\mathbb{R}^{2n}\times\mathbb{R}^{2n}$ and $H:\Lambda\times \mathbb{R}\times{\R}^{2n}\to\R$
is $\tau$-periodic in the $\mathbb{R}$-variable,  see, e.g.
\cite{Ba1}, \cite{BaSz}, \cite{FiPeRe00},  \cite{Han}, \cite{IJWa},
\cite{MeHa92}, \cite{Rab86}, \cite{Rad10}, \cite{RadRy10}, \cite{Szu}, \cite{Wat15} and references therein.
Most of them consider a family of Hamiltonian functions parameterized affinely by a real number
 or invariant for the action of a Lie group
on $\mathbb{R}^{2n}$, and used either
Liapunov-Schmidt reductions (\cite{Ba1}, \cite{BaSz}, \cite{Rab86}, \cite{Rad10},
\cite{Szu}) or spectral flow methods (\cite{FiPeRe00}, \cite{IJWa}, \cite{Rad10},
\cite{RadRy10}, \cite{Wat15})
or the Poincar\'e map  (\cite{MeHa92}, \cite{Han}).

 In this work, using the abstract bifurcation theory recently developed by the author in \cite{Lu8},
 \cite{Lu10} and Appendix~\ref{app:Th3.5}
we can prove many new bifurcation results for solutions of four types of non-autonomous Hamiltonian boundary value problems
nonlinearly depending on parameters. Our methods are also used to derive some bifurcation results  starting at
 a non-equilibrium (i.e., nonconstant) periodic solution of an autonomous Hamiltonian system.
Though  some necessary or sufficient conditions for
bifurcations of these Hamiltonian problems may also be derived from previous abstract theories,
{\it the most interesting and important alternative results in this paper} can only be proved with our
generalizations in \cite{Lu8}, \cite{Lu10} and Appendix~\ref{app:Th3.5} for the famous
 Rabinowitz's alternative bifurcation theorem \cite{Rab77}.
Before precisely stating these we begin with  the following.\\

\noindent{\textsf{Notation and conventions}}. Let $S_\tau:=\R/\tau\Z$ for $\tau>0$.
All vectors in $\R^m$ will be understood as column vectors. The transpose of a matrix $M\in\R^{m\times m}$
is denoted by $M^T$. ${\rm Ker}(M)=\{x\in\R^m\,|\,Mx=0\}$.
We denote by
$(\cdot,\cdot)_{\mathbb{R}^m}$  the standard Euclidean inner product in $\R^m$
and by $|\cdot|$ the corresponding norm.
Let $\mathcal{L}_s(\mathbb{R}^{m})$ be the set of all real symmetric
matrixes of order $m$, and let ${\rm Sp}(2n,\mathbb{R})$ be the symplectic
group of real symplectic matrixes  of order $2n$, i.e., ${\rm Sp}(2n,\mathbb{R})=\{M\in{\rm GL}(2n,\mathbb{R})\,|\, M^TJ_nM=J_n\}$,
where
\begin{equation}\label{e:standcompl}
J_n=\left(\begin{array}{cc}
             0 & -I_n \\
             I_n & 0 \\
           \end{array}
         \right)
\end{equation}
with the identity matrix $I_n$ of order $n$, which gives  the standard complex structure  on $\mathbb{R}^{2n}$,
$$
 (q_1,\cdots,q_n, p_1,\cdots, p_n)^T\mapsto J_n(q_1,\cdots,q_n, p_1,\cdots, p_n)^T=(-p_1,\cdots,-p_n, q_1,\cdots, q_n)^T.
$$

 For a map between Banach spaces $f:X\to Y$,
 $Df(x)$ (resp. $df(x)$ or $f'(x)$) denotes the G\^ateaux (resp. Fr\'echet) derivative of $f$ at $x\in X$,
 which is an element in $\mathscr{L}(X,Y)$. Of course, we also use $f'(x)$ to denote $Df(x)$
 when no confusion occurs.
 When $Y=\mathbb{R}$, $f'(x)\in \mathscr{L}(X,\mathbb{R})=X^\ast$,
 and if $X$ is a Hilbert space $\mathcal{H}$ we call the Riesz representation of $f'(x)$ in $\mathcal{H}$
 gradient of $f$ at $x$, denoted by $\nabla f(x)$. The Fr\'echet (or G\^ateaux) derivative of $\nabla f$ at $x\in \mathcal{H}$
 is denoted by $f''(x)$ or $\nabla^2 f(x)$, which is an element in $\mathscr{L}_s(\mathcal{H})$.
(Precisely, $f''(x)=(f')'(x)\in \mathscr{L}(\mathcal{H};\mathscr{L}(\mathcal{H};\mathbb{R}))$ is a symmetric bilinear form on $\mathcal{H}$,
 and is identified with $D(\nabla f)(x)$ after $\mathscr{L}(\mathcal{H},\mathbb{R})=\mathcal{H}^\ast$ is identified with $\mathcal{H}$ via the Riesz representation theorem.)
For $\varepsilon>0$ and a point $x$ in a Banach space $X$ we write $B_X(x,\varepsilon)=\{y\in X\,|\,\|x-y\|<\varepsilon\}$.\\

We will give a detailed description of our main results in four subsections.
Section~1.1 lists bifurcation results for solutions of two types of non-autonomous Hamiltonian boundary value problems.
Results in the second type immediately follows from that of the first type.
Section~1.2  concerns with bifurcations for generalized periodic solutions of autonomous Hamiltonian systems.
 In Section~1.3 we consider bifurcation for brake orbits of periodic Hamiltonian systems.
In Section~1.4 we study bifurcation for Hamiltonian paths connecting affine Lagrangian subspaces.\\

\noindent{1.1. \; \bf  Bifurcations for solutions of two types of non-autonomous Hamiltonian systems}.
We begin with the following two closely related assumptions.
\begin{assumption}\label{ass:BasiAss1}
{\rm  For a real $\tau>0$,  a symplectic matrix $M\in{\rm Sp}(2n,\mathbb{R})$,
 and a topological space $\Lambda$,
let $H:\Lambda\times [0,\tau]\times{\R}^{2n}\to\R$ be a continuous function satisfying:
\begin{itemize}
\item[(i)] $H(\lambda, \tau, Mz)=H(\lambda, 0, z)$ for all $(\lambda, z)\in \Lambda\times {\R}^{2n}$.
\item[(ii)] For each $(\lambda,t)\in\Lambda\times [0,\tau]$, $H(\lambda,t,\cdot):{\R}^{2n}\to\R$ is $C^2$.
\item[(iii)] The Euclidean gradient and the Hessian of  $H(\lambda,t, z)$ with respect to
$z\in{\R}^{2n}$ are denoted by $\nabla_zH(\lambda,t, z)$ and $\nabla^2_zH(\lambda,t, z)$,
 respectively. Here, $\nabla^2_zH(\lambda,t, z)=D_z(\nabla_zH(\lambda,t,z))\in \mathcal{L}_s(\mathbb{R}^{2n})$. They depend continuously on the variables
 $(\lambda, t, z)\in\Lambda\times [0,\tau]\times\mathbb{R}^{2n}$.
 \end{itemize}
  For each $\lambda\in\Lambda$, let $u_\lambda$ be a solution of the Hamiltonian boundary value problem
\begin{equation}\label{e:Hboundary}
\dot{u}(t)=J\nabla_z H({\lambda,t}, u(t))\;\forall t\in [0,\tau]\quad\hbox{and}\quad u(\tau)=Mu(0),
\end{equation}
 and $\Lambda\times [0,\tau]\ni (\lambda,t)\mapsto u_\lambda(t)\in\mathbb{R}^{2n}$ is also continuous. }
\end{assumption}

  Under this assumption it follows from (\ref{e:Hboundary}) that
 $\Lambda\times [0,\tau]\ni (\lambda,t)\mapsto \dot{u}_\lambda(t)\in\mathbb{R}^{2n}$ is also continuous.
 For any given $\mu\in\Lambda$ using nets we can prove $\|u_\lambda-u_\mu\|_{C^1}\to 0$ as $\lambda\to\mu$.
Note that Assumption~\ref{ass:BasiAss1} cannot guarantee that
 each $u_\lambda:[0,\tau]\to\R^{2n}$ is $C^2$ since we have not assumed that
 each $H(\lambda,\cdot,\cdot):[0, \tau]\times {\R}^{2n}\to\R$, $\lambda\in\Lambda$, is $C^1$.

\begin{assumption}\label{ass:BasiAss}
{\rm For a $\tau>0$,  a symplectic matrix $M\in{\rm Sp}(2n,\mathbb{R})$,
 and a topological space $\Lambda$,  let
$H:\Lambda\times\R\times{\R}^{2n}\to\R$ be a continuous function such that:
\begin{itemize}
\item[(i)] For each  $(\lambda, t)\in\Lambda\times \R$, the map $H(\lambda,t,\cdot):{\R}^{2n}\to\R$
is of class $C^2$.
\item[(ii)] The Euclidean gradient $\nabla_zH(\lambda,t, z)$ (with respect to
$z\in{\R}^{2n}$) is continuous in  $(\lambda, t, z)\in\Lambda\times \R\times\mathbb{R}^{2n}$.
\item[(iii)] The following identity holds for all  $(\lambda, t, z)\in\Lambda\times \R\times\mathbb{R}^{2n}$:
\begin{eqnarray}\label{e:M-invariant1}
H(\lambda, t+\tau, Mz)=H(\lambda, t, z).
\end{eqnarray}
\end{itemize}
For each $\lambda\in\Lambda$ let $v_\lambda:\R\to\mathbb{R}^{2n}$ satisfy the following generalized periodic Hamiltonian system
\begin{equation}\label{e:PPer}
\dot{v}(t)=J\nabla_z H(\lambda,t, v(t))\quad\hbox{and}\quad v(t+\tau)=Mv(t)\;\;\forall t\in \R,
\end{equation}
 and $\Lambda\times \R\ni (\lambda,t)\mapsto v_\lambda(t)\in\mathbb{R}^{2n}$ is also continuous.
}
\end{assumption}

(In some papers solutions of (\ref{e:PPer}) are called $(M,\tau)$-\textsf{periodic} or \textsf{affine $\tau$-periodic}.)
Clearly, a solution $v$ of (\ref{e:PPer}) gives one of (\ref{e:Hboundary}), $v|_{[0,\tau]}$.
Conversely, since (\ref{e:M-invariant1}) implies
\begin{eqnarray}\label{e:M-invariant2}
\nabla_z H(\lambda, t+\tau, Mz)=(M^{-1})^T\nabla_z H(\lambda, t, z)\quad\forall (\lambda,t,z)\in\Lambda\times\R\times\mathbb{R}^{2n},
\end{eqnarray}
every solution $u$ of (\ref{e:Hboundary}) gives rise to a solution $u^M$ of (\ref{e:PPer})
 defined by
\begin{equation}\label{e:extend}
u^M(t):=M^ku(t-k\tau)\;\;\hbox{for $k\tau<t\le (k+1)\tau$},\;k\in\Z.
\end{equation}
Moreover, this correspondence between solutions of (\ref{e:Hboundary}) and (\ref{e:PPer}) is one-to-one.

The linearized problems of (\ref{e:Hboundary}) and (\ref{e:PPer}) along $u_\lambda$ and $v_\lambda$ are, respectively,
\begin{eqnarray}\label{e:linear1}
&&\dot{u}(t)=J\nabla^2_zH(\lambda,t, u_\lambda(t))u(t)\; \forall t\in [0,\tau]\quad\hbox{and}\quad u(\tau)=Mu(0),\\
&&\dot{v}(t)=J\nabla^2_zH(\lambda,t, v_\lambda(t))v(t)\quad\hbox{and}\quad v(t+\tau)=Mv(t)\;\forall t\in \R.\label{e:linear2*}
\end{eqnarray}

Under Assumption~\ref{ass:BasiAss1}, for some $\mu\in\Lambda$, we say that
 $(\mu, u_\mu)$  is a \textsf{bifurcation point along sequences} of (\ref{e:Hboundary})
 with respect to the branch $\{(\lambda, u_\lambda)\,|\,\lambda\in\Lambda\}$
if there exists a sequence $(\lambda_k)\subset\Lambda$ converging to $\mu\in\Lambda$
 and corresponding solutions $\bar{u}_k\ne u_{\lambda_k}$  of (\ref{e:Hboundary})
 (for $\lambda=\lambda_k\in\Lambda$),  $k=1,2,\cdots$, such that
 ${\bar{u}_k}\to {u_{\mu}}$  in $C^1([0,\tau];\R^{2n})$ (or equivalently in $C^{0}([0,\tau];\R^{2n})$).

Under Assumption~\ref{ass:BasiAss}, consider a sequence
$(\lambda_k)\subset\Lambda$ converges to $\mu\in\Lambda$, and  $\bar{v}_k:\R\to\R^{2n}$ is a solution of (\ref{e:PPer}) (for $\lambda=\lambda_k\in\Lambda$), $k=1,2,\cdots$.
Then $\bar{v}_k|_{[0,\tau]}\to v_\mu|_{[0,\tau]}$ in
$C^1([0,\tau];\mathbb{R}^{2n})$ (resp. $C^{0}([0,\tau];\mathbb{R}^{2n})$)
 if and only if $\bar{v}_k|_{[a, b]}\to v_\mu|_{[a, b]}$ in $C^1([a, b];\mathbb{R}^{2n})$ (resp. $C^{0}([a, b];\mathbb{R}^{2n})$) for any bounded interval $[a,b]\subset\mathbb{R}$.
Accordingly, in this paper we define $\bar{v}_k\to v_\mu$ in $C^1_{\rm loc}$ (resp. $C^{0}_{\rm loc}$)
 whenever the latter condition holds.
Consequently, in the sense of Definition~\ref{def:seqbifur} and by Proposition~\ref{prop:threeBifu}(i),
 we say that  $(\mu, v_\mu)$ is a \textsf{bifurcation point along sequences} of (\ref{e:PPer})
 with respect to the branch  $\{(\lambda, v_\lambda)\,|\,\lambda\in\Lambda\}$
if there exists a sequence $(\lambda_k)\subset\Lambda$ converging to $\mu\in\Lambda$
 and corresponding distinct solutions $\bar{v}_k\ne v_{\lambda_k}$  of (\ref{e:PPer})
 for $\lambda=\lambda_k\in\Lambda$,  $k=1,2,\cdots$, such that
 $\bar{v}_k\to v_\mu$ in $C^1_{\rm loc}$ (or equivalently in $C^{0}_{\rm loc}$).

 \begin{proposition}\label{prop:threeBifu}
 \begin{enumerate}
\item[\rm (i)] Under Assumption~\ref{ass:BasiAss1},  $(\mu, u_\mu)$ is a bifurcation point
 along sequences of the problem (\ref{e:Hboundary}) in $\Lambda\times C^0([0,\tau];\R^{2n})$
with respect to $\{(\lambda,u_\lambda)\,|\,\lambda\in\Lambda\}$ if and only if
it is also one in  $\Lambda\times C^1([0,\tau];\R^{2n})$ with respect to $\{(\lambda,u_\lambda)\,|\,\lambda\in\Lambda\}$.
\item[\rm (ii)]   Let $\mathbb{E}:=E^{\frac{1}{2}}_M$ be the Hilbert space as in Appendix~\ref{app:threeBifu},
 which is the fractional Sobolev space $H^{1/2}(S_\tau;\R^{2n})$ if $M=I_{2n}$. Under Assumption~\ref{ass:BasiAss1}, if $H$ also satisfies
\begin{equation}\label{e:growth}
|\nabla_z H(\lambda,t, z)|\le c_1+ c_2|z|^r\quad
\forall (\lambda, t, z)\in\Lambda\times [0,\tau]\times\mathbb{R}^{2n},
\end{equation}
where $c_1, c_2\ge 0$ and $r\in (1, 2]$ (resp. $r>1$) are constants if $M\ne I_{2n}$ (resp. $M=I_{2n}$),
then $(\mu, u_\mu)$ is a bifurcation point along sequences of the problem (\ref{e:Hboundary}) in
$\Lambda\times \mathbb{E}$ with respect to $\{(\lambda,u_\lambda)\,|\,\lambda\in\Lambda\}$ if and only if it
 is also one in  $\Lambda\times C^1([0,\tau];\R^{2n})$ with respect to
 $\{(\lambda,u_\lambda)\,|\,\lambda\in\Lambda\}$.
\end{enumerate}
\end{proposition}

This result will be proved in Appendix~\ref{app:threeBifu}.
It also shows that, under Assumption~\ref{ass:BasiAss} and the condition (\ref{e:growth}),
 the property of being a bifurcation point along sequences for problem (\ref{e:PPer}) with
 $M=I_{2n}$ is intrinsic (i.e., independent of the choice of the function spaces
 $H^{1/2}(S_\tau;\R^{2n})$, $W^{1,2}(S_\tau;\R^{2n})$ and $C^i(S_\tau;\R^{2n})$ for $i=0,1$).

Let $\gamma_\lambda:[0,\tau]\to {\rm Sp}(2n,\mathbb{R})$ be the fundamental matrix solution of
\begin{equation}\label{e:fundaSolution}
\dot{Z}(t)=J\nabla^2_zH(\lambda,t, u_\lambda(t))Z(t).
\end{equation}
The Maslov type index of $\gamma_\lambda$ relative to $M$ is  a pair of integers
$(i_{\tau,M}(\gamma_\lambda), \nu_{\tau,M}(\gamma_\lambda))$
(cf. Appendix~\ref{app:Index} for details),
 where $\nu_{\tau, M}(\gamma_\lambda)=\dim{\rm Ker}(\gamma_\lambda(\tau)-M)$ is the dimension
of solution space of (\ref{e:linear1}).
When $M=I_{2n}$,
$(i_{\tau,M}(\gamma_\lambda), \nu_{\tau,M}(\gamma_\lambda))$ is just the Maslov type index
$(i_{\tau}(\gamma_\lambda), \nu_{\tau}(\gamma_\lambda))$
defined in \cite{Long02},
in particular, $i_\tau(\gamma_\lambda)$ is the Conley-Zehnder index $i_{\rm CZ}(\gamma_\lambda)$ of $\gamma_\lambda$
 if $\nu_\tau(\gamma_\lambda)=0$ (\cite{CoZe}, \cite{LongZeh}, \cite{SaZe2}).
 Consider the Banach subspaces of $C^i([0,\tau];\R^{2n})$,
 $$
 C^i_{M}([0,\tau];\R^{2n})=\{u\in C^i([0,\tau];\R^{2n})\,|\,u(\tau)=Mu(0)\},\quad i\in\N\cup\{0\}.
 $$
 For bifurcations of the boundary value problem (\ref{e:Hboundary}) we have:

\begin{theorem}\label{th:bif-ness}
Under Assumption~\ref{ass:BasiAss1}, let  $\gamma_\lambda$ be as above.
\begin{enumerate}
\item[\rm (I)]{\rm (\textsf{Necessary condition}):}
$\nu_{\tau,M}(\gamma_{\mu})\ne 0$ if $(\mu, u_\mu)$ is a bifurcation point along sequences of
 (\ref{e:Hboundary}) in $\Lambda\times C^0_{M}([0,\tau];\R^{2n})$
 with respect to the branch $\{(\lambda,u_\lambda)\,|\,\lambda\in\Lambda\}$, i.e.,
 there exists a sequence $(\lambda_k)\subset\Lambda$ converging to $\mu$ and
solutions $u^k\ne u_{\lambda_k}$ of (\ref{e:Hboundary}) with $\lambda=\lambda_k$
such that $u^k\to u_\mu$ in $C^0([0,\tau], \mathbb{R}^{2n})$.

\item[\rm (II)]{\rm (\textsf{Sufficient condition}):}
Let $\Lambda$ be first countable.
Suppose for some $\mu\in\Lambda$ that
there exist two sequences in  $\Lambda$ converging to $\mu$, $(\lambda_k^-)$ and
$(\lambda_k^+)$,  such that
for each $k\in\mathbb{N}$,
$$
[i_{\tau,M}(\gamma_{\lambda_k^-}), i_{\tau,M}(\gamma_{\lambda_k^-})+\nu_{\tau,M}(\gamma_{\lambda_k^-})]\cap[i_{\tau,M}(\gamma_{\lambda_k^+}), i_{\tau,M}(\gamma_{\lambda_k^+})+\nu_{\tau,M}(\gamma_{\lambda_k^+})]=\emptyset
$$
 and either $\nu_{\tau,M}(\gamma_{\lambda_k^+})=0$ or $\nu_{\tau,M}(\gamma_{\lambda_k^-})=0$.
Let $\hat{\Lambda}:=\{\mu,\lambda^+_k, \lambda^-_k\,|\,k\in\mathbb{N}\}$.
  Then  $(\mu, u_\mu)$ is a bifurcation point  of (\ref{e:Hboundary})
 in $\hat\Lambda\times C^1_M([0,\tau];\R^{2n})$ with respect to the branch $\{(\lambda,u_\lambda)\,|\,\lambda\in\hat\Lambda\}$.
 In particular, $(\mu, u_\mu)$ is a bifurcation point along sequences of (\ref{e:Hboundary})
 in $\Lambda\times C^1_M([0,\tau];\R^{2n})$ with respect to the branch $\{(\lambda,u_\lambda)\,|\,\lambda\in\Lambda\}$.

\item[\rm (III)]{\rm (\textsf{Existence for bifurcations}):}
   Let $\Lambda$ be path-connected. If
  there exist two  points $\lambda^+, \lambda^-\in\Lambda$ such that
  $[i_{\tau,M}(\gamma_{\lambda^-}), i_{\tau,M}(\gamma_{\lambda^-})+\nu_{\tau,M}(\gamma_{\lambda^-})]\cap[i_{\tau,M}(\gamma_{\lambda^+}), i_{\tau,M}(\gamma_{\lambda^+})+\nu_{\tau,M}(\gamma_{\lambda^+})]=\emptyset$
 and either $\nu_{\tau,M}(\gamma_{\lambda^+})=0$ or $\nu_{\tau,M}(\gamma_{\lambda^-})=0$,
 then for any path $\alpha:[0,1]\to\Lambda$ connecting $\lambda^+$ to $\lambda^-$
 there exists a sequence $(t_k)\subset [0, 1]$ converging to some $\bar{t}$
   and solutions $u^k\ne u_{\alpha(t_k)}$ of (\ref{e:Hboundary}) with $\lambda=\alpha(t_k)$, $k=1,2,\cdots$,
  such that $\|u^k-u_{\alpha(t_k)}\|_{C^1}\to 0$ (and so $\|u^k-u_{\alpha(\bar{t})}\|_{C^1}\to 0$).
     Moreover,  $\alpha(\bar{t})$ is not equal to $\lambda^+$ (resp. $\lambda^-$) if $\nu_{\tau,M}(\gamma_{\lambda^+})=0$ (resp. $\nu_{\tau,M}(\gamma_{\lambda^-})=0$).
In particular, $(\alpha(\bar{t}), u_{\alpha(t_k)})$  is  a bifurcation point along sequences of (\ref{e:Hboundary})
  in $\Lambda\times C^1_M([0,\tau];\mathbb{R}^{2n})$.
  \end{enumerate}
\end{theorem}

From this theorem we immediately obtain:

\begin{theorem}\label{th:bif-per1}
Under Assumption~\ref{ass:BasiAss}, let  $\gamma_\lambda$ be
 the fundamental matrix solution of
 $$
 \dot{Z}(t)=J\nabla_z^2H(\lambda,t, v_\lambda(t))Z(t).
 $$
 \begin{enumerate}
\item[\rm (I)]{\rm (\textsf{Necessary condition}):}
$\nu_{\tau,M}(\gamma_{\mu})\ne 0$ if $(\mu, v_\mu)$ is a bifurcation point along sequences of (\ref{e:PPer})
 with respect to the branch $\{(\lambda, v_\lambda)\,|\,\lambda\in\Lambda\}$.

\item[\rm (II)]{\rm (\textsf{Sufficient condition}):}
Let $\Lambda$ be first countable.
Suppose that for some $\mu\in\Lambda$,
there exist two sequences $(\lambda_k^-)$ and
$(\lambda_k^+)$ in  $\Lambda$ converging to $\mu$  such that
for each $k\in\mathbb{N}$,
$$
[i_{\tau,M}(\gamma_{\lambda_k^-}), i_{\tau,M}(\gamma_{\lambda_k^-})+\nu_{\tau,M}(\gamma_{\lambda_k^-})]\cap[i_{\tau,M}(\gamma_{\lambda_k^+}), i_{\tau,M}(\gamma_{\lambda_k^+})+\nu_{\tau,M}(\gamma_{\lambda_k^+})]=\emptyset,
$$
 and either $\nu_{\tau,M}(\gamma_{\lambda_k^+})=0$ or $\nu_{\tau,M}(\gamma_{\lambda_k^-})=0$.
Let $\hat{\Lambda}:=\{\mu,\lambda^+_k, \lambda^-_k\,|\,k\in\mathbb{N}\}$.
  Then  $(\mu, v_\mu)$ is a bifurcation point  of (\ref{e:PPer})
  with respect to the branch $\{(\lambda, v_\lambda)\,|\,\lambda\in\hat\Lambda\}$
  (and so $\{(\lambda, v_\lambda)\,|\,\lambda\in\Lambda\}$).

\item[\rm (III)]{\rm (\textsf{Existence for bifurcations}):}
 Let $\Lambda$ be path-connected. If there exist two  points $\lambda^+, \lambda^-\in\Lambda$ such that
  $[i_{\tau,M}(\gamma_{\lambda^-}), i_{\tau,M}(\gamma_{\lambda^-})+\nu_{\tau,M}(\gamma_{\lambda^-})]\cap[i_{\tau,M}(\gamma_{\lambda^+}), i_{\tau,M}(\gamma_{\lambda^+})+\nu_{\tau,M}(\gamma_{\lambda^+})]=\emptyset$
 and either $\nu_{\tau,M}(\gamma_{\lambda^+})=0$ or $\nu_{\tau,M}(\gamma_{\lambda^-})=0$,
   then for any path $\alpha:[0,1]\to\Lambda$ connecting $\lambda^+$ to $\lambda^-$
   there exists a sequence $(t_k)\subset [0, 1]$ converging to some $\bar{t}$
     and solutions $v^k\ne v_{\alpha(t_k)}$ of (\ref{e:PPer}) with $\lambda=\alpha(t_k)$, $k=1,2,\cdots$,
  such that $(v^k)$ converges to  $v_{\alpha(\bar{t})}$ on any compact
  interval $I\subset\R$ in the $C^1$-topology as $k\to\infty$.
   Moreover,  $\alpha(\bar{t})$ is not equal to $\lambda^+$ (resp. $\lambda^-$) if $\nu_{\tau,M}(\gamma_{\lambda^+})=0$ (resp. $\nu_{\tau,M}(\gamma_{\lambda^-})=0$).
  \end{enumerate}
\end{theorem}

\begin{remark}\label{rm:bif-suffict1}
{\rm \begin{enumerate}
\item[(i)] Let us compare Theorem~\ref{th:bif-per1}(I) with previous related results.
If $M=I_{2n}$, $\tau=2\pi$ and  $\Lambda$ is an open subset in a real finite dimensional Banach space,
the conclusion of Theorem~\ref{th:bif-per1}(I) can also be derived from \cite[Proposition~26.1]{Am90} provided that
$\nabla_zH$ and mappings  $\Lambda\times [0, 2\pi]\ni (\lambda,t)\mapsto v_\lambda(t)\in\mathbb{R}^{2n}$
and  $\Lambda\times [0, 2\pi]\ni (\lambda,t)\mapsto \dot{v}_\lambda(t)\in\mathbb{R}^{2n}$ are also $C^1$.
In fact, under our conditions  we can apply \cite[Proposition~26.1]{Am90} to $X=E=\mathbb{R}^{2n}$, $k=1$ and $f(\lambda,t,\xi)=J\nabla_zH(\lambda,t,\xi+v_\lambda(t))-\dot{v}_\lambda(t)$
to conclude $\nu_{2\pi,M}(\gamma_{\mu})\ne 0$ because the assumption guarantees that
 $(\mu,0)$ is also a bifurcation point of $\dot{x}=f(\lambda,t, x)$ in
the sense of \cite[page 369]{Am90}.

Let $H\in C^2([a,b]\times\mathbb{R}\times\mathbb{R}^{2n})$ be $2\pi$-periodic in the second variable
and satisfy $H(\lambda,t,0)\equiv 0$.
(Therefore it satisfies Assumption~\ref{ass:BasiAss} with $\Lambda=[a,b]$ and $M=I_{2n}$.)  Suppose that
\begin{description}
\item[(H1)] There exist constants $r>1$ and $c_i>0$, $i=1,2$, such that for all $(\lambda, t, z)\in\Lambda\times \mathbb{R}\times\mathbb{R}^{2n}$,
$$
|\nabla_z H(\lambda,t, z)|\le c_1+ c_2|z|^r\quad\hbox{and}\quad
|\nabla^2_z H(\lambda,t, z)|\le c_1+ c_2|z|^r.
$$
\item[(H2)] There is a path $\lambda\mapsto A_\lambda$ of  time-independent symmetric $2n\times 2n$
real matrices such that
$$
H(\lambda,t,z)=(A_\lambda z,z)_{\mathbb{R}^{2n}}/2+ R(\lambda,t, z)
$$
where $\nabla_z R(\lambda,t, z)=o(|z|)$ as $|z|\to 0$. (Without using $H(\lambda,t,0)\equiv 0$, this implies that
$\nabla_z H(\lambda,t, 0)=0$ and $\nabla^2_z H(\lambda,t, 0)=A_\lambda$ for all $(\lambda,t)$. In particular, $\lambda\mapsto A_\lambda$
is continuous because $H$ is $C^2$.)
\end{description}
Let $\gamma_\lambda$ be the fundamental matrix solution of $\dot{Z}(t)=JA_\lambda Z(t)$.
It was claimed in \cite[page~22]{FiPeRe00} that $\nu_{2\pi}(\gamma_{\mu})=\dim{\rm Ker}(\gamma_{\mu}(2\pi)-I_{2n})\ne 0$ is a necessary condition
for $(\mu, 0)$ to be a bifurcation point of (\ref{e:PPer}) with respect to the trivial branch in
$[a,b]\times H^{1/2}(S_{2\pi};\R^{2n})$ [or equivalently in $[a,b]\times C^1(S_{2\pi};\R^{2n})$
by Proposition~\ref{prop:threeBifu}(ii)]. Clearly, this result is contained in Theorem~\ref{th:bif-per1}(I).

\item[(ii)]  Theorem~\ref{th:bif-per1}(III)  with $M=I_{2n}$ and $\tau=2\pi$
generalizes \cite[Theorem~1.1]{FiPeRe00} and \cite[Theorem~2.2]{FiPeRe00}.
Let $H\in C^\infty([a,b]\times\mathbb{R}\times\mathbb{R}^{2n})$ be $2\pi$-periodic in the second variable
and satisfy the above growth condition (H1).
Assume that $\dot{v}(t)=J\nabla_z H(\lambda,t, v(t))$ possesses a known family $\{v_\lambda\}$ of $2\pi$-periodic solutions
smoothly parametrized by $\lambda$ in $[a,b]$. Denote by $\gamma_\lambda$  the fundamental matrix solution of $\dot{Z}(t)=J\nabla^2_zH(\lambda,t, v_\lambda(t)) Z(t)$.
For each $\lambda\in [a, b]$, we define
$L_\lambda\in\mathscr{L}_s(H^{1/2}(S_{2\pi};\R^{2n}))$  by
\begin{eqnarray*}
 (L_\lambda v,w)_{H^{1/2}}=\int^{2\pi}_0(J\dot{v}(t)+\nabla^2_zH(\lambda,t, v_\lambda(t))v(t), w(t))_{\mathbb{R}^{2n}} dt
 \end{eqnarray*}
 as in the final line of \cite[page 24]{FiPeRe00},
and $\mathcal{A}_\lambda:H^{1}(S_{2\pi};\R^{2n})\to L^2(S_{2\pi};\R^{2n})$
by
$$
(\mathcal{A}_\lambda v)(t)=J\dot{v}(t)+\nabla^2_zH(\lambda,t, v_\lambda(t))v(t).
$$
Let ${\rm sf}(\mathcal{A})$ and ${\rm sf}(L)$ be the spectral flows of the families $\{\mathcal{A}_\lambda\,|\, a\le\lambda\le b\}$
and $\{L_\lambda\,|\, a\le\lambda\le b\}$, respectively.
Consider a symplectic path $\Gamma_{2\pi}:[a, b]\to {\rm Sp}(2n,\mathbb{R})$ given by $\Gamma_{2\pi}(\lambda)=\gamma_\lambda(2\pi)$.
Let $\mu_{\rm CZ}(\Gamma_{2\pi})$ be defined by (\ref{e:RobbinSaCZ}).
We assert the following formula:
\begin{equation}\label{e:RobbinSaCZ*}
\mu_{\rm CZ}(\Gamma_{2\pi})=\mu_{\rm CZ}(\gamma_{b})-\mu_{\rm CZ}(\gamma_{a}).
\end{equation}
In fact, consider the continuous map
$$
\Lambda:[a, b]\times [0, 2\pi]\to {\rm Sp}(2n,\mathbb{R}),\; (\lambda, t)\mapsto \gamma_\lambda(t),
$$
and paths
\begin{eqnarray*}
&&\alpha_1: [0, 1]\to [a, b]\times [0, 2\pi],\;t\mapsto (a,2\pi t),\\
&&\beta_1: [0, 1]\to [a, b]\times [0, 2\pi],\;t\mapsto (a+(b-a)t, 2\pi),\\
&&\alpha_2: [0, 1]\to [a, b]\times [0, 2\pi],\;t\mapsto (a+(b-a)t, 0),\\
&&\beta_2: [0, 1]\to [a, b]\times [0, 2\pi],\;t\mapsto (2\pi t, b).
\end{eqnarray*}

Then $\alpha_1\ast\beta_1$ and $\alpha_2\ast\beta_2$ are two paths with fixed end points
in the convex subset $[a, b]\times [0, 2\pi]\subset\mathbb{R}^2$. Therefore,
we have a homotopy $\{h_s\}_{0\le s\le 1}$ from the paths $\alpha_1\ast\beta_1$ to
$\alpha_2\ast\beta_2$ with fixed end points in $[a, b]\times [0, 2\pi]$.
It follows that $\{\Lambda\circ h_s\}_{0\le s\le 1}$ is a homotopy
from the paths $\Lambda\circ(\alpha_1\ast\beta_1)$ to
$\Lambda\circ(\alpha_2\ast\beta_2)$ with fixed end points in
${\rm Sp}(2n,\mathbb{R})$. Let
$$
\mu_{\rm CZ}(\Lambda\circ(\alpha_1\ast\beta_1))\quad\hbox{and}\quad
\mu_{\rm CZ}(\Lambda\circ(\alpha_2\ast\beta_2))
$$
 be the Conley-Zehnder index  of $\Lambda\circ(\alpha_1\ast\beta_1)$ and
 $\Lambda\circ(\alpha_2\ast\beta_2)$  defined  by
(\ref{e:RobbinSaCZ}) (or \cite[(2.7)]{FiPeRe00}), respectively.
By formula (\ref{e:RobbinSa5}) we obtain
$$
\mu_{\rm CZ}(\Lambda\circ(\alpha_1\ast\beta_1))=
\mu_{\rm CZ}(\Lambda\circ(\alpha_2\ast\beta_2)).
$$
Clearly, $\Lambda\circ(\alpha_i\ast\beta_i)=
(\Lambda\circ\alpha_i)\ast(\Lambda\circ\beta_i)$ for $i=1,2$.
Since  $\mu_{\rm CZ}$ is additive  under concatenation of paths, it follows that
\begin{equation}\label{e:RobbinSaCZ*+}
\mu_{\rm CZ}(\Lambda\circ\alpha_1)+\mu_{\rm CZ}(\Lambda\circ\beta_1))
=\mu_{\rm CZ}(\Lambda\circ\alpha_2)+
\mu_{\rm CZ}(\Lambda\circ\beta_2)
\end{equation}
Note that
 $\mu_{\rm CZ}(\Lambda\circ\alpha_2)=0$ by the vanishing property of the Conley-Zehnder index
 (since $\Lambda\circ\alpha_2$ is a constant path).
Moreover, $\mu_{\rm CZ}(\Lambda\circ\alpha_1)=\mu_{\rm CZ}(\gamma_{a})$,
$\mu_{\rm CZ}(\Lambda\circ\beta_1))=\mu_{\rm CZ}(\Gamma_{2\pi})$
and $\mu_{\rm CZ}(\Lambda\circ\beta_2)=\mu_{\rm CZ}(\gamma_{b})$
because  paths $\Lambda\circ\alpha_1$, $\Lambda\circ\beta_1$ and
$\Lambda\circ\beta_2$ are  reparameterizations of $\gamma_a$, $\Gamma_{2\pi}$ and $\gamma_b$, respectively.
 Substituting these equalities into (\ref{e:RobbinSaCZ*+}) yields
  the desired equality (\ref{e:RobbinSaCZ*}).
By  \cite{SaZe2} (see also  \cite[Theorem~11.1]{CLM}) and \cite[Proposition~2.1]{FiPeRe00} there holds
$$
{\rm sf}(\mathcal{A})=\mu_{\rm CZ}(\Gamma_{2\pi})
={\rm sf}(L)
$$
(see also \cite[Theorem~4.1]{Wat15}). Combining this formula with (\ref{e:RobbinSaCZ*}) and
(\ref{e:RobbinSa3}) we directly obtain the following claim.

\noindent{\bf Claim}. {\it If $\nu_{2\pi}(\gamma_{a})=\nu_{2\pi}(\gamma_{b})=0$, then
$$
\mu_{\rm CZ}(\gamma_{\star})=i_{\rm CZ}(\gamma_{\star})=i_{2\pi}(\gamma_{\star}),\quad\star=a,b,
$$
and consequently,
\begin{eqnarray}\label{e:RobbinSaCZ**}
&&\mu_{\rm CZ}(\Gamma_{2\pi})=i_{2\pi}(\gamma_{b})-i_{2\pi}(\gamma_{a}),\\
&&i_{2\pi}(\gamma_{b})\ne i_{2\pi}(\gamma_{a}) \Longleftrightarrow\mu_{\rm CZ}(\Gamma_{2\pi})\ne 0 \Longleftrightarrow {\rm sf}(L)\ne 0.\label{e:RobbinSaCZ***}
\end{eqnarray}}

In view of this claim, the results of \cite[Theorem~1.1]{FiPeRe00} and \cite[Theorem~2.2]{FiPeRe00} can be restated as follows:

\noindent{\bf Theorem}(\cite[Theorem~1.1]{FiPeRe00}). {\it Let $H\in C^2([a,b]\times\mathbb{R}\times\mathbb{R}^{2n})$ be $2\pi$-periodic in the second variable
and satisfy (H1)-(H2) in (i) above and
\begin{description}
\item[(H3)] The matrices $JA_a$ and $JA_b$ have no eigenvalues that are integral multiples
of $\sqrt{-1}$, and therefore $\nu_{2\pi}(\gamma_{a})=\nu_{2\pi}(\gamma_{b})=0$.
\end{description}
If $i_{2\pi}(\gamma_{b})\ne i_{2\pi}(\gamma_{a})$, then there exists  $\mu\in (a,b)$ such that
$(\mu, 0)$ is a bifurcation point of (\ref{e:PPer}) with $(\Lambda, M, \tau)=([a,b], I_{2n}, 2\pi)$
 in $[a,b]\times H^{1/2}(S_{2\pi};\R^{2n})$ with respect to the trivial branch.}

\noindent{\bf Theorem}(\cite[Theorem~2.2]{FiPeRe00}). {\it Let $H\in C^\infty([a,b]\times\mathbb{R}\times\mathbb{R}^{2n})$
 be $2\pi$-periodic in the second variable and satisfy (H1) in (i), and let $\{v_\lambda\,|\,\lambda\in [a,b]\}$
 be a family  of $2\pi$-periodic solutions of $\dot{v}(t)=J\nabla_z H(\lambda,t, v(t))$
smoothly parameterized by $\lambda$ in $[a,b]$. If $\nu_{2\pi}(\gamma_{a})=\nu_{2\pi}(\gamma_{b})=0$ and
$i_{2\pi}(\gamma_{b})\ne i_{2\pi}(\gamma_{a})$, then there exists $\mu\in (a,b)$ such that
$(\mu, v_\mu)$ is a bifurcation point of (\ref{e:PPer}) with $(\Lambda, M, \tau)=([a,b], I_{2n}, 2\pi)$
 in $[a,b]\times H^{1/2}(S_{2\pi};\R^{2n})$ with respect to the branch $\{(\lambda, v_\lambda)\,|\,\lambda\in\Lambda\}$.}

When $v_\lambda\equiv 0$ in this result, the condition ``$H\in C^\infty([a,b]\times\mathbb{R}\times\mathbb{R}^{2n})$''
was weakened as
``\textsf{$H\in C^0([a,b]\times\mathbb{R}\times\mathbb{R}^{2n})$, each $H(\lambda,\cdot,\cdot)$, $\lambda\in [a,b]$, is $C^2$
and all possible partial derivatives of $H$ depend continuously on the parameter $\lambda$}''
in lines 1-4 of \cite[page 743]{Wat15}.

Clearly, the assumptions on $H$ in the above two theorems are stronger than those in Assumption~\ref{ass:BasiAss}
with $(\Lambda, M, \tau)=([a,b], I_{2n}, 2\pi)$.
 \end{enumerate} }
\end{remark}

To the best of the author's knowledge, the following alternative bifurcation result
(and so its consequence, Theorem~\ref{th:bif-per2}) is completely new,
 and no similar results have been found.

\begin{theorem}[\textsf{Alternative bifurcations of Rabinowitz's type and of Fadell-Rabinowitz's type}]\label{th:bif-suffict}
Let Assumption~\ref{ass:BasiAss1} with $\Lambda$ being a real interval be  satisfied,
and let $\gamma_\lambda$ be as in Theorem~\ref{th:bif-ness} for each $\lambda\in\Lambda$.
Suppose for some interior point $\mu$ of $\Lambda$ that
 $\dim{\rm Ker}(\gamma_\mu(\tau)-M)\ne 0$ and  $\dim{\rm Ker}(\gamma_\lambda(\tau)-M)=0$
 for each $\lambda\in\Lambda\setminus\{\mu\}$ near $\mu$, and that
  $i_{\tau,M}(\gamma_\lambda)$ takes, respectively, values $i_{\tau,M}(\gamma_\mu)$ and
  $i_{\tau,M}(\gamma_\mu)+ \nu_{\tau,M}(\gamma_\mu)$
 as $\lambda\in\Lambda$ varies in
 two deleted half neighborhoods  of $\mu$.
 Then at least one of the following assertions holds:
\begin{enumerate}
\item[\rm (i)]
The problem (\ref{e:Hboundary}) with $\lambda=\mu$ has a sequence of solutions, $u_\mu^k\ne u_\mu$, $k=1,2,\cdots$,
which converges to $u_\mu$ in $C^1_M([0,\tau];\R^{2n})$.

\item[\rm (ii)]  For every $\lambda\in\Lambda\setminus\{\mu\}$ near $\mu$ there is a  solution $\bar{u}_\lambda\ne u_\lambda$ of
(\ref{e:Hboundary}) with parameter value $\lambda$, such that
$\bar{u}_\lambda-u_\lambda$ converges to zero in $C^1_M([0,\tau];\R^{2n})$ as $\lambda\to \mu$.

\item[\rm (iii)] For a given neighborhood $\mathcal{W}$ of $u_\mu$ in $C^1_M([0,\tau];\R^{2n})$
there is a one-sided  neighborhood $\Lambda^0$ of $\mu$ such that
for any $\lambda\in\Lambda^0\setminus\{\mu\}$, (\ref{e:Hboundary}) with parameter value $\lambda$
has at least two distinct solutions $u_\lambda^1\ne u_\lambda$ and $u_\lambda^2\ne u_\lambda$ in $\mathcal{W}$,
which can also be required to satisfy
\begin{eqnarray*}
\int^{\tau}_0\left[\frac{1}{2}(J\dot{u}_\lambda^1(t),u^1_\lambda(t))_{\mathbb{R}^{2n}}+ H(\lambda, t, {u}_\lambda^1(t))\right]dt
\ne \int^{\tau}_0\left[\frac{1}{2}(J\dot{u}^2_\lambda(t),u_\lambda^2(t))_{\mathbb{R}^{2n}}+ H(\lambda, t,
{u}_\lambda^2(t))\right]dt
\end{eqnarray*}
provided that $\nu_{\tau,M}(\gamma_\mu)>1$ and (\ref{e:Hboundary}) with parameter value $\lambda$
has only finitely many  solutions in $\mathcal{W}$.
\end{enumerate}
Moreover, if $u_\lambda=0\;\forall\lambda$, and all $H(\lambda,t,\cdot)$ are even, then
at least one of  (i) and (iv) occurs:
\begin{enumerate}
\item[\rm (iv)] There exist left and right  neighborhoods $\Lambda^-$ and $\Lambda^+$ of $\mu$ in $\Lambda$
and integers $n^+, n^-\ge 0$, such that $n^++n^-\ge \nu_{\tau,M}(\gamma_\mu)$,
and for $\lambda\in\Lambda^-\setminus\{\mu\}$ (resp. $\lambda\in\Lambda^+\setminus\{\mu\}$),
(\ref{e:Hboundary}) with parameter value $\lambda$  has at least $n^-$ (resp. $n^+$) distinct pairs of nontrivial solutions,
$\{u_\lambda^i, -u_\lambda^i\}$, $i=1,\cdots,n^-$ (resp. $n^+$),
which  converge to zero in $C^1_M([0,\tau];\R^{2n})$  as $\lambda\to\mu$.
\end{enumerate}
\end{theorem}

As above this directly leads to the following bifurcation result for the system (\ref{e:PPer}).

\begin{theorem}[\textsf{Alternative bifurcations of Rabinowitz's type and of Fadell-Rabinowitz's type}]\label{th:bif-per2}
Let Assumption~\ref{ass:BasiAss} with $\Lambda$ being a real interval be  satisfied,
and let $\gamma_\lambda$ be as in Theorem~\ref{th:bif-per1} for each $\lambda\in\Lambda$.
Suppose for some interior point $\mu$ of $\Lambda$ that
 $\dim{\rm Ker}(\gamma_\mu(\tau)-M)\ne 0$ and  $\dim{\rm Ker}(\gamma_\lambda(\tau)-M)=0$
 for each $\lambda\in\Lambda\setminus\{\mu\}$ near $\mu$, and that
  $i_{\tau,M}(\gamma_\lambda)$ takes, respectively, values $i_{\tau,M}(\gamma_\mu)$ and $i_{\tau,M}(\gamma_\mu)+ \nu_{\tau,M}(\gamma_\mu)$
 as $\lambda\in\Lambda$ varies in
 two deleted half neighborhoods  of $\mu$.
 Then at least one of the following assertions holds:
\begin{enumerate}
\item[\rm (i)]
The problem  (\ref{e:PPer}) with $\lambda=\mu$ has a sequence of solutions, $v_\mu^k\ne v_\mu$, $k=1,2,\cdots$,
such that  $v_\mu^k|_{[0,\tau]}\to v_\mu|_{[0,\tau]}$ in $C^1([0,\tau];\R^{2n})$.

\item[\rm (ii)]  For every $\lambda\in\Lambda\setminus\{\mu\}$ near $\mu$ there is a  solution $\bar{v}_\lambda\ne v_\lambda$ of
(\ref{e:PPer}) with parameter value $\lambda$
such that  $\bar{v}_\lambda|_{[0,\tau]}-v_\lambda|_{[0,\tau]}\to 0$ in $C^1([0,\tau];\R^{2n})$  as $\lambda\to \mu$.

\item[\rm (iii)] For a given neighborhood $\mathcal{W}$ of $v_\mu|_{[0,\tau]}$ in $C^1_M([0,\tau];\R^{2n})$
there is an one-sided  neighborhood $\Lambda^0$ of $\mu$ such that
for any $\lambda\in\Lambda^0\setminus\{\mu\}$,  (\ref{e:PPer}) with parameter value $\lambda$
has at least two distinct solutions $v_\lambda^1\ne v_\lambda$ and $v_\lambda^2\ne v_\lambda$ such that
$v_\lambda^1|_{[0,\tau]}$ and $v_\lambda^2|_{[0,\tau]}$ are in $\mathcal{W}$,
which can also be required to satisfy
\begin{eqnarray*}
\int^{\tau}_0\left[\frac{1}{2}(J\dot{v}_\lambda^1(t),v^1_\lambda(t))_{\mathbb{R}^{2n}}+ H(\lambda, t, {v}_\lambda^1(t))\right]dt
\ne \int^{\tau}_0\left[\frac{1}{2}(J\dot{v}^2_\lambda(t), v_\lambda^2(t))_{\mathbb{R}^{2n}}+ H(\lambda, t,
{v}_\lambda^2(t))\right]dt
\end{eqnarray*}
provided that $\nu_{\tau,M}(\gamma_\mu)>1$ and (\ref{e:PPer}) with parameter value $\lambda$
has only finitely many  solutions whose restrictions to $[0,\tau]$ belong to $\mathcal{W}$.
\end{enumerate}
Moreover, if $v_\lambda=0\;\forall\lambda$, and all $H(\lambda,t,\cdot)$ are even, then
at least one of  (i) and (iv) occurs:
\begin{enumerate}
\item[\rm (iv)] There exist left and right  neighborhoods $\Lambda^-$ and $\Lambda^+$ of $\mu$ in $\Lambda$
and integers $n^+, n^-\ge 0$, such that $n^++n^-\ge \nu_{\tau,M}(\gamma_\mu)$,
and for $\lambda\in\Lambda^-\setminus\{\mu\}$ (resp. $\lambda\in\Lambda^+\setminus\{\mu\}$),
(\ref{e:PPer}) with parameter value $\lambda$  has at least $n^-$ (resp. $n^+$) distinct pairs of nontrivial solutions,
$\{v_\lambda^i, -v_\lambda^i\}$, $i=1,\cdots,n^-$ (resp. $n^+$),
such that their restrictions to $[0,\tau]$  converge to zero in $C^1([0,\tau];\R^{2n})$  as $\lambda\to\mu$.
\end{enumerate}
\end{theorem}

The first part is an alternative bifurcation result of Rabinowitz's type,
and the second part is of alternative bifurcations of Fadell-Rabinowitz's type.

Theorem~\ref{th:bif-suffict} has the following consequence:

\begin{corollary}\label{cor:necess-suffi}
 For a real $\tau>0$, and a symplectic matrix $M\in{\rm Sp}(2n,\mathbb{R})$,
let $H_0, \hat{H}: [0,\tau]\times{\R}^{2n}\to\R$
two continuous functions satisfying:
\begin{itemize}
\item[\rm (i)] $H_0(\tau, Mz)=H_0(0, z)$ and $\hat{H}(\tau, Mz)=\hat{H}(0, z)$ for all $z\in {\R}^{2n}$.
\item[\rm (ii)] For each $t\in [0,\tau]$, $H_0(t,\cdot)$ and $\hat{H}(t,\cdot)$ are
  $C^2$ on ${\R}^{2n}$.
\item[\rm (iii)] Both the Euclidean gradients, $\nabla_zH_0(t, z)$ and $\nabla_z\hat{H}(t, z)$,
 and the Hessian matrices, $\nabla^2_zH_0(t, z)$ and $\nabla^2_z\hat{H}(t, z)$,
  depend continuously on $(t, z)\in [0,\tau]\times\mathbb{R}^{2n}$.
 \end{itemize}
Let  $\bar{u}:[0,\tau]\to\R^{2n}$ be a solution of
$$
\dot{u}(t)=J\nabla_z H_0(t,u(t))\quad\text{and}\quad u(\tau)=Mu(0).
$$
Suppose that $\nabla_z\hat{H}(t,\bar{u}(t))=0$ for all $t\in [0,\tau]$, and that
 the Hessian $\nabla^2_z\hat{H}(t,\bar{u}(t))$ is either positive definite for all
 $t$ or negative definite for all $t$.
Define
$$
H(\lambda,t, z)=H_0(t,z)+\lambda\hat{H}(t,z),\quad\forall (\lambda,t,z)\in\mathbb{R}\times [0,\tau]\times\mathbb{R}^{2n}.
$$
Then for the fundamental matrix solution $\gamma_\lambda$ of
$\dot{Z}(t)=J\nabla^2_zH(\lambda,t, \bar{u}(t))Z(t)$,
$$
\Sigma:=\{\lambda\in\R\,|\, \nu_{\tau,M}(\gamma_\lambda)>0\}
$$
is a discrete set in $\R$,
and $(\mu, \bar{u})\in\R\times W^{1,2}_M([0,\tau];\R^{2n})$
is a bifurcation point  for (\ref{e:Hboundary}) with this $H$
if and only if $\nu_{\tau,M}(\gamma_\mu)>0$.
Moreover, for each $\mu\in\Sigma$
and a small enough $\rho>0$ it holds that
\begin{eqnarray}\label{e:case1}
i_{\tau,M}(\gamma_{\lambda})=\left\{
\begin{array}{ll}
 i_{\tau,M}(\gamma_{\mu})\;&\forall\lambda\in [\mu-\rho,\mu),\\ 
 i_{\tau,M}(\gamma_{\mu})+ \nu_{\tau,M}(\gamma_{\mu}) \;&\forall\lambda\in (\mu, \mu+\rho]
\end{array}\right.
\end{eqnarray}
if $\nabla_z^2\hat{H}(t, \bar{u}(t))>0\;\forall t\in [0,\tau]$, and
\begin{eqnarray}\label{e:case2}
i_{\tau,M}(\gamma_{\lambda})=\left\{
\begin{array}{ll}
 i_{\tau,M}(\gamma_{\mu})+ \nu_{\tau,M}(\gamma_{\mu}) \;&\forall\lambda\in [\mu-\rho,\mu),\\
 i_{\tau,M}(\gamma_{\mu})\;&\forall\lambda\in (\mu, \mu+\rho]
\end{array}\right.
\end{eqnarray}
if $\nabla_z^2\hat{H}(t, \bar{u}(t))<0\;\forall t\in [0,\tau]$.
Consequently, for each $\mu\in\Sigma$,
the conclusions in the first part of Theorem~\ref{th:bif-suffict}
holds, and the second one of Theorem~\ref{th:bif-suffict} is also true provided that
$\bar{u}=0$ and all $H_0(t,\cdot), \hat{H}(t,\cdot)$ are even.
\end{corollary}

If $\bar{u}$ is constant and $H_0=0$, Proposition~\ref{prop:Index}
gives stronger results than (\ref{e:case1}) and (\ref{e:case2}).

As a consequence of Theorem~\ref{th:bif-suffict}
we can also obtain a  result corresponding to Theorem~5.6 of \cite{Lu9}.

\begin{corollary}\label{cor:bif-deform}
Let $M\in{\rm Sp}(2n,\mathbb{R})$ be a symplectic matrix, and
let $H: {\R}^{2n}\to\R$ be a $C^2$ function satisfying $H(Mz)=H(z)$
for all $z\in\mathbb{R}^{2n}$. Suppose that
 $\bar{u}\in{\rm Ker}(M-I_{2n})$ satisfies $\nabla H(\bar{u})=0$. Let
$B=\nabla^2{H}(\bar{u})$. Then the fundamental matrix solution of $\dot{Z}(t)=JBZ(t)$
is $\Upsilon_B(t)=\exp(tJB)$.
Let $\xi_{2n}(t)=I_{2n}$ for $t\in \mathbb{R}$.
For a given  real $\tau>0$, we have:
\begin{description}
\item[(I)]
Suppose  that $B>0$ and $i_{\tau,M}(\Upsilon_B)\ne i_{\tau,M}(\xi_{2n}|_{[0,\tau]})+ \dim{\rm Ker}(I_{2n}-M)$.
Then there exists at least one and at most  finitely many numbers
in $(0, \tau)$, $\tau_1<\cdots<\tau_l$, such that
$$
\nu_{\tau_k,M}(\Upsilon_B):=\dim{\rm Ker}(\Upsilon_B(\tau_k)-M)\ne 0,\quad k=1,\cdots, l,
$$
and  for each $\tau_k$, at least  one of the following assertions holds:
\begin{enumerate}
\item[\rm (I-1)]
The boundary value problem
\begin{equation*}
\hbox{$\dot{u}(t)=J\nabla H(u(t))\;\forall t\in [0,\tau_k]$ \quad and\quad $u(\tau_k)=Mu(0)$}
\end{equation*}
 has a sequence of solutions $u_m$ converging to $\bar{u}$ in $C^1_M([0, \tau_k];\R^{2n})$, $u_m\ne \bar{u}$, $m=1,2,\cdots$.

\item[\rm (I-2)]  For every $s\ne\tau_k$ near $\tau_k$,
the boundary value problem
\begin{equation}\label{e:Deform-Equ}
\hbox{$\dot{u}(t)=J\nabla H(u(t))\;\forall t\in [0,s]$ \quad and\quad $u(s)=Mu(0)$}
\end{equation}
has a  solution $u_s\ne \bar{u}$, which
converges to $\bar{u}$ in $C^1_M([0,s];\R^{2n})$ as $s\to \tau_k$.

\item[\rm (I-3)] There is a one-sided  neighborhood $\Delta^0$ of $\tau_k$ in $[0, \tau]$ such that
for any $s\in\Delta^0\setminus\{\tau_k\}$, (\ref{e:Deform-Equ})
has at least two distinct solutions $u_s^1\ne \bar{u}$ and $u_s^2\ne \bar{u}$ near $\bar{u}$,
 which converges to $\bar{u}$ in $C^1_M([0,s];\R^{2n})$ as $s\to \tau_k$; besides,
$u_s^1$ and $u_s^2$ can also be required to satisfy
{\small
$$
\int^{s}_0\left[\frac{1}{2}(J\dot{u}_s^1(t),u^1_s(t))_{\mathbb{R}^{2n}}+ H({u}_s^1(t))\right]dt
\ne \int^{s}_0\left[\frac{1}{2}(J\dot{u}^2_s(t),u_s^2(t))_{\mathbb{R}^{2n}}+
 H({u}_s^2(t))\right]dt
$$}
provided that $\nu_{\tau_k,M}(\Upsilon_B)>1$ and (\ref{e:Deform-Equ})
has only finitely many  solutions near $\bar{u}$.
\end{enumerate}
Furthermore, if $\bar{u}=0$ and  $H(\cdot)$ is even, then at least one of (I-1) and (I-4)
holds:
\begin{enumerate}
\item[\rm (I-4)] There exist left and right  neighborhoods $\Delta^-$ and $\Delta^+$ of $\tau_k$ in $[0, \tau]$
and integers $n^+, n^-\ge 0$, such that $n^++n^-\ge \nu_{\tau_k,M}(\Upsilon_B)$,
and for $s\in\Delta^-\setminus\{\tau_k\}$ (resp. $s\in\Delta^+\setminus\{\tau_k\}$),
(\ref{e:Deform-Equ})  has at least $n^-$ (resp. $n^+$) distinct pairs of nontrivial solutions,
$\{u_s^i, -u_s^i\}$, $i=1,\cdots,n^-$ (resp. $n^+$),
which  converge to zero in $C^1_M([0, s];\R^{2n})$  as $s\to\tau_k$.
\end{enumerate}
\item[(II)] Suppose that $B<0$ and
$i_{\tau,M}(\Upsilon_{\hat{B}})\ne i_{\tau,M}(\xi_{2n}|_{[0,\tau]})+ \dim{\rm Ker}(I_{2n}-M)$, where
$\Upsilon_{\hat{B}}(t)=\exp(-tJB)$.
Then there exists  at least one and at most  finitely many numbers
in $(0, \tau)$, $\delta_1<\cdots<\delta_l$, such that
$$
\nu_{\delta_k,M}(\Upsilon_{\hat{B}}):=\dim{\rm Ker}(\Upsilon_{\hat{B}}(\delta_k)-M)\ne 0,\quad k=1,\cdots, l,
$$
and  for each $\delta_k$, at least  one of the following assertions holds:
\begin{enumerate}
\item[\rm (II-1)]
The boundary value problem
\begin{equation*}
\hbox{$\dot{w}(t)=\frac{\tau}{\delta_k}J\nabla_z {H}(w(t))\;\forall t\in [\tau-\delta_k^2/\tau,\tau]$ \quad and\quad $Mw(\tau)=w(\tau-\delta_k^2/\tau)$}
\end{equation*}
 has a sequence of solutions, $w_m\ne \bar{u}$, $m=1,2,\cdots$,
which converges to $\bar{u}$ in $C^1_M([\tau-\delta_k^2/\tau,\tau];\R^{2n})$.

\item[\rm (II-2)]  For every $\rho\ne\delta_k$ near $\delta_k$,
the boundary value problem
\begin{equation}\label{e:Deform-EquNew}
\hbox{$\dot{w}(t)=\frac{\tau}{\rho}J\nabla_z {H}(w(t))\;\forall t\in [\tau-\rho^2/\tau,\tau]$ \quad and\quad $Mw(\tau)=w(\tau-\rho^2/\tau)$}
\end{equation}
has a  solution $w_\rho\ne \bar{u}$, which
converges to $\bar{u}$ in $C^1_M([\tau-\rho^2/\tau,\tau];\R^{2n})$ as $\rho\to \delta_k$.

\item[\rm (II-3)] There is a one-sided  neighborhood $\Delta^0$ of $\delta_k$ in $[0, \tau]$ such that
for any $\rho\in\Delta^0\setminus\{\tau_k\}$, (\ref{e:Deform-EquNew})
has at least two distinct solutions $w_\rho^1\ne \bar{u}$ and $w_\rho^2\ne \bar{u}$ near $\bar{u}$,
 which converges to $\bar{u}$ in $C^1_M([\tau-\rho^2/\tau,\tau];\R^{2n})$ as $\rho\to \tau_k$; besides,
$w_\rho^1$ and $w_\rho^2$ can also be required to satisfy
{\small
\begin{eqnarray*}
&&\int^{\tau}_{\tau-\rho^2/\tau}\left[\frac{1}{2}(J\dot{w}_\rho^1(t),w^1_\rho(t))_{\mathbb{R}^{2n}}+ \frac{\tau}{\rho}H({w}_\rho^1(t))\right]dt\\
&&\ne \int^{\tau}_{\tau-\rho^2/\tau}\left[\frac{1}{2}(J\dot{w}^2_\rho(t),w_\rho^2(t))_{\mathbb{R}^{2n}}+ \frac{\tau}{\rho}H({u}_s^2(t))\right]dt
\end{eqnarray*}}
provided that $\nu_{\delta_k,M}(\Upsilon_{\hat{B}})>1$ and (\ref{e:Deform-EquNew})
has only finitely many  solutions near $\bar{u}$.
\end{enumerate}
Furthermore, if $\bar{u}=0$ and  $H(\cdot)$ is even, then at least one
 (II-1) and (II-4) holds:
\begin{enumerate}
\item[\rm (II-4)] There exist left and right  neighborhoods $\Delta^-$ and $\Delta^+$ of $\delta_k$ in $[0, \tau]$
and integers $n^+, n^-\ge 0$, such that $n^++n^-\ge \nu_{\delta_k,M}(\Upsilon_{\hat{B}})$,
and for $\rho\in\Delta^-\setminus\{\delta_k\}$ (resp. $\rho\in\Delta^+\setminus\{\delta_k\}$),
(\ref{e:Deform-EquNew})  has at least $n^-$ (resp. $n^+$) distinct pairs of nontrivial solutions,
$\{w_\rho^i, -u_\rho^i\}$, $i=1,\cdots,n^-$ (resp. $n^+$),
which  converge to zero in $C^1_M([\tau-\rho^2/\tau,\tau];\R^{2n})$   as $\rho\to\delta_k$.
\end{enumerate}
\end{description}
\end{corollary}

Clearly, (I) of Corollary~\ref{cor:bif-deform} always implies that there exists a sequence $(s_m)\subset (0,\tau)$ such that
(\ref{e:Deform-Equ}) with $s=s_m$ has a solution $u_m\ne \bar{u}$ for each $m$, and that
$\|u_m-\bar{u}\|_{C^1([0,s_m];\R^{2n})}\to 0$ as $m\to\infty$. By (II) of Corollary~\ref{cor:bif-deform}
we have also  a similar conclusion.

\begin{remark}\label{rm:bif-deform}
{\rm  By \cite[Remark~4.3]{Do06},  the conditions
in Corollary~\ref{cor:bif-deform}
\begin{eqnarray*}
&&\hbox{``$i_{\tau,M}(\Upsilon_B)\ne i_{\tau,M}(\xi_{2n})+ \dim{\rm Ker}(I_{2n}-M)$''\quad\hbox{and}}\\
&& \hbox{``$i_{\tau,M}(\Upsilon_{\hat{B}})\ne i_{\tau,M}(\xi_{2n})+ \dim{\rm Ker}(I_{2n}-M)$''}
\end{eqnarray*}
 respectively  become
 \begin{center}
 ``$i_{\tau}(\Upsilon_B)\ne n$'' and ``$i_{\tau}(\Upsilon_{\hat{B}})\ne n$''
 \end{center}
if $M=I_{2n}$, and
 \begin{center}
 ``$i_{\tau,M}(\Upsilon_B)\ne \kappa$'' and ``$i_{\tau,M}(\Upsilon_{\hat{B}})\ne \kappa$''
 \end{center}
if $M={\rm diag}\{-I_{n-\kappa}, I_{\kappa},  -I_{n-\kappa}, I_{\kappa}\}$.
}
\end{remark}

Corollary~\ref{cor:necess-suffi} implies

\begin{corollary}\label{cor:bif-per2}
 For a real $\tau>0$, and a symplectic matrix $M\in{\rm Sp}(2n,\mathbb{R})$,
let $H_0, \hat{H}: \R\times{\R}^{2n}\to\R$ be two continuous functions satisfying:
\begin{itemize}
\item[\rm (i)] $H_0(t+\tau, Mz)=H_0(t, z)$ and $\hat{H}(t+\tau, Mz)=\hat{H}(t,z)$ for all $(t,z)\in\R\times\mathbb{R}^{2n}$.

\item[\rm (ii)] For each $t\in\mathbb{R}$, $H_0(t,\cdot)$ and $\hat{H}(t,\cdot)$ are
  $C^2$ on ${\R}^{2n}$.
\item[\rm (iii)] Both the Euclidean gradients, $\nabla_zH_0(t, z)$ and $\nabla_z\hat{H}(t, z)$,
 and the Hessian matrices, $\nabla^2_zH_0(t, z)$ and $\nabla^2_z\hat{H}(t, z)$,
  depend continuously on $(t, z)\in \mathbb{R}\times\mathbb{R}^{2n}$.
 \end{itemize}
Let $\bar{v}:\R\to\R^{2n}$ satisfy
\begin{enumerate}
\item[\rm (c)] $\dot{v}(t)=J\nabla_zH_0(t, v(t))$ and $v(t+\tau)=Mv(t)\;\forall t$.
\item[\rm (d)]
$\nabla_z\hat{H}(t, \bar{v}(t))=0$ for all $t\in \R$, and
 either $\nabla^2_z\hat{H}(t, \bar{v}(t))>0\;\forall t$
 or $\nabla^2_z\hat{H}(t, \bar{v}(t))<0\;\forall t$.
\end{enumerate}
Take $H(\lambda,t, z)=H_0(t,z)+\lambda\hat{H}(t,z)$  in (\ref{e:M-invariant1}).
Then
\begin{enumerate}
\item[\rm (A)] $\Sigma:=\{\lambda\in\R\,|\, \nu_{\tau,M}(\gamma_\lambda)>0\}$
is a discrete set in $\R$.
\item[\rm (B)] $(\mu, \bar{v})$ with $\mu\in\R$ is a bifurcation point  for (\ref{e:PPer})
if and only if $\nu_{\tau,M}(\gamma_\mu)>0$.
\item[\rm (C)] For each $\mu\in\Sigma$ and a small enough $\rho>0$, (\ref{e:case1}) and (\ref{e:case2})
hold, and therefore the conclusions in the first part of Theorem~\ref{th:bif-per2}
holds, and the second one of Theorem~\ref{th:bif-per2} is also true provided that
$\bar{v}=0$ and all $H_0(t,\cdot), \hat{H}(t,\cdot)$ are even.
\end{enumerate}
\end{corollary}

\noindent{1.2. \;\bf  Bifurcations for generalized periodic solutions of autonomous Hamiltonian systems}
\begin{assumption}\label{ass:BasiAss2}
{\rm Let $M\in{\rm Sp}(2n,\mathbb{R})$ be an orthogonal symplectic matrix and $\Lambda$ be
a topological space. Moreover, let
$H:\Lambda\times{\R}^{2n}\to\R$ be a continuous function such that each
$H_\lambda:{\R}^{2n}\to\R$, $\lambda\in\Lambda$, is $M$-invariant and $C^2$, and all its partial derivatives depend continuously on
the parameter $\lambda\in\Lambda$. For some fixed real $\tau>0$
and each $\lambda\in\Lambda$ let $v_\lambda:\R\to\mathbb{R}^{2n}$ satisfy the following generalized periodic Hamiltonian system
\begin{equation}\label{e:PPer1}
\dot{v}(t)=J\nabla_z H(\lambda, v(t))\quad\hbox{and}\quad v(t+\tau)=Mv(t)\;\;\forall t\in \R,
\end{equation}
 and $\Lambda\times \R\ni (\lambda,t)\mapsto v_\lambda(t)\in\mathbb{R}^{2n}$ is also continuous.}
\end{assumption}

Under this assumption,  each $v_\lambda$ is $C^2$, and
$\Lambda\times \R\ni (\lambda,t)\mapsto \dot{v}_\lambda(t)\in\mathbb{R}^{2n}$ and $\Lambda\times \R\ni (\lambda,t)\mapsto \ddot{v}_\lambda(t)\in\mathbb{R}^{2n}$ are also continuous.
Moreover, each element in $\R\cdot v_\lambda:=\{v_\lambda(\theta+\cdot)\,|\,\theta\in\R\}$ (\textsf{$\R$-orbit})
 also satisfies (\ref{e:PPer1}).  It follows that
 \begin{equation}\label{e:linear3*0}
\hbox{$\ddot{v}_\lambda(t)=J\nabla_z^2H({\lambda}, v_\lambda)\dot{v}_\lambda(t)$\quad and\quad $\dot{v}_\lambda(t+\tau)=M\dot{v}_\lambda(t)\;\forall t$}.
\end{equation}
  Thus if $v_\mu$ is nonconstant (i.e., $\dot{v}_\mu\ne 0$) for some $\mu\in\Lambda$,
then $(\mu, v_\mu)$  is a bifurcation point of (\ref{e:PPer1})
in the sense stated in the paragraph above Proposition~\ref{prop:threeBifu}.
In order to give an exact description for bifurcation pictures of solutions of
 (\ref{e:PPer1}) near $\R\cdot v_\lambda$
some concepts are needed. Two solutions $v_1$ and $v_2$ of (\ref{e:PPer1})
 with parameter value $\lambda$ is said to be \textsf{$\R$-distinct} if they belong to different $\R$-orbits.
We call $(\mu, \R\cdot v_\mu)$ a \textsf{bifurcation $\R$-orbit along sequences} of (\ref{e:PPer1})
with respect to the branch $\{(\lambda, \R\cdot v_\lambda)\,|\,\lambda\in\Lambda\}$
if there exists a sequence $(\lambda_k)\subset\Lambda$ converging to $\mu$, and a solution $\bar{v}_k$ of (\ref{e:PPer1}) with parameter value $\lambda_k$ for each $k$,
such that: (i) $\bar{v}_k\notin\R\cdot v_{\lambda_k}\;\forall k$, (ii)  all $\bar{v}_k$ are $\R$-distinct,
(iii) $\bar{v}_k\to v_\mu$ in $C^1_{\rm loc}$.

When each $v_\lambda$ is constant,  we have the following
 Theorem~\ref{th:bif-per3} and Corollaries~\ref{cor:bif-per4},~\ref{cor:bif-per5}.

\begin{theorem}[\textsf{Alternative bifurcations of Fadell-Rabinowitz's type and of Rabinowitz's type}]\label{th:bif-per3}
Let Assumption~\ref{ass:BasiAss2}, with $\Lambda$ being a real interval, be satisfied.
Suppose  that the orthogonal symplectic matrix $M\in{\rm Sp}(2n,\mathbb{R})$
satisfies $M^l=I_{2n}$ for some integer $l\ge 1$, and that
each $v_\lambda$ is constant, (hence viewed as a constant vector in
${\rm Ker}(M-I_{2n})\cap\{dH_\lambda=0\}\subset\R^{2n}$).
Let $\gamma_\lambda(t)=\exp(tJ\nabla_z^2H({\lambda}, v_\lambda))$, i.e.,
 the fundamental matrix solution of $\dot{v}(t)=J\nabla^2_zH({\lambda},v_\lambda)v(t)$.
 For some interior point $\mu$ of $\Lambda$, assume that the following two conditions are satisfied:
\begin{enumerate}
\item[\rm (a)] The following problem
\begin{equation}\label{e:linear3*}
\hbox{$\dot{v}(t)=J\nabla_z^2H({\mu}, v_\mu)v(t)$\quad and\quad $v(t+\tau)=Mv(t)\;\forall t$}
\end{equation}
has no nonzero constant solutions, i.e.,  ${\rm Ker}(M-I_{2n})\cap{\rm Ker}(\nabla_z^2H({\mu},v_\mu))=\{0\}$.
\item[\rm (b)]  $\dim{\rm Ker}(\gamma_\mu(\tau)-M)\ne 0$,  $\dim{\rm Ker}(\gamma_\lambda(\tau)-M)=0$
 for each $\lambda\in\Lambda\setminus\{\mu\}$ near $\mu$, and
  $i_{\tau,M}(\gamma_\lambda)$ takes, respectively, values $i_{\tau,M}(\gamma_\mu)$ and $i_{\tau,M}(\gamma_\mu)+ \nu_{\tau,M}(\gamma_\mu)$
 as $\lambda\in\Lambda$ varies in  two deleted half neighborhoods  of $\mu$.
 \end{enumerate}
 Then one of the following alternatives occurs for (\ref{e:PPer1}):
 \begin{enumerate}
\item[\rm (i)] The problem (\ref{e:PPer1})  with $\lambda=\mu$ has a sequence of $\R$-distinct solutions,
$\bar{v}_k\notin\mathbb{R}\cdot v_\mu$, $k=1,2,\cdots$,
which  converges to $v_\mu$ on any compact interval $I\subset\R$ in $C^1$-topology.
\item[\rm (ii)] There exist left and right  neighborhoods $\Lambda^-$ and $\Lambda^+$ of $\mu$ in $\Lambda$
and integers $n^+, n^-\ge 0$, such that $n^++n^-\ge \nu_{\tau,M}(\gamma_\mu)/2$,
and for $\lambda\in\Lambda^-\setminus\{\mu\}$ (resp. $\lambda\in\Lambda^+\setminus\{\mu\}$),
(\ref{e:PPer1}) with parameter value $\lambda$  has at least $n^-$ (resp. $n^+$) $\R$-distinct
 solutions, $v_\lambda^i\notin\mathbb{R}\cdot v_\lambda$, $i=1,\cdots,n^-$ (resp. $n^+$)
such that all $v_\lambda^i-v_\lambda$  converge to zero on any compact interval
 $I\subset\R$ in $C^1$-topology as $\lambda\to\mu$.
\end{enumerate}
In particular,  $(\mu, v_\mu)=(\mu, \R\cdot v_\mu)$ is a bifurcation $\R$-orbit of (\ref{e:PPer1}).
Moreover, if $\nu_{\tau,M}(\gamma_\mu)\ge 3$, 
then at least one
of (i) and the following (iii) and (iv) holds:
\begin{enumerate}
\item[\rm (iii)]  For every $\lambda\in\Lambda\setminus\{\mu\}$ near $\mu$ there is a
 solution $\bar{v}_\lambda\notin \mathbb{R}\cdot v_\lambda$  of
(\ref{e:PPer1}) with parameter value $\lambda$, such that $\bar{v}_\lambda-v_\lambda$ converges to zero on any compact interval $I\subset\R$ in $C^1$-topology as $\lambda\to\mu$.

\item[\rm (iv)] For a given $\varepsilon>0$  there is an one-sided  neighborhood $\Lambda^0$ of $\mu$ in $\Lambda$ such that
for any $\lambda\in\Lambda^0\setminus\{\mu\}$, (\ref{e:PPer1}) with parameter value $\lambda$
has either infinitely many $\mathbb{R}$-distinct solutions
$\bar{v}_\lambda^k\notin\mathbb{R}\cdot v_\lambda$
such that $\|\bar{v}_\lambda^k|_{[0,\tau]}-v_\lambda|_{[0,\tau]}\|_{C^1}<\varepsilon$,
 $k=1,2,\cdots$, or
at least two $\mathbb{R}$-distinct solutions $\hat{v}_\lambda^1\notin\mathbb{R}\cdot v_\lambda$ and $\hat{v}_\lambda^2\notin\mathbb{R}\cdot v_\lambda$ such that $\|\hat{v}_\lambda^i|_{[0,\tau]}-v_\lambda|_{[0,\tau]}\|_{C^1}<\varepsilon$, $i=1,2$,
and that
\begin{eqnarray*}
\int^{\tau}_0\left[\frac{1}{2}(J\dot{\hat{v}}_\lambda^1(t),\hat{v}^1_\lambda(t))_{\mathbb{R}^{2n}}+
 H(\lambda, \hat{v}_\lambda^1(t))\right]dt
\ne \int^{\tau}_0\left[\frac{1}{2}(J\dot{\hat{v}}^2_\lambda(t), \hat{v}_\lambda^2(t))_{\mathbb{R}^{2n}}+ H(\lambda, \hat{v}_\lambda^2(t))\right]dt.
\end{eqnarray*}
\end{enumerate}
\end{theorem}

The first part is an alternative bifurcation result of Fadell-Rabinowitz's type,
and the second part is that of  Rabinowitz's type.

The following corollary, which follows directly from Theorem~\ref{th:bif-per3},
is essentially a restatement of Corollary~\ref{cor:necess-suffi}.

\begin{corollary}\label{cor:bif-per4}
Let $M\in{\rm Sp}(2n,\mathbb{R})$ be an orthogonal symplectic matrix,
let $\bar{v}\in{\rm Ker}(M-I_{2n})$, and let $H_0, \hat{H}: {\R}^{2n}\to\R$ be $M$-invariant $C^2$-functions.
Suppose that $dH_0(\bar{v})=d\hat{H}(\bar{v})=0$  and
that either $\hat{H}''(\bar{v})>0$ or $\hat{H}''(\bar{v})<0$.
Then for $\gamma_\lambda(t)=\exp(tJH_0''(\bar{v})+ \lambda tJ\hat{H}''(\bar{v}))$ and any $\tau>0$,
$\Sigma_\tau:=\{\lambda\in\R\,|\, \nu_{\tau,M}(\gamma_\lambda)>0\}$ is a discrete set in $\R$;
and for each $\mu\in\Sigma_\tau$ such that
 (\ref{e:linear3*}) with, $H(\mu, x)=H_0(x)+\mu\hat{H}(x)$ and $v_\mu=\bar{v}$,
  has no nonzero constant solutions,
  the conclusions of Theorem~\ref{th:bif-per3}
 holds  for (\ref{e:PPer1}) with $H(\lambda, x)=H_0(x)+\lambda\hat{H}(x)$ and $v_\lambda\equiv \bar{v}\;\forall\lambda$.
\end{corollary}

In particular, taking $\tau=1$, $H_0=0$ and $H=\hat{H}$ we obtain
the following existence result of $M$-rotating periodic orbits of
$\dot{v}=J\nabla H(v)$ near a $M$-equilibruim $\bar{v}$.

\begin{corollary}\label{cor:bif-per5}
 Let $M\in{\rm Sp}(2n,\mathbb{R})$ be an orthogonal symplectic matrix,
 let $\bar{v}\in{\rm Ker}(M-I_{2n})$,
  and let $H: {\R}^{2n}\to\R$ be a $M$-invariant $C^2$-function satisfying
 $dH(\bar{v})=0$. Suppose that either ${H}''(\bar{v})>0$ or ${H}''(\bar{v})<0$.
 Then
 \begin{description}
  \item[(A)] $\Gamma(H, \bar{v}, M):=\{\lambda\in\mathbb{R}\setminus\{0\}\,|\, \nu_{1,M}(\gamma_\lambda)=\dim{\rm Ker}(\exp(\lambda J{H}''(\bar{v}))-M)>0\}$
is a discrete set.

\item[(B)] If ${H}''(\bar{v})>0$ and $\Gamma(H, \bar{v}, M)\cap(0, \infty)\ne\emptyset$
(resp. ${H}''(\bar{v})<0$ and $\Gamma(H, \bar{v}, M)\cap(-\infty, 0)\ne\emptyset$),
then for $\mu\in \Gamma(H, \bar{v}, M)\cap(0, \infty)$ (resp. $\Gamma(H, \bar{v}, M)\cap(-\infty, 0)$),
 one of the following alternatives occurs:
  \begin{enumerate}
\item[\rm (B.i)] The problem
 \begin{equation}\label{e:PPer3.5}
\dot{v}(t)=J\nabla H(v(t))\quad\hbox{and}\quad v(t+\mu)=Mv(t),\;\forall t\in \R
\end{equation}
 has a sequence of $\R$-distinct solutions, $v_k$, $k=1,2,\cdots$,
such that each $v_k$ is $\mathbb{R}$-distinct from $\bar{v}$, and that $(v_k)$  converges to
$\bar{v}$ on any compact interval $I\subset\R$ in $C^1$-topology.
\item[\rm (B.ii)] There exist left and right  neighborhoods $\Lambda^-$ and $\Lambda^+$ of $\mu$ in $\mathbb{R}\setminus\{0\}$
and integers $n^+, n^-\ge 0$, such that $n^++n^-\ge \frac{1}{2}\dim{\rm Ker}(\exp(\mu{H}''(\bar{v}))-M)$,
and for $\lambda\in\Lambda^-\setminus\{\mu\}$ (resp. $\lambda\in\Lambda^+\setminus\{\mu\}$)
the problem
 \begin{equation}\label{e:PPer3.5.1}
\dot{v}(t)=J\nabla H(v(t))\quad\hbox{and}\quad v(t+\lambda)=Mv(t),\;\forall t\in \R
\end{equation}
 has at least $n^-$ (resp. $n^+$) $\R$-distinct solutions
 solutions, $v_\lambda^i$, $i=1,\cdots,n^-$ (resp. $n^+$),
such that each of them is $\R$-distinct from $\bar{v}$ and converges to  $\bar{v}$ on any compact interval
$I\subset\R$ in $C^1$-topology as $\lambda\to\mu$.
\end{enumerate}
Moreover, if $\dim{\rm Ker}(\exp(\mu{H}''(\bar{v}))-M)\ge 3$,
then either (B.i) holds or one of the following alternatives occurs:
\begin{enumerate}
\item[\rm (B.iii)]  For every $\lambda\in\mathbb{R}\setminus\{\mu\}$ near $\mu$,
the problem (\ref{e:PPer3.5.1}) has a  solution $v_\lambda$, which is $\R$-distinct from $\bar{v}$ and converges to $\bar{v}$
on any compact interval $I\subset\R$ in $C^1$-topology as $\lambda\to\mu$.

\item[\rm (B.iv)]
For a given $\varepsilon>0$  there is an one-sided  neighborhood $\Lambda^0$ of $\mu$ in $\R\setminus\{0\}$ such that
for any $\lambda\in\Lambda^0\setminus\{\mu\}$, the problem (\ref{e:PPer3.5.1}) with parameter value $\lambda$
has either infinitely many $\mathbb{R}$-distinct solutions $\bar{v}_\lambda^k$
such that each of them is $\R$-distinct from $\bar{v}$ and $\|\bar{v}_\lambda^k|_{[0,|\lambda|]}-\bar{v}\|_{C^1}<\varepsilon$, $k=1,2,\cdots$, or
at least two $\mathbb{R}$-distinct solutions $\hat{v}_\lambda^1$ and $\hat{v}_\lambda^2$
such that: {\rm a)} they are $\R$-distinct from $\bar{v}$, {\rm b)}
$\|\hat{v}_\lambda^i-\bar{v}\|_{C^0([0,|\lambda|])}
+|\lambda|\|\hat{v}_\lambda^i\|_{C^0([0,|\lambda|])}
<\varepsilon$ for $i=1,2$,
{\rm c)}
{\small
\begin{eqnarray*}
\int^{\lambda}_0\left[\frac{1}{2}(J\dot{\hat{v}}_\lambda^1(t),\hat{v}^1_\lambda(t))_{\mathbb{R}^{2n}}+ H(\lambda, \hat{v}_\lambda^1(t))\right]dt
\ne \int^{\lambda}_0\left[\frac{1}{2}(J\dot{\hat{v}}^2_\lambda(t), \hat{v}_\lambda^2(t))_{\mathbb{R}^{2n}}+ H(\lambda, \hat{v}_\lambda^2(t))\right]dt.
\end{eqnarray*}}
\end{enumerate}
  \end{description}
\end{corollary}

\begin{remark}\label{rm:twoBifu}
{\rm
Suppose  $H''(\bar{v})>0$. By Williamson theorem (cf. \cite[page 41]{HoZe94})
 there exists a symplectic matrix $S\in{\rm Sp}(2n,\R)$ such that
$S^TH''(\bar{v})S=D_{H''(\bar{v})}={\rm diag}(\varrho_1,\cdots,\varrho_n, \varrho_1,\cdots,\varrho_n)$,
where $0<\varrho_j\le\varrho_k$ for $j\le k$, are the symplectic eigenvalues of $H''(\bar{v})$, namely $(\pm\sqrt{-1}\varrho_j)$ for $1\le j\le n$
are all eigenvalues of $JH''(\bar{v})$. Therefore we can reduce the study of the Hamiltonian system $\dot{v}=J\nabla H(v)$
with $M$-rotating periods near the $M$-equilibruim $\bar{v}$
to that of the Hamiltonian system $\dot{v}=J\nabla G(v)$ with $S^{-1}MS$-rotating periods near $\bar{w}=S^{-1}\bar{v}$, where $G(z)=H(Sz)$
and so $G'(\bar{w})=S^TH'(S\bar{w})=0$ and
$G''(\bar{w})=S^TH''(\bar{v})S=D_{H''(\bar{v})}$. Since $S$ is symplectic,
$JG''(\bar{w})=S^{-1}JH''(\bar{v})S$ and
$\exp(\lambda JH''(\bar{v}))=S\exp(\lambda JD_{H''(\bar{v})})S^{-1}$ for any $\lambda\in\mathbb{R}$.
The former implies that both $G''(\bar{w})$ and $H''(\bar{v})$ have the same symplectic eigenvalues.
The latter  leads to
\begin{eqnarray*}
\Gamma(H, \bar{v}, M)\cap(0,\infty)&=&\left\{\lambda>0\,|\, \dim{\rm Ker}(\exp(\lambda JH''(\bar{v}))-M)>0\right\}\\
&=&\Gamma(G, \bar{w}, S^{-1}MS)\cap(0,\infty).
\end{eqnarray*}
In particular, from this and  (\ref{e:ex.4-}) we derive
\begin{eqnarray*}
\Gamma(H, \bar{v}, I_{2n})\cap(0,\infty)
=\cup^n_{j=1}\left(\frac{2\pi}{\varrho_j}\mathbb{N}\right).
\end{eqnarray*}
because $\dim{\rm Ker}(\exp(\lambda JH''(\bar{v}))-I_{2n})$ is equal to
\begin{eqnarray}\label{e:ex.14-}
&&\dim{\rm Ker}\left(\left(\begin{array}{cc}
             {\rm diag} (\cos(\varrho_1\lambda), . . . , \cos(\varrho_n\lambda)) &  -{\rm diag} (\sin(\varrho_1\lambda), . . . , \sin(\varrho_n\lambda)) \\
             {\rm diag} (\sin(\varrho_1\lambda), . . . , \sin(\varrho_n\lambda)) &  {\rm diag} (\cos(\varrho_1\lambda), . . . ,
             \cos(\varrho_n\lambda)) \nonumber\\
           \end{array}
         \right)-I_{2n}\right)\\
         &=&\sum^n_{j=1}\dim{\rm Ker}\left(\left(\begin{array}{cc}
            \cos(\varrho_j\lambda) &  -\sin(\varrho_j\lambda) \\
             \sin(\varrho_j\lambda) &   \cos(\varrho_j\lambda)
           \end{array}
         \right)-I_{2}\right).
\end{eqnarray}
This is exactly the set of periods of all periodic solutions of $\dot{v}=JH''(\bar{v})v$.

Let us compare previous results with Corollary~\ref{cor:bif-per5}.
Very recently, J. Xing, X. Yang and Y. Li \cite[Corollary~1.2]{XYL}
gave a Lyapunov type theorem on $M$-rotating periodic orbits for an
orthogonal $M\in{\rm Sp}(2n,\mathbb{R})$,
which can be stated as follows in our notations:

\begin{center}
  \textsf{Let   $\bar{v}\in{\rm Ker}(M-I_{2n})$,
  and let $H: {\R}^{2n}\to\R$ be a $M$-invariant $C^2$-function satisfying
 $H'(\bar{v})=0$ and ${H}''(\bar{v})>0$. Suppose that  the symplectic eigenvalues
 of $H''(\bar{v})$ as above,
 $\varrho_1,\cdots,\varrho_n$, are nonresonant, i.e.,
 $\varrho_i/\varrho_j\notin\mathbb{Q}$ for $1\le i,j\le n$, $i\ne j$. Then
 for each sufficiently small $\varepsilon>0$, the energy surface
$H(v) = H(\bar{v}) +\varepsilon^2$ contains at least $n$ $M$-rotating periodic orbits of $\dot{v}=JH'(v)$
with periods close to the ones of the linear system $\dot{v}=JH''(\bar{v})v$.}
\end{center}
Note that the set of $M$-rotating periods of the linear system $\dot{v}=JH''(\bar{v})v$ is exactly
the discrete set $\Gamma(H, \bar{v}, M)$ in Corollary~\ref{cor:bif-per5}(A).
This theorem and Corollary~\ref{cor:bif-per5} belong to
 two categories of fixed energy and fixed period problems about $M$-rotating periodic orbits, respectively.

Previous other related results require $M=I_{2n}$.
 To compare with them we may assume  $\bar{v}=0$ and
$H''(0)={\rm diag}(\varrho_1,\cdots,\varrho_n, \varrho_1,\cdots,\varrho_n)$ by the above arguments.
Then for $T>0$,  $\dot{v}=JH''(0)v$ has a nontrivial $T$-periodic solution if and only if
$T$ belongs to $\cup^n_{j=1}\left(\frac{2\pi}{\varrho_j}\mathbb{N}\right)$. For $1\le i\le n$ let
$$
\mathbb{R}^{2}_i=\{(q_1,\cdots,q_n, p_1,\cdots, p_n)^T\in \mathbb{R}^{2n}\,|\,
q_j=p_j=0\;\forall j\ne i\}.
$$
It is an invariant subspace of $JH''(0)$. For a given
$T\in \cup^n_{j=1}\left(\frac{2\pi}{\varrho_j}\mathbb{N}\right)$,
$E=\oplus_{T\varrho_i\in 2\pi\mathbb{N}}\mathbb{R}^{2}_i$ and $F=\oplus_{T\varrho_i\notin 2\pi\mathbb{N}}\mathbb{R}^{2}_i$
are invariant subspaces of $JH''(0)$, and $\mathbb{R}^{2n}=E\oplus F$. It is easy to see:
\begin{description}
\item[$\bullet$] $\dim E=2\sharp\{j\in\{1,\cdots,n\}\,|\, T\varrho_j\in 2\pi\mathbb{N}\}=
\dim{\rm Ker}(\exp(TJH''(0))-I_{2n})$ [by (\ref{e:ex.14-})].

\item[$\bullet$] Each solution of $\dot{v}=JH''(0)v$ in $E$ is $T$-periodic, and
no solutions of $\dot{v}=JH''(0)v$ with initial data in $F\setminus\{0\}$ is $T$-periodic.
\end{description}
Since $H''(0)>0$, $\dot{v}=JH''(0)v$ has no constant solutions in $E\setminus\{0\}$ and
the quadratic form $E\ni z\mapsto (H''(0)z,z)$ has signature equal to $\dim E$. From
\cite[Theorem~8.4]{FaRa2} (or Theorem on page 88 of \cite{FlZe})  we can only derive:

\textsf{either (B.i) or (B.ii)  holds with $(\bar{v}, \mu, M)=(0, T, I_{2n})$}.
But when $2\sharp\{j\in\{1,\cdots,n\}\,|\, T\varrho_j\in 2\pi\mathbb{N}\}\ge 3$,
Corollary~\ref{cor:bif-per5}  also gives
\textsf{either (B.i) or (B.iii)  or (B.iv) with $(\bar{v}, \mu, M)=(0, T, I_{2n})$.}
Moreover,  \cite[Corollaries~4.2,4.3]{Szu}  leads to: \textsf{for each $1\le j\le n$
there exists a sequence $(v_k)$ of nonconstant
periodic solutions of $\dot{v}= J\nabla H(v)$ with (not necessarily minimal)
period $T_k$ such that $\|z_k\|_{L^\infty}\to 0$ and $T_k\to 2\pi/\varrho_j$}.

Assume $\{\varrho_1,\cdots, \varrho_n\}=\{\beta_1,\cdots,\beta_p\}$ with
$\beta_1<\cdots<\beta_p$. For $\beta\in\{\beta_1,\cdots,\beta_p\}$ let
$H_\beta$ be the sum of the generalized eigenspaces of $JH''(0)$ in $\mathbb{C}^{2n}$
associated to the eigenvalues of the form $\pm ik\beta$, $k\in\mathbb{N}$, i.e.,
$H_\beta=\sum_{\lambda\in\Gamma_\beta}\cup^\infty_{l=1}{\rm Ker}(\lambda I_{2n}-JH''(0))$,
where $\Gamma_\beta=\{ik\beta\,|\,k\in\mathbb{Z}\setminus\{0\}\}\cap\{-i\beta_1,\cdots, -i\beta_p, i\beta_1,\cdots, i\beta_p\}$.
Put
$$
\Xi_\beta=\{\xi,\eta\in\mathbb{R}^{2n}\,|\, \xi+i\eta\in H_\beta\}.
$$
They are invariant subspaces of $JH''(0)$ in $\mathbb{C}^{2n}$ and $\mathbb{R}^{2n}$, respectively.
$\dim \Xi_\beta$ is even and $H''(0)|_{\Xi_\beta}>0$ implies: a)
each solution of $\dot{v}=JH''(0)v$ with initial value in $\Xi_\beta$ has (not necessary minimal) period $2\pi/\beta$ and
no solutions of $\dot{v}=JH''(0)v$ with initial data in $\mathbb{R}^{2n}\setminus\Xi_\beta$ has the period $2\pi/\beta$;
b) $\dot{v}=JH''(0)v$ has $\dim\Xi_\beta/2$ linearly independent periodic solutions with period $2\pi/\beta$
as claimed below (H2) on the page 128 of \cite{Ba1}. The latter sentence shows $\dim\Xi_\beta=\dim{\rm Ker}(\exp(TJH''(0))-I_{2n})$
with period $T=2\pi/\beta$. Therefore
using \cite[Theorem~9.5]{Ba1} we  obtain:

\textsf{For each $\beta\in\{\beta_1,\cdots,\beta_p\}$ and $T=2\pi/\beta$,
either (B.i)  or (B.ii)  holds with $(\bar{v}, \mu, M)=(0, T, I_{2n})$.}

\cite[Theorem~9.5]{Ba1} gives no conclusions for $T\in \cup^n_{j=1}\left(\frac{2\pi}{\varrho_j}\mathbb{N}\right)\setminus\{2\pi/\beta_k\,|\,k=1,\cdots,p\}$.

 In summary, Corollary~\ref{cor:bif-per5} with $M=I_{2n}$  cannot be included by the previous results.}
\end{remark}

Most of the  works on bifurcations of Hamiltonian systems have very little
on the case where all $v_\lambda$
in Assumption~\ref{ass:BasiAss2} are equal to a nonequilibrium (i.e., nonconstant) solution $\bar{v}$.
Using the iso-energetic Poincar\'e mapping some results were obtained, see \cite{MeHa92} and \cite{Han}.
Here are our three bifurcation results of $M$-rotating periodic solutions from  $\bar{v}$.

\begin{theorem}[\hbox{\textsf{Necessary condition}}]\label{th:bif-ness-orbit}
Let $\Lambda, M$ and $H$ be as in Assumption~\ref{ass:BasiAss2}, and let
$\bar{v}:\R\to\R^{2n}$ be a nonconstant solution of (\ref{e:PPer1})  for each $\lambda\in\Lambda$.
Suppose that for some  $\mu\in\Lambda$
there exists a sequence $(\lambda_k)\subset\Lambda$ converging to $\mu$ such that
(\ref{e:PPer1}) with $\lambda=\lambda_k$ has  solutions $v_k$, $k=1,2,\cdots$,
which are $\R$-distinct each other and satisfies
$v_k|_{[0,\tau]}\to \bar{v}|_{[0,\tau]}$ in $C^0([0,\tau];\R^{2n})$.
Then the following problem
\begin{equation}\label{e:linear4}
\hbox{$\dot{v}(t)=J\nabla^2_zH({\mu}, \bar{v}(t))v(t)$\quad and\quad $v(t+\tau)=Mv(t)\;\forall t$}
\end{equation}
 has at least two linearly independent  nontrivial solutions, i.e.,
$\dim{\rm Ker}(\gamma_\mu(\tau)-M)\ge 2$, where $\gamma_\lambda(t)$ is  the fundamental matrix solution of $\dot{v}(t)=J\nabla^2_zH(\lambda, \bar{v}(t))v(t)$.
\end{theorem}

Under the assumptions of Theorem~\ref{th:bif-ness-orbit}, since $\dot{\bar{v}}(0)\ne 0$ and
 $\gamma_\lambda(\tau)\dot{\bar{v}}(0)=\dot{\bar{v}}(\tau)=M\dot{\bar{v}}(0)$ we get $\dim{\rm Ker}(\gamma_{\lambda}(\tau)-M)\ge 1$.
 (Though $1$ and $-1$ as eigenvalues of a symplectic matrix $S\in{\rm Sp}(2n;\mathbb{R})$ always have even algebraic multiplicity it is possible
 that the geometric multiplicity of $1$ (or $-1$) is equal to one.)
When $\dim{\rm Ker}(\gamma_{\lambda}(\tau)-M)=1$ we call $\bar{v}$ a \textsf{transversally nondegenerate}
$M$-rotating $\tau$-periodic solution of (\ref{e:PPer1}). Therefore the conclusion of Theorem~\ref{th:bif-ness-orbit}
says $\bar{v}$ to be a transversally degenerate $M$-rotating $\tau$-periodic solution of $\dot{v}(t)=J\nabla_zH({\mu}, v(t))$.

Clearly, Theorem~\ref{th:bif-ness-orbit} is  a strengthen version of Theorem~\ref{th:bif-per1}(I).
With respect to (II) and (III) of Theorem~\ref{th:bif-per1} and (B) and (C) of
 Corollary~\ref{cor:bif-per2} respectively, we have also:

\begin{theorem}[\hbox{\textsf{Sufficient condition}}]\label{th:bif-suffict1-orbit}
Let $\Lambda, M$, $H$, $\bar{v}$ and $\gamma_\lambda$ be as in Theorem~\ref{th:bif-ness-orbit},
and let $\Lambda$ be first countable.
Suppose that the following conditions are satisfied:
\begin{enumerate}
\item[\rm (a)]  The orbit ${\cal O}:=\R\cdot \bar{v}$ is an embedded circle
(i.e., $\mathbb{R}_{\bar{v}}$ is an infinite cyclic subgroup of $\R$ with generator $p>0$).
\item[\rm (b)] For some $\mu\in\Lambda$, $\dim{\rm Ker}(\gamma_{\mu}(\tau)-M)\ge 2$ and
there exist two sequences in  $\Lambda$ converging to $\mu$, $(\lambda_k^-)$ and
$(\lambda_k^+)$,  such that for each $k\in\mathbb{N}$,
$$
[i_{\tau,M}(\gamma_{\lambda_k^-}), i_{\tau,M}(\gamma_{\lambda_k^-})+\nu_{\tau,M}(\gamma_{\lambda_k^-})-1]\cap[i_{\tau,M}(\gamma_{\lambda_k^+}), i_{\tau,M}(\gamma_{\lambda_k^+})+\nu_{\tau,M}(\gamma_{\lambda_k^+})-1]=\emptyset,
$$
 and either $\nu_{\tau,M}(\gamma_{\lambda_k^+})=1$ or $\nu_{\tau,M}(\gamma_{\lambda_k^-})=1$.
 \item[\rm (c)] For any solution $v$ of (\ref{e:PPer1})  with $\lambda=\mu$,
 if there exists a sequence $(s_k)$ of reals such that  $s_k\cdot v$
  converges to  $\bar{v}$ on any compact interval $I\subset\R$ in $C^0$-topology (and so
  in the $C^1$-topology by Proposition~\ref{prop:threeBifu}), then $v$ is periodic.
  (Clearly, this holds if $M^l=I_{2n}$ for some $l\in\mathbb{N}$.)
  \end{enumerate}
   Then  there exists a sequence $(\lambda_k)\subset\hat{\Lambda}:=\{\mu,\lambda^+_k, \lambda^-_k\,|\,k\in\mathbb{N}\}$ converging to $\mu$ and
  solutions $v_k$ of (\ref{e:PPer1}) with $\lambda=\lambda_k$, $k=1,2,\cdots$,
  such that any two of these $v_k$ are $\R$-distinct and that
  $(v_k)$ converges to  $\bar{v}$ on any compact interval $I\subset\R$ in $C^1$-topology as $k\to\infty$.
    \end{theorem}

\begin{theorem}[\hbox{\textsf{Existence for bifurcations}}]\label{th:bif-existence-orbit}
Let $\Lambda, M$, $H$, $\bar{v}$ and $\gamma_\lambda$ be as in Theorem~\ref{th:bif-ness-orbit}.
Suppose that $\Lambda$ is path-connected, the conditions (a) and (c) in Theorem~\ref{th:bif-suffict1-orbit}
and the following are satisfied:
\begin{enumerate}
\item[\rm (d)] There exist two  points $\lambda^+, \lambda^-\in\Lambda$ such that
  $$
  [i_{\tau,M}(\gamma_{\lambda^+}), i_{\tau,M}(\gamma_{\lambda^+})+\nu_{\tau,M}(\gamma_{\lambda^+})-1]\cap[i_{\tau,M}(\gamma_{\lambda^-}), i_{\tau,M}(\gamma_{\lambda^-})+\nu_{\tau,M}(\gamma_{\lambda^-})-1]=\emptyset,
  $$
 and either $\nu_{\tau,M}(\gamma_{\lambda^+})=1$ or $\nu_{\tau,M}(\gamma_{\lambda^-})=1$.
  \end{enumerate}
 Then for any path $\alpha:[0,1]\to\Lambda$ connecting $\lambda^+$ to $\lambda^-$
   there exists a sequence $(t_k)\subset [0, 1]$ converging to some $\bar{t}$
   and solutions $v_k\ne \bar{v}$ of (\ref{e:PPer1}) with $\lambda=\alpha(t_k)$, $k=1,2,\cdots$,
  such that any two of these $v_k$ are $\R$-distinct and that
  $(v_k)$ converges to  $\bar{v}$ on any compact interval $I\subset\R$ in $C^1$-topology as $k\to\infty$.
     Moreover,  $\alpha(\bar{t})$ is not equal to $\lambda^+$ (resp. $\lambda^-$) if $\nu_{\tau,M}(\gamma_{\lambda^+})=1$ (resp. $\nu_{\tau,M}(\gamma_{\lambda^-})=1$).
 \end{theorem}

\begin{theorem}[\textsf{Alternative bifurcations of Rabinowitz's type and of Fadell-Rabinowitz's type}]\label{th:bif-suffict-orbit}
Let $\tau>0$ and  $\Lambda$ be a real interval, and let
$H:\Lambda\times{\R}^{2n}\to\R$ be a continuous function such that each
$H_\lambda:{\R}^{2n}\to\R$, $\lambda\in\Lambda$, is  $C^2$ and all its
partial derivatives depend continuously on
the parameter $\lambda\in\Lambda$.
Let $\bar{v}:\R\to\R^{2n}$ be a nonconstant $\tau$-periodic solution of
$\dot{v}(t)=J\nabla_z H(\lambda, v(t))$  for each $\lambda\in\Lambda$,
(with minimal period $\tau/p$, $p\in\mathbb{N}$),
and let $\gamma_\lambda(t)$ be  the fundamental matrix solution of
$\dot{v}(t)=JH''_{\lambda}(\bar{v}(t))v(t)$.
For some interior point $\mu$ of $\Lambda$, suppose that
 $\nu_{\tau}(\gamma_\mu):=\dim{\rm Ker}(\gamma_\mu(\tau)-I_{2n})>1$,
 $\nu_{\tau}(\gamma_\lambda):=\dim{\rm Ker}(\gamma_\lambda(\tau)-I_{2n})=1$
 for each $\lambda\in\Lambda\setminus\{\mu\}$ near $\mu$, and that
  $i_{\tau}(\gamma_\lambda)$ takes values $i_{\tau}(\gamma_\mu)$ and $i_{\tau}(\gamma_\mu)+ \nu_{\tau}(\gamma_\mu)-1$
  as $\lambda\in\Lambda$ varies in  two deleted half neighborhoods of $\mu$. Then
\begin{description}
\item[(I)] The conclusions of Theorem~\ref{th:bif-suffict1-orbit} (after $\hat\Lambda$ is replaced by $\Lambda$) hold.

\item[(II)] If  $\nu_{\tau}(\gamma_\mu)>2$, one of the following alternatives occurs:
\begin{enumerate}
\item[\rm (II.1)] Equation  (\ref{e:PPer1}) with $\lambda=\mu$ has a sequence of $\R$-distinct solutions, $v_k$, $k=1,2,\cdots$,
such that $(v_k)$ converges to $\bar{v}$ on any compact interval $I\subset\R$ in $C^1$-topology as $k\to\infty$.

\item[\rm (II.2)]  For every $\lambda\in\Lambda\setminus\{\mu\}$ near $\mu$ there is a  solution $v_\lambda$ of
(\ref{e:PPer1}) with parameter value $\lambda$, which is $\R$-distinct with $\bar{v}$ and
converges to $\bar{v}$ on any compact interval $I\subset\R$ in $C^1$-topology as $\lambda\to \mu$.

\item[\rm (II.3)] There is an one-sided  neighborhood $\Lambda^0$ of $\mu$ such that
for any $\lambda\in\Lambda^0\setminus\{\mu\}$, (\ref{e:PPer1}) with parameter value $\lambda$
has either infinitely many $\R$-distinct solutions $\bar{v}_\lambda^k\notin\mathbb{R}\cdot \bar{v}$, $k=1,2,\cdots$,
or at least two $\R$-distinct solutions, $v_\lambda^1\notin\mathbb{R}\cdot \bar{v}$ and $v_\lambda^2\notin\mathbb{R}\cdot \bar{v}$,
such that
{\small
\begin{eqnarray*}
\int^{\tau}_0\left[\frac{1}{2}(J\dot{v}_\lambda^1(t),v^1_\lambda(t))_{\mathbb{R}^{2n}}+ H(\lambda, {v}_\lambda^1(t))\right]dt
\ne \int^{\tau}_0\left[\frac{1}{2}(J\dot{v}^2_\lambda(t), v_\lambda^2(t))_{\mathbb{R}^{2n}}+ H(\lambda, {v}_\lambda^2(t))\right]dt;
\end{eqnarray*}}
and these $\bar{v}_\lambda^k$, $v^1_\lambda$ and $v^2_\lambda$ converge to $\bar{v}$ on any compact interval $I\subset\R$ in $C^1$-topology as $\lambda\to \mu$.
\end{enumerate}
\item[(III)] If $\nu_{\tau/p}(\gamma_\mu)=1$, (which implies $p\ge 2$ because $\nu_{\tau}(\gamma_\mu)>1$), and $p$ is equal to $2$ (resp. a prime greater than $2$), then
  one of (II.1) and the following (III.1) [resp. (III.2)] occurs:
  \begin{enumerate}
\item[\rm (III.1)] There exist left and right  neighborhoods $\Lambda^-$ and $\Lambda^+$ of $\mu$ in $\Lambda$
and integers $n^+, n^-\ge 0$, such that $n^++n^-\ge \nu_{\tau}(\gamma_\mu)$,
and for $\lambda\in\Lambda^-\setminus\{\mu\}$ (resp. $\lambda\in\Lambda^+\setminus\{\mu\}$),
(\ref{e:PPer1}) with parameter value $\lambda$  has at least $n^-$ (resp. $n^+$) $\R$-distinct
 solutions, $v_\lambda^i\notin\mathbb{R}\cdot \bar{v}$, $i=1,\cdots,n^-$ (resp. $n^+$),
which
converge to  $\bar{v}$ on any compact interval $I\subset\R$ in $C^1$-topology as $\lambda\to\mu$.

\item[\rm (III.2)] Replace $\nu_{\tau}(\gamma_\mu)$ with $\nu_{\tau}(\gamma_\mu)/2$ in (III.1).
\end{enumerate}
\end{description}
\end{theorem}

\noindent{1.3. \;\bf  Bifurcations for brake orbits of periodic Hamiltonian systems}

\begin{assumption}\label{ass:brake}
{\rm
For a real $\tau>0$ and a topological space $\Lambda$,
let $H:\Lambda\times\R\times{\R}^{2n}\to\R$ be a continuous function such that
each $H(\lambda,t,\cdot):{\R}^{2n}\to\R$, $(\lambda,t)\in\Lambda\times \R$, is $C^2$,
 and all its partial derivatives of $H$ depend continuously on
 $(\lambda, t, z)\in\Lambda\times \R\times\mathbb{R}^{2n}$,
and that
\begin{eqnarray}\label{e:brake-invariant0}
H(\lambda, t+\tau, z)=H(\lambda, t, z)\quad\forall (\lambda,t,z)\in\Lambda\times\R\times\mathbb{R}^{2n},\\
H(\lambda, -t, Nz)=H(\lambda, t, z)\quad\forall (\lambda,t,z)\in\Lambda\times\R\times\mathbb{R}^{2n},\label{e:brake-invariant1}
\end{eqnarray}
where $N=\left(\begin{array}{cc}-I_n & 0\\0& I_n\end{array}\right)$.
For each $\lambda\in\Lambda$ let $v_\lambda:\R\to\mathbb{R}^{2n}$ satisfy  the following  periodic Hamiltonian system
\begin{equation}\label{e:Pbrake}
\dot{v}(t)=J\nabla_z H(\lambda,t, v(t)),\quad v(t+\tau)=v(t)\quad\hbox{and}\quad v(-t)=Nv(t)\;\forall t\in\R,
\end{equation}
 and $\Lambda\times [0,\tau]\ni (\lambda,t)\mapsto v_\lambda(t)\in\mathbb{R}^{2n}$ is  continuous.
}
\end{assumption}

As pointed out below Assumption~\ref{ass:BasiAss1} it follows from Assumption~\ref{ass:brake} that
$\Lambda\times [0,\tau]\ni (\lambda,t)\mapsto \dot{v}_\lambda(t)\in\mathbb{R}^{2n}$ is  continuous.
When $H$ is independent of $t$,  solutions of (\ref{e:Pbrake}) are called \textsf{brake orbits}.
An even brake orbit  $\bar v$ must be \emph{constant}, hence an equilibrium point (i.e.,
a critical point of $H(\lambda,\cdot)$).
In fact, let $\bar{v}(t)=(x(t)^\top, y(t)^\top)^\top$. Then $\bar{v}(-t)=\bar{v}(t)$ and
 $\bar{v}(-t)=N\bar{v}(t)$ lead to $\bar{v}(t)=N\bar{v}(t)$ and so $x(t)\equiv 0$.
 Moreover, (\ref{e:brake-invariant1}) implies $\nabla_zH(Nz)=N\nabla_zH(z)$, and hence
 $$
 \dot{\bar{v}}(t)=N\dot{\bar{v}}(t)=J\nabla_zH(N\bar{v}(t))=JN\nabla_zH(\bar{v}(t))=
-NJ\nabla_zH(\bar{v}(t))=-N\dot{\bar{v}}(t).
 $$
It follows that $\dot{y}(t)\equiv 0$. Therefore, $\bar v$ is constant.

 We call $(\mu, v_\mu)$  with some $\mu\in\Lambda$ a \textsf{bifurcation point} of (\ref{e:Pbrake})
if each neighborhood of it in $\Lambda\times W^{1,2}(S_\tau; \R^{2n})$
(or equivalently $\Lambda\times C^{1}(S_\tau; \R^{2n})$)
 contains elements $(\lambda, v)$, where $v\ne v_\lambda$ is a solution of
 (\ref{e:Pbrake})  for the parameter value $\lambda$.
Similarly, we can define a \textsf{bifurcation point along sequence} of
 (\ref{e:Pbrake}) in the sense of Definition~\ref{def:seqbifur}.

Under Assumption~\ref{ass:brake} the  linearized problem of (\ref{e:Pbrake}) along $v_\lambda$ is
\begin{equation}\label{e:linearbrake}
\dot{v}(t)=J\nabla^2_zH(\lambda,t, v_\lambda(t))v(t),\quad v(t+\tau)=v(t)\quad\hbox{and}\quad v(-t)=Nv(t)
\end{equation}
for all $t\in \R$. Note that (\ref{e:brake-invariant1}) implies
$$
N^T\nabla_z H(\lambda,-t, Nz)=\nabla_z H(\lambda,t, z)\quad\hbox{and}\quad
N^T\nabla^2_z H(\lambda,-t, Nz)N=\nabla^2_z H(\lambda,t, z).
$$
 Hence each
 $B_\lambda(t):= \nabla^2_zH(\lambda,t, v_\lambda(t))$
satisfies Assumption~\ref{ass:brakeB}.
Let $\gamma_\lambda:\R\to {\rm Sp}(2n,\mathbb{R})$ be the fundamental matrix solution of
$\dot{Z}(t)=JB_\lambda(t)Z(t)$. It has the Maslov-type index  $(\mu_{1,\tau}(\gamma_{\lambda}),\nu_{1,\tau}(\gamma_{\lambda}))$ defined by
(\ref{e:2.3M}) and (\ref{e:2.4M}).
For a solution of (\ref{e:Pbrake}), since $NJ=-JN$ and $N^T=N$, it is not hard to see that
$Nu$ does not necessarily satisfy (\ref{e:Pbrake}) [even if $H$ is independent of $t$].

Corresponding to Theorems~\ref{th:bif-ness}, ~\ref{th:bif-suffict}  we have:

\begin{theorem}\label{th:bif-nessbrake}
Let Assumption~\ref{ass:brake} be satisfied.
\begin{enumerate}
\item[\rm (I)]{\rm (\textsf{Necessary condition}):}
 If $(\mu, v_\lambda)$ is a bifurcation point along sequences of
(\ref{e:Pbrake}), that is,
 there exists a sequence $(\lambda_k)\subset\Lambda$ converging to $\mu$ and
solutions $v^k\ne v_{\lambda_k}$ of (\ref{e:Pbrake}) with $\lambda=\lambda_k$
such that $v^k\to v_\mu$ on any compact interval $I\subset\R$ in $C^0$-topology as $k\to\infty$,
then  $\nu_{1,\tau}(\gamma_{\mu})\ne 0$.

\item[\rm (II)]{\rm (\textsf{Sufficient condition}):}
Let $\Lambda$ be first countable.
Suppose that some $\mu\in\Lambda$ 
there exist two sequences in  $\Lambda$ converging to $\mu$, $(\lambda_k^-)$ and
$(\lambda_k^+)$,  such that
for each $k\in\mathbb{N}$,
$$
[\mu_{1,\tau}(\gamma_{\lambda_k^-}), \mu_{1,\tau}(\gamma_{\lambda_k^-})+\nu_{1,\tau}(\gamma_{\lambda_k^-})]\cap[\mu_{1,\tau}(\gamma_{\lambda_k^+}), \mu_{1,\tau}(\gamma_{\lambda_k^+})+\nu_{1,\tau}(\gamma_{\lambda_k^+})]=\emptyset
$$
 and either $\nu_{1,\tau}(\gamma_{\lambda_k^+})=0$ or $\nu_{1,\tau}(\gamma_{\lambda_k^-})=0$.
Let $\hat{\Lambda}:=\{\mu,\lambda^+_k, \lambda^-_k\,|\,k\in\mathbb{N}\}$.
  Then  $(\mu, v_\mu)$ is a bifurcation point  of (\ref{e:Pbrake})
  with respect to the branch $\{(\lambda, v_\lambda)\,|\,\lambda\in\hat\Lambda\}$
  (and so $\{(\lambda, v_\lambda)\,|\,\lambda\in\Lambda\}$).

  \item[\rm (III)]{\rm (\textsf{Existence for bifurcations}):} Let $\Lambda$ be path-connected. Suppose that
  there exist two  points $\lambda^+, \lambda^-\in\Lambda$ such that
  $[\mu_{1,\tau}(\gamma_{\lambda^+}), \mu_{1,\tau}(\gamma_{\lambda^+})+\nu_{1,\tau}(\gamma_{\lambda^+})]\cap[\mu_{1,\tau}(\gamma_{\lambda^-}), \mu_{1,\tau}(\gamma_{\lambda^-})+\nu_{1,\tau}(\gamma_{\lambda^-})]=\emptyset$
 and either $\nu_{1,\tau}(\gamma_{\lambda^+})=0$ or $\nu_{1,\tau}(\gamma_{\lambda^-})=0$.
 Then for any path $\alpha:[0,1]\to\Lambda$ connecting $\lambda^+$ to $\lambda^-$
   there exists a sequence $(t_k)\subset [0, 1]$ converging to some $\bar{t}$
   and solutions $v^k\ne v_{\alpha(t_k)}$ of (\ref{e:Pbrake})  with $\lambda=\alpha(t_k)$, $k=1,2,\cdots$,
  such that  $(v^k)$ converges to  ${v}_{\alpha(\bar{t})}$ on any compact interval $I\subset\R$ in $C^1$-topology as $k\to\infty$.
   Moreover,  $\alpha(\bar{t})$ is not equal to $\lambda^+$ (resp. $\lambda^-$) if $\nu_{1,\tau}(\gamma_{\lambda^+})=0$ (resp. $\nu_{1,\tau}(\gamma_{\lambda^-})=0$).
 \end{enumerate}
\end{theorem}

\begin{theorem}[\textsf{Alternative bifurcations of Rabinowitz's type and of Fadell-Rabinowitz's type}]\label{th:bif-suffictbrake}
Let Assumption~\ref{ass:brake} with $\Lambda$ being a real interval be satisfied.
Suppose that for some interior point $\mu$ of $\Lambda$,
 $\nu_{1,\tau}(\gamma_\mu)\ne 0$,  $\nu_{1,\tau}(\gamma_\lambda)=0$
 for each $\lambda\in\Lambda\setminus\{\mu\}$ near $\mu$, and
 $\mu_{1,\tau}(\gamma_\lambda)$ takes, respectively, values $\mu_{1,\tau}(\gamma_\mu)$ and $\mu_{1,\tau}(\gamma_\mu)+ \nu_{1,\tau}(\gamma_\mu)$
 as $\lambda\in\Lambda$ varies in  two deleted half neighborhoods  of $\mu$.
  Then at least one of the following three claims occurs:
\begin{enumerate}
\item[\rm (i)] Equation (\ref{e:Pbrake}) with $\lambda=\mu$ has a sequence of solutions,
$\bar{v}_k\ne v_\mu$, $k=1,2,\cdots$,
which converges to $v_\mu$ in $C^1(S_\tau;\R^{2n})$.

\item[\rm (ii)]  For every $\lambda\in\Lambda\setminus\{\mu\}$ near $\mu$ there is a  solution $\bar{v}_\lambda\ne v_\lambda$ of
(\ref{e:Pbrake}) with parameter value $\lambda$,
such that $\bar{v}_\lambda-v_\lambda$ converges to zero
 in $C^1(S_\tau;\R^{2n})$ as $\lambda\to \mu$.

\item[\rm (iii)]
For a given neighborhood $\mathcal{W}$ of $v_\mu$ in $C^1(S_\tau;\R^{2n})$
there is an one-sided  neighborhood $\Lambda^0$ of $\mu$ such that
for any $\lambda\in\Lambda^0\setminus\{\mu\}$, (\ref{e:Pbrake}) with parameter value $\lambda$
has at least two distinct solutions $v_\lambda^1\ne v_\lambda$ and $v_\lambda^2\ne v_\lambda$ in $\mathcal{W}$,
which can also be required to satisfy
\begin{eqnarray*}
\int^{\tau}_0\left[\frac{1}{2}(J\dot{v}_\lambda^1(t),v^1_\lambda(t))_{\mathbb{R}^{2n}}+ H(\lambda, t, {v}_\lambda^1(t))\right]dt
\ne \int^{\tau}_0\left[\frac{1}{2}(J\dot{v}^2_\lambda(t), v_\lambda^2(t))_{\mathbb{R}^{2n}}+ H(\lambda, t,
{v}_\lambda^2(t))\right]dt
\end{eqnarray*}
provided that $\nu_{1,\tau}(\gamma_\mu)>1$ and (\ref{e:Pbrake}) with parameter value $\lambda$
has only finitely many  solutions in $\mathcal{W}$.
\end{enumerate}
Moreover, if $v_\lambda=0\;\forall\lambda$, and all $H(\lambda,t,\cdot)$ are even,
then either (i) holds or the following  occurs:
\begin{enumerate}
\item[\rm (iv)] There exist left and right  neighborhoods $\Lambda^-$ and $\Lambda^+$ of $\mu$ in $\Lambda$
and integers $n^+, n^-\ge 0$, such that $n^++n^-\ge \nu_{1,\tau}(\gamma_\mu)$,
and for $\lambda\in\Lambda^-\setminus\{\mu\}$ (resp. $\lambda\in\Lambda^+\setminus\{\mu\}$),
(\ref{e:Pbrake}) with parameter value $\lambda$  has at least $n^-$ (resp. $n^+$) distinct pairs of nontrivial solutions,
$\{v_\lambda^i, -v_\lambda^i\}$, $i=1,\cdots,n^-$ (resp. $n^+$),
which  converge to zero in $C^1(S_\tau;\R^{2n})$  as $\lambda\to\mu$.
\end{enumerate}
\end{theorem}

Corresponding to Corollary~\ref{cor:necess-suffi} or \ref{cor:bif-per2}, we have also

\begin{corollary}\label{cor:bif-per2brake}
Let $C^1$ functions $H_0, \hat{H}: \R\times{\R}^{2n}\to\R$ satisfy the following conditions:
\begin{enumerate}
\item[\rm (A)] For all $(t,z)\in\Lambda\times\R\times\mathbb{R}^{2n}$ it holds that
$$
H_0(-t, Nz)=H_0(t,z)=H_0(t+\tau, z)\quad\hbox{and}\quad \hat{H}(-t, Nz)=\hat{H}(t,z)=\hat{H}(t+\tau,z).
$$
\item[\rm (B)]  All $H_0(t,\cdot), \hat{H}(t,\cdot):{\R}^{2n}\to\R$ are $C^2$ and all their
partial derivatives depend continuously on  $(t, z)\in \R\times\mathbb{R}^{2n}$.
\end{enumerate}
Suppose that  $\bar{v}:\R\to\R^{2n}$ satisfy
$$
\hbox{$\dot{v}(t)=J\nabla_zH_0(t, v(t))$,\quad\hbox{$v(t+\tau)=v(t)$\quad\hbox{and}\quad $v(-t)=Nv(t)\;\forall t$,}}
$$
and that $\nabla_z\hat{H}(t, \bar{v}(t))=0$ for all $t\in \R$ and
$$
\hbox{either $\nabla^2_z\hat{H}(t, \bar{v}(t))>0\;\forall t$\quad or\quad $\nabla^2_z\hat{H}(t,\bar{v}(t))<0\;\forall t$.}
$$
Then, taking $H(\lambda,t, z)=H_0(t,z)+\lambda\hat{H}(t,z)$  in (\ref{e:Pbrake})
and letting 
$\gamma_\lambda:\R\to {\rm Sp}(2n,\mathbb{R})$ be the fundamental matrix solution of
$\dot{Z}(t)=JB_\lambda(t)Z(t)$ with $B_\lambda(t):= \nabla^2_zH(\lambda,t, \bar{v}(t))$,
the following holds:
\begin{enumerate}
\item[\rm (i)]  $\Sigma_1:=\{\lambda\in\R\,|\, \nu_{1,\tau}(\gamma_\lambda)>0\}$ is a discrete set in $\R$.
\item[\rm (ii)] $(\mu, \bar{v})$ with $\mu\in\R$ is a bifurcation point  for (\ref{e:Pbrake})
if and only if $\nu_{1,\tau}(\gamma_\mu)>0$.
\item[\rm (iii)] For each $\mu\in\Sigma_1$ and a small enough $\rho>0$ it holds that
\begin{eqnarray*}
\mu_{1,\tau}(\gamma_{\lambda})=\left\{
\begin{array}{ll}
 \mu_{1,\tau}(\gamma_{\mu})\;&\forall\lambda\in [\mu-\rho,\mu),\\
 \mu_{1,\tau}(\gamma_{\mu})+ \nu_{1,\tau}(\gamma_{\mu}) \;&\forall\lambda\in (\mu, \mu+\rho]
\end{array}\right.
\end{eqnarray*}
if $\nabla^2_z\hat{H}(t, \bar{v}(t))>0$ for all $t$, and
\begin{eqnarray*}
\mu_{1,\tau}(\gamma_{\lambda})=\left\{
\begin{array}{ll}
 \mu_{1,\tau}(\gamma_{\mu})+ \nu_{1,\tau}(\gamma_{\mu}) \;&\forall\lambda\in [\mu-\rho,\mu),\\
 \mu_{1,\tau}(\gamma_{\mu})\;&\forall\lambda\in (\mu, \mu+\rho]
\end{array}\right.
\end{eqnarray*}
if $\nabla^2\hat{H}(t, \bar{v}(t))<0$ for all $t$.
\item[\rm (iv)]  One of (i)-(iii) in Theorem~\ref{th:bif-suffictbrake} with this $H$ and $v_\lambda\equiv\bar{v}$ occurs.
\item[\rm (v)] If $\bar{v}=0$ and all $H_0(t,\cdot), \hat{H}(t,\cdot)$ are even, then
for this $H$, either (i) or (iv) in Theorem~\ref{th:bif-suffictbrake} holds;
in the case of (iv), the solution is trivial, i.e., $v_\lambda\equiv 0$.
\end{enumerate}
\end{corollary}

\begin{theorem}[\textsf{Alternative bifurcations of Rabinowitz's type and of Fadell-Rabinowitz's type}]\label{th:bif-per3brake}
Let Assumption~\ref{ass:brake} with $\Lambda$ being a real interval  be satisfied.
Suppose that $H$ is independent of time $t$, and that $\bar{v}\in\R^{2n}$ is
a common critical point of all functions $H(\lambda,\cdot)$.
Assume also that $N\bar{v}=\bar{v}$.
Consequently,
for any $\tau>0$ and all $\lambda\in\Lambda$, the constant function
$\bar{v}(t)\equiv\bar{v}$ is an even solution of
\begin{equation}\label{e:Pbrake++}
\dot{v}(t)=J\nabla_z H(\lambda, v(t)),\quad v(t+\tau)=v(t)\quad\hbox{and}\quad v(-t)=Nv(t)\;\forall t\in\R.
\end{equation}
Let $\gamma_\lambda(t)=\exp(tJ\nabla^2_zH(\lambda, \bar{v}))$, i.e.,
 the fundamental matrix solution of the linearized system
$\dot{v}(t)=J\nabla^2_zH(\lambda, \bar{v})v(t)$.
Suppose that for some interior point $\mu$ of $\Lambda$  and some $\tau>0$, the following  holds.
\begin{enumerate}
\item[\rm (a)]  $\nu_{1,\tau}(\gamma_\mu)\ne 0$,  $\nu_{1,\tau}(\gamma_\lambda)=0$
 for each $\lambda\in\Lambda\setminus\{\mu\}$ near $\mu$, and
 $\mu_{1,\tau}(\gamma_\lambda)$ takes, respectively, values $\mu_{1,\tau}(\gamma_\mu)$ and $\mu_{1,\tau}(\gamma_\mu)+ \nu_{1,\tau}(\gamma_\mu)$
 as $\lambda\in\Lambda$ varies in  two deleted half neighborhoods  of $\mu$.
\item[\rm (b)] The following problem
\begin{equation}\label{e:linear3*brake}
\hbox{$\dot{v}(t)=J\nabla^2_zH(\mu, \bar{v}(t))v(t)$,\quad\hbox{$v(t+\tau)=v(t)$\quad\hbox{and}\quad $v(-t)=Nv(t)\;\forall t$}}
\end{equation}
has no nonzero even solutions.
 \end{enumerate}
 Then one of the following alternatives occurs:
 \begin{enumerate}
\item[\rm (i)] Equation (\ref{e:Pbrake++})  with $\lambda=\mu$ has a sequence of solution pairs, $\{v_k, Nv_k\}$, $k=1,2,\cdots$,
such that $v_k$ and $Nv_k$  converge to $\bar{v}$ on any compact interval $I\subset\R$ in $C^1$-topology.
\item[\rm (ii)] There exist left and right  neighborhoods $\Lambda^-$ and $\Lambda^+$ of $\mu$ in $\Lambda$
and integers $n^+, n^-\ge 0$, such that $n^++n^-\ge \nu_{1,\tau}(\gamma_\mu)$,
and for $\lambda\in\Lambda^-\setminus\{\mu\}$ (resp. $\lambda\in\Lambda^+\setminus\{\mu\}$),
(\ref{e:Pbrake++}) with parameter value $\lambda$  has at least $n^-$ (resp. $n^+$) distinct solution pairs
  $\{v_\lambda^i, Nv_\lambda^i\}$, $i=1,\cdots,n^-$ (resp. $n^+$),
such that all $v_\lambda^i$ and $Nv_\lambda^i$  converge to  $\bar{v}$ on any compact interval $I\subset\R$ in $C^1$-topology as $\lambda\to\mu$.
\end{enumerate}
Furthermore,  if $\bar{v}=0$,  each $H(\lambda,\cdot)$ is even
and the assumption (b) is removed, then
one of the following alternatives occurs:
 \begin{enumerate}
\item[\rm (iii)] Equation (\ref{e:Pbrake++})  with $\lambda=\mu$ has a sequence of solution quadruples, $\{v_k, -v_k, Nv_k, -Nv_k\}$, $k=1,2,\cdots$,
such that $v_k$  converges to $0$ on any compact interval $I\subset\R$ in $C^1$-topology.
\item[\rm (iv)] There exist left and right  neighborhoods $\Lambda^-$ and $\Lambda^+$ of $\mu$ in $\Lambda$
and integers $n^+, n^-\ge 0$, such that $n^++n^-\ge \nu_{1,\tau}(\gamma_\mu)$,
and for $\lambda\in\Lambda^-\setminus\{\mu\}$ (resp. $\lambda\in\Lambda^+\setminus\{\mu\}$),
(\ref{e:Pbrake++}) with parameter value $\lambda$  has at least $n^-$ (resp. $n^+$) distinct solution quadruples
  $\{v_\lambda^i, -v_\lambda^i, Nv_\lambda^i, -Nv_\lambda^i\}$, $i=1,\cdots,n^-$ (resp. $n^+$),
such that  $v_\lambda^i$  converges to  $0$ on any compact interval $I\subset\R$ in $C^1$-topology as $\lambda\to\mu$.
\end{enumerate}
\end{theorem}

\begin{remark}\label{rm:bif-per3brake+}
{\rm  For the first part in Theorem~\ref{th:bif-per3brake},  if we do not assume the condition (b),
 using \cite[Theorem~3.6]{Lu10} we can deduce that at least one among (i), (ii') and (iii') holds:
 \begin{enumerate}
\item[\rm (ii')]  For every $\lambda\in\Lambda\setminus\{\mu\}$ near $\mu$ there is a  solution ${v}_\lambda\ne \bar{v}$ of
(\ref{e:Pbrake++}) with parameter value $\lambda$, which
converges to $\bar{v}$ in $C^1(S_\tau;\R^{2n})$ as $\lambda\to \mu$.

\item[\rm (iii')]
For a given neighborhood $\mathcal{W}$ of $\bar{v}$ in $C^1(S_\tau;\R^{2n})$
there is a one-sided  neighborhood $\Lambda^0$ of $\mu$ such that
for any $\lambda\in\Lambda^0\setminus\{\mu\}$, (\ref{e:Pbrake++}) with parameter value $\lambda$
has at least two distinct solutions $v_\lambda^1\ne \bar{v}$ and $v_\lambda^2\ne \bar{v}$ in $\mathcal{W}$,
which can also be required to satisfy
\begin{eqnarray*}
\int^{\tau}_0\left[\frac{1}{2}(J\dot{v}_\lambda^1(t),v^1_\lambda(t))_{\mathbb{R}^{2n}}+ H(\lambda, {v}_\lambda^1(t))\right]dt
\ne \int^{\tau}_0\left[\frac{1}{2}(J\dot{v}^2_\lambda(t), v_\lambda^2(t))_{\mathbb{R}^{2n}}+ H(\lambda, {v}_\lambda^2(t))\right]dt
\end{eqnarray*}
provided that $\nu_{1,\tau}(\gamma_\mu)>1$ and (\ref{e:Pbrake++}) with parameter value $\lambda$
has only finitely many  solutions in $\mathcal{W}$.
\end{enumerate}
Clearly, these have already been implied in the first part of Theorem~\ref{th:bif-suffictbrake}.
Of course, for $v_\lambda$ in (ii'), $Nv_\lambda$ is also a solution of (\ref{e:Pbrake++})
with parameter value $\lambda$.
But we cannot affirm $Nv_\lambda\ne v_\lambda$. Similarly, if $\nu_{1,\tau}(\gamma_\mu)=1$
it is uncertain that
$v_\lambda^1\ne Nv_\lambda^2$ in (iii'). If $\nu_{1,\tau}(\gamma_\mu)>1$, (iii') implies that
(\ref{e:Pbrake++}) with parameter value $\lambda$  has at least two distinct solution pairs
 $\{v_\lambda^1, Nv_\lambda^2\}$ and
$\{v_\lambda^2, Nv_\lambda^2\}$ in $\mathcal{W}$.
}
\end{remark}

\begin{corollary}\label{cor:bif-per4brake}
 Let  $H_0, \hat{H}: {\R}^{2n}\to\R$ be $N$-invariant $C^2$-functions.
Suppose that $dH_0(0)=d\hat{H}(0)=0$  and
that either $\hat{H}''(0)>0$ or $\hat{H}''(0)<0$.
Then for $\gamma_\lambda(t)=\exp(tJH_0''(0)+ \lambda tJ\hat{H}''(0))$ and any $\tau>0$,
$\Sigma_\tau:=\{\lambda\in\R\,|\, \nu_{1,\tau}(\gamma_\lambda)>0\}$ is a discrete set in $\R$.
Moreover, for each $\mu\in\Sigma_\tau$
and a small enough $\rho>0$ it holds that
\begin{eqnarray}\label{e:case1brake}
\mu_{1,\tau}(\gamma_{\lambda})=\left\{
\begin{array}{ll}
 \mu_{1,\tau}(\gamma_{\mu})\;&\forall\lambda\in [\mu-\rho,\mu),\\
 \mu_{1,\tau}(\gamma_{\mu})+ \nu_{1,\tau}(\gamma_{\mu}) \;&\forall\lambda\in (\mu, \mu+\rho]
\end{array}\right.
\end{eqnarray}
if $\hat{H}''(0)>0\;\forall t\in [0,\tau]$, and
\begin{eqnarray}\label{e:case2brake}
\mu_{1,\tau}(\gamma_{\lambda})=\left\{
\begin{array}{ll}
 \mu_{1,\tau}(\gamma_{\mu})+ \nu_{1,\tau}(\gamma_{\mu}) \;&\forall\lambda\in [\mu-\rho,\mu),\\
 \mu_{1,\tau}(\gamma_{\mu})\;&\forall\lambda\in (\mu, \mu+\rho]
\end{array}\right.
\end{eqnarray}
if $\hat{H}''(0)<0\;\forall t\in [0,\tau]$.
Consequently, for each $\mu\in\Sigma_\tau$,
if (\ref{e:linear3*brake}) with $H(\mu, x)=H_0(x)+\mu\hat{H}(x)$ and $\bar{v}=0$ has no nonzero even solution,
 then the first conclusion of Theorem~\ref{th:bif-per3brake}
 holds  for (\ref{e:Pbrake++}) with $H(\lambda, x)=H_0(x)+\lambda\hat{H}(x)$.
 Suppose that both $H_0$ and $\hat{H}$ are also even. Then
 the second conclusion of Theorem~\ref{th:bif-per3brake}
 holds  for (\ref{e:Pbrake++}) with $H(\lambda, x)=H_0(x)+\lambda\hat{H}(x)$.
\end{corollary}

Corresponding to Corollary~\ref{cor:bif-per5} we have

\begin{corollary}\label{cor:bif-per5brake}
 Let $H: {\R}^{2n}\to\R$ be a $N$-invariant $C^2$-function satisfying
 $dH(0)=0$. Suppose that either ${H}''(0)>0$ or ${H}''(0)<0$.
 Then   $\Delta(H):=\{\lambda\in\mathbb{R}\setminus\{0\}\,|\, \nu_{1,1}(\gamma^H_\lambda)>0\}$
is a discrete set, where $\gamma^H_\lambda(t)=\exp(\lambda tJ{H}''(0))$.

 If ${H}''(0)>0$ and $\mu\in\Delta(H)\cap(0, \infty)$ is such that
 the following problem
\begin{equation}\label{e:linear3*brake+}
\hbox{$\dot{v}(t)=\mu JH''(0)v(t)$,\quad\hbox{$v(t+1)=v(t)$\quad\hbox{and}\quad $v(-t)=Nv(t)\;\forall t$}}
\end{equation}
has no nonzero even solutions, then  one of the following alternatives occurs:
 \begin{enumerate}
\item[\rm (i)] The problem
 \begin{equation}\label{e:PPer3.5brake}
\dot{v}(t)=J\nabla H(v(t)),\quad\hbox{$v(t+\mu)=v(t)$\quad\hbox{and}\quad $v(-t)=Nv(t)\;\forall t$}
\end{equation}
 has a sequence of solution pairs, $\{v_k, Nv_k\}$, $k=1,2,\cdots$,
such that $v_k$ and $Nv_k$  converge to $0$ on any compact interval $I\subset\R$ in $C^1$-topology.
\item[\rm (ii)] There exist left and right  neighborhoods $\Lambda^-$ and $\Lambda^+$ of $\mu$ in $\Lambda$
and integers $n^+, n^-\ge 0$, such that $n^++n^-\ge \nu_{1,1}(\gamma^H_\mu)$,
and for $\lambda\in\Lambda^-\setminus\{\mu\}$ (resp. $\lambda\in\Lambda^+\setminus\{\mu\}$),
the problem
 \begin{equation}\label{e:PPer3.5brake+}
\dot{v}(t)=J\nabla H(v(t)),\quad\hbox{$v(t+\lambda)=v(t)$\quad\hbox{and}\quad $v(-t)=Nv(t)\;\forall t$}
\end{equation}
 has at least $n^-$ (resp. $n^+$) distinct solution pairs
  $\{v_\lambda^i, Nv_\lambda^i\}$, $i=1,\cdots,n^-$ (resp. $n^+$),
such that all $v_\lambda^i$ and $Nv_\lambda^i$  converge to  $0$ on any compact interval $I\subset\R$ in $C^1$-topology as $\lambda\to\mu$.
\end{enumerate}
If ${H}''(0)<0$ and $\mu\in\Delta(H)\cap(-\infty,0)$ is such that
(\ref{e:linear3*brake+}) has no nonzero even solutions, then  one of the above (i) and (ii) with this $\mu$ holds.

Suppose further that $H$ is also even.
Then ``$\{v_k, Nv_k\}$'' in (i) and  ``$\{v_\lambda^i, Nv_\lambda^i\}$'' in (ii)
are replaced by ``$\{v_k, Nv_k, -v_k, -Nv_k\}$''  and  ``$\{v_\lambda^i, Nv_\lambda^i, -v_\lambda^i, -Nv_\lambda^i\}$'',
respectively.
\end{corollary}

As in Remark~\ref{rm:twoBifu} we can show that this result cannot be included in \cite[Theorem~9.12]{Ba1}.

\begin{example}\label{ex:twoBifuBrake}
{\rm Let $0<\varrho_1\le\cdots\le\varrho_n$ and $H:\R^{2n}\to\R$ be given by
$$
H(x_1,\cdots,x_n,y_1,\cdots,y_n)=\sum^n_{j=1}\frac{\varrho_j}{2}(x_j^2+ y_j^2).
$$
Then by \cite[Example~3.2]{LoZZ} $\gamma^H_\lambda(t):=\exp(\lambda tJ{H}''(0))$ is equal to
\begin{equation}\label{e:ex.4-}
\left(\begin{array}{cc}
             {\rm diag} (\cos(\varrho_1\lambda t), . . . , \cos(\varrho_n\lambda t)) &  -{\rm diag} (\sin(\varrho_1\lambda t), . . . , \sin(\varrho_n\lambda t)) \\
             {\rm diag} (\sin(\varrho_1\lambda t), . . . , \sin(\varrho_n\lambda t)) &  {\rm diag} (\cos(\varrho_1\lambda t), . . . , \cos(\varrho_n\lambda t)) \\
           \end{array}
         \right)
\end{equation}
and
\begin{equation}\label{e:ex.4}
\mu_{1,\tau}(\gamma^H_\lambda)=\mu_{2,\tau}(\gamma^H_\lambda)=n-k+ \sum^n_{i=1}\left[\frac{\lambda\varrho_i\tau}{\pi}\right],
\end{equation}
where $k=\sharp\{i\in\{1,\cdots,n\}\,|\, \lambda\varrho_i\tau=0\;{\rm mod}\;\pi\}$.
Moreover, by (\ref{e:2.4M}) it is easily proved that
\begin{eqnarray*}
\nu_{1,\tau}(\gamma^H_\lambda)=\nu_{2,\tau}(\gamma^H_\lambda)&=&\dim{\rm Ker}({\rm diag} (\sin(\varrho_1\lambda \tau/2), . . . , \sin(\varrho_n\lambda \tau/2)))\\
&=&\sharp\{i\in\{1,\cdots,n\}\,|\, \lambda\varrho_i\tau=0\;{\rm mod}\;2\pi\}.
\end{eqnarray*}
Taking $\tau=1$ we obtain $\Delta(H)\cap(0,\infty)=\cup^n_{i=1}\{\lambda>0\,|\, \lambda\varrho_i=0\;{\rm mod}\;2\pi\}$.
For $\mu\in \Delta(H)\cap(0,\infty)$, all nontrivial solutions of the problem
$$
\dot{v}(t)=\mu JH''(0)v(t)=\mu J{\rm diag}(\varrho_1,\cdots,\varrho_n, \varrho_1,\cdots,\varrho_n)v(t)\quad\hbox{and}\quad v(t+1)=v(t)
$$
have forms $v(t)=\gamma^H_\mu(t) c$, where components $c_i$ and $c_{n+i}$ in $c=(c_1,\cdots,c_{2n})^T\in\mathbb{R}^{2n}$
are equal to zeros if $\mu\varrho_i\notin 2\pi\mathbb{Z}$.
Since $NJ=-JN$ and $NH''(0)=H''(0)N$ we deduce that $Nv(t)=N\gamma^H_\mu(t) c=\gamma^H_\mu(-t)Nc$.
This implies that $Nv(t)=v(-t)\;\forall t$ if and only $Nc=c$, i.e., $c_1=\cdots=c_n=0$.
Thus if the above $v(t)$ satisfies $Nv(t)=v(-t)\;\forall t$, then
$$
v(t)=\big(-c_{n+1}\sin(\lambda\varrho_1t),\cdots, -c_{2n}\sin(\lambda\varrho_nt),
c_{n+1}\cos(\lambda\varrho_1t),\cdots, c_{2n}\cos(\lambda\varrho_1t)\big)^T,
$$
where $c_{n+i}=0$ if $\mu\varrho_i\notin 2\pi\mathbb{Z}$. Clearly, if it is also even then $v\equiv 0$.
Hence the problem (\ref{e:linear3*brake+}) with this $H$ has no nonzero even solutions.
Suppose that $H$ is also $N$-invariant. By Corollary~\ref{cor:bif-per5brake},
for each $\mu\in\Delta(H)\cap(0,\infty)$ one of the statements (i) and (ii) in Corollary~\ref{cor:bif-per5brake}
holds after $\nu_{1,1}(\gamma_\mu)$ is replaced by $\nu_{1,1}(\gamma^H_\mu)$.}
\end{example}

\begin{question}\label{que:Will}
{\rm For a positive definite matrix $A$ of order $2n$,
Williamson theorem gives rise to a symplectic matrix $S\in{\rm Sp}(2n,\R)$ such that
$S^TAS={\rm diag}(\varrho_1,\cdots,\varrho_n, \varrho_1,\cdots,\varrho_n)$,
where $0<\varrho_j\le\varrho_k$ for $j\le k$. Suppose $AN=NA$. Can the above symplectic matrix $S$ be chosen
to satisfy $NS=SN$ ?}
\end{question}

If this question is solved in an affirmative way, the conclusions of
Corollary~\ref{cor:bif-per5brake} in the case  ${H}''(0)>0$ are the same as those of
Example~\ref{ex:twoBifuBrake}.\\

%

\noindent{1.4. \;\bf Bifurcations for Hamiltonian trajectories connecting affine Lagrangian subspaces.}\;
Recall that an \textsf{\textsf{affine Lagrangian}} subspace of $\mathbb{R}^{2n}$ is a
subset of $\mathbb{R}^{2n}$ of the form $w + L$, where $w\in \mathbb{R}^{2n}$ is a fixed point
and $L$ is a Lagrangian subspace of $\mathbb{R}^{2n}$.

\begin{assumption}\label{ass:TwoPoint}
{\rm  For a real $\tau>0$ and a topological space $\Lambda$,
 let $H:\Lambda\times [0, \tau]\times{\R}^{2n}\to\R$ be a continuous function such that
   each $H(\lambda,t,\cdot):{\R}^{2n}\to\R$, $(\lambda,t)\in\Lambda\times [0, \tau]$,
   is $C^2$ and all  possible partial derivatives of $H$ depend continuously on
 $(\lambda, t, z)\in\Lambda\times [0, \tau]\times\mathbb{R}^{2n}$.
 Let $L$ and $L'$ be two  Lagrangian subspaces of $\mathbb{R}^{2n}$,
 and let $\Lambda\ni\lambda\mapsto w_\lambda\in\mathbb{R}^{2n}$ and $\Lambda\ni\lambda\mapsto w'_\lambda\in\mathbb{R}^{2n}$
 be two continuous maps.
For each $\lambda\in\Lambda$ let $u_\lambda:[0, \tau]\to\R^{2n}$ be a differentiable path
  satisfying the Hamiltonian boundary value problem
  \begin{equation}\label{e:TwoPoint1}
\left.\begin{array}{ll}
\dot{u}(t)=J\nabla_z H(\lambda,t, u(t))\;\forall t\in [0, \tau],\\
u(0)\in w_\lambda+L,\quad u(\tau)\in w'_\lambda+L'.
\end{array}\right\}
\end{equation}
 Suppose also that the map $\Lambda\times [0,\tau]\ni(\lambda,t)\to u_\lambda(t)\in \mathbb{R}^{2n}$
 is continuous.}
  \end{assumption}

As noted below Assumption~\ref{ass:BasiAss1} the above assumption implies that
 $\Lambda\times [0,\tau]\ni (\lambda,t)\mapsto \dot{u}_\lambda(t)\in\mathbb{R}^{2n}$ is also continuous.
But $u_\lambda$ cannot be assured to be $C^2$.

Under Assumptions~\ref{ass:TwoPoint}, we say
$(\mu, u_\mu)$  to be a \textsf{ bifurcation point along sequences} of the problem (\ref{e:TwoPoint1})
with respect to the trivial branch $\{(\lambda, u_\lambda)\,|\,\lambda\in\Lambda\}$
if there exists a sequence $(\lambda_k)\subset\Lambda$ converging to $\mu$ and
solutions $v_k\ne u_{\lambda_k}$ of (\ref{e:TwoPoint1}) with $\lambda=\lambda_k$ for each $k=1,2,\cdots$,
such that $v_k\to u_\mu$ in $C^1([0, \tau], \mathbb{R}^{2n})$.

 Let $\gamma_\lambda:[0,\tau]\to{\rm Sp}(2n)$ be the fundamental matrix solution of
$\dot{u}(t)=J\nabla^2_zH(\lambda,t, u_\lambda(t))u(t)$.
For any symplectic path $\gamma:[0,\tau]\to{\rm Sp}(2n)$ starting at the identity,
Liu-Wang-Lin \cite{LiuWL11}  introduced a $(L, L')$-index
 \begin{equation}\label{e:(L,L')-index}
 (i^{L'}_L(\gamma), \nu^{L'}_L(\gamma))\in\mathbb{Z}\times\{0,1,\cdots,2n\}
 \end{equation}
  where
$\nu^{L'}_L(\gamma)=\dim(\gamma(\tau)L\cap L')$.
When $L'=L$, $(i^{L'}_L(\gamma), \nu^{L'}_L(\gamma))$ is equal to $L$-index
$(i_L(\gamma), \nu_L(\gamma))$ of $\gamma$ introduced by Liu \cite{Liu07}. Let $L_0=\{0\}\times\mathbb{R}^n\subset\mathbb{R}^{2n}$ and let $O\in {\rm Sp}(2n,\mathbb{R})$
be an orthogonal symplectic matrix such that $OL_0=L$. According to Liu \cite{Liu07}
\begin{equation}\label{e:L-index}
i_L(\gamma)=i_{L_0}\left(\{O^{-1}\gamma(t)O\}_{0\le t\le\tau}\right).
\end{equation}
If this $\gamma$ is also piecewise smooth, the equality above
\cite[Remark~2.14]{Liu17} and \cite[(5.99)]{Liu17}
show that the right side is equal to
$$
\mu^{\rm CLM}(f)-n=\mu^{\rm CLM}(\tilde{f})-n,
$$
where $f(t)=(L_0,  O^{-1}\gamma(t)OL_0)$ and $\tilde{f}(t)=(L,  \gamma(t)L)$
for $0\le t\le\tau$, where  $\mu^{\rm CLM}(f)$ is the Cappell-Lee-Miller index of $f$ characterized by properties I-VI of \cite[pp. 127-128]{CLM}.

\begin{theorem}\label{th:bif-nessHam}
Under Assumption~\ref{ass:TwoPoint}, let  $\gamma_\lambda$ be as above.
\begin{description}
\item[(I)]{\rm (\textsf{Necessary condition}):}
If $(\mu, u_\mu)$  is a bifurcation point along sequences of the problem (\ref{e:TwoPoint1})
with respect to the {trivial branch} $\{(\lambda, u_\lambda)\,|\,\lambda\in\Lambda\}$,
i.e.,  there exists a sequence $(\lambda_k)\subset\Lambda$ converging to $\mu$ and
solutions $u^k\ne u_{\lambda_k}$ of (\ref{e:TwoPoint1}) with $\lambda=\lambda_k$
such that $u^k\to u_\mu$ in $C^0([0,\tau], \mathbb{R}^{2n})$,
then $\nu^{L'}_L(\gamma_\mu)\ne 0$.

  \item[(II)]{\rm (\textsf{Sufficient condition}):}
Let $\Lambda$ be first countable.
Suppose that for some $\mu\in\Lambda$,
there exist two sequences in  $\Lambda$ converging to $\mu$, $(\lambda_k^-)$ and
$(\lambda_k^+)$,  such that for each $k\in\mathbb{N}$,
$$
[i^{L'}_L(\gamma_{\lambda_k^-}), i^{L'}_L(\gamma_{\lambda_k^-})+\nu^{L'}_L(\gamma_{\lambda_k^-})]\cap[i^{L'}_L(\gamma_{\lambda_k^+}), i^{L'}_L(\gamma_{\lambda_k^+})+\nu^{L'}_L(\gamma_{\lambda_k^+})]=\emptyset
$$
 and either $\nu^{L'}_L(\gamma_{\lambda_k^+})=0$ or $\nu^{L'}_L(\gamma_{\lambda_k^-})=0$.
Let $\hat{\Lambda}:=\{\mu,\lambda^+_k, \lambda^-_k\,|\,k\in\mathbb{N}\}$.
  Then  $(\mu, u_\mu)$ is a bifurcation point  of (\ref{e:TwoPoint1})
  with respect to the branch $\{(\lambda, u_\lambda)\,|\,\lambda\in\hat\Lambda\}$
  (and so $\{(\lambda, u_\lambda)\,|\,\lambda\in\Lambda\}$).

  \item[(III)]{\rm (\textsf{Existence for bifurcations}):} Let $\Lambda$ be path-connected. Suppose that
  there exist two points $\lambda^+, \lambda^-\in\Lambda$ such that
  $[i^{L'}_L(\gamma_{\lambda^+}), i^{L'}_L(\gamma_{\lambda^+})+\nu^{L'}_L(\gamma_{\lambda^+})]\cap[i^{L'}_L(\gamma_{\lambda^-}), i^{L'}_L(\gamma_{\lambda^-})+\nu^{L'}_L(\gamma_{\lambda^-})]=\emptyset$,
 and either $\nu^{L'}_L(\gamma_{\lambda^+})=0$ or $\nu^{L'}_L(\gamma_{\lambda^-})=0$.
  Then for any path $\alpha:[0,1]\to\Lambda$ connecting $\lambda^+$ to $\lambda^-$
   there exists a sequence $(t_k)\subset [0, 1]$ converging to some $\bar{t}$
   and solutions $u^k\ne u_{\alpha(t_k)}$ of (\ref{e:TwoPoint1})
   with $\lambda=\alpha(t_k)$, $k=1,2,\cdots$,
  such that  $(u^k)$ converges to  ${u}_{\alpha(\bar{t})}$  in $C^1$-topology as $k\to\infty$.
   Moreover,  $\alpha(\bar{t})$ is not equal to $\lambda^+$ (resp. $\lambda^-$)
    if $\nu^{L'}_L(\gamma_{\lambda^+})=0$ (resp. $\nu^{L'}_L(\gamma_{\lambda^-})=0$).
 \end{description}
\end{theorem}

\begin{theorem}[\textsf{Alternative bifurcations of Rabinowitz's type and of Fadell-Rabinowitz's type}]\label{th:bif-suffHam}
Let Assumption~\ref{ass:TwoPoint} with $\Lambda$ being a real interval be  satisfied,
and let $\gamma_\lambda$ be as in Theorem~\ref{th:bif-nessHam} for each $\lambda\in\Lambda$.
Suppose that for some $\mu\in{\rm Int}(\Lambda)$,
  $\nu^{L'}_L(\gamma_\mu)\ne 0$
   and  $\nu^{L'}_L(\gamma_\lambda)=0$  for each $\lambda\in\Lambda\setminus\{\mu\}$ near $\mu$, and that
  $i^{L'}_L(\gamma_{\lambda})$ take, respectively, values $i^{L'}_L(\gamma_{\mu})$ and
  $i^{L'}_L(\gamma_{\mu})+ \nu^{L'}_L(\gamma_{\mu})$
 as $\lambda\in\Lambda$ varies in two deleted half neighborhoods  of $\mu$.
Then  one of the following alternatives occurs:
\begin{enumerate}
\item[\rm (i)] The problem (\ref{e:TwoPoint1})
 with $\lambda=\mu$ has a sequence of solutions, $v_k\ne u_\mu$, $k=1,2,\cdots$,
which converges to $u_\mu$ in $C^1([0, \tau], \mathbb{R}^{2n})$.

\item[\rm (ii)]  For every $\lambda\in\Lambda\setminus\{\mu\}$ near $\mu$ there is a
solution $v_\lambda\ne u_\lambda$ of
the problem (\ref{e:TwoPoint1}) with parameter value $\lambda$, such that
 $v_\lambda-u_\lambda$ converges to zero in  $C^1([0, \tau], \mathbb{R}^{2n})$ as $\lambda\to \mu$.

\item[\rm (iii)] For a given neighborhood $\mathcal{W}$ of $u_\mu$ in $C^1([0, \tau];\R^{2n})$ there is an one-sided
 neighborhood $\Lambda^0$ of $\mu$ such that for any $\lambda\in\Lambda^0\setminus\{\mu\}$, the problem (\ref{e:TwoPoint1})
  with parameter value $\lambda$ has  at least two distinct solutions, $v_\lambda^1\ne u_\lambda$ and $v_\lambda^2\ne u_\lambda$ in $\mathcal{W}$,
 which can also be required to satisfy
 \begin{eqnarray*}
&&\int^{\tau}_0\left[\frac{1}{2}(J\dot{v}_\lambda^1(t), v^1_\lambda(t))_{\mathbb{R}^{2n}}+ H(\lambda, t, {v}_\lambda^1(t))\right]dt\\
&&\ne \int^{\tau}_0\left[\frac{1}{2}(J\dot{v}^2_\lambda(t),v_\lambda^2(t))_{\mathbb{R}^{2n}}+ H(\lambda, t,
{v}_\lambda^2(t))\right]dt;
\end{eqnarray*}
 provided that  $\nu^{L'}_L(\gamma_\mu)>1$ and (\ref{e:TwoPoint1}) with parameter value $\lambda$
has only finitely many  solutions in $\mathcal{W}$.
\end{enumerate}
Moreover, if $u_\lambda=0\;\forall\lambda$, and all $H(\lambda,t,\cdot)$ are even,
then either (i) holds or the following  occurs:
\begin{enumerate}
\item[\rm (iv)] There exist left and right  neighborhoods $\Lambda^-$ and $\Lambda^+$ of $\mu$ in $\Lambda$
and integers $n^+, n^-\ge 0$, such that $n^++n^-\ge \nu^{L'}_L(\gamma_\mu)$,
and for $\lambda\in\Lambda^-\setminus\{\mu\}$ (resp. $\lambda\in\Lambda^+\setminus\{\mu\}$),
(\ref{e:TwoPoint1}) with parameter value $\lambda$  has at least $n^-$ (resp. $n^+$) distinct pairs of nontrivial solutions,
$\{u_\lambda^i, -u_\lambda^i\}$, $i=1,\cdots,n^-$ (resp. $n^+$),
which  converge to zero in $C^1([0, \tau];\R^{2n})$  as $\lambda\to\mu$.
\end{enumerate}
\end{theorem}

By \cite[Lemma~1.6]{LiuWL11} Theorems~\ref{th:bif-nessbrake},\ref{th:bif-suffictbrake}
can also be derived from the above two results.

\begin{example}[\textsf{Bifurcation of the Sturm-Liouville problem}]\label{ex:Sturm}
{\rm Let $\Lambda$ be a topological space, let $P \in C\left(\Lambda\times [0,\tau], \mathbb{R}^{n\times n}\right)$
and $V\in C(\Lambda\times [0, \tau]\times{\R}^{n})$ satisfy the following conditions:
\begin{enumerate}
\item[\rm (i)] Each $P(\lambda,t)$, $(\lambda,t)\in\Lambda\times [0,\tau]$, is positive definite.
 \item[\rm (ii)] For each  $(\lambda,t)\in\Lambda\times [0, \tau]$,
 $V(\lambda,t,\cdot):{\R}^{n}\to\R$ is $C^2$ and all
  possible partial derivatives of $V$ depend continuously on
 $(\lambda, t, q)\in\Lambda\times [0, \tau]\times\mathbb{R}^{n}$.
\end{enumerate}
Given two reals $0 \leqslant \alpha, \beta \leqslant \pi$,
and $\lambda\in\Lambda$,
a differentiable path $q:[0, \tau]\to\R^{n}$ satisfies the following Sturm-Liouville problem
\begin{equation}\label{e:Sturm}
\left.\begin{array}{l}
\frac{d}{dt}\left(P(\lambda, t) \dot{q}\right)-\nabla_qV(\lambda, t, q)=0, \\
 (\cos\alpha)q(0)- (\sin\alpha)P(\lambda, 0) \dot{q}(0)=0, \\
(\cos\beta)q(\tau) - (\sin\beta)P(\lambda, \tau) \dot{q}(\tau)=0
\end{array}\right\}
\end{equation}
if and only if it is a critical point of the functional
$$
\Phi_\lambda(q)=\int^\tau_0\left[\frac{1}{2}(P(\lambda,t)\dot{q}(t),\dot{q}(t))-V(\lambda,t,q(t))\right]dt
$$
on the Hilbert subspace $W^{1,2}_{\alpha,\beta}([0,\tau],\mathbb{R}^n)$ consisting of
$q\in W^{1,2}([0,\tau],\mathbb{R}^n)$ satisfying the boundary conditions in (\ref{e:Sturm}).
(In this case we may deduce that $t\mapsto P(\lambda,t)\dot{q}(t)$  (and thus $q$) is $C^1$.
But we cannot obtain that $q$ is $C^2$ if $t\mapsto P(\lambda,t)\dot{q}(t)$ is not $C^1$.)
These functionals $\Phi_\lambda$ are $C^2$ and have finite Morse indexes and nullities
at all critical points of them.
For each $\lambda\in\Lambda$,  let $q_\lambda:[0, \tau]\to\R^{n}$ be
a  differentiable path  satisfying  (\ref{e:Sturm}).
(From the conditions (i)-(ii) and (\ref{e:Sturm}) we derive that
$t\mapsto P(\lambda, t) \dot{q}(t)$ is $C^1$. This can only lead to
the $C^1$-smoothness of $q$ since we have not assumed that $t\mapsto P(\lambda,t)$
is differentiable.) Denote by $m^-(\Phi_\lambda, q_\lambda)$ and $m^0(\Phi_\lambda, q_\lambda)$
the Morse index and nullity of $\Phi_\lambda$ at $q_\lambda$.
For some $\mu\in\Lambda$ we call $(\mu, q_\mu)$  to be a \textsf{ bifurcation point along sequences}
of the problem (\ref{e:Sturm}) with respect to the trivial branch $\{(\lambda, q_\lambda)\,|\,\lambda\in\Lambda\}$ if there exists a sequence $(\lambda_k)\subset\Lambda$
converging to $\mu$ and solutions $y_k\ne q_{\lambda_k}$ of (\ref{e:TwoPoint1}) with $\lambda=\lambda_k$
for $k=1,2,\cdots$, such that $y_k\to q_\mu$ in $C^1([0, \tau], \mathbb{R}^{2n})$.

Suppose also that $(\lambda,t)\mapsto q_\lambda(t)$ is continuous.
From Theorems~\ref{th:bif-nessHam},~\ref{th:bif-suffHam}
we can obtain:
\begin{description}
\item[(I)]{\rm (\textsf{Necessary condition}):}
If $(\mu, q_\mu)$  is a bifurcation point along sequences of the problem (\ref{e:Sturm})
with respect to the {trivial branch} $\{(\lambda, q_\lambda)\,|\,\lambda\in\Lambda\}$,
then $m^0(\Phi_\mu, q_\mu)>0$.

\item[(II)]{\rm (\textsf{Sufficient condition}):}
Let $\Lambda$ be first countable.
Suppose that for some $\mu\in\Lambda$,
there exist two sequences in  $\Lambda$ converging to $\mu$, $(\lambda_k^-)$ and
$(\lambda_k^+)$,  such that
for each $k\in\mathbb{N}$,
{\small
$$
[m^-(\Phi_{\lambda_k^-}, q_{\lambda_k^-}), m^-(\Phi_{\lambda_k^-}, q_{\lambda_k^-})+
m^0(\Phi_{\lambda_k^-}, q_{\lambda_k^-})]\cap
[m^-(\Phi_{\lambda_k^+}, q_{\lambda_k^+}), m^-(\Phi_{\lambda_k^+}, q_{\lambda_k^+})+
m^0(\Phi_{\lambda_k^+}, q_{\lambda_k^+})]=\emptyset
$$}
 and either $m^0(\Phi_{\lambda_k^+}, q_{\lambda_k^+})=0$ or $m^0(\Phi_{\lambda_k^-}, q_{\lambda_k^-})=0$.
Let $\hat{\Lambda}:=\{\mu,\lambda^+_k, \lambda^-_k\,|\,k\in\mathbb{N}\}$.
  Then  $(\mu, q_\mu)$ is a bifurcation point  of (\ref{e:Sturm})
  with respect to the branch $\{(\lambda, q_\lambda)\,|\,\lambda\in\hat\Lambda\}$
  (and so $\{(\lambda, q_\lambda)\,|\,\lambda\in\Lambda\}$).

   \item[(III)]{\rm (\textsf{Existence for bifurcations}):} Let $\Lambda$ be path-connected. Suppose that
  there exist two  points $\lambda^+, \lambda^-\in\Lambda$ such that
 {\small
  $$
[m^-(\Phi_{\lambda^+}, q_{\lambda^+}), m^-(\Phi_{\lambda^+}, q_{\lambda^+})+
m^0(\Phi_{\lambda^+}, q_{\lambda^+})]\cap
[m^-(\Phi_{\lambda^-}, q_{\lambda^-}), m^-(\Phi_{\lambda^-}, q_{\lambda^-})+
m^0(\Phi_{\lambda^-}, q_{\lambda^-})]=\emptyset
$$}
  and either $m^0(\Phi_{\lambda^+}, q_{\lambda^+})=0$ or $m^0(\Phi_{\lambda^-}, q_{\lambda^-})=0$.
   Then for any path $\alpha:[0,1]\to\Lambda$ connecting $\lambda^+$ to $\lambda^-$
   there exists a sequence $(t_k)\subset [0, 1]$ converging to some $\bar{t}$
   and solutions $q^k\ne q_{\alpha(t_k)}$ of (\ref{e:Sturm})  with $\lambda=\alpha(t_k)$, $k=1,2,\cdots$,
  such that  $(q^k)$ converges to  ${q}_{\alpha(\bar{t})}$  in $C^1$-topology as $k\to\infty$.
   Moreover,  $\alpha(\bar{t})$ is not equal to $\lambda^+$ (resp. $\lambda^-$)
    if $m^0(\Phi_{\lambda^+}, q_{\lambda^+})=0$ (resp. $m^0(\Phi_{\lambda^-}, q_{\lambda^-})=0$).

  \item[(IV)]{\rm (\textsf{Alternative bifurcations of Rabinowitz's type and of Fadell-Rabinowitz's type}):}
   If $\Lambda$ is a real interval,  $\mu\in{\rm Int}(\Lambda)$ is such that
 $m^0(\Phi_\mu, q_\mu)>0$  and $m^0(\Phi_\lambda, q_\lambda)=0$  for each $\lambda\in\Lambda\setminus\{\mu\}$ near $\mu$, and that
  $m^-(\Phi_\lambda, q_\lambda)$ take, respectively, values $m^-(\Phi_\mu, q_\mu)$ and
  $m^-(\Phi_\mu, q_\mu)+ m^0(\Phi_\mu, q_\mu)$
 as $\lambda\in\Lambda$ varies in two deleted half neighborhoods  of $\mu$.
Then at least one of the following alternatives occurs:
\begin{enumerate}
\item[\rm (IV.1)] The problem (\ref{e:Sturm})
 with $\lambda=\mu$ has a sequence of solutions, $y_k\ne q_\mu$, $k=1,2,\cdots$,
which converges to $q_\mu$ in $C^1([0, \tau], \mathbb{R}^{n})$.

\item[\rm (IV.2)]  For every $\lambda\in\Lambda\setminus\{\mu\}$ near $\mu$ there is a
solution $y_\lambda\ne q_\lambda$ of
the problem (\ref{e:Sturm}) with parameter value $\lambda$, such that
 $y_\lambda-q_\lambda$ converges to zero in  $C^1([0, \tau], \mathbb{R}^{n})$ as $\lambda\to \mu$.

\item[\rm (IV.3)] For a given neighborhood $\mathcal{W}$ of $q_\mu$ in $C^1([0, \tau];\R^{n})$
there is a one-sided
 neighborhood $\Lambda^0$ of $\mu$ such that for any $\lambda\in\Lambda^0\setminus\{\mu\}$, the problem (\ref{e:Sturm}) with parameter value $\lambda$
has  at least two distinct solutions, $y_\lambda^1\ne q_\lambda$ and $y_\lambda^2\ne q_\lambda$ in $\mathcal{W}$,
which can also be required to satisfy $\Phi_\lambda(y_\lambda^1)\ne\Phi_\lambda(y_\lambda^2)$
  provided that  $m^0(\Phi_\mu, q_\mu)>1$ and (\ref{e:Sturm}) with parameter value $\lambda$
has only finitely many  solutions in $\mathcal{W}$.
\end{enumerate}
Moreover, if $q_\lambda=0\;\forall\lambda$, and all $V(\lambda,t,\cdot)$ are even,
then either (IV.1) holds or the following  occurs:
\begin{enumerate}
\item[\rm (IV.4)] There exist left and right  neighborhoods $\Lambda^-$ and $\Lambda^+$ of $\mu$ in $\Lambda$
and integers $n^+, n^-\ge 0$, such that $n^++n^-\ge m^0(\Phi_\mu, q_\mu)$,
and for $\lambda\in\Lambda^-\setminus\{\mu\}$ (resp. $\lambda\in\Lambda^+\setminus\{\mu\}$),
(\ref{e:Sturm}) with parameter value $\lambda$  has at least $n^-$ (resp. $n^+$) distinct pairs of nontrivial solutions,
$\{u_\lambda^i, -u_\lambda^i\}$, $i=1,\cdots,n^-$ (resp. $n^+$),
which  converge to zero in $C^1([0, \tau];\R^{2n})$  as $\lambda\to\mu$.
\end{enumerate}
   \end{description}
Indeed, let $H(\lambda, t, p,q)=1 / 2\left(P(\lambda, t)^{-1}p, p\right)_{\mathbb{R}^{n}}-V(\lambda, t, q)$.
Then the conditions (i)-(ii) implies that $H$ satisfies Assumption~\ref{ass:TwoPoint}.
Clearly, $q(t)$ satisfies (\ref{e:Sturm}) if and only if $u(t):=(p(t)^T, q(t)^T)^T$ with $p(t)=P(\lambda, t)\dot{q}(t)$ satisfies
 \begin{equation}\label{e:Sturm1}
\left.\begin{array}{ll}
\dot{u}(t)=J\nabla_z H(\lambda,t, u(t))\;\forall t\in [0, \tau],\\
u(0)\in L_\alpha,\quad u(\tau)\in L_\beta,
\end{array}\right\}
\end{equation}
 where $L_\xi=\left\{(p^T, q^T)^T\in\mathbb{R}^{2n}
  \,\big|\, (\cos\xi)q- (\sin\xi)p=0, p, q\in \mathbb{R}^{n}\right\}$ for $\xi\in\mathbb{R}$.
 In particular, $u_\lambda(t):=(p_\lambda(t)^T, q_\lambda(t)^T)^T$ with $p_\lambda(t)=P(\lambda, t)\dot{q}_\lambda(t)$
 satisfies (\ref{e:Sturm1}), and $(\mu, q_\mu)$ is
 a  bifurcation point of the problem
(\ref{e:Sturm}) with respect to the trivial branch $\{(\lambda, q_\lambda)\,|\,\lambda\in\Lambda\}$
if and only if $(\mu, u_\mu)$ is that of (\ref{e:Sturm1}) with respect to the trivial branch $\{(\lambda, u_\lambda)\,|\,\lambda\in\Lambda\}$.
Let $\gamma_\lambda:[0,\tau]\to{\rm Sp}(2n)$ be the fundamental matrix solution of
$\dot{u}(t)=J\nabla^2_zH(\lambda,t, u_\lambda(t))u(t)$. It was stated in \cite[page 66, line 2]{LiuWL11}
that $i^{L_\beta}_{L_\alpha}(\gamma_\lambda)$ and $\nu^{L_\beta}_{L_\beta}(\gamma_\lambda)$
are equal to $m^-(\Phi_\lambda, q_\lambda)$ and $m^0(\Phi_\lambda, q_\lambda)$, respectively.
The above conclusions in (I)-(III) immediately follow from Theorems~\ref{th:bif-nessHam}, \ref{th:bif-suffHam}.
}
\end{example}

{\it Note}: Since the Maslov-type index in (\ref{e:(L,L')-index}) is homotopy invariant
under certain path deformations, it is easier to compute
$i^{L_\beta}_{L_\alpha}(\gamma_\lambda)$ and $\nu^{L_\beta}_{L_\beta}(\gamma_\lambda)$
 than to compute $m^-(\Phi_\lambda, q_\lambda)$ and $m^0(\Phi_\lambda, q_\lambda)$.
Therefore, it is better to replace the latter with the former
in the  statements (I)--(IV) of Example~\ref{ex:Sturm}.

\begin{example}[\textsf{Bifurcation of the Bolza Problem}]\label{ex:Bolza}
{\rm Split $z\in\mathbb{R}^{2n}$ into $(p^T, q^T)^T$, with $p$ and $q \in \mathbb{R}^{n}$. Fix two points $q_{0}, q_{1} \in \mathbb{R}^{n}$.
In Assumption~\ref{ass:TwoPoint} choose $w=(0, q_0^T)^T$, $w'=(0, q_1^T)^T$ and $L=L'=\{(p^T, q^T)^T\,|\,p\in\mathbb{R}^n, q=0\}$.
(\ref{e:TwoPoint1}) becomes  the following Bolza bifurcation problem:
$$\left.\begin{array}{l}
\dot{p}=-H_{q}^{\prime}(\lambda, t, p, q) \\
\dot{q}=H_{p}^{\prime}(\lambda, t, p, q) \\
q(0)=q_{0} \\
q(\tau)=q_{1}.
\end{array}\right\}
$$
}
\end{example}

Our final result is an analogy of Corollary~\ref{cor:necess-suffi}.
Let $H:{\R}^{2n}\to\R$ be $C^2$, and let $u:[0,\tau]\to {\R}^{2n}$ satisfy $\dot{u}(t)=J\nabla_z H(u(t))$.
For two Lagrangian subspaces $L$ and $L'$ of $\mathbb{R}^{2n}$, and $0<s\le\tau$,
 we say $u(s)$ to be \textsf{$(L,L')$-conjugate to $u(0)$ along $u$} if
$\nu^{L'}_L(\gamma)>0$ for the fundamental matrix solution  $\gamma:[0, s]\to{\rm Sp}(2n)$  of
$\dot{v}(t)=J\nabla^2_zH(u(t))v(t)$ on $[0,s]$. The number $\nu^{L'}_L(\gamma)$
is called the \textsf{multiplicity} of $u(s)$ as a $(L,L')$-conjugate point to $u(0)$ along $u$.
In particular, when $L'=L$,  $u(s)$ is said to be \textsf{$L$-conjugate to $u(0)$ along $u$} if
$\nu_L(\gamma)>0$ for the above $\gamma$, and the number $\nu_L(\gamma)$ is
called the \textsf{multiplicity} of $u(s)$.
We call $\mu\in (0, \tau]$ a \textsf{bifurcation instant for $(H, u, L)$} if there exists a sequence $(\tau_k)\subset (0,\tau]$ converging to $\mu$
such that for each $k\in\mathbb{N}$ the boundary value problem
\begin{equation}\label{e:TwoPointMu}
\left.\begin{array}{ll}
\dot{v}(t)=J\nabla_z H(v(t))\;\forall t\in [0, \tau_k],\\
v(0)\in u(0)+L,\quad v(\tau_k)\in u(\tau_k)+ L
\end{array}\right\}
\end{equation}
has a solution $v_k\ne u|_{[0,\tau_k]}$ such that $\|v_k-u\|_{C^1([0,\tau_k],\mathbb{R}^{2n})}\to 0$ as $k\to\infty$.

\begin{theorem}\label{th:two-point}
Let $H:{\R}^{2n}\to\R$ be $C^2$,  let $u:[0,\tau]\to {\R}^{2n}$ satisfy $\dot{u}(t)=J\nabla_z H(u(t))$,
 let  $\gamma:[0, \tau]\to{\rm Sp}(2n)$ be the fundamental matrix solution   of
$\dot{v}(t)=J\nabla^2_zH(u(t))v(t)$ on $[0, \tau]$, and let $L$ be a Lagrangian subspace  of $\mathbb{R}^{2n}$. Then:
 \begin{description}
  \item[(A)] If $\mu\in (0, \tau]$ is a bifurcation instant for $(H, u, L)$, then $u(\mu)$ is $L$-conjugate to $u(0)$ along $u$.
  \item[(B)] Suppose that $\nabla^2_zH(u(t))$ is positive definite for all $t\in [0, \tau]$.
Then
$$
\Gamma(H, u, L):=\{\lambda\in (0, \tau)\,|\, \dim (\gamma(\lambda)L)\cap L>0\}
$$
contains at most finitely many points, that is, there exist at most finitely
 many $L$-conjugate points to $u(0)$ along $u$; and  each $\mu\in \Gamma(H, u, L)$
is a bifurcation instant for $(H, u, L)$, precisely
  one of the following alternatives occurs:
  \begin{enumerate}
\item[\rm (B.i)] The problem
\begin{equation}\label{e:TwoPointMu1}
\left.\begin{array}{ll}
\dot{v}(t)=J\nabla_z H(v(t))\;\forall t\in [0, \mu],\\
v(0)\in u(0)+L,\quad v(\mu)\in u(\mu)+ L
\end{array}\right\}
\end{equation}
 has a sequence of distinct solutions, $v_k\ne u|_{[0,\mu]}$, $k=1,2,\cdots$,
such that $v_k\to u|_{[0,\mu]}$ in $C^1([0,\mu],\mathbb{R}^{2n})$ as $k\to\infty$.

\item[\rm (B.ii)]  For every $\lambda\in (0, \tau)\setminus\{\mu\}$ near $\mu$, the problem
\begin{equation}\label{e:TwoPointMu2}
\left.\begin{array}{ll}
\dot{v}(t)=J\nabla_z H(v(t))\;\forall t\in [0, \lambda],\\
v(0)\in u(0)+L,\quad v(\lambda)\in u(\lambda)+ L
\end{array}\right\}
\end{equation}
 has a solutions $v_\lambda\ne u|_{[0,\mu]}$
such that $\|v_\lambda-u\|_{C^1([0,\lambda],\mathbb{R}^{2n})}\to 0$ as $\lambda\to\mu$.

\item[\rm (B.iii)] There is an one-sided  neighborhood $\Lambda^0$ of $\mu$ in $[0, \tau]$ such that
for any $\lambda\in\Lambda^0\setminus\{\mu\}$  the problem (\ref{e:TwoPointMu2})
has at least two distinct solutions, $v_\lambda^i\ne u|_{[0,\lambda]}$, $i=1,2$,
to satisfy the condition that  $\|v^i_\lambda-u|_{[0,\lambda]}\|_{C^1([0,\lambda],\mathbb{R}^{2n})}\to 0$  as $\lambda\to \mu$, $i=1,2$.
Moreover, if $\dim (\gamma(\lambda)L)\cap L>1$ and the problem (\ref{e:TwoPointMu2})
has only finitely many distinct solutions near $u|_{[0,\lambda]}$ in $C^1([0,\lambda],\mathbb{R}^{2n})$
then the above two distinct solutions $v_\lambda^i$ can also be chosen to satisfy
 \begin{eqnarray*}
\int^{\lambda}_0\left[\frac{1}{2}(J\dot{v}_\lambda^1(t), v^1_\lambda(t))_{\mathbb{R}^{2n}}+ H({v}_\lambda^1(t))\right]dt
\ne \int^{\lambda}_0\left[\frac{1}{2}(J\dot{v}^2_\lambda(t),v_\lambda^2(t))_{\mathbb{R}^{2n}}+ H({v}_\lambda^2(t))\right]dt.
\end{eqnarray*}
\end{enumerate}
\end{description}
Finally, if $L$ is equal to $L_0:=\{0\}\times\mathbb{R}^n\subset\mathbb{R}^{2n}$ (resp. $L_1:=\mathbb{R}^n\times\{0\}\subset\mathbb{R}^{2n}$),
the condition ``$\nabla^2_zH(u(t))$ is positive definite for all $t\in [0, \tau]$'' in (B)
 can be substituted with the condition ``$B_{22}(t)$ (resp. $B_{11}(t)$)
 is positive definite for all $t\in [0, \tau]$'', where
 $\nabla^2_zH(u(t))=\left(\begin{array}{cc}
             B_{11}(t) & B_{12}(t) \\
            B_{21}(t) & B_{22}(t) \\
           \end{array}\right)$,  $B_{ij}(t)\in\mathbb{R}^{n\times n}$, $i,j=1, 2$.
\end{theorem}

\begin{remark}\label{rm:free-time}
{\rm Let $L$ and $L'$ be two  Lagrangian subspaces of $\mathbb{R}^{2n}$,
 and  let $G:{\R}^{2n}\to\R$ be a $C^2$ function.
  Motivated by the notion of an orbit cylinder (\cite[Definition~1.5]{MerPat})
 a Hamiltonian path $v:[0, \tau]\to\mathbb{R}^{2n}$ of $X_G=J\nabla G$ connecting $L$ and $L'$
 is said to admit a \textsf{path trapezoid}
if there exist $\varepsilon>0$ together with a family $\mathcal{O}=(v_\lambda)_{\lambda\in(-\varepsilon,\varepsilon)}$ of
Hamiltonian paths $v_\lambda:[0, T(\lambda)]\to \mathbb{R}^{2n}$ of $X_G$ connecting $L$ and $L'$
 with $v_0=v$ and $T(0)=\tau$ such that
$(-\varepsilon, \varepsilon)\ni \lambda\mapsto T(\lambda)\in\mathbb{R}$
is continuous and $G(v_\lambda)=G(v)+\lambda$ for all $\lambda\in (-\varepsilon, \varepsilon)$. Define
$$
{H}:(-\varepsilon,\varepsilon)\times\mathbb{R}^{2n}\to\mathbb{R},\;(\lambda,z)\mapsto \frac{T(\lambda)}{\tau}G(z)
$$
and $u_\lambda:[0, \tau]\to\mathbb{R}^{2n}$ by $u_\lambda(t)=v_\lambda(\frac{T(\lambda)}{\tau}t)$ for $\lambda\in(-\varepsilon, \varepsilon)$.
Then $H$, $u_\lambda$  satisfy Assumption~\ref{ass:TwoPoint} with
 $\Lambda=(-\varepsilon,\varepsilon)$ and $w_\lambda=0=w'_\lambda$ for all $\lambda$.
Applying Theorems~\ref{th:bif-nessHam},~\ref{th:bif-suffHam} to them some interesting results may be obtained.}
\end{remark}

\noindent{1.5. \;\bf Research methods.}\;
Our proofs for the above results in Sections 1.1-1.4 are based on
the abstract bifurcation theory recently developed by author in \cite{Lu8} and \cite{Lu10}
through Morse theory methods.
These theorems cannot be used for the associated functionals as in (\ref{e:1.5.1})
because  all $\Phi_\lambda$ are strongly indefinite, that is, each critical point of them
 has infinite Morse index and co-index. One cannot directly study bifurcations of
 $\nabla\Phi_\lambda=0$  by changes of the Morse indexes of the Hessians $\nabla^2\Phi_\lambda$
 at given solutions.  There are some ways to overcome this difficulty.
 Using spectral flow methods as in  \cite{FiPeRe00}, \cite{IJWa}, \cite{Rad10},
 \cite{RadRy10}, \cite{Szu} and \cite{Wat15},
 we can only obtain necessary conditions and sufficient criteria,
 and cannot get alternative results of Rabinowitz type.
   There exists a powerful way to arrive at the latter aim,
   the saddle point reduction by Amann and Zehnder (\cite{Am}, \cite{AmZe}).
   This method  is to reduce the original problem to a corresponding one in the finite dimension space
 via a Lyapunov-Schmidt reduction.   For example, references \cite{Ba1}, \cite{Ba2}
 and \cite{MaWi} studied bifurcations of periodic solutions of
 Hamiltonian systems via this method. A detailed process will be demonstrated
  in the proofs of Theorems~\ref{th:bif-nessHam}, \ref{th:bif-suffHam} and \ref{th:two-point}
  in Section~\ref{sec:twoLagr}.

Though the same methods can be used to prove the results in Sections 1.1-1.3,
 because of the lower smoothness for variational functionals in our
  abstract bifurcation theorems in \cite{Lu8} and \cite{Lu10}
it is possible to prove them via another way-----the  dual variational principle by Clarke and Ekeland.
In order to show our ideas let us consider the case of Theorem~\ref{th:bif-ness}.
Under Assumption~\ref{ass:BasiAss1} each $u_\lambda$ is $C^2$ and continuously depends on $\lambda$.
As done above the proof of Theorem~\ref{th:bif-ness}(I) in Section~\ref{sec:HamBif}
the problem is reduced to the case:
\begin{equation}\label{e:reductionCase}
 \hbox{$u_\lambda\equiv 0\;\forall\lambda$ and $H$ satisfies (\ref{e:modifyH}) and so (\ref{e:modifyH1}).}
\end{equation}
 Consider the Hilbert subspace
\begin{equation}\label{e:M-SobolevSpace}
W^{1,2}_{M}([0,\tau];\R^{2n}):=\{u\in W^{1,2}([0,\tau];\mathbb{R}^{2n})\,|\, u(\tau)=Mu(0)\}
\end{equation}
of  $W^{1,2}([0,\tau];\mathbb{R}^{2n})$,
and $C^2$  functionals $\Phi_\lambda: W^{1,2}_{M}([0,\tau];\R^{2n})\to\mathbb{R}$ defined by
\begin{eqnarray}\label{e:1.5.1}
\Phi_\lambda(v)=\int^{\tau}_0\left[\frac{1}{2}(J\dot{v}(t),v(t))_{\mathbb{R}^{2n}}+ H(\lambda,t, {v}(t))\right]dt
\end{eqnarray}
under the assumption (\ref{e:reductionCase}).
The critical set of $\Phi_\lambda$ is exactly the solution set of (\ref{e:Hboundary}).
Let us choose $\kappa\in\mathbb{R}$ such that
 each $H_\kappa(\lambda, t, z):=H(\lambda, t, z)- \frac{\kappa}{2}|z|^2$ is strictly convex in $z$
 and so that we may get  the dual action of $\Phi_\lambda$, $\Psi_K(\lambda,\cdot):
W^{1,2}_{M}([0,\tau];\R^{2n})\to\R$ defined by
\begin{equation*}
\Psi_\kappa(\lambda, v)=\int^{\tau}_0\left[\frac{1}{2}(J\dot{v}(t)+\kappa v(t),v(t))_{\mathbb{R}^{2n}}+
H_\kappa^\ast(\lambda, t;-J\dot{v}(t)-\kappa v(t))\right]dt.
\end{equation*}
The latter functional is $C^1$, twice G\^ateaux-differentiable,  and
has the same critical set as that of $\Phi_\lambda$ (Corollary~\ref{cor:MWTh.2.6}).
However, $\Psi_\kappa(\lambda, \cdot)$ is not $C^2$ in general.
It is this reason that the Morse theory method for this dual functional has not been used well before.
 Our abstract bifurcation theorems in \cite{Lu8} and \cite{Lu10}, which were developed on the base of the author's Morse theory  for non-twice continuously differentiable functionals
on Hilbert spaces  (\cite{Lu1}, \cite{Lu3} and \cite{Lu7}),
do not require that potential operators are $C^1$, and therefore make  this way to become possible.
With a Banach space isomorphism
$\Lambda_{M,\tau,\kappa I_{2n}}:W^{1,2}_{M}([0,\tau];\R^{2n})\to L^{2}([0,\tau];\R^{2n})$
in (\ref{e:Bisom}) we consider the functional
$$
\psi_\kappa(\lambda, \cdot)=\Psi_\kappa(\lambda, \cdot)\circ(-\Lambda_{M,\tau,\kappa I_{2n}})^{-1}:
L^2([0,\tau];\R^{2n})\to\R
$$
given by (\ref{e:HAction**}). Then the bifurcation problem of $\nabla\Phi_\lambda=0$ near $(\mu,0)\in\Lambda\times W^{1,2}_{M}([0,\tau];\R^{2n})$
 is reduced to the bifurcation one of
 $\nabla\psi_\kappa(\lambda, \cdot)=0$ near $(\mu,0)\in\Lambda\times L^{2}([0,\tau];\R^{2n})$.
At each critical point $w$ of $\psi_\kappa(\lambda, \cdot)$,
$\nabla^2\psi_\kappa(\lambda, w)$ has finite Morse index and nullity,
$m^-(\psi_\kappa(\lambda, \cdot), w)$ and
$m^0(\psi_\kappa(\lambda, \cdot), w)$.
 By (\ref{e:psi-MorseIndex+}) and (\ref{e:psi-Nullity+}) it holds that
\begin{eqnarray*}
  m^-(\psi_\kappa(\lambda, \cdot), 0) &=&i_{\tau,M}(\gamma_\lambda)-i_{\tau,M}(\Upsilon_{\kappa I_{2n}})-\nu_{\tau,M}(\Upsilon_{\kappa I_{2n}}),\\
  m^0(\psi_\kappa(\lambda, \cdot), 0)&=&\nu_{\tau,M}(\gamma_\lambda)=\dim{\rm Ker}(\gamma_\lambda(\tau)-M),
 \end{eqnarray*}
 where $\Upsilon_{\kappa I_{2n}}(t)=\exp(t\kappa J)$,
 $\gamma_\lambda:[0,\tau]\to {\rm Sp}(2n,\mathbb{R})$ is the fundamental matrix solution of
$\dot{Z}(t)=J\nabla^2_zH(\lambda,t, 0)Z(t)$, and $(i_{\tau,M}(\gamma_\lambda), \nu_{\tau,M}(\gamma_\lambda))$
(resp. $(i_{\tau,M}(\Upsilon_{\kappa I_{2n}}), \nu_{\tau,M}(\Upsilon_{\kappa I_{2n}}))$) is
the $M$-Maslov-type index of $\gamma_\lambda$ (resp. $\Upsilon_{\kappa I_{2n}}$) defined in \cite{Do06}.
 When $M$ is an orthogonal symplectic matrix, using these  we can prove that  $\mathcal{F}_\lambda(\cdot)=\psi_\kappa(\lambda, \cdot)$, ${\rm H}=L^2([0,\tau];\R^{2n})$ and $X=L^2([0,\tau];\R^{2n})$
satisfy conditions of \cite[Theorem~3.1]{Lu8} (Theorem~\ref{th:A.10}) under the assumptions of  Theorem~\ref{th:bif-ness}(I).
We can also prove that  $\mathcal{L}_\lambda(\cdot)=\psi_\kappa(\lambda, \cdot)$, ${\rm H}=L^2([0,\tau];\R^{2n})$ and $X=C^0_M([0,\tau];\R^{2n})$
satisfy conditions of \cite[Theorem~3.6]{Lu8} (resp. \cite[Theorem~3.6]{Lu10}) under the assumptions of Theorem~\ref{th:bif-ness}(I) (resp. Theorem~\ref{th:bif-suffict}).
For the case of a general symplectic matrix  $M$,
we show, via a simple transformation, that these results can be derived from their
already-proved special cases with $M=I_{2n}$.\\

\noindent{1.6. \;\bf Further researches.}\;
There are a variety of natural continuations to this work:
\begin{enumerate}
\item[(i)] In order to study the existence of periodic solutions to certain class of
scalar delay differential equations,  Kaplan and Yorke \cite{KaYo74} introduced a new technique,
that is,  the original problem was reduced to finding periodic solutions to an associated (generalized) Hamiltonian system. Our theorems and methods in this paper can be used to derive some bifurcation results for delay (Hamiltonian) equations (\cite{Lu11}),
 and for homoclinic orbits of Hamiltonian systems
and solutions of other nonlinear (e.g. Dirac and wave) equations. 

\item[(ii)]  As done by  Ciriza \cite{Cir}, we may generalize results in this paper to
Hamiltonian systems on general symplectic manifolds with a parametrized version of Darboux's theorem (cf. \cite[\S9]{SaZe2}).
We can also combine our methods with that of \cite{KKK}.
They  would appear elsewhere.
\end{enumerate}

\noindent{1.7.  \bf  Organization of the article.}\;
In Section~\ref{sec:pre}  we collect some  preliminaries, including
a review about properties of a perturbation Hamiltonian operator (Section~2.1)
and  statements of Clarke-Ekeland dual variational principles and their versions used in this paper (Section~2.2). Section~\ref{sec:HamBif}  proves Theorems~\ref{th:bif-ness},~\ref{th:bif-suffict}
and Corollaries~\ref{cor:necess-suffi},~\ref{cor:bif-deform}.
 Section~\ref{sec:orbit} is devoted to the proofs of Theorems~\ref{th:bif-per3},~\ref{th:bif-ness-orbit},
\ref{th:bif-suffict1-orbit}, \ref{th:bif-suffict-orbit} and Corollary~\ref{cor:bif-per5}.
 Section~\ref{sec:brake} proves Theorems~\ref{th:bif-nessbrake},~\ref{th:bif-suffictbrake},~\ref{th:bif-per3brake}
and their Corollaries~\ref{cor:bif-per2brake}--\ref{cor:bif-per5brake}.
In Section~\ref{sec:twoLagr} we complete the proofs of Theorems~\ref{th:bif-nessHam}, \ref{th:bif-suffHam} and \ref{th:two-point}.
For completeness, we add three appendixes. In Appendix~\ref{app:Index} we first briefly review various related
Maslov-type indexes for symplectic paths and their relations to Morse indexes,
and then give some relations  (Theorems~\ref{th:brakeIndex},~\ref{th:brakeIndexMono}) between Morse indexes and the Maslov-type indexes defined by
Long, Zhang and Zhu \cite{LoZZ}. In Appendix~\ref{app:threeBifu} we present the proof of Proposition~\ref{prop:threeBifu}.
In Appendix~\ref{app:Th3.5} we first give a remark to clarify how the compactness assumption of
the parameter topological space $\Lambda$ \cite[Theorems~A.1,A.2]{Lu3} was precisely used.
Then we prove generalizations and refinements  of \cite[Theorems~3.2, 3.5]{Lu10} in great detail.\\

\section{Preliminaries}\label{sec:pre}
\setcounter{equation}{0}

\noindent{2.1. \bf Sobolev spaces and a perturbation Hamiltonian operator.}
We begin with the following elementary fact of linear functional analysis.

\begin{proposition}\label{prop:abstract}
Let $\mathscr{H}$ be a Hilbert space with norm $\|\cdot\|$ and let ${A}: D(A)\subset
\mathscr{H}\to \mathscr{H}$ be an unbounded linear operator that is densely defined and closed.
Suppose that ${A}$ is self-adjoint and that there exist a complete norm $|\cdot|$ on $D({A})$
for which the following three conditions are satisfied:
\begin{itemize}
\item[\rm (i)] The inclusion $\iota:(D({A}),|\cdot|)\hookrightarrow \mathscr{H}$ is compact.
\item[\rm (ii)]  ${A}:(D({A}), |\cdot|)\to \mathscr{H}$ is a bounded linear operator.
\item[\rm (iii)] There exists a constant $C>0$ such that
$|x|\le C(\|x\|+\|{A}x\|)$ for all $x\in D({A})$.
\end{itemize}
Then $\dim{\rm Ker}({A})<\infty$, $R({A})$ is closed in $\mathscr{H}$  and  there exists an orthogonal decomposition
$\mathscr{H}={\rm Ker}({A})\oplus R({A})$. Moreover
$$
\hat{{A}}:(D({A})\cap R({A}), |\cdot|)\to (R({A}),\|\cdot\|),\;x\mapsto {A}x
$$
is  invertible and $\iota\circ \hat{{A}}^{-1}$ as an operator on the Hilbert subspace $R({A})$ of $\mathscr{H}$ is compact and self-adjoint.
\end{proposition}

Applying [10, Exercise 6.9] to $F=\mathscr{H}$, $G=\mathscr{H}$, $E=(D(A), |\cdot|)$, $K=\iota$ and $T={A}$,
we obtain the first two claims immediately.
Since $A$ is self-adjoint, by \cite[Theorem~2.19]{Bre}
we have $({\rm Ker}(A))^\bot=R(A^\ast)=R(A)$. Other conclusions are obvious.

Let $L^{2}([0,\tau];\mathbb{R}^{2n})=(L^{2}([0,\tau];\mathbb{R}))^{2n}$
and $W^{1,2}([0,\tau];\mathbb{R}^{2n})=(W^{1,2}([0,\tau];\mathbb{R}))^{2n}$
be the Hilbert spaces equipped with $L^{2}$-inner product and $W^{1,2}$-inner product
\begin{eqnarray}\label{e:innerP}
(u,v)_{2}&=&\int^\tau_0(u(t),v(t))_{\mathbb{R}^{2n}}dt,\\
(u,v)_{1,2}&=&\int^\tau_0[(u(t),v(t))_{\mathbb{R}^{2n}}+ (\dot{u},\dot{v})_{\mathbb{R}^{2n}}]dt,\label{e:innerP2}
\end{eqnarray}
respectively. The corresponding norms are denoted by $\|\cdot\|_{2}$ and $\|\cdot\|_{1,2}$, respectively.
(As usual each $u\in L^{2}([0,\tau];\R^{2n})$ will be identified with any fixed representative of it; in particular,
we do not distinguish $u\in W^{1,2}([0,\tau];\mathbb{R}^{2n})$ with its unique continuous representation.)

For $A\in L^\infty([0,\tau];\mathcal{L}_s(\mathbb{R}^{2n}))$ and a symplectic matrix $M\in{\rm Sp}(2n,\mathbb{R})$
 let $\Lambda_{M,\tau, A}$  be the operator on $L^2([0,\tau];\mathbb{R}^{2n})$
defined by
\begin{equation}\label{e:LambdaKA}
(\Lambda_{M,\tau, A}u)(t):=J\frac{d}{dt}u(t)+A(t)u(t)
\end{equation}
 with domain
$$
{\rm dom}(\Lambda_{M,\tau, A})=W^{1,2}_{M}([0,\tau]; \mathbb{R}^{2n}).
$$
Denote by $\Upsilon_A$  the fundamental solution of the linear Hamiltonian system
\begin{equation}\label{e:fundamental}
\dot{x}=JA(t)x\quad\forall t\in [0,\tau].
\end{equation}
Clearly, ${\rm Ker}(\Lambda_{M,\tau, A})=\{\Upsilon_A(\cdot)\xi\,|\,\xi\in\mathcal{E}_{M,\tau,A}\}$, where
 $$
 \mathcal{E}_{M,\tau,A}={\rm Ker}(\Upsilon_A(\tau)-M)\quad\hbox{and}\quad
\mathcal{E}_{M,\tau,A}^\bot=\{u\in \mathbb{R}^{2n}\,|\, (u,v)_{\mathbb{R}^{2n}}=0\;\forall v\in\mathcal{E}_{M,\tau,A}\}.
$$
Note that  $\Upsilon_{\kappa I_{2n}}(t)=e^{\kappa tJ}$ for a real constant $\kappa$.

 It was proved in Lemma 4 of \cite[page 102]{Ek90} that
the operator $\Lambda_{M,\tau, 0}$ is closed and self-adjoint.
But $\Lambda_{M,\tau, A}$ is the sum of  $\Lambda_{M,\tau, 0}$
and the following continuous linear self-adjoint operator
 $$
 L^{2}([0,\tau];\R^{2n})\to L^{2}([0,\tau];\R^{2n}),\;u\mapsto A(\cdot)u.
 $$
An elementary functional analysis fact (cf. \cite[Exercise 2.20]{Bre})
shows that $\Lambda_{M,\tau, A}$ is also closed and self-adjoint.
Since $A\in L^\infty([0,\tau];\mathcal{L}_s(\mathbb{R}^{2n}))$,
for any $u\in W^{1,2}_{M}([0,\tau]; \mathbb{R}^{2n})$,
 (\ref{e:LambdaKA}) yields
$$
\|\Lambda_{M,\tau, A}u\|_2\ge \|\dot{u}\|_2-\|A\|_{L^\infty}\|u\|_2
$$
and hence
$$
\|u\|_{1,2}\le \|\dot{u}\|_2+\|u\|_2\le (1+\|A\|_{L^\infty})(\|u\|_2+ \|\Lambda_{M,\tau, A}u\|_2).
$$
That is, the conditions in Proposition~2.1 are satisfied with $\mathscr{H}=L^{2}([0,\tau];\R^{2n})$
and ${A}=\Lambda_{M,\tau, A}$.
Hence Proposition~\ref{prop:abstract} leads to  the following result, which was directly proved
by Dong \hbox{\cite[Proposition~3.12]{Do06}}.

\begin{proposition}\label{prop:Dong}
$\Lambda_{M,\tau, A}$ has the closed  range $R(\Lambda_{M,\tau, A})$ in $L^2([0,\tau];\mathbb{R}^{2n})$,
 the kernel ${\rm Ker}(\Lambda_{M,\tau, A})$ is of dimension at most $2n$, and there exists an orthogonal decomposition
$$
L^2([0,\tau];\mathbb{R}^{2n})={\rm Ker}(\Lambda_{M,\tau, A})\oplus R(\Lambda_{M,\tau, A}).
$$
Moreover, the restriction $\hat{\Lambda}_{M,\tau, A}$ of $\Lambda_{M,\tau, A}$
to $R(\Lambda_{M,\tau, A})\cap{\rm dom}(\Lambda_{M,\tau, A})$
 is invertible and the inverse $(\hat{\Lambda}_{M,\tau, A})^{-1}:
 R(\Lambda_{M,\tau, A})\to R(\Lambda_{M,\tau, A})$
is compact and self-adjoint if  $R(\Lambda_{M,\tau, A})$ is  endowed with the $L^2$-norm.
\end{proposition}

 For $v(t)=\Upsilon_A(t)\xi$ with $\xi\in\mathbb{R}^{2n}$,
and $u\in L^2([0,\tau];\mathbb{R}^{2n})$,
$$
(u,v)_2=-\int^\tau_0(J\Upsilon_A(t)^{-1}Ju(t),\xi)_{\mathbb{R}^{2n}}dt=
-\left(\int^\tau_0J\Upsilon_A(t)^{-1}Ju(t)dt,\xi\right)_{\mathbb{R}^{2n}}
$$
implies
\begin{eqnarray*}
\tilde{L}^{2}_{M, A}([0, \tau];\R^{2n}):=R(\Lambda_{M,\tau, A})
=\left\{u\in L^2([0,\tau];\mathbb{R}^{2n})\;\Big|\;  \int^\tau_0J\Upsilon_A(t)^{-1}Ju(t)dt\in \mathcal{E}_{M,\tau,A}^\bot\right\}.
\end{eqnarray*}
Moreover,  for  $w\in W^{1,2}([0,\tau];\mathbb{R}^{2n})$ and $v(t)=\Upsilon_A(t)\xi$ with $\xi\in\mathbb{R}^{2n}$, it holds that
\begin{eqnarray*}
(w,v)_{1,2}&=&\int^\tau_0(w(t),\Upsilon_A(t)\xi)_{\mathbb{R}^{2n}}dt+\int^\tau_0(\dot{w}(t),JA(t)\Upsilon_A(t)\xi)_{\mathbb{R}^{2n}}dt\\
&=&\int^\tau_0(-A(t)J\dot{w}(t)+ w(t),\Upsilon_A(t)\xi)_{\mathbb{R}^{2n}}dt\\
&=&\int^\tau_0(J\Upsilon_A(t)^{-1}J^{-1}[-A(t)J\dot{w}(t)+ w(t)],\xi)_{\mathbb{R}^{2n}}dt\\
&=&\int^\tau_0(-J\Upsilon_A(t)^{-1}J[-A(t)J\dot{w}(t)+ w(t)],\xi)_{\mathbb{R}^{2n}}dt\\
&=&\int^\tau_0(J\Upsilon_A(t)^{-1}J[A(t)J\dot{w}(t)- w(t)],\xi)_{\mathbb{R}^{2n}}dt.
\end{eqnarray*}
Let
$$
\tilde{W}^{1,2}_{M, A}([0, \tau];\R^{2n})=\left\{w\in{W}^{1,2}_{M}([0, \tau];\R^{2n})\,\Big|\,
\int^\tau_0J\Upsilon_A(t)^{-1}J[A(t)J\dot{w}(t)- w(t)]dt\in \mathcal{E}_{M,\tau,A}^\bot\right\},
$$
thus, we have an orthogonal  decomposition of Hilbert spaces
\begin{equation}\label{e:orthCom}
W^{1,2}_{M}([0,\tau]; \mathbb{R}^{2n})={\rm Ker}(\Lambda_{M,\tau, A})\oplus
\tilde{W}^{1,2}_{M, A}([0, \tau];\R^{2n}).
\end{equation}
Clearly, $\Lambda_{M,\tau, A}$ restricts to a Banach space isomorphism
$$
\tilde{\Lambda}_{M,\tau, A}:\tilde{W}^{1,2}_{M, A}([0, \tau];\R^{2n})\to\tilde{L}^{2}_{M, A}([0, \tau];\R^{2n}).
$$
It was explicitly written out in \cite{Do06}. Indeed,  by the orthogonal splitting
$$
\R^{2n}=(I_{2n}-\Upsilon_A(\tau)^{-1}M)(\R^{2n})\oplus J{\rm Ker}(\Upsilon_A(\tau)-M)
$$
(\cite[(3.8)]{Do06}), $I_{2n}-\Upsilon_A(\tau)^{-1}M$ restricts to an invertible operator from $(I_{2n}-\Upsilon_A(\tau)^{-1}M)(\R^{2n})$
to itself. Therefore we can define
$$
\mathfrak{J}:\R^{2n}\to\R^{2n},\;x\mapsto \mathfrak{J}x=(I_{2n}-\Upsilon_A(\tau)^{-1}M)x_1+ x_2,
$$
where $x=x_1+ x_2$ with $x_1\in(I_{2n}-\Upsilon_A(\tau)^{-1}M)(\R^{2n})$ and $x_2\in J{\rm Ker}(\Upsilon_A(\tau)-M)$. Then
$\mathfrak{J}$ is invertible. By \cite[(3.6) and (3.7)]{Do06} we have for $u\in\tilde{L}^{2}_{M,A}([0,\tau];\R^{2n})$,
\begin{equation}\label{e:inver}
[(\tilde{\Lambda}_{M,\tau, A})^{-1}u](t)=\Upsilon_A(t)\mathfrak{J}^{-1}\int^\tau_0\Upsilon_A(s)^{-1}Ju(s)ds-
\Upsilon_A(t)\int^t_0\Upsilon_A(s)^{-1}Ju(s)ds.
\end{equation}
Clearly, if $\mathcal{E}_{M,\tau,A}=\{0\}$, i.e., $\det(\Upsilon_A(\tau)-M)\ne 0$,
then $\Lambda_{M,\tau,A}$ is a Banach space isomorphism from ${W}^{1,2}_{M}([0,\tau];\R^{2n})$ onto ${L}^{2}([0,\tau];\R^{2n})$.\\

\noindent{2.2. \bf Clarke-Ekeland adjoint-action principle.}
Let $X$ be a Banach space, and $X^\ast$ its dual, the duality pairing being
denoted by $\langle f,x\rangle$ for $(x,f)\in X\times X^\ast$.
For a function $G:X\to\R\cup\{+\infty\}$
 which is proper, i.e., it is  not identically $+\infty$,
its Legendre-Fenchel conjugate is the  function $G^\ast:X^\ast\to\R\cup\{+\infty\}$  defined by
  \begin{eqnarray}\label{e:FenchelConjugate}
G^\ast(u^\ast)=\sup\{\langle u^\ast, u\rangle-G(u)\,|\, u\in X\}
\end{eqnarray}
with effective domain  $D(G^\ast)=\{v^\ast\in X^\ast\,|\, G^*(v^*)<+\infty\}$.
It is convex and lower semi-continuous.

From now on, we assume that $X$ is a reflexive Banach space, i.e.,
the canonical embedding (or the James map) $J_X: X\to X^{\ast\ast}$ is a Banach space isomorphism.
Let $\textsf{A} : D(\textsf{A})\subset X\to X^\ast$ be a  closed, densely defined,
 unbounded, and  self-adjoint linear operator.
  Further,  let $F : X\to\mathbb{R}\cup\{ +\infty\}$ be a
  proper, convex and lower semi-continuous (l.s.c.) function.
Define a functional $\Phi: D(\textsf{A})\to\mathbb{R}\cup\{+\infty\}$
and its dual functional $\Psi:D(\textsf{A})\to\R\cup\{+\infty\}$ by
\begin{eqnarray*}
 \Phi(v)=\frac{1}{2}\langle \textsf{A}v,v\rangle+ F(v)\quad\hbox{and}\quad
 \Psi(x)=\frac{1}{2}\langle \textsf{A}x,x\rangle+ [F^\ast\circ(-\textsf{A})](x),
\end{eqnarray*}
where $F^\ast\circ(-\textsf{A}):X\to\R\cup\{+\infty\}$ is defined by
$$
(F^\ast\circ(-\textsf{A}))(x)=\left\{\begin{array}{ll}
F^\ast( -\textsf{A}x)\;\hbox{if $x\in D(\textsf{A})$},\\
+\infty\;\hbox{otherwise}.
\end{array}\right.
$$
If $x\in D(\textsf{A})$ satisfies the condition that $0\in \textsf{A}x+\partial F(x)$ (resp.
$0\in \textsf{A}x+ \partial[F^\ast\circ(-\textsf{A})](x)$), it is called
a \textsf{critical point} of $\Phi$ (resp. $\Psi$), where $\partial F(x)$
is the subgradient of $F$ at $x\in X$  defined by
  \begin{eqnarray*}
\partial F(x)=\{u^\ast\in X^\ast\,|\, F(u)\ge F(x)+\langle u^\ast, u-x\rangle\;\forall u\in X\}.
\end{eqnarray*}

The Clarke duality was introduced by Clarke \cite{Cl79}
and developed by Clarke-Ekeland \cite{ClEk80}.

\begin{theorem}[\hbox{\cite[page 100, Theorem~2]{Ek90}}]\label{th:Ek}
Let the above assumptions be satisfied. If $\bar{u}\in D(\textsf{A})$ is a critical point of $\Phi$,
then all $u\in{\rm Ker}(\textsf{A})+\bar{u}$ are  critical points of
$\Psi$ and $\Psi({u})=-\Phi(\bar{u})$. Conversely, if $\bar{v}\in D(\textsf{A})$
is a critical point of $\Psi$, and
\begin{equation}\label{e:interior}
0\in {\rm Int}(D(F^\ast)-R(\textsf{A}))
 \end{equation}
 then there is some $\bar{w}\in{\rm Ker}(\textsf{A})$ such that
$\bar{u}:=\bar{v}-\bar{w}$
is a critical point of $\Phi$ with $\Phi(\bar{u})=-\Psi(\bar{v})$.
\end{theorem}

Let us state three convenient corollaries of this theorem.

By \cite[page 92, Proposition 9]{Ek90}  and Theorem~\ref{th:Ek} we may deduce:

\begin{corollary}\label{th:Ek9}
Under the assumptions of Theorem~\ref{th:Ek}, suppose that
 $F : X\to\R\cup\{+\infty\}$
(resp. $F^\ast : X\to\R\cup\{+\infty\}$) is
Gateaux-differentiable at each point in $D(\textsf{A})$ (resp. $R(\textsf{A})$),
and that $D(\textsf{A})$ itself is a Banach space with norm $|\cdot|$ such that
 $$
 \iota:(D(\textsf{A}),|\cdot|)\hookrightarrow X\quad\hbox{and}\quad \textsf{A}: (D(\textsf{A}),|\cdot|)\to X^\ast
 $$
are continuous, which implies
\begin{enumerate}
\item[\rm (i)]  $F:(D(\textsf{A}),|\cdot|)\to\R$
(and so $\Phi:(D(\textsf{A}),|\cdot|)\to\R$) is Gateaux-differentiable,
\item[\rm (ii)] $F^\ast\circ(-\textsf{A}):(D(\textsf{A}),|\cdot|)\to\R$ (and so $\Psi:(D(\textsf{A}),|\cdot|)\to\R$)
 is Gateaux-differentiable.
\end{enumerate}
Then  for any critical point $\bar{u}$ of $\Phi$ on $(D(\textsf{A}),|\cdot|)$, each $u\in{\rm Ker}(\textsf{A})+\bar{u}$
is a critical point of
$\Psi$ on $(D(\textsf{A}),|\cdot|)$ and $\Psi({u})=-\Phi(\bar{u})$.
Conversely, if $0\in {\rm Int}(D(F^\ast)-R(\textsf{A}))$, then for any critical point $\bar{v}$ of $\Psi$ on $(D(\textsf{A}),|\cdot|)$,
 there is some $\bar{w}\in{\rm Ker}(\textsf{A})$ such that
$\bar{u}:=\bar{v}-\bar{w}$ is a critical point of $\Phi$ on $(D(\textsf{A}),|\cdot|)$ with $\Phi(\bar{u})=-\Psi(\bar{v})$.
\end{corollary}

Suppose that $F\in C^1(X)$ is strictly convex, i.e.,
$\langle v-w, F'(v)-F'(w)\rangle>0$ if $v\ne w$.
Then $F'$ is injective, $F^\ast$ is finite on the range of $F'$, and $\partial F^\ast(F'(v))=\{v\}$
for any $v\in X$. If, in addition, $F'$ is strongly monotone and coercive in the sense that
$$
\langle v-w, F'(v)-F'(w)\rangle\ge \alpha(\|v-w\|)\|v-w\|\quad\forall v,w\in X,
$$
where $\alpha:[0,\infty)\to [0,\infty)$ is a non-decreasing function vanishing only at $0$ and such that
$\alpha(r)\to\infty$ as $r\to\infty$, then $F^\ast\in C^1(X^\ast)$ and $F'$ is a homeomorphism of $X$ onto $X^\ast$.

Similarly,  \cite[page 92, Proposition 9]{Ek90}  and Theorem~\ref{th:Ek}  give:

\begin{corollary}\label{th:Ek10}
  Let $\textsf{A} : D(\textsf{A})\subset X\to X^\ast$  be as in Theorem~\ref{th:Ek},
 and let $F\in C^1(X)$ be strictly convex and $F'$ be strongly monotone
 and coercive (so that $R(F') = X$ and $\textsf{A}(D(\textsf{A}))\subset R(F')$).
 Suppose that $D(\textsf{A})$ itself is a Banach space with norm $|\cdot|$ and
 $$
 \iota:(D(\textsf{A}),|\cdot|)\hookrightarrow X\quad\hbox{and}\quad \textsf{A}: (D(\textsf{A}),|\cdot|)\to X^\ast
 $$
are continuous, which implies that $F^\ast\circ(-\textsf{A}):(D(\textsf{A}),|\cdot|)\to\R$ (and so $\Psi:(D(\textsf{A}),|\cdot|)\to\R$)
 is $C^1$.
Then for a critical point $\bar{u}$ of $\Phi$ on $(D(\textsf{A}),|\cdot|)$,  each $u\in{\rm Ker}(\textsf{A})+\bar{u}$
is a  critical point of
$\Psi$ on $(D(\textsf{A}),|\cdot|)$ and $\Psi({u})=-\Phi(\bar{u})$. Conversely, if $\bar{v}$ is a critical point of $\Psi$ on $(D(\textsf{A}),|\cdot|)$,
 then there is some $\bar{w}\in{\rm Ker}(\textsf{A})$ such that
$\bar{u}:=\bar{v}-\bar{w}$
is a critical point of $\Phi$ on $(D(\textsf{A}),|\cdot|)$ with $\Phi(\bar{u})=-\Psi(\bar{v})$.
\end{corollary}

Using \cite[Lemma~2.1]{Ch81}  we can also derive from Theorem~\ref{th:Ek}:

\begin{corollary}\label{cor:Ek1}
Under the assumptions of Theorem~\ref{th:Ek} suppose that $D(\textsf{A})$ itself is a Banach space with norm $|\cdot|$ such that
 the inclusion $\iota:(D(\textsf{A}),|\cdot|)\hookrightarrow X$ is continuous. Assume also  that both $F$ and $F^\ast$
are convex and locally Lipschitz continuous functions on $X$. Then for a critical point
 $\bar{u}$ of $\Phi$ on $(D(\textsf{A}),|\cdot|)$,  all $u\in{\rm Ker}(\textsf{A})+\bar{u}$ are
  critical points of $\Psi$ on $(D(\textsf{A}),|\cdot|)$ and satisfy $\Psi({u})=-\Phi(\bar{u})$.
  Conversely, if $\bar{v}\in D(\textsf{A})$ is a critical point of $\Psi$ on $(D(\textsf{A}),|\cdot|)$,
and (\ref{e:interior}) holds,  then there is some $\bar{w}\in{\rm Ker}(\textsf{A})$ such that
$\bar{u}:=\bar{v}-\bar{w}$
is a critical point of $\Phi$ on $(D(\textsf{A}),|\cdot|)$ with $\Phi(\bar{u})=-\Psi(\bar{v})$.
\end{corollary}

By \cite[Exercise~2.20]{Bre} it is easy to deduce the following result, which
can be used to enlarge applications of the above theorems and corollaries.

\begin{proposition}\label{prop:Bre8}
Let $X$ be a reflexive Banach space, $\textsf{T}\in\mathscr{L}(X,X^\ast)$ and let $\textsf{A} : D(\textsf{A})\subset X\to X^\ast$
 be an unbounded linear operator that is densely defined and closed. Suppose that
 $\textsf{T}$ and $\textsf{A}$ are self-adjoint. Then the
operator $\textsf{B}: D(\textsf{B})\subset X\to X^\ast$ defined by $D(\textsf{B})=D(\textsf{A})$ and $\textsf{B}=\textsf{A}+\textsf{T}$
is also closed, and self-adjoint.
\end{proposition}

Let $\Omega$ be a Borel subset of  $\R^n$ with finite Lebesgue measure, and let
 $f:\Omega\times\R^N\to\R\cup\{+\infty\}$ be a Borel function satisfying the following conditions:
 \begin{eqnarray}
&&f\;\hbox{is bounded from below},\label{e:borel1}\\
&&f(\omega,\cdot)\;\hbox{is lower semi-continuous (l.s.c.) on}\;
\R^N\; \text{for every } \omega \in \Omega,\label{e:borel2}\\
&&f(\omega,\cdot)\;\hbox{is convex on}\;\R^N\;
\text{for every } \omega \in \Omega.\label{e:borel3}
\end{eqnarray}

Denote by $f^\ast(\omega;\xi)=(f(\omega,\cdot))^\ast(\xi)$
the convex conjugate with respect to the second variable.
As noted in the first two lines of \cite[p. 96]{Ek90}, we have:

\begin{proposition}\label{prop:Ek12}
Let \(\Omega \subset \mathbb{R}^n\) be as above, and let
 \(f: \Omega \times \mathbb{R}^N \to \mathbb{R} \cup \{+\infty\}\)
 be a Borel function satisfying (\ref{e:borel1}).
 Let \(1 \le \alpha \le +\infty\) and let \(\beta = \alpha/(\alpha-1)\) be its H\"older conjugate exponent,
  with the conventions \(1/\infty = 0\) and \(1/0 = \infty\).
Define the functional \(F: L^{\alpha}(\Omega; \mathbb{R}^N) \to \mathbb{R} \cup \{+\infty\}\) by
\[
F(u) = \int_\Omega f(x, u(x)) \, dx, \qquad u \in L^{\alpha}(\Omega; \mathbb{R}^N),
\]
where the integral is taken in the sense of Lebesgue. Then the following holds:
\begin{enumerate}
\item[\rm (1)] If $f$ also satisfies (\ref{e:borel2}), then $F$
is lower semi-continuous \cite[Proposition~11, p. 96]{Ek90}.

\item[\rm (2)] If $f$ also satisfies (\ref{e:borel3}),  then $F$ is convex
\cite[line 1, p. 94]{Ek90}.

\item[\rm (3)] If $f$ also satisfies (\ref{e:borel2})--(\ref{e:borel3}) and if
\begin{equation}\label{e:borel4}
    \int_\Omega f(x, \bar{u}(x)) \, dx < +\infty
\end{equation}
for some $\bar{u} \in L^\infty(\Omega; \mathbb{R}^N)$, then
\begin{equation*}
    F^*(u) = \int_\Omega f^*(x; u(x)) \, dx, \qquad \forall\, u \in L^\beta(\Omega;\mathbb{R}^N),
\end{equation*}
as stated in \cite[Theorem~2, p. 94]{Ek90}.

\item[\rm (4)] Under the assumptions of (3), if
\begin{equation}\label{e:borel5}
    \int_\Omega |f^*(x; \bar{v}(x))| \, dx < +\infty
\end{equation}
for some $\bar{v} \in L^\beta(\Omega; \mathbb{R}^N)$, then
\begin{equation*}
    \partial F(u) = \bigl\{ x^* \in L^\beta(\Omega; \mathbb{R}^N) \mid x^*(\omega) \in \partial_\xi f(\omega, u(\omega)) \ \text{a.e.} \bigr\},
\end{equation*}
as stated in \cite[Corollary~2, p. 96]{Ek90}.
\end{enumerate}
 \end{proposition}

\begin{proposition}\label{prop:Ek13}
Let $H : [0,\tau]\times \mathbb{R}^{2n}\to\mathbb{R}, (t, z)\to H(t, z)$
  be measurable in $t$ for each $z\in \mathbb{R}^{2n}$ and strictly convex and continuously differentiable
in $z$ for almost every $t\in [0,\tau]$. Denote by $H^\ast(t;\cdot)$
the Fenchel conjugate of $H(t,\cdot)$ as defined by (\ref{e:FenchelConjugate}).
(It is also strictly convex and continuously differentiable.)
Assume that there exists
 $\alpha>0$, $\delta>0$, and
$\beta,\gamma\in L^2([0,\tau];\mathbb{R}^+)$,
   such that for a.e. $t\in [0,\tau]$ and every $z\in \mathbb{R}^{2n}$, one has
\begin{eqnarray}\label{e:H-condition}
\delta|z|^2-\beta(t)\le H(t,z)\le \alpha|z|^2+ \gamma(t).
\end{eqnarray}
Then the dual  of the functional
  $$
 \mathcal{H}:L^2([0,\tau];\R^{2n})\to\mathbb{R},\;u\mapsto \int^T_0H(t,u(t))dt
 $$
 is given by
$$
\mathcal{H}^\ast:L^2([0,\tau]; \R^{2n})\to\R\cup\{+\infty\},\;u\mapsto \int^\tau_0H^\ast(t; u(t))dt.
$$
Moreover, both  $\mathcal{H}$ and  $\mathcal{H}^\ast$
 are  convex,  bounded on any bounded subsets,  of class $C^1$, and
 \begin{eqnarray}\label{e:H-gradient}
\nabla\mathcal{H}(u)=\nabla_z H(\cdot,u(\cdot))\quad\hbox{and}\quad
\nabla\mathcal{H}^\ast(u)=\nabla_z H^\ast(\cdot;u(\cdot)).
\end{eqnarray}
\end{proposition}

\begin{proof}[\bf Proof]
By (\ref{e:H-condition}) and \cite[Proposition 2.4]{MaWi},
$H^\ast(t;z)$ for a.e. $t\in [0,\tau]$ is strictly convex and continuously differentiable in $z\in \mathbb{R}^{2n}$, and also satisfies
\begin{eqnarray}\label{e:H-condition*}
\alpha^{-1}|z|^2-\gamma(t)\le H^\ast(t;z)\le\delta^{-1}|z|^2+ \beta(t).
\end{eqnarray}
Note that  (\ref{e:H-condition}) and  (\ref{e:H-condition*})  imply
$$
|H(t,z)|\le \beta(t)+ \gamma(t)+ \alpha |z|^2\quad\hbox{and}\quad |H^\ast(t;z)|\le \beta(t)+\gamma(t)+\delta^{-1}|z|^2,
$$
 respectively. Thus
$$
 \int^\tau_0|H(t,0)|\le \int^\tau_0(\beta(t)+\gamma(t))dt<+\infty\quad\hbox{and}\quad
 \int^\tau_0|H^\ast(t;0)|\le \int^\tau_0(\beta(t)+\gamma(t))dt<+\infty.
$$
For any $u\in L^2([0,\tau];\R^{2n})$ we derive from (3) to (4) in Proposition~\ref{prop:Ek12} that
\begin{eqnarray*}
&\mathcal{H}^\ast(u)=\int^\tau_0H^\ast(t; u(t))dt,\\
&\partial \mathcal{H}(u)=\{x^\ast\in L^2([0,\tau];\R^{2n})\,|\,x^\ast(t)\in\partial_z H(t,u(t))\;a.e.\}=\{\nabla_{z}H(\cdot,u(\cdot))\},\\
&\partial \mathcal{H}^\ast(u)=\{x^\ast\in L^2([0,\tau];\R^{2n})\,|\,x^\ast(t)\in\partial_z H^\ast(t;u(t))\;a.e.\}=\{\nabla_z H^\ast(\cdot;u(\cdot))\}.
\end{eqnarray*}
By these and \cite[page 92, Proposition 9]{Ek90} we obtain that both $\mathcal{H}$ and $\mathcal{H}^\ast$
are Gateaux-differentiable at each $u\in L^2([0,\tau];\R^{2n})$ and their gradients
are given by (\ref{e:H-gradient}).

 It follows from (\ref{e:H-condition}), (\ref{e:H-condition*})  and
 a theorem of Krasnosel'skii (cf. \cite[page 98, Corollary 5]{Ek90}) that the operators
  ${\bf H}, {\bf H}_\ast:L^2([0,\tau];\R^{2n})\to L^1([0,\tau];\R)$ defined by
$$
{\bf H}(u)(t)=H(t, u(t))\quad\hbox{and}\quad{\bf H}_\ast(u)(t)=H^\ast(t; u(t))
$$
are continuous and maps bounded sets into bounded ones. Hence the functionals
  $\mathcal{H}$ and $\mathcal{H}^\ast$ are convex,  continuous and bounded on any bounded subsets.

Finally, as in the proof of \cite[Theorem~2.3]{MaWi},
 (\ref{e:H-condition}) and (\ref{e:H-condition*})  lead to
\begin{eqnarray}\label{e:H-condition***}
 |\nabla_z H(t, z)|&\le& 1+ \alpha|z|+ \alpha(\beta(t)+\gamma(t)),\\
 |\nabla_z H^\ast(t; z)|&\le& 1+ \delta^{-1}|z|+ \delta^{-1}(\beta(t)+\gamma(t))],\label{e:H-condition**}
\end{eqnarray}
respectively, where constants $\alpha$ and $\delta$ are as in (\ref{e:H-condition})
and (\ref{e:H-condition*}).
As above we derive that
both $\nabla\mathcal{H}$ and $\nabla\mathcal{H}^\ast$ are continuous. Hence
$\mathcal{H}$ and $\mathcal{H}^\ast$ are of class $C^1$.
\end{proof}

\begin{proposition}\label{prop:Ek14}
Under the assumptions of Proposition~\ref{prop:Ek13}, suppose also that $H$
is twice  continuously differentiable in $z$ for almost every $t\in [0,\tau]$, and there
exist constants $C_2>C_1>0$ such that for a.e. $t\in [0,\tau]$,
\begin{equation}\label{e:Hess}
C_1|\eta|^2\le (\nabla^2_{z}H(t, \xi)\eta,\eta)_{\mathbb{R}^{2n}}\le C_2|\eta|^2,\quad\forall\xi,\eta\in\mathbb{R}^{2n}.
\end{equation}
Then both $\nabla\mathcal{H}$ and $\nabla\mathcal{H}^\ast$ are Gateaux differentiable,
strongly monotone and coercive, and therefore
 homeomorphisms from $L^2([0,\tau];\mathbb{R}^{2n})$ to itself. In particular,
 $\mathcal{H}$ and $\mathcal{H}^\ast$ are twice Gateaux  differentiable, and
 \begin{eqnarray*}
 &&\mathcal{H}''(x)[u,v]=(D(\nabla\mathcal{H})(x)[u], v)_{2}=\int^\tau_0(\nabla^2_{z}H(t,x(t))u(t), v(t))_{\mathbb{R}^{2n}} dt,\\
 &&(\mathcal{H}^\ast)''(x)[u,v]=(D(\nabla\mathcal{H}^\ast)(x)[u], v)_{2}=\int^\tau_0(\nabla^2_{z}H^\ast(t;x(t))u(t), v(t))_{\mathbb{R}^{2n}} dt
 \end{eqnarray*}
for all $x,u,v\in L^2([0,\tau];\mathbb{R}^{2n})$. Clearly, both $D(\nabla\mathcal{H})(x)$ and $D(\nabla\mathcal{H}^\ast)(x)$
are positive definite.
\end{proposition}
\begin{proof}[\bf Proof]
For $s\in (-1,1)\setminus\{0\}$, using the mean value theorem we get
 \begin{eqnarray*}
&&\left(\frac{1}{s}\big(\nabla\mathcal{H}(x+su)-\nabla\mathcal{H}(x)\big)-\nabla^2_z H(\cdot, x(\cdot))u(\cdot), v\right)_2\\
&=&\int^\tau_0\left(\frac{1}{s}(\nabla_z H(t, x(t)+su(t))-\nabla_z H(t, x(t)))-\nabla^2_z H(t, x(t))u(t), v(t)\right)_{\mathbb{R}^{2n}}dt\\
&=&\int^\tau_0\left(\nabla^2_z H(t, x(t)+\theta(s,t)su(t))u(t)-\nabla^2_z H(t, x(t))u(t), v(t)\right)_{\mathbb{R}^{2n}}dt\\
\end{eqnarray*}
for some function $(s,t)\mapsto\theta(s,t)\in (0,1)$. This leads to
 \begin{eqnarray}\label{e:Hess+}
&&\left\|\frac{1}{s}\big(\nabla\mathcal{H}(x+su)-\nabla\mathcal{H}(x)\big)-\nabla^2_z H(\cdot, x(\cdot))u(\cdot)\right\|_2\nonumber\\
&\le &\left(\int^\tau_0\left|\nabla^2_z H(t, x(t)+\theta(s,t)su(t))u(t)-\nabla^2_z H(t, x(t))u(t)\right|^2_{\mathbb{R}^{2n}}dt\right)^{1/2}.
\end{eqnarray}
Since $|\nabla^2_z H(t, x(t)+\theta(s,t)su(t))u(t)-\nabla^2_z H(t, x(t))u(t)|^2_{\mathbb{R}^{2n}}\le 4C_2^2|u(t)|^2_{\mathbb{R}^{2n}}$,
by the Lebesgue dominated convergence theorem the integration in (\ref{e:Hess+})
converges to zero as $s\to 0$. Hence
$$
D(\nabla\mathcal{H})(x)[u]=\nabla^2_z H(\cdot, x(\cdot))u(\cdot).
$$

Moreover, (\ref{e:Hess}) implies that for a.e. $t\in [0,\tau]$, $\nabla_z H(t,\cdot)$
 is strongly monotone and coercive, and therefore $\nabla_z H(t,\cdot):\mathbb{R}^{2n}\to \mathbb{R}^{2n}$
 is a homeomorphism. By  \cite[page 92, Proposition 10]{Ek90}, for $\xi\in \mathbb{R}^{2n}$ and $\xi^\ast=\nabla_z H(t,\xi)$
 it holds that $\xi=\nabla_z H^\ast(t;\xi^\ast)$ and $I_{2n}=\nabla^2_{z} H^\ast(t;\xi^\ast)\nabla^2_z H(t,\xi)$.
 These and (\ref{e:Hess}) imply
 \begin{equation}\label{e:Hess*}
C_2^{-1}|\eta|^2\le (\nabla^2_{z} H^\ast(t;\xi^\ast)\eta,\eta)_{\mathbb{R}^{2n}}\le C_1^{-1}|\eta|^2,\quad\forall\xi,\eta\in\mathbb{R}^{2n}.
\end{equation}
As above, we can derive from this that $\nabla\mathcal{H}^\ast$ is G\^ateaux differentiable
and
$$
D(\nabla\mathcal{H}^\ast)(x)[u]=\nabla^2_z H^\ast(\cdot; x(\cdot))u(\cdot).
$$

Other claims follow from (\ref{e:Hess}) and (\ref{e:Hess*}).
 \end{proof}

The following is a generalization of \cite[Theorem~2.3]{MaWi}.

\begin{theorem}\label{th:MWTh.2.3}
Let $H : [0,\tau]\times \mathbb{R}^{2n}\to\mathbb{R}$ be as in Proposition~\ref{prop:Ek14},
 $M\in{\rm Sp}(2n,\mathbb{R})$, and $K\in L^\infty([0,\tau];\mathscr{L}_s(\mathbb{R}^{2n}))$.
 Then  functionals
\begin{eqnarray}\label{e:Action}
&&\Phi_{H,K}(v)=\int^\tau_0\left[\frac{1}{2}(J\dot{v}(t)+K(t)v(t),v(t))_{\mathbb{R}^{2n}}+ H(t,{v}(t))\right]dt,\\
&&\Psi_{H,K}(v)=\int^\tau_0\left[\frac{1}{2}(J\dot{v}(t)+K(t)v(t),v(t))_{\mathbb{R}^{2n}}+ H^\ast(t;-J\dot{v}(t)-K(t)v(t))\right]dt\label{e:Action1}
\end{eqnarray}
are $C^1$ and twice G\^ateaux  differentiable on $W^{1,2}_{M}([0,\tau],\R^{2n})$, and the following holds:
\begin{enumerate}
\item[\rm (i)] If $\bar{u}\in W^{1,2}_{M}([0,\tau],\R^{2n})$ is a critical point of $\Phi_{H,K}$,
i.e., a solution of
\begin{equation}\label{e:Ham}
J\dot{u}(t)=-K(t)u(t)-\nabla_z H(t, u(t))\;\forall t\in [0,\tau]\quad\hbox{and}\quad u(\tau)=Mu(0),
\end{equation}
then each $u\in \bar{u}+{\rm Ker}(\Lambda_{M,\tau, K})$
 is also a critical point of $\Psi_{H,K}$ on $W^{1,2}_{M}([0,\tau],\R^{2n})$ and $\Psi_{H,K}({u})=-\Phi_{H,K}(\bar{u})$.
Conversely, if $\bar{v}\in W^{1,2}_{M}([0,\tau],\R^{2n})$ is a critical point  of $\Psi_{H,K}$,
  then there exists a unique $\xi\in\mathcal{E}_{M,\tau,K}$ such that $\bar{u}:=\bar{v}-v_\xi$, where $v_\xi(t)=\Upsilon_K(t)\xi$,
   is a  critical point  of $\Phi_{H,K}$ on $W^{1,2}_{M}([0,\tau],\R^{2n})$
 and satisfies  $\Phi_{H,K}(\bar{u})=-\Psi_{H,K}(\bar{v})$; this unique $\xi$ is equal to
 $(\Upsilon_K(t))^{-1}\big(\bar{v}(t)-\nabla_z H^\ast(t; -J\dot{\bar{v}}(t)-K(t)\bar{v}(t))\big)$
 for any $t\in [0,\tau]$ and therefore
 $\bar{u}(t)=\nabla_z H^\ast(t;-J\dot{\bar{v}}(t)-K(t)\bar{v}(t))$.

\item[\rm (ii)] If $\det(\Upsilon_K(\tau)-M)\ne 0$, $\Phi_{H,K}$ and $\Psi_{H,K}$  have same critical point sets
on  ${W}^{1,2}_{M}([0,\tau];\R^{2n})$ and  $\Phi_{H,K}(u)=-\Psi_{H,K}(u)$ for each critical point $u$ of them;
in fact, $v\mapsto\Gamma(v)$, where $\Gamma(v)(t)=\nabla_z H^\ast(t; -J\dot{v}(t)-K(t)v(t))$,
 is a map from ${\rm Crit}({\Psi}_{H,K})={\rm Crit}(\Phi_{H,K})$ to itself.

\item[\rm (iii)] The restrictions of $\Phi_{H,K}$ and $\Psi_{H,K}$ to $\tilde{W}^{1,2}_{M,K}([0,\tau];\R^{2n})$,
denoted by $\tilde{\Phi}_{H,K}$ and $\tilde{\Psi}_{H,K}$, have same critical point sets and
 $\Phi_{H,K}(u)=-\Psi_{H,K}(u)$ for each critical point $u$ of them.

 \item[\rm (iv)]
 Any critical point of  $\tilde{\Psi}_{H,K}$
 is also one of $\Psi_{H,K}$ on ${W}^{1,2}_{M}([0,\tau];\R^{2n})$.
 \end{enumerate}
    \end{theorem}
\begin{proof}[\bf Proof]
By Propositions~\ref{prop:Ek13}, \ref{prop:Ek14}, the first two claims  can be derived.
  Taking $X=L^2([0,\tau];\R^{2n})$, $\textsf{A}=\Lambda_{M,\tau,K}$, $D(\textsf{A})=W^{1,2}_M([0,\tau];\R^{2n})$ and
  \begin{eqnarray}\label{e:Action2}
&&\Phi=\Phi_{H,K}(v)=\frac{1}{2}\langle \Lambda_{M,\tau,K} v, v\rangle_{L^2}+ \mathcal{H}(v),\\
&&\Psi=\Psi_{H,K}(v)=\frac{1}{2}\langle \Lambda_{M,\tau,K} v, v\rangle_{L^2}+ \mathcal{H}^\ast(-\Lambda_{M,\tau,K} v),\label{e:Action3}
\end{eqnarray}
the desired conclusions in (i), (ii) and (iii)  follow from Proposition~\ref{prop:Dong} and
Corollary~\ref{th:Ek9} immediately. For example, the second claim in (i) can be proved as follows.
Using Proposition~\ref{prop:Dong} and Corollary~\ref{th:Ek9}
 we can get some $\xi\in\mathcal{E}_{M,\tau,K}$ such that  $\bar{u}:=\bar{v}-v_\xi$,
where $v_\xi(t)=\Upsilon_K(t)\xi$,  is a  critical point  of $\Phi_{H,K}$ on $W^{1,2}_{M}([0,\tau],\R^{2n})$
 and satisfies  $\Phi_{H,K}(\bar{u})=-\Psi_{H,K}(\bar{v})$.
Since this $\bar{u}$ satisfies (\ref{e:Ham}) and $\Lambda_{M,\tau,K}\bar{v}-\Lambda_{M,\tau,K}\bar{u}=\Lambda_{M,\tau,K} v_\xi=0$, we derive that for each $t\in [0,\tau]$,
$$
-J\dot{\bar{v}}(t)-K(t)\bar{v}(t)=-J\dot{\bar{u}}(t)-K(t)\bar{u}(t)=\nabla_z H(t, \bar{u}(t))
$$
and so $\bar{v}(t)-v_{\xi}(t)=\bar{u}(t)=\nabla_z H^\ast(t;-J\dot{\bar{v}}(t)-K(t)\bar{v}(t))$
by strict convexity of $H(t,\cdot)$.

({\bf Note}: Since $\mathcal{H}$ and $\mathcal{H}^\ast$
are convex and of class $C^1$ on $L^2([0,\tau];\R^{2n})$, by
 \cite[page 92, Proposition 9]{Ek90} we may also use Corollary~\ref{cor:Ek1} instead of Corollary~\ref{th:Ek9}.)

As to (iv), let $v$ be a critical point of the restriction of
 $\Psi_{H,K}$ to $\tilde{W}^{1,2}_{M,K}([0,\tau];\R^{2n})$.
 For any $h\in W^{1,2}_{M}([0,\tau],\R^{2n})$, since
$$
\int^\tau_0\left[\frac{1}{2}(J\dot{v}(t)+K(t)v(t),h(t))_{\mathbb{R}^{2n}}\right]dt= \int^\tau_0\left[\frac{1}{2}(J\dot{h}(t)+K(t)h(t), v(t))_{\mathbb{R}^{2n}}\right]dt,
$$
we deduce that
\begin{eqnarray}\label{e:Action4}
d\Psi_{H,K}(v)[h]&=&\int^\tau_0\left[\frac{1}{2}(J\dot{v}(t)+K(t)v(t),h(t))_{\mathbb{R}^{2n}}+  \frac{1}{2}(J\dot{h}(t)+K(t)h(t), v(t))_{\mathbb{R}^{2n}}\right]dt\nonumber\\
&+&\int^\tau_0\left[\Big(\nabla_z H^\ast\big(t;-J\dot{v}(t)-K(t)v(t)\big),  -J\dot{h}(t)-K(t)h(t)\Big)_{\mathbb{R}^{2n}}\right]dt\nonumber\\
&=&\int^\tau_0\left[(v(t), J\dot{h}(t)+K(t)h(t))_{\mathbb{R}^{2n}}\right]dt\nonumber\\
&+&\int^\tau_0\left[\Big(\nabla_z H^\ast\big(t;-J\dot{v}(t)-K(t)v(t)\big),  -J\dot{h}(t)-K(t)h(t)\Big)_{\mathbb{R}^{2n}}\right]dt\nonumber\\
&=&\int^\tau_0\left[\Big(v(t)-\nabla_z H^\ast\big(t;-J\dot{v}(t)-K(t)v(t)\big), \Lambda_{M,\tau,K}h \Big)_{\mathbb{R}^{2n}}\right]dt.
\end{eqnarray}
By (\ref{e:orthCom}),
 we can decompose  $h$ into $h_0+h_1$, where $h_0\in{\rm Ker}(\Lambda_{M,\tau,K})$
 and $h_1\in \tilde{W}^{1,2}_{M,K}([0,\tau];\R^{2n})$. Hence (\ref{e:Action4}) leads to
 \begin{eqnarray*}
 d\Psi_{H,K}(v)[h]&=&\int^\tau_0\left[\Big(v(t)-\nabla_z H^\ast\big(t;-J\dot{v}(t)-K(t)v(t)\big), \Lambda_{M,\tau,K}h_1 \Big)_{\mathbb{R}^{2n}}\right]dt\\
 &=&d\tilde{\Psi}_{H,K}(v)[h_1]=0.
 \end{eqnarray*}
That is,  $v$ is a critical point of $\Psi_{H,K}$.
\end{proof}

\begin{remark}\label{rm:MWTh.2.5}
{\rm As pointed out below (\ref{e:orthCom}),
$\Lambda_{M,\tau, K}$
restricts to a Banach space isomorphism $\tilde{\Lambda}_{M,\tau, K}$ from $\tilde{W}^{1,2}_{M, K}([0, \tau];\R^{2n})$ onto
$\tilde{L}^{2}_{M, K}([0, \tau];\R^{2n})$.
Thus under the assumptions of Theorem~\ref{th:MWTh.2.3}, $v$ is a critical point of $\tilde{\Psi}_{H,K}$ on  $\tilde{W}^{1,2}_{M, K}([0,\tau];\R^{2n})$
if and only if $u:=-\tilde{\Lambda}_{M,\tau,K}v$ is that of the  functional $\tilde{\psi}_{H,K}:\tilde{L}^{2}_{M, K}([0, \tau];\R^{2n})\to\mathbb{R}$ given by
$$
\tilde{\psi}_{H,K}(w):=\tilde\Psi_{H,K}\circ(-\tilde{\Lambda}_{M,\tau,K})^{-1}(w)
=\int^\tau_0\left[\frac{1}{2}(w(t), ((\tilde{\Lambda}_{M,\tau,K})^{-1}w)(t))_{\mathbb{R}^{2n}}+
 H^\ast(t; w(t))\right]dt.
 $$
In particular, if $\det(\Upsilon_K(\tau)-M)\ne 0$,
then $v$ is a critical point of $\Psi_{H,K}$ on  ${W}^{1,2}_{M}([0,\tau];\R^{2n})$
if and only if $w:=-\Lambda_{M,\tau,K}v$ is one of the $C^1$ functional
\begin{equation}\label{e:Action***}
\psi_{H,K}(u):=\int^\tau_0\left[\frac{1}{2}(u(t),
((\Lambda_{M,\tau,K})^{-1}u)(t))_{\mathbb{R}^{2n}}+ H^\ast(t; u(t))\right]dt
\end{equation}
on ${L}^{2}([0,\tau];\R^{2n})$. ({\bf Note}: the functional ${\Psi}_{H,K}$ and so $\tilde{\psi}_{H,K}$
is mistakenly considered to be of class $C^2$ in many references. In general, they cannot
be $C^2$, see \cite[Chap.5, Sec.5.1, Theorem~1]{Skr1}.)

For a critical point $w$ of $\tilde{\psi}_{H,K}$, $v:=(-\tilde{\Lambda}_{M,\tau,K})^{-1}w$
is a critical point of $\tilde{\Psi}_{H,K}$ (and so that of ${\Psi}_{H,K}$ by Theorem~\ref{th:MWTh.2.3}(iv)).
It follows from Theorem~\ref{th:MWTh.2.3}(i)  that there exists a unique $\xi\in\mathcal{E}_{M,\tau,K}$ such that  $u:=v-v_\xi$,
where $v_\xi(t)=\Upsilon_K(t)\xi$, is a  critical point  of $\Phi_{H,K}$ on $W^{1,2}_{M}([0,\tau];\R^{2n})$ and
satisfies  $\Phi_{H,K}({u})=-\Psi_{H,K}({v})$. Then for any $t\in [0, \tau]$ we have
$$
-\nabla_z H(t, u(t))=J\dot{u}(t)+K(t)u(t)=J\dot{v}(t)+K(t)v(t)=-w(t)
$$
and hence $\nabla^2_{z}H^\ast(t;w(t))\nabla^2_z H(t,u(t))=I_{2n}$.
A direct computation shows
\begin{eqnarray*}
(\tilde{\psi}_{H,K})''(w)[\xi,\eta]&=&\int^\tau_0\left[(((\tilde{\Lambda}_{M,\tau,K})^{-1}\xi)(t), \eta(t))_{\mathbb{R}^{2n}}+
(\nabla^2_{z}H^\ast(t;w(t))\xi(t),\eta(t))_{\mathbb{R}^{2n}})\right]dt\\
&=&\int^\tau_0\left[(((\tilde{\Lambda}_{M,\tau,K})^{-1}\xi)(t)+
[\nabla^2_z H(t,u(t))]^{-1}\xi(t),\eta(t))_{\mathbb{R}^{2n}})\right]dt
\end{eqnarray*}
for $\xi,\eta\in\tilde{L}^{2}_{M, K}([0, \tau];\R^{2n})$.
By the arguments below Theorem~\ref{th:twoIndex}, $(\tilde{\psi}_{H,K})''(w)$ is
a Legendre form on $\tilde{L}^{2}_{M, K}([0, \tau];\R^{2n})$, and hence has finite
 Morse index and the nullity, denoted by $m^-(\tilde{\psi}_{H,K}, w)$ and $m^0(\tilde{\psi}_{H,K}, w)$, respectively.
By (\ref{e:MorseIndex}) and (\ref{e:Nullity}) we have
  \begin{eqnarray}
   m^-(\tilde{\psi}_{H,K}, w)&=&i_{\tau,M}(\Upsilon_{K+\nabla^2_z H(\cdot, u(\cdot))})-i_{\tau,M}(\Upsilon_{K})-\nu_{\tau,M}(\Upsilon_{K}),\label{e:psi-MorseIndex}\\
  m^0(\tilde{\psi}_{H,K}, w)&=&\nu_{\tau,M}(\Upsilon_{K+\nabla^2_z H(\cdot, u(\cdot))})=\dim{\rm Ker}(\Upsilon_{K+\nabla^2_z H(\cdot, u(\cdot))}(\tau)-M),\label{e:psi-Nullity}
 \end{eqnarray}
 where $i_{\tau,M}$ is defined by  (\ref{e:dongIndex}) and  $\Upsilon_{K+\nabla^2_z H(\cdot, u(\cdot))}$ is the fundamental matrix solution of
 $$
\dot{Z}(t)=J(K(t)+ \nabla^2_z H(t, u(t)))Z(t)\;\forall t\in [0,\tau].
$$
A direct proof of (\ref{e:psi-Nullity}) can also be given by
following \hbox{\cite[Proposition~3.14]{Do06}}. When $\det(\Upsilon_K(\tau)-M)\ne 0$, $\tilde{\psi}_{H,K}$ in
(\ref{e:psi-MorseIndex}) and (\ref{e:psi-Nullity})  is replaced by ${\psi}_{H,K}$ in (\ref{e:Action***}),
and $u=v=(-\tilde{\Lambda}_{M,\tau,K})^{-1}w$.
}
\end{remark}

\begin{corollary}\label{cor:MWTh.2.6}
 For   $M\in{\rm Sp}(2n,\mathbb{R})$, let
 $H: [0,\tau]\times \mathbb{R}^{2n}\to\mathbb{R}, (t, u)\to H(t, u)$
  be measurable in $t$ for each $u\in \mathbb{R}^{2n}$ and twice continuously differentiable
in $u$ for almost every $t\in [0,\tau]$. Assume that there exists $\ell>0$, $\alpha>0$, $\delta>0$, and
$\beta,\gamma\in L^2(0,\tau;\mathbb{R}^+)$ such that
 for almost every $t\in [0,\tau]$,
\begin{eqnarray*}
&&|(\nabla^2_zH(t,\xi)\eta,\eta)_{\mathbb{R}^{2n}}|\le \ell|\eta|^2\quad\forall \xi, \eta\in\mathbb{R}^{2n},\\
&&-\delta |\xi|^2-\beta(t)\le H(t,\xi)\le\alpha |\xi|^2+ \gamma(t)\quad\forall \xi\in\mathbb{R}^{2n}.
\end{eqnarray*}
Then for any real $\kappa<\min\{-\ell, -2\delta\}$, and ${H}_{\kappa}(t,\xi):=H(t,\xi)- \frac{\kappa}{2}|\xi|^2$,  functionals
\begin{eqnarray*}
&&\Phi(v)=\int^\tau_0\left[\frac{1}{2}(J\dot{v}(t),v(t))_{\mathbb{R}^{2n}}+ H({v}(t), t)\right]dt=\Phi_{H_{\kappa},\kappa I_{2n}}(v),\\
&&\Psi_\kappa(v):=\Psi_{H_{\kappa},\kappa I_{2n}}(v)=\int^\tau_0\left[\frac{1}{2}(J\dot{v}(t)+\kappa v(t),v(t))_{\mathbb{R}^{2n}}+ ({H}_{\kappa})^\ast(t;-J\dot{v}(t)-\kappa v(t))\right]dt
\end{eqnarray*}
are $C^1$ and twice G\^ateaux  differentiable on $W^{1,2}_{M}([0,\tau],\R^{2n})$, and the following holds:
\begin{enumerate}
\item[\rm (i)] If $\bar{u}\in W^{1,2}_{M}([0,\tau],\R^{2n})$ is a critical point of $\Phi$,
then each $u\in \bar{u}+{\rm Ker}(\Lambda_{M,\tau, \kappa I_{2n}})$
 is also a critical point of $\Psi_{\kappa}$ on $W^{1,2}_{M}([0,\tau],\R^{2n})$
  and $\Psi_{\kappa}({u})=-\Phi(\bar{u})$.
Conversely, if $\bar{v}\in W^{1,2}_{M}([0,\tau],\R^{2n})$ is a critical point  of $\Psi_{\kappa}$,
  then there exists a unique $\xi\in\mathcal{E}_{M,\tau,\kappa I_{2n}}$ such that $\bar{u}:=\bar{v}-v_\xi$, where
  $v_\xi(t)=\Upsilon_{\kappa I_{2n}}(t)\xi$,
   is a  critical point  of $\Phi$ on $W^{1,2}_{M}([0,\tau],\R^{2n})$
 and satisfies  $\Phi(\bar{u})=-\Psi_{\kappa}(\bar{v})$; this unique $\xi$ is equal to
 $(\Upsilon_{\kappa I_{2n}}(t))^{-1}\big(\bar{v}(t)-\nabla_z (H_\kappa)^\ast(t;-J\dot{\bar{v}}(t)-\kappa \bar{v}(t))\big)$
 and therefore
 $\bar{u}(t)=\nabla_z (H_\kappa)^\ast(t;-J\dot{\bar{v}}(t)-\kappa\bar{v}(t))$.

\item[\rm (ii)] If $\det(\Upsilon_{\kappa I_{2n}}(\tau)-M)\ne 0$, $\Phi$ and $\Psi_{\kappa}$  have the same critical point sets
on  ${W}^{1,2}_{M}([0,\tau];\R^{2n})$ and  $\Phi(u)=-\Psi_{\kappa}(u)$ for each critical point $u$ of them.

\item[\rm (iii)] The restrictions of $\Phi$ and $\Psi_{\kappa}$ to $\tilde{W}^{1,2}_{M,\kappa I_{2n}}([0,\tau];\R^{2n})$,
denoted by $\tilde{\Phi}$ and $\tilde{\Psi}_{\kappa}$, have the same critical point sets and
 $\Phi(u)=-\Psi_{\kappa}(u)$ for each critical point $u$ of them.

 \item[\rm (iv)]
 Any critical point of  $\tilde{\Psi}_{\kappa}$  is also one of $\Psi_{\kappa}$ on ${W}^{1,2}_{M}([0,\tau];\R^{2n})$.
 \end{enumerate}
   Moreover, we can take a real $\kappa<\min\{-\ell, -2\delta\}$ such that
$\det(\Upsilon_{\kappa I_{2n}}(\tau)-M)\ne 0$, and therefore
 $\bar{v}$ is a critical point of $\Psi_{\kappa}$ on  ${W}^{1,2}_{M}([0,\tau];\R^{2n})$
if and only if $\bar{w}:=-\Lambda_{M,\tau,\kappa I_{2n}}\bar{v}$ is that of the $C^1$ and twice G\^ateaux  differentiable functional
\begin{equation}\label{e:Action****}
\psi_{\kappa}(w):=\int^\tau_0\left[\frac{1}{2}(w(t), ((\Lambda_{M,\tau,\kappa I_{2n}})^{-1}w)(t))_{\mathbb{R}^{2n}}+ (H_{\kappa})^\ast(t; w(t))\right]dt
\end{equation}
on ${L}^{2}([0,\tau];\R^{2n})$; in this case there holds
 \begin{eqnarray}
   m^-({\psi}_{\kappa}, \bar{w})&=&i_{\tau,M}(\Upsilon_{\nabla^2_zH(\cdot,\bar{v}(\cdot))})-i_{\tau,M}(\Upsilon_{\kappa I_{2n}})-\nu_{\tau,M}(\Upsilon_{\kappa I_{2n}}),\label{e:psi-MorseIndex+}\\
  m^0({\psi}_{\kappa}, \bar{w})&=&\nu_{\tau,M}(\Upsilon_{\nabla^2_zH(\cdot,\bar{v}(\cdot))})=\dim{\rm Ker}(\Upsilon_{\nabla^2_zH(\cdot,\bar{v}(\cdot))}(\tau)-M),\label{e:psi-Nullity+}
 \end{eqnarray}
 where $\Upsilon_{\kappa I_{2n}}(t)=\exp(t\kappa J)$, and
  $\Upsilon_{\nabla^2_zH(\cdot,\bar{v}(\cdot))}$ is the fundamental matrix solution of
 $$
 \dot{Z}(t)=J(H_t)''(\bar{v}(t))Z(t)=J\nabla^2_zH(t,\bar{v}(t))Z(t).
 $$
 \end{corollary}
\begin{proof}[\bf Proof]
For any $\kappa<\min\{-\ell, -2\delta\}$,  by the assumptions it is easily seen that $H_{\kappa}$ satisfies
(\ref{e:Hess}) and (\ref{e:H-condition}). Therefore applying Theorem~\ref{th:MWTh.2.3}
to $K(t)=\kappa I_{2n}$, $H=H_\kappa$ and therefore $\tilde{\psi}_{H,K}=\psi_{\kappa}$
we obtain the conclusions above ``Moreover''.

Since $\Upsilon_{\kappa I_{2n}}(t)=\exp(t\kappa J)$, by Proposition~\ref{prop:Index}
we can choose a real $\kappa<\min\{-\ell, -2\delta\}$ such that
$\det(\Upsilon_{\kappa I_{2n}}(\tau)-M)\ne 0$.
Other claims follow from Remark~\ref{rm:MWTh.2.5} and the fact that
$\kappa I_{2n}+\nabla^2_z H_\kappa(\cdot, \bar{v}(\cdot))=\nabla^2_z H(\cdot, \bar{v}(\cdot))$.
\end{proof}

\section{Proofs of Theorems~\ref{th:bif-ness},
~\ref{th:bif-suffict} and Corollaries~\ref{cor:necess-suffi},~\ref{cor:bif-deform}
}\label{sec:HamBif}\setcounter{equation}{0}

The following remark about parameter space $\Lambda$ is effective for this section and next sections.

\begin{remark}\label{rm:effective}
{\rm The parameter space $\Lambda$ may be, respectively, replaced by its subsets
\begin{description}
\item[$\bullet$] $\{\mu, \lambda_k\,|\, k\in\mathbb{N}\}$ in
Theorem~\ref{th:bif-ness}(I), Theorem~\ref{th:bif-ness-orbit}, Theorem~\ref{th:bif-nessbrake}(I),Theorem~\ref{th:bif-nessHam}(I),
\item[$\bullet$] $\{\mu, \lambda_k^+, \lambda_k^-\,|\, k\in\mathbb{N}\}$  in Theorem~\ref{th:bif-ness}(II), Theorem~\ref{th:bif-suffict1-orbit}, Theorem~\ref{th:bif-nessbrake}(II), Theorem~\ref{th:bif-nessHam}(II),
\item[$\bullet$] $\alpha([0,1])$ in Theorem~\ref{th:bif-ness}(III), Theorem~\ref{th:bif-existence-orbit}, Theorem~\ref{th:bif-nessbrake}(III), Theorem~\ref{th:bif-nessHam}(III),
\item[$\bullet$] $[\mu-\epsilon, \mu+\epsilon]$, $\epsilon>0$, in Theorem~\ref{th:bif-suffict}, Theorems~\ref{th:bif-per3}, Theorem~\ref{th:bif-suffict-orbit},
Theorems~\ref{th:bif-suffictbrake},~\ref{th:bif-per3brake}, Theorem~\ref{th:bif-suffHam}.
\end{description}
These four sets $\{\mu, \lambda_k\,|\, k\in\mathbb{N}\}$, $\{\mu, \lambda_k^+, \lambda_k^-\,|\, k\in\mathbb{N}\}$,
$\alpha([0,1])$ and $[\mu-\epsilon, \mu+\epsilon]$ are compact and
sequentially compact subsets of $\Lambda$.}
\end{remark}

\subsection{Proofs of Theorems~\ref{th:bif-ness},
~\ref{th:bif-suffict} and Corollary~\ref{cor:necess-suffi}
for an orthogonal symplectic matrix $M$}\label{sec:HamBif}

 The assumption that the symplectic matrix $M\in{\rm Sp}(2n,\mathbb{R})$
is orthogonal is essentially used in two places:
\begin{itemize}
\item[(1)] in modifying $H$ outside an open neighborhood of $0\in\mathbb{R}^{2n}$, as described above equation (\ref{e:modifyH});
\item[(2)] in Step~4 of the proof of Theorem~\ref{th:bif-ness}(II), specifically for verifying the conditions of Theorem~\ref{th:A.9}.
\end{itemize}

Our proofs will be completed with theorems  in \cite{Lu8, Lu10}
and Theorems~\ref{th:A.9},~\ref{th:A.9+},~\ref{th:A.10}.
To this goal let us define
$\bar{H}:\Lambda\times[0,\tau]\times{\R}^{2n}\to\R$ by
\begin{equation}\label{e:ModifiedH}
\bar{H}(\lambda,t,z)={H}(\lambda,t,z+u_\lambda(t))-(z, \nabla_z{H}(\lambda,t, u_\lambda(t)))_{\mathbb{R}^{2n}}.
\end{equation}
Since each $\Lambda\times [0,\tau]\ni(\lambda,t)\mapsto u_\lambda(t)\in\R^{2n}$ is  continuous,
it is clear that $\bar{H}$ satisfies Assumption~\ref{ass:BasiAss1}
and $\bar{u}\equiv 0\in\mathbb{R}^{2n}$ satisfies
\begin{equation}\label{e:Hboundary*}
\dot{u}(t)=J\nabla_z\bar{H}(\lambda,t, u(t))\;\forall t\in [0,\tau]\quad\hbox{and}\quad u(\tau)=Mu(0)
\end{equation}
for each $\lambda\in\Lambda$.
Note that a function $u:[0,\tau]\to\R^{2n}$ satisfies
(\ref{e:Hboundary}) with the parameter value $\lambda$ if and only if
the function $w(t):=u(t)-u_\lambda(t)$
is a solution of (\ref{e:Hboundary*}). Hence  we have

\begin{claim}\label{cl:Equiv}
For $X=W^{1,2}_{M}([0,\tau];\R^{2n})$ or $C^1_{M}([0,\tau];\R^{2n})$, the bifurcation problem of (\ref{e:Hboundary}) in $\Lambda\times X$
with respect to the branch $\{(\lambda,u_\lambda)\,|\,\lambda\in\Lambda\}$
is equivalent to that of (\ref{e:Hboundary*}) in $\Lambda\times X$
with respect to the trivial branch $\{(\lambda, 0)\,|\,\lambda\in\Lambda\}$.
\end{claim}

Since $\nabla^2_z\bar{H}(\lambda,t,0)=\nabla^2_zH(\lambda,t, u_\lambda(t))$
for all $(\lambda,t)\in \Lambda\times [0,\tau]$, $\gamma_\lambda:[0,\tau]\to {\rm Sp}(2n,\mathbb{R})$
is also the  fundamental matrix solution of
$$
\dot{Z}(t)=J\nabla^2_z\bar{H}(\lambda,t, 0)Z(t).
$$
\textsf{Therefore in what follows we only need to prove Theorems~\ref{th:bif-ness},~\ref{th:bif-suffict}
in the case where $u_\lambda\equiv 0$ for all} $\lambda\in\Lambda$. (This means $\nabla_z{H}(\lambda,t, 0)=0\;\forall (\lambda,t)$.)\\

\noindent{\bf Modify $H$ outside an open neighborhood of $0\in\mathbb{R}^{2n}$}.
By Remark~\ref{rm:effective} \textsf{we always assume that $\Lambda$ is compact and
 sequentially compact} in what follows.
By Assumption~\ref{ass:BasiAss1}
$$
\Lambda\times [0,\tau]\ni (\lambda,t)\mapsto \nabla^2_zH(\lambda,t,0)\in\mathscr{L}_s(\mathbb{R}^{2n})
$$
is continuous.  Therefore we can get a constant $C>0$ and
 a compact neighborhood $U$ of $0$ in $\mathbb{R}^{2n}$ such that
$$
-CI_{2n}\le \nabla^2_zH(\lambda,t,z)\le CI_{2n},\quad\forall (\lambda,t, z)\in \Lambda\times [0,\tau]\times U.
$$
Take positive numbers $0<\delta_1<\delta_2$ such that
the closed ball  $\bar{B}^{2n}(0,\delta_2^2)\subset U$, and
a smooth cut-off function $\rho:[0, \infty)\to [0, 1]$ such that
$\rho(t)=1$ for $t\le\delta_1^2$ and $\rho(t)=0$ for $t\ge\delta_2^2$.
Define a function $\chi:\mathbb{R}^{2n}\to [0, 1]$ by
$$
\chi(z)=\rho(|z|^2),\quad\forall z\in\mathbb{R}^{2n},
$$
and
$$
\tilde{H}:\Lambda\times [0, \tau]\times\mathbb{R}^{2n}\to \R,\;(\lambda,t,z)\mapsto\chi(z)H(\lambda,t,z).
$$
Since $M$ is an orthogonal symplectic matrix, we have $\chi(Mz)=\chi(z)$ and hence
 $$
 \tilde{H}(\lambda,\tau, Mz)=\chi(Mz)H(\lambda,\tau, Mz)=\chi(z)H(\lambda, 0, z)=
\tilde{H}(\lambda, 0, z).
$$
 Then it is easily computed that for any $\xi,\eta\in\mathbb{R}^{2n}$,
\begin{eqnarray*}
(\nabla^2_z\tilde{H}({\lambda,t},z)\xi,\eta)_{\mathbb{R}^{2n}}&=&(H_{\lambda,t})'(z)[\xi]\chi'(z)[\eta]+ \chi(z)(\nabla^2_zH(\lambda,t,z)\xi,\eta)_{\mathbb{R}^{2n}}\\
&&+(H_{\lambda,t})'(z)[\eta]\chi'(z)[\xi]+ H(\lambda,t, z)(\chi''(z)\xi,\eta)_{\mathbb{R}^{2n}}.
\end{eqnarray*}
Since $H$ and $\nabla_zH$ are continuous by Assumption~\ref{ass:BasiAss1},
it follows from the compactness of $\Lambda$ that there exists a constant $C'>0$ such that
$$
-C'I_{2n}\le\nabla^2_z\tilde{H}(\lambda,t,z)\le C'I_{2n},\quad\forall (\lambda,t, z)\in \Lambda\times [0,\tau]\times \mathbb{R}^{2n}.
$$
Note that for some small $\epsilon>0$, that  $\|u\|_{1,2}\le\epsilon$ implies $u([0,\tau])\subset U_0$, and
that we are only concerned with solutions of (\ref{e:Hboundary})
near $0\in W_M^{1,2}([0,\tau];\mathbb{R}^n)$.
Replacing $H$ by $\tilde{H}$ we can assume  that $H$ satisfies
\begin{equation}\label{e:modifyH}
-CI_{2n}\le \nabla^2_zH(\lambda,t,z)\le CI_{2n},
\quad\forall (\lambda,t, z)\in \Lambda\times [0,\tau]\times \mathbb{R}^{2n}.
\end{equation}

For a given $(\lambda,t, z)\in \Lambda\times [0,\tau]\times \mathbb{R}^{2n}$,
using Taylor expansion we have $\theta\in (0, 1)$ such that
\begin{eqnarray*}
H(\lambda,t, z)&=&H(\lambda,t, 0)+ (\nabla_z H(\lambda,t,0), z)_{\mathbb{R}^{2n}}+
 \frac{1}{2}\big(\nabla_z^2H(\lambda,t,\theta z)z, z\big)_{\mathbb{R}^{2n}}\\
&=&H(\lambda,t, 0)+ \frac{1}{2}\big(\nabla^2_zH(\lambda,t,\theta z)z, z\big)_{\mathbb{R}^{2n}}.
\end{eqnarray*}
Since  $(\lambda,t)\to H(\lambda,t, 0)$ is continuous,
and hence bounded,  from the above expression it follows
that there exist constants $c'_1>0, c'_2>0$ such that
 \begin{equation}\label{e:modifyH1}
-c'_1|z|^2-c'_2\le H(\lambda,t, z)\le c'_1|z|^2+ c'_2,\quad\forall (\lambda,t, z)\in \Lambda\times [0,\tau]\times \mathbb{R}^{2n}.
\end{equation}

\begin{proof}[\bf Proof of Theorem~\ref{th:bif-ness}(I)]
By the above arguments we can  assume that $H$ satisfies
 (\ref{e:modifyH}) and (\ref{e:modifyH1}). Then we may choose $\kappa<0$,
 $c_i>0$, $i=1,2,3$   such that
   \begin{enumerate}
   \item[(i)] $\det(e^{\tau \kappa J}M-I_{2n})\ne 0$;
   \item[(ii)] each $H_\kappa (\lambda, t, z):=H(\lambda, t, z)- \frac{\kappa}{2}|z|^2$ satisfies
 $H_\kappa (\lambda, \tau, Mz)=H_\kappa (\lambda, 0, z)$ and
   \begin{equation}\label{e:HPositive}
   c_1I_{2n}\le \nabla^2_zH_\kappa(\lambda,t,z)\le c_2I_{2n},\quad\forall (\lambda,t,z);
   \end{equation}
  \item[(iii)] $\frac{c_1}{2} |z|^2- c_3\le H_\kappa(\lambda, t, z)\le c_2|z|^2+ c_3$
  for all $(\lambda, t, z)$.
   \end{enumerate}
   Let $(H_\kappa)^\ast(\lambda,t;z)=(H_\kappa(\lambda,t,\cdot))^\ast(z)$. These and Assumption~\ref{ass:BasiAss1} imply:
  \begin{enumerate}
   \item[(iv)]  For $\xi\in \mathbb{R}^{2n}$ and $\xi^\ast=\nabla_z H_\kappa(\lambda,t, \xi)$  it holds that
 \begin{equation}\label{e:HPositive+}
 \xi=\nabla_{z}(H_\kappa)^\ast(\lambda,t;\xi^\ast)\quad\hbox{and}\quad I_{2n}=\nabla^2_z(H_\kappa)^\ast(\lambda,t;\xi^\ast) \nabla^2_zH_\kappa(\lambda,t,\xi).
  \end{equation}
  \item[(v)] Because of (\ref{e:HPositive}), it holds for all $(\lambda,t,z)$  that
 \begin{equation}\label{e:HPositive1}
  \frac{1}{c_2}I_{2n}\le \nabla_z^2(H_\kappa)^\ast(\lambda,t;z)\le \frac{1}{c_1}I_{2n}.
  \end{equation}
  \item[(vi)] $(H_\kappa)^\ast:\Lambda\times [0,\tau]\times{\R}^{2n}\to\R$
also satisfies Assumption~\ref{ass:BasiAss1}, that is, it
is a continuous function such that each
$(H_\kappa)^\ast(\lambda,t;\cdot):{\R}^{2n}\to\R$, $(\lambda,t)\in\Lambda\times [0,\tau]$, is $C^2$ and all possible partial derivatives of it depend continuously on
 $(\lambda, t, z)\in\Lambda\times [0,\tau]\times\mathbb{R}^{2n}$.
 (These  follow from (\ref{e:HPositive+}) and the implicit function theorem.)
  \item[(vii)] $\frac{1}{c_2} |z|^2- c_3\le (H_\kappa)^\ast(\lambda, t; z)\le \frac{2}{c_1}|z|^2+ c_3$ for all $(\lambda, t, z)$.
    \end{enumerate}
By these and Corollary~\ref{cor:MWTh.2.6}  the functionals  $\Phi(\lambda,\cdot):
W^{1,2}_{M}([0,\tau];\R^{2n})\to\R$ defined by
\begin{eqnarray}\label{e:HActionK}
\Phi(\lambda, v):=\int^{\tau}_0\left[\frac{1}{2}(J\dot{v}(t),v(t))_{\mathbb{R}^{2n}}+ H(\lambda,t, {v}(t))\right]dt
\end{eqnarray}
and $\Psi_\kappa(\lambda,\cdot):
W^{1,2}_{M}([0,\tau];\R^{2n})\to\R$ defined by
\begin{equation}\label{e:HAction*}
\Psi_\kappa(\lambda, v)=\int^{\tau}_0\left[\frac{1}{2}(J\dot{v}(t)+\kappa v(t),v(t))_{\mathbb{R}^{2n}}+
 (H_\kappa)^\ast(\lambda, t;-J\dot{v}(t)-\kappa v(t))\right]dt,
\end{equation}
are  $C^1$ and twice G\^ateaux-differentiable, and have same critical point sets,
 which exactly correspond to  solutions of (\ref{e:Hboundary}) with the parameter value $\lambda$.
Moreover,  $\Phi(\lambda, u)=-\Psi_\kappa(\lambda, u)$ for any critical point $u$ of them.
As noted in Remark~\ref{rm:MWTh.2.5},
(i) also implies that
\begin{equation}\label{e:Bisom}
\Lambda_{M,\tau,\kappa I_{2n}}:W^{1,2}_{M}([0,\tau];\R^{2n})\to L^{2}([0,\tau];\R^{2n})
\end{equation}
is a Banach space isomorphism, and thus the functional
$$
\psi_\kappa(\lambda, \cdot)\stackrel{\rm def}{=}\Psi_\kappa(\lambda, \cdot)\circ(-\Lambda_{M,\tau,\kappa I_{2n}})^{-1}:
L^2([0,\tau];\R^{2n})\to\R
$$
given by
\begin{equation}\label{e:HAction**}
\psi_\kappa(\lambda, u)=\int^{\tau}_0\left[\frac{1}{2}(u(t), ((\Lambda_{M,\tau,\kappa I_{2n}})^{-1}u)(t))_{\mathbb{R}^{2n}}+ (H_\kappa)^\ast(\lambda, t; u(t))\right]dt
\end{equation}
has the same analytical properties as those of $\Psi_\kappa(\lambda, \cdot)$. In particular, $\psi_\kappa(\lambda, \cdot)$
is $C^1$ and twice G\^ateaux-differentiable.

Let us prove that \cite[Theorem~3.1]{Lu8} (Theorem~\ref{th:A.10}) is applicable to
$$
\mathcal{F}_\lambda(\cdot)=\psi_\kappa(\lambda, \cdot),\quad
H=L^2([0,\tau];\R^{2n}),\quad X=L^2([0,\tau];\R^{2n})
$$
(thus any Banach space $X$ which is dense in $H=L^2([0,\tau];\R^{2n})$).
 To this end we define
\begin{equation}\label{e:Lambda-inverse}
A_{M,\tau,\kappa}:L^{2}([0,\tau];\R^{2n})\to L^{2}([0,\tau];\R^{2n}),\;u\mapsto  \iota\circ(\Lambda_{M,\tau,{\kappa I_{2n}}})^{-1}u,
\end{equation}
 where
$\iota:W^{1,2}_{M}([0,\tau];\R^{2n})\to L^2([0,\tau];\R^{2n})$ is the inclusion. By (\ref{e:inver})
\begin{equation}\label{e:inver*}
[({\Lambda}_{M,\tau, \kappa I_{2n}})^{-1}u](t)=\Upsilon_{\kappa I_{2n}}(t)\mathfrak{J}^{-1}\int^\tau_0\Upsilon_{\kappa I_{2n}}(s)^{-1}Ju(s)ds-
\Upsilon_{\kappa I_{2n}}(t)\int^t_0\Upsilon_{\kappa I_{2n}}(s)^{-1}Ju(s)ds.
\end{equation}
 Then $A_{M,\tau,\kappa}$ is a compact self-adjoint operator.
It is easily proved that $\psi_\kappa(\lambda,\cdot)$ has the gradient
\begin{equation}\label{e:p-gradient}
\nabla_v\psi_\kappa(\lambda, v)=A_{M,\tau,\kappa}v+ \nabla_z(H_\kappa)^\ast(\lambda, \cdot; v(\cdot))
\end{equation}
and that $L^{2}([0,\tau];\R^{2n})\ni v\mapsto \nabla_v\psi_\kappa(\lambda, v)\in L^{2}([0,\tau];\R^{2n})$
has a G\^ateaux derivative
\begin{equation}\label{e:p-self-adjoint}
B_\lambda(v):=D_v\nabla_v\psi_\kappa(\lambda, v)=A_{M,\tau,\kappa}+ \nabla^2_z(H_\kappa^\ast)(\lambda, \cdot; v(\cdot))
\in \mathscr{L}_s(L^{2}([0,\tau];\R^{2n})).
\end{equation}
Denote by
$$
P_\lambda(v)=\nabla^2_z(H_\kappa)^\ast(\lambda, \cdot; v(\cdot))\quad\hbox{and}\quad Q_\lambda(v)=A_{M,\tau,\kappa}.
$$
 Both are in $\mathscr{L}_s(L^{2}([0,\tau];\R^{2n}))$.
Clearly, $Q_\lambda$ satisfies  (iii)-(iv) of \cite[Theorem~3.1]{Lu8}.
By (\ref{e:HPositive1}),
$P_\lambda(v)\in \mathscr{L}_s(L^{2}([0,\tau];\R^{2n}))$
is uniformly positive definite with respect to $(\lambda,v)$, i.e., it satisfies
(ii) of \cite[Theorem~3.1]{Lu8}.

We also need to prove that $P_\lambda$ satisfies (i) of \cite[Theorem~3.1]{Lu8},
that is, for any $h\in L^{2}([0,\tau];\R^{2n})$ there holds
\begin{eqnarray*}
\|P_{\lambda_k}(v_k)h-P_{\mu}(0)h\|_2^2=\int^{\tau}_0|\nabla^2_z(H_\kappa)^\ast(\lambda_k,t;v_k(t))h(t)-
\nabla_z^2(H_\kappa)^\ast(\mu,t; 0)h(t)|^2dt\to 0
\end{eqnarray*}
provided that $(v_k)\subset L^{2}([0,\tau];\R^{2n})$
 and $(\lambda_k)\subset\Lambda$ converge to $0\in L^{2}([0,\tau];\R^{2n})$ and $\mu\in\Lambda$, respectively.
 Arguing by contradiction,  we can assume after passing to subsequences if necessary that
 \begin{equation}\label{e:contrad}
 \hbox{$v_k\to 0$ a.e., and
$\|P_{\lambda_k}(v_k)h-P_{\mu}(0)h\|_2\ge \varepsilon_0$ for some $\varepsilon_0>0$ and all $k$.}
\end{equation}
By the above (iv) the map
$(\lambda,t,z)\mapsto \nabla_z^2(H_\kappa^\ast)(\lambda,t; z)$
 is continuous, and thus
 $$
 \hbox{$\nabla_z^2(H_\kappa)^\ast(\lambda_k,t;v_k(t))h(t)\to \nabla_z^2(H_\kappa)^\ast(\mu,t; 0)h(t)$ almost everywhere.}
 $$
Observe that (\ref{e:HPositive1}) implies
\begin{eqnarray*}
&&|\nabla^2_z(H_\kappa)^\ast(\lambda_k,t;v_k(t))h(t)-\nabla^2_z(H_\kappa)^\ast(\mu,t; 0)h(t)|^2\\
&&\le 4(|\nabla_z^2(H_\kappa)^\ast(\lambda_k,t;v_k(t))h(t)|^2+
|\nabla^2_z(H_\kappa)^\ast(\mu,t; 0)h(t)|^2)\le \frac{8}{c_1^2}|h(t)|^2.
\end{eqnarray*}
Using Lebesgue dominated convergence theorem we deduce
\begin{eqnarray*}
\lim_{k\to\infty}\int^{\tau}_0|\nabla^2_z(H_\kappa)^\ast(\lambda_k,t;v_k(t))h(t)-
\nabla_z^2(H_\kappa)^\ast(\mu,t; 0)h(t)|^2dt=0.
\end{eqnarray*}
This contradicts (\ref{e:contrad}). The desired claim is proved.

Since $(\mu, 0)\in\Lambda\times W^{1,2}_{M}([0,\tau];\R^{2n})$ is
a bifurcation point along sequences of
the problem  (\ref{e:Hboundary}) if and only if it is that of solutions of
$\nabla_u\Phi(\lambda,u)=0$ (or equivalently $\nabla_u\Psi_\kappa(\lambda,u)=0$)
in $\Lambda\times W^{1,2}_{M}([0,\tau];\R^{2n})$, we deduce that $(\mu, 0)\in\Lambda\times W^{1,2}_{M}([0,\tau];\R^{2n})$ is a bifurcation point along sequences
of the problem  (\ref{e:Hboundary}) if and only if $(\mu, 0)\in\Lambda\times L^2([0,\tau];\R^{2n})$
 is a bifurcation point along sequences for solutions of $\nabla_w\psi_\kappa(\lambda, w)=0$
in $\Lambda\times L^2([0,\tau];\R^{2n})$.

It follows from the assumption of  Theorem~\ref{th:bif-ness}(I) and  \cite[Theorem~3.1]{Lu8} (Theorem~\ref{th:A.10}) that $0\in L^2([0,\tau];\R^{2n})$ is a degenerate critical point of
$\psi_\kappa(\mu,\cdot)$ and therefore
$$
\dim{\rm Ker}(\Upsilon_{\nabla_z^2H(\mu,\cdot,0)}(\tau)-M)=m^0(\psi_\kappa(\mu,\cdot), 0)>0
$$
by  (\ref{e:psi-Nullity+}).
Note that $\gamma_\mu=\Upsilon_{\nabla_z^2H(\mu,\cdot,0)}$ by the definition of $\gamma_\lambda$. Hence
\begin{eqnarray*}
\nu_{\tau,M}(\gamma_\mu)=\dim{\rm Ker}(\gamma_{\mu}(\tau)-M)=\dim{\rm Ker}(\Upsilon_{\nabla^2_zH(\mu,\cdot,0)}(\tau)-M)>0.
 \end{eqnarray*}
 {\it Note}: \cite[Theorem~3.2]{Lu8} (Theorem~\ref{th:A.10}) can be also used because its conditions are satisfied by the following proof of Theorem~\ref{th:bif-ness}(II).
 \end{proof}

\begin{proof}[\bf Proof of Theorem~\ref{th:bif-ness}(II)]
 The arguments before the final paragraph in the proof of Theorem~\ref{th:bif-ness}(I)
are still valid.
Clearly, the isomorphism $\Lambda_{M,\tau,\kappa I_{2n}}$ in (\ref{e:Bisom})
gives rise to a Banach space isomorphism from $C^1_{M}([0,\tau];\R^{2n})$ onto $C^0_M([0,\tau];\R^{2n})$,
denoted by $\Lambda_{M,\tau,\kappa I_{2n}}^c$ for the sake of clearness.
If $w\in C^0_{M}([0,\tau];\R^{2n})$ then $u:=A_{M,\tau,\kappa}w$ belongs to
 $u\in W^{1,2}_{M}([0,\tau];\R^{2n})$ and satisfies $J\dot{u}+ \kappa u=\Lambda_{M,\tau,\kappa I_{2n}}u=w$.
It follows that $J\dot{u}=w-u\in C^0_{M}([0,\tau];\R^{2n})$ and hence
$u\in C^1_{M}([0,\tau];\R^{2n})$.

\textsf{Let us prove in five steps that the family
$$
\{\mathcal{L}_\lambda(\cdot)=\psi_\kappa(\lambda, \cdot)\,|\,
 \lambda\in \Lambda\}
 $$
 satisfies the conditions of Theorem~\ref{th:A.9}
with $\lambda^\ast=\mu$, $H=L^2([0,\tau];\R^{2n})$ and $X=C^0_M([0,\tau];\R^{2n})$
except for the condition that ${\rm Ker}(B_{\lambda^\ast}(0))\ne\{0\}$.}

{\bf Step 1}  ({\it Prove $\psi_\kappa$ to be continuous}). By a contradiction we assume that
 $\psi_\kappa$ is not continuous at some point
 $(\bar{\lambda}, \bar{u})\in\Lambda\times L^2([0,\tau];\R^{2n})$.
 Since $\Lambda\times L^2([0,\tau];\R^{2n})$ is first countable,
  there exists $\varepsilon>0$ and a sequence
  $\{(\lambda_k,u_k)\}^\infty_{k=1}$ in $\Lambda\times L^2([0,\tau];\R^{2n})$ converging to
$(\bar{\lambda}, \bar{u})$ such that
\begin{equation}\label{e:Th1.3(II)1}
|\psi_\kappa({\lambda_k}, {u}_k)-\psi_\kappa(\bar{\lambda}, \bar{u})|\ge\varepsilon,\quad\forall k=1,2,\cdots.
\end{equation}
By (vii) in the proof of Theorem~\ref{th:bif-ness}(I) we have
 \begin{equation}\label{e:HPositive2}
 |H^\ast_\kappa(\lambda, t; z)|\le \frac{1}{c_1}|z|^2+ c_3\quad\forall (\lambda, t, z).
  \end{equation}
Since $H_\kappa^\ast:\Lambda\times [0,\tau]\times{\R}^{2n}\to\R$ is continuous,
it follows from Lebesgue dominated convergence theorem that
$$
\lim_{\lambda\to\bar{\lambda}}\int^\tau_0H_\kappa^\ast(\lambda, t; \bar{u}(t))dt=\int^\tau_0 H_\kappa^\ast(\bar{\lambda}, t; \bar{u}(t))dt.
$$
Moreover, (\ref{e:HPositive2})  and \cite[Proposition C.1]{Lu9} imply that maps
\begin{equation*}
 L^2([0,\tau];\R^{2n})\to L^1([0,\tau]),\; u\mapsto H_\kappa^\ast(\lambda, \cdot; u(\cdot))
\end{equation*}
are uniformly continuous at $\bar{u}$ with respect to $\lambda\in\Lambda^\ast:=\{\bar{\lambda},\lambda_k\,|\,k\in\mathbb{N}\}$. Then
we have a neighborhood $\Lambda_0$ of $\bar{\lambda}$ in $\Lambda$ and a natural number $N$  such that
\begin{eqnarray*}
&&\left|\int^\tau_0H_\kappa^\ast(\lambda, t; \bar{u}(t))dt-\int^\tau_0 H_\kappa^\ast(\bar{\lambda}, t; \bar{u}(t))dt\right|<\frac{\varepsilon}{4},
\quad\forall\lambda\in\Lambda_0,\\
&&\left|\int^\tau_0H_\kappa^\ast(\lambda_k, t; {u}_k(t))dt-\int^\tau_0 H_\kappa^\ast({\lambda}_k, t; \bar{u}(t))dt\right|<\frac{\varepsilon}{4}\quad\forall k>N, \\
&&\frac{1}{2}\left|\int^{\tau}_0(u_k(t), ((\Lambda_{M,\tau,\kappa I_{2n}})^{-1}u_k)(t))_{\mathbb{R}^{2n}}dt-
\int^{\tau}_0(\bar{u}(t), ((\Lambda_{M,\tau,\kappa I_{2n}})^{-1}\bar{u})(t))_{\mathbb{R}^{2n}}dt\right|<\frac{\varepsilon}{4}
\end{eqnarray*}
if $k>N$.
We can assume $\lambda_k\in\Lambda_0$ for all $k>N$. It follows from these and (\ref{e:HAction**}) that
$$
|\psi_\kappa({\lambda}_k, {u}_k)-\psi_\kappa(\bar{\lambda}, \bar{u})|<\varepsilon,\quad\forall k>N.
$$
This contradicts (\ref{e:Th1.3(II)1}).
Hence $\{\mathcal{L}_\lambda\,|\, \lambda\in \Lambda\}$ is a continuous family.

{\bf Step 2}. By (\ref{e:HPositive+})-(\ref{e:HPositive1}) and (\ref{e:p-gradient})
it is easily seen that $\nabla_v\psi_\kappa(\lambda, v)\in X$ for each $v\in X$.
Thus
$$
\mathscr{A}:\Lambda\times X\to X,\;(\lambda,v)\mapsto \nabla_v\psi_\kappa(\lambda, v)
=\iota^c\circ(\Lambda^c_{M,\tau,\kappa I_{2n}})^{-1}v+ \nabla_zH_\kappa^\ast(\lambda, \cdot; v(\cdot))
$$
is well-defined, where  $\iota^c:C^1_{M}([0,\tau];\R^{2n})\hookrightarrow C^0_M([0,\tau];\R^{2n})$
is the inclusion.
By a contradiction we assume that $\mathscr{A}$  is not continuous at some point $(\bar{\lambda}, \bar{u})\in \Lambda\times X$.
Then there exists $\varepsilon>0$ and a sequence $\{(\lambda_k,u_k)\}^\infty_{k=1}$ in $\Lambda\times X$ converging to
$(\bar{\lambda}, \bar{u})$ such that
\begin{equation}\label{e:Th1.3(II)2}
\|\mathscr{A}({\lambda_k}, {u}_k)-\mathscr{A}(\bar{\lambda}, \bar{u})\|_{C^0}\ge\varepsilon,\quad\forall k=1,2,\cdots.
\end{equation}
(Here we have also used the first countability of $\Lambda$ and $X$.)
Since $\iota^c\circ(\Lambda^c_{M,\tau,\kappa I_{2n}})^{-1}$ is continuous,
we deduce that $\|\iota^c\circ(\Lambda^c_{M,\tau,\kappa I_{2n}})^{-1}u_k-
\iota^c\circ(\Lambda^c_{M,\tau,\kappa I_{2n}})^{-1}\bar{u}\|_{C^0}\to 0$ as $k\to\infty$.
Therefore   there exists a natural number $N$ such that
$\|\iota^c\circ(\Lambda^c_{M,\tau,\kappa I_{2n}})^{-1}u_k-
\iota^c\circ(\Lambda^c_{M,\tau,\kappa I_{2n}})^{-1}\bar{u}\|_{C^0}<\varepsilon/2$ for all $k\ge N$.
It follows from this and (\ref{e:Th1.3(II)2}) that
\begin{equation}\label{e:Th1.3(II)3}
\|\nabla_zH_\kappa^\ast(\lambda_k, \cdot; {u}_k(\cdot))-\nabla_zH_\kappa^\ast(\bar{\lambda}, \cdot; \bar{u}(\cdot))\|_{C^0}\ge\frac{\varepsilon}{2},\quad\forall k\ge N.
\end{equation}
These mean that there exists a sequence $\{t_k\}_{k=N}^\infty$ in $[0,\tau]$ such that
\begin{equation}\label{e:Th1.3(II)4}
|\nabla_zH_\kappa^\ast(\lambda_k, t_k; {u}_k(t_k))-\nabla_zH_\kappa^\ast(\bar{\lambda}, t_k; \bar{u}(t_k))|_{\mathbb{R}^{2n}}\ge\frac{\varepsilon}{2},\quad\forall k\ge N.
\end{equation}
Passing to a subsequence we can assume $t_k\to\bar{t}\in [0,\tau]$.
 Since $\nabla_zH_\kappa^\ast(\lambda, t; z)$ is continuous and
\begin{eqnarray*}
|{u}_k(t_k)-\bar{u}(\bar{t})|_{\mathbb{R}^{2n}}&\le& |{u}_k(t_k)-\bar{u}(t_k)|_{\mathbb{R}^{2n}}+|\bar{u}({t}_k)-\bar{u}(\bar{t})|_{\mathbb{R}^{2n}}\\
&\le& \|{u}_k-\bar{u}\|_{C^0}+|\bar{u}({t}_k)-\bar{u}(\bar{t})|_{\mathbb{R}^{2n}}\to 0,
\end{eqnarray*}
it follows from (\ref{e:Th1.3(II)4}) that
$0=|\nabla_zH_\kappa^\ast(\bar{\lambda}, \bar{t}; \bar{u}(\bar{t}))-\nabla_zH_\kappa^\ast(\bar{\lambda}, \bar{t}; \bar{u}(\bar{t}))|_{\mathbb{R}^{2n}}\ge\frac{\varepsilon}{2}$.
This contradiction shows that $\mathscr{A}$ is  continuous.

{\bf Step 3}. We claim that $\mathscr{A}_\lambda$ is $C^1$.
In fact, as in the proof of Proposition~\ref{prop:Ek14} we obtain that
$\mathscr{A}_\lambda(\cdot):=\mathscr{A}(\lambda,\cdot)$ has the G\^ateaux derivative at $w\in X$,
$$
D\mathscr{A}_\lambda(w):X\to\mathscr{L}(X),\;\xi\mapsto \nabla_z^2(H_\kappa)^\ast(\lambda, \cdot; w(\cdot))\xi(\cdot)+ \iota^c\circ(\Lambda^c_{M,\tau,\kappa I_{2n}})^{-1}.
$$
For any given $w\in X$, since  $(t,z)\mapsto\nabla^2_z(H_\kappa)^\ast(\lambda, t; z)$ is uniformly continuous
on a compact neighborhood of $[0,\tau]\times w([0,\tau])$ in $[0,\tau]\times\R^{2n}$, using the inequality
\begin{eqnarray*}
\|D\mathscr{A}_\lambda(u)-D\mathscr{A}_\lambda(w)\|_{\mathscr{L}(X)}\le\max_{0\le t\le\tau}
\|\nabla_z^2(H_\kappa)^\ast(\lambda, t; u(t))-\nabla_z^2(H_\kappa)^\ast(\lambda, t; w(t))\|_{\mathbb{R}^{2n\times 2n}}
\end{eqnarray*}
for $u\in X$  we can deduce that $\|D\mathscr{A}_\lambda(u)-D\mathscr{A}_\lambda(w)\|_{\mathscr{L}(X)}\to 0$
as $\|u-w\|_{C^0}\to 0$ and therefore that $\mathscr{A}_\lambda$ is $C^1$.

{\bf Step 4}.   For $B_\lambda(v)$ given by (\ref{e:p-self-adjoint}) we claim
\begin{equation}\label{e:C-D-D1}
\{u\in H\,|\,{B}_\lambda(0)u=su,\, s\le 0\}\subset X\quad\hbox{and}\quad
\{u\in H\,|\,{B}_\lambda(0)u\in X\}\subset X.
\end{equation}
That is, both (D1) in \cite[Hypothesis~1.1]{Lu8} and (C) in
\cite[Hypothesis~1.3]{Lu8} hold with $B_\lambda(0)$.

In fact, suppose that  $u\in H$
and $\varrho\le 0$ satisfy $B_\lambda(0)u=\varrho u$.
Then
$$
 (A_{M,\tau,\kappa}u)(t)+ \nabla^2_z(H_\kappa)^\ast(\lambda, t; 0)u(t)=\varrho u(t).
$$
By (\ref{e:HPositive1}), $\nabla^2_z(H_\kappa)^\ast(\lambda, t; 0)-\varrho I_{2n}\ge \frac{1}{c_2}I_{2n}$
and so $\nabla^2_z(H_\kappa)^\ast(\lambda, t; 0)-\varrho I_{2n}$ is invertible.
Since
$$
A_{M,\tau,\kappa}u=\iota\circ(\Lambda_{M,\tau,\kappa I_{2n}})^{-1}u=(\Lambda_{M,\tau,\kappa I_{2n}})^{-1}u\in W^{1,2}_M([0,\tau];\R^{2n}),
$$
we deduce that $u(t)=-[\nabla^2_z(H_\kappa)^\ast(\lambda, t; 0)-\varrho I_{2n}]^{-1}(A_{M,\tau,\kappa I_{2n}}u)(t)$ is continuous. To prove that $u\in X$,
 it suffices to show that  $u(\tau)=Mu(0)$.
 Above (\ref{e:HPositive}) we have shown
  that $H_\kappa (\lambda, \tau, Mz)=H_\kappa (\lambda, 0, z)$ for all $z\in\mathbb{R}^{2n}$.
  This implies
$$
\nabla_{z}^{2}H(\lambda, \tau, Mz)=(M^{-1})^T\nabla_{z}^{2}H(\lambda, 0, z)M^{-1}.
$$
It follows from (\ref{e:HPositive+}) that
\begin{eqnarray*}
\nabla_{z}^{2}(H_{\kappa})^{*}(\lambda, \tau;0)
&=&(\nabla_{z}^{2}H_{\kappa}(\lambda, \tau;0))^{-1}\\
&=&((M^{-1})^T\nabla_{z}^{2}H(\lambda, 0, 0)M^{-1})^{-1}\\
&=&M(\nabla_{z}^{2}H(\lambda, 0, 0))^{-1}M^{T}\\
&=&M\nabla_{z}^{2}(H_{\kappa})^{*}(\lambda, 0;0)M^{-1}
\end{eqnarray*}
because $M$ is an orthogonal symplectic matrix. Hence
\begin{eqnarray*}
u(\tau)&=&-[\nabla^2_z(H_\kappa)^\ast(\lambda, \tau; 0)-\varrho I_{2n}]^{-1}(A_{M,\tau,\kappa I_{2n}}u)(\tau)\\
&=&-[M\nabla_{z}^{2}(H_{\kappa})^{*}(\lambda, 0;0)M^{-1}-\varrho I_{2n}]^{-1}M(A_{M,\tau,\kappa I_{2n}}u)(0)\\
&=&-M[\nabla_{z}^{2}(H_{\kappa})^{*}(\lambda, 0;0)-\varrho I_{2n}]^{-1}M^{-1}M(A_{M,\tau,\kappa I_{2n}}u)(0)\\
&=&Mu(0).
\end{eqnarray*}
The first claim is thus proved.

In order to prove the second inclusion in (\ref{e:C-D-D1})
let $u\in H$ be such that $v:=B_\lambda(0)u\in X=C^0_M([0,\tau];\R^{2n})$. Then
$$
 (A_{M,\tau,\kappa}u)(t)+ \nabla^2_z(H_\kappa)^\ast(\lambda, t; 0)u(t)=v(t).
$$
By (\ref{e:Bisom}), $A_{M,\tau,\kappa}u\in W^{1,2}_{M}([0,\tau];\R^{2n})\subset X$.
This implies  that $u$ is continuous since $t\mapsto \nabla^2_z(H_\kappa)^\ast(\lambda, t; 0)$
 is continuous and
 $\nabla^2_z(H_\kappa)^\ast(\lambda, t; 0)$ is invertible.
Let $w(t)=v(t)-(A_{M,\tau,\kappa}u)(t)$. Then $w\in X$ and
$u(t)=[\nabla^2_z(H_\kappa)^\ast(\lambda, t; 0)]^{-1}w(t)$. As above, we have
\begin{eqnarray*}
u(\tau)&=&[\nabla^2_z(H_\kappa)^\ast(\lambda, \tau; 0)]^{-1}w(\tau)\\
&=&[M\nabla_{z}^{2}(H_{\kappa})^{*}(\lambda, 0;0)M^{-1}]^{-1}Mw(0)\\
&=&M[\nabla_{z}^{2}(H_{\kappa})^{*}(\lambda, 0;0)]^{-1}M^{-1}Mw(0)\\
&=&Mu(0).
\end{eqnarray*}
Hence $u\in X$, which proves the second claim.

{\bf Step 5}.
 In the proof of Theorem~\ref{th:bif-ness}(I) we have proved that
$P_\lambda(v)=\nabla^2_z(H_\kappa)^\ast(\lambda, \cdot; v(\cdot))$ and $Q_\lambda(v)=A_{M,\tau,\kappa}$ satisfy
the conditions (i), (ii), (iii) and (iv) in
Theorem~\ref{th:A.9} with $\lambda^\ast=\mu$.
Combining these with the above  Steps 3, 4 we see that
Hypothesis~\ref{hyp:A.5} and so the condition (v) in Theorem~\ref{th:A.9} with $\lambda^\ast=\mu$ is satisfied.

It follows from (\ref{e:psi-MorseIndex+}) and (\ref{e:psi-Nullity+})  that the Morse index and nullity of
$\mathcal{L}_{{\lambda}}$ at $0\in H$ are
\begin{eqnarray}
m^-_\lambda&=& m^-({\psi}_{\kappa}(\lambda,\cdot), 0)=i_{\tau,M}(\gamma_\lambda)- i_{\tau,M}(\Upsilon_{\kappa I_{2n}})-\nu_{\tau,M}(\Upsilon_{\kappa I_{2n}}),\label{e:L-MorseIndex}\\
m^0_\lambda&=&m^0({\psi}_{\kappa}(\lambda, \cdot), 0)=\nu_{\tau,M}(\gamma_\lambda)=\dim{\rm Ker}(\gamma_\lambda(\tau)-M),\label{e:L-Nullity}
\end{eqnarray}
 respectively, where $\Upsilon_{\kappa I_{2n}}(t)=\exp({t\kappa J})$.
Let $\Xi= -i_{\tau,M}(\Upsilon_{\kappa I_{2n}})-\nu_{\tau,M}(\Upsilon_{\kappa I_{2n}})$.
Under the assumptions in Theorem~\ref{th:bif-ness}(II), for each $k\in\mathbb{N}$ we derive from (\ref{e:L-MorseIndex})-(\ref{e:L-Nullity})
\begin{eqnarray*}
&&[m^-_{\lambda_k^-}, m^-_{\lambda_k^-}+m^0_{\lambda_k^-}]\cap[m^-_{\lambda_k^+},
m^-_{\lambda_k^+}+m^0_{\lambda_k^+}]\\
&=&[i_{\tau,M}(\gamma_{\lambda_k^-})+\Xi, i_{\tau,M}(\gamma_{\lambda_k^-})+\nu_{\tau,M}(\gamma_{\lambda_k^-})+\Xi]\cap[i_{\tau,M}(\gamma_{\lambda_k^+})
+\Xi, i_{\tau,M}(\gamma_{\lambda_k^+})+\nu_{\tau,M}(\gamma_{\lambda_k^+})+\Xi]\\
&=&\Xi+ [i_{\tau,M}(\gamma_{\lambda_k^-}), i_{\tau,M}(\gamma_{\lambda_k^-})+\nu_{\tau,M}(\gamma_{\lambda_k^-})]\cap[i_{\tau,M}(\gamma_{\lambda_k^+}), i_{\tau,M}(\gamma_{\lambda_k^+})+\nu_{\tau,M}(\gamma_{\lambda_k^+})]=\emptyset
\end{eqnarray*}
 and either $m^0_{\lambda_k^+}=\nu_{\tau,M}(\gamma_{\lambda_k^+})=0$ or $m^0_{\lambda_k^-}=\nu_{\tau,M}(\gamma_{\lambda_k^-})=0$.
Hence from the conclusion (B) of Theorem~\ref{th:A.9} we immediately conclude that
there exists a sequence $\{(\lambda_k, w_k)\}_{k\ge 1}$ in
$\hat\Lambda\times L^2([0,\tau];\R^{2n})$ converging to
$(\mu, 0)$ such that each $w_k\ne 0$ and satisfies $\nabla_w\psi_\kappa(\lambda_k, w_k)=0$,
$k=1,2,\cdots$. Then
$$
u_k:=-(\Lambda_{M,\tau,\kappa I_{2n}})^{-1}w_k+u_{\lambda_k}\in W^{1,2}_M([0,\tau];\mathbb{R}^{2n}),\quad k=1,2,\cdots,
$$
satisfy (\ref{e:Hboundary}) with $\lambda=\lambda_k$
and $0<\|u_k-u_{\lambda_k}\|_{1,2}=\|(\Lambda_{M,\tau,\kappa I_{2n}})^{-1}w_k\|_{1,2}\to 0$
as $k\to\infty$.
We also need to prove $\|\dot{u}_k-\dot{u}_{\lambda_k}\|_{C^0}\to 0$.
It is easily seen  that
\begin{eqnarray*}
|\dot{{u}}_k(t)-\dot{u}_{\lambda_k}(t)|&=&|J\nabla_zH(\lambda_k,t, {u}_k(t))-
J\nabla_zH(\lambda_k,t, u_{\lambda_k}(t))|\\
&=&|\nabla_zH(\lambda_k,t, {u}_k(t))-\nabla_zH(\lambda_k,t, u_{\lambda_k}(t))|\\
&\le&\int^1_0\|\nabla^2_zH(\lambda_k,t, s{u}_k(t)+(1-s)u_{\lambda_k}(t))\|\cdot|{u}_k(t)-u_{\lambda_k}(t)|ds\\
&\le&\|{u}_k-u_{\lambda_k}\|_{C^0}\int^1_0\|\nabla^2_zH(\lambda_k,t, s{u}_k(t)+(1-s)u_{\lambda_k}(t))\|ds.
\end{eqnarray*}
Let $\Lambda_0=\{\mu,\lambda_k\,|\,k\in\mathbb{N}\}$. It is a compact set.
Since $\|u_k-u_{\lambda_k}\|_{C^0}\to 0$ and $\Lambda\times [0, \tau]\ni (\lambda,t)\mapsto u_\lambda(t)\in\mathbb{R}^{2n}$ is continuous,
there exists a $C>0$ such that $|u_{\lambda_k}(t)|+|u_{k}(t)|\le C$ for all $(t,k)\in[0,\tau]\times\mathbb{N}$.
Note that $\nabla^2_zH$ is continuous. It is bounded on the compact subset
$$
\Lambda_0\times [0,\tau]\times \{z\in\mathbb{R}^{2n}\,|\,|z|\le C\}.
$$
It follows that $|\dot{{u}}_k(t)-\dot{u}_{\lambda_k}(t)|$ uniformly converges to zero on $[0,\tau]$
as $k\to\infty$. Hence  $0<\|u_k-u_{\lambda_k}\|_{C^1}\to 0$.
The required conclusions are proved.
\end{proof}

\begin{proof}[\bf Proof of Theorem~\ref{th:bif-ness}(III)]
Recall that we can assume $\Lambda=\alpha([0,1])$.
From the proof of Theorem~\ref{th:bif-ness}(II)
it is easily seen that the conditions of Theorem~\ref{th:A.9+}
are satisfied. Then there exists a sequence
$\{(t_k, w_k)\}_{k\ge1}\subset [0, 1]\times L^2([0,\tau];\mathbb{R}^{2n})$ such that
 \begin{enumerate}
\item[$\bullet$]  $t_k\to\bar{t}$ and $0<\|w_k\|_{2}\to 0$,
 \item[$\bullet$] each $w_k$ satisfies $\nabla_w\psi_\kappa(\alpha(t_k), w_k)=0$, $k=1,2,\cdots$,
\item[$\bullet$]  $\alpha(\bar{t})$ is not equal to $\lambda^+$ (resp. $\lambda^-$) if $m^0_{\lambda^+}=0$ (resp. $m^0_{\lambda^-}=0$),
where $m^0_\lambda$ is given by (\ref{e:L-Nullity}).
\end{enumerate}
As above we define $u_k:=-(\Lambda_{M,\tau,\kappa I_{2n}})^{-1}w_k+u_{\alpha(t_k)}\in W^{1,2}_M([0,\tau];\mathbb{R}^{2n})$,
$k=1,2,\cdots$. Then $u_k$ satisfies (\ref{e:Hboundary}) with $\lambda=\alpha(t_k)$
and $0<\|u_k-u_{\alpha(t_k)}\|_{1,2}=\|(\Lambda_{M,\tau,\kappa I_{2n}})^{-1}w_k\|_{1,2}\to 0$
as $k\to\infty$.
As in the proof of Theorem~\ref{th:bif-ness}(II), taking $\lambda_k=\alpha(t_k)$ and $\mu=\alpha(\bar{t})$
we obtain $\|u_k-u_{\alpha(t_k)}\|_{C^1}\to 0$.
\end{proof}

\begin{remark}\label{rm:MW89}
{\rm Even if $(\Lambda, M, \tau)=(I, I_{2n}, 2\pi)$, where $I$ is
 an interval containing $\mu$ as an interior point,
we cannot still guarantee that $\psi_\kappa$ is $C^1$ on $\Lambda\times L^{2}([0,\tau];\mathbb{R}^{2n})$. Therefore the conclusion cannot be derived from
\cite[Theorem~4.6]{BaSzWi} (a result due to Mawhin and Willem \cite[Theorem~8.9]{MaWi}).
 }
 \end{remark}

\begin{proof}[\bf Proof of Theorem~\ref{th:bif-suffict}]
With notations in the proofs of Theorems~\ref{th:bif-ness}.
But now the parameter space $\Lambda$ is a real interval.
Comparing the conditions in Theorem~\ref{th:A.9}
 with those of \cite[Theorem~3.6]{Lu10}
we  have actually checked in  the proof of Theorem~\ref{th:bif-ness}(II)
 that $\mathcal{L}_\lambda(\cdot)=\psi_\kappa(\lambda, \cdot)$ with $\lambda^\ast=\mu$, $H=L^2([0,\tau];\R^{2n})$
and $X=C^0_M([0,\tau];\R^{2n})$ satisfy conditions in \cite[Theorem~3.6]{Lu10} (Theorem~\ref{th:A.11})
except for the condition (f).

We claim that the latter can be also satisfied.
In fact, by the assumptions of Theorem~\ref{th:bif-suffict} and (\ref{e:L-MorseIndex})-(\ref{e:L-Nullity}), $m^0({\psi}_{\kappa}(\mu,\cdot), 0)>0$
and $m^0({\psi}_{\kappa}(\lambda,\cdot), 0)=0$
if $\lambda\in\Lambda\setminus\{\mu\}$ is close to $\mu$. Moreover,
$m^-({\psi}_{\kappa}(\lambda, \cdot), 0)$ takes, respectively,
values $m^-({\psi}_{\kappa}(\mu,\cdot), 0)$ and
$m^-({\psi}_{\kappa}(\mu, \cdot), 0)+ m^0({\psi}_{\kappa}(\mu, \cdot), 0)$
 as $\lambda\in\Lambda$ varies in two deleted half neighborhoods  of $\mu$.

Then by Theorem~\ref{th:A.11} (\cite[Theorem~3.6]{Lu10}),  one of the following assertions holds:
 \begin{enumerate}
\item[(i)] $(\mu, 0)$ is not an isolated solution  in  $\{\mu\}\times L^2([0,\tau];\R^{2n})$
 of $\nabla_w\psi_\kappa(\lambda, w)=0$.

\item[(ii)]  For each  $\lambda\in\Lambda\setminus\{\mu\}$ near $\mu$,
$\nabla_w\psi_\kappa(\lambda, w)=0$ has a  solution $w_\lambda\in C^0_M([0,\tau];\R^{2n})$
different from $0$,
 which  converges to $0$ in $C^0_M([0,\tau];\R^{2n})$ as $\lambda\to \mu$.

\item[(iii)] For a given neighborhood $\mathcal{U}$ of $0$ in $C^0_M([0,\tau];\R^{2n})$
there is an one-sided  neighborhood $\Lambda^0$ of $\mu$ such that for any $\lambda\in\Lambda^0\setminus\{\mu\}$, $\nabla_w\psi_\kappa(\lambda, w)=0$
has at least two distinct nonzero solutions  $w_\lambda^1$ and $w_\lambda^2$ in $\mathcal{U}$,
which can also be required to satisfy
$\psi_\kappa(\lambda, w_\lambda^1)\ne\psi_\kappa(\lambda, w_\lambda^2)$
provided that $\nu_{\tau,M}(\gamma_\mu)>1$ and $\nabla_w\psi_\kappa(\lambda, w)=0$
has only finitely many  solutions in $\mathcal{U}$.
\end{enumerate}

Since the Banach space isomorphism  $-\Lambda_{M,\tau,\kappa I_{2n}}^c: C^1_{M}([0,\tau];\R^{2n})\to C^0_M([0,\tau];\R^{2n})$
maps the set of solutions of $\nabla_u\Psi_\kappa(\lambda, u)=0$
in $C^1_{M}([0,\tau];\R^{2n})$ onto that of $\nabla_w\psi_\kappa(\lambda, w)=0$ in $C^0_{M}([0,\tau];\R^{2n})$, and
for each critical point $w$ of $\psi_\kappa(\lambda, \cdot)$ it holds that
$$
\psi_\kappa(\lambda, w)=\Psi_\kappa\left(\lambda, (-\Lambda_{M,\tau,\kappa I_{2n}}^c)^{-1}w\right)=-\Phi\left(\lambda,
(-\Lambda_{M,\tau,\kappa I_{2n}}^c)^{-1}w\right),
$$
the conclusions in the first part of  Theorem~\ref{th:bif-suffict} follow from these and Proposition~\ref{prop:threeBifu}(i) as in the proof of Theorem~\ref{th:bif-ness}(III).

For the part of ``Moreover", note that the Banach space isomorphism
$\Lambda_{M,\tau,\kappa I_{2n}}$ is equivariant with respect to
 the natural actions of $\Z_2=\{id,-id\}$ on spaces
$$
W^{1,2}_{M}([0,\tau];\R^{2n}), \quad L^{2}([0,\tau];\R^{2n}),\quad C^1_{M}([0,\tau];\R^{2n})
\quad\hbox{and}\quad C^0_{M}([0,\tau];\R^{2n})
$$
 and that functionals
$\Psi_\kappa(\lambda, \cdot)$ and $\psi_\kappa(\lambda, \cdot)$ in (\ref{e:HAction*})
and (\ref{e:HAction**})
are invariant under the $\Z_2$-action, i.e., even.
Comparing the conditions in Theorem~\ref{th:A.11} (\cite[Theorem~3.6]{Lu10}
or \cite[Theorem~4.6]{Lu8}) with those of
Theorem~\ref{th:A.12} (\cite[Theorem~3.7]{Lu10} or \cite[Theorem~5.12]{Lu8})
we immediately use the latter and \cite[Remark~3.9]{Lu10} (or \cite[Remark~5.14]{Lu8})
to obtain the desired claims.
\end{proof}

\begin{proof}[\bf Proof of Corollary~\ref{cor:necess-suffi}]
If $(\mu, \bar{u})\in\R\times W^{1,2}_M([0,\tau];\R^{2n})$
is a bifurcation point  for (\ref{e:Hboundary}) with
$H(\lambda,t, z)=H_0(t,z)+\lambda\hat{H}(t,z)$ and $\Lambda=\R$,
then $\nu_{\tau,M}(\gamma_\mu)>0$ by Theorem~\ref{th:bif-ness}(I).

Suppose now that $\nu_{\tau,M}(\gamma_\mu)>0$.
Take positive numbers $0<R_1<R_2<R_3$ such that
 $\bar{u}([0,\tau])$ is contained in the closed ball  $\bar{B}^{2n}(0, R_1^2)$.
 Therefore we can get a constant $C'>0$  such that
$$
-C'I_{2n}\le \nabla_z^2H_0(t, z),\; \nabla_z^2\hat{H}(t, z)\le C'I_{2n},
\quad\forall (\lambda,t, z)\in \Lambda\times [0,\tau]\times \bar{B}^{2n}(0, R_3^2).
$$
Let us choose a smooth cut-off function $\rho:[0, \infty)\to [0, 1]$ such that
$\rho(t)=1$ for $t\le R_2^2$ and $\rho(t)=0$ for $t\ge R_3^2$.
Define a function $\chi:\mathbb{R}^{2n}\to [0, 1]$ by
$$
\chi(z)=\rho(|z|^2),\quad\forall z\in\mathbb{R}^{2n}.
$$
We modify the functions $H_0$ and $\hat{H}$ by multiplying them by the cutoff function $\chi$.
That is, we define
$$
H_0^\star(t,z):=\chi(z)H_0(t,z),\qquad \hat{H}^\star(t,z):=\chi(z)\hat{H}(t,z),
$$
for all $(t,z)\in [0, \tau]\times\mathbb{R}^{2n}$.
After this modification, we continue to denote these new functions by
$H_0$ and $\hat{H}$. They then satisfy the following uniform estimates for some constants
 $C>0, c'_1>0, c'_2>0$:
\begin{eqnarray*}
-CI_{2n}\le \nabla_z^2H_0(t, z),\; \nabla_z^2\hat{H}(t, z)\le CI_{2n},\quad\forall (t, z)\in [0,\tau]\times \mathbb{R}^{2n},\\
-c'_1|z|^2-c'_2\le H_0(t, z),\; \hat{H}(t, z)\le c'_1|z|^2+ c'_2,\quad\forall (t, z)\in [0,\tau]\times \mathbb{R}^{2n}
\end{eqnarray*}
 Therefore there exist constants $\kappa<0$ and
 $c_i>0$, $i=1,2,3$   such that (i)-(iii) in the proof of Theorem~\ref{th:bif-ness}(I)
are satisfied with
$$
H_\kappa(\lambda, t, z):=H_0(t, z)+\lambda\hat{H}(t,z)- \frac{\kappa}{2}|z|^2
$$
for all $(\lambda, t, z)\in [\mu-1, \mu+1]\times [0,\tau]\times\R^{2n}$. In particular, (\ref{e:HPositive}) implies
\begin{equation}\label{e:HPositive*}
   c_1I_{2n}\le \nabla_z^2H_0(t,z)+ \lambda\nabla_z^2\hat{H}(t,z)-\kappa I_{2n}
      \le c_2I_{2n}
   \end{equation}
for all $(\lambda, t, z)\in [\mu-1, \mu+1]\times [0,\tau]\times\R^{2n}$.
Let $\bar{w}=-\Lambda_{M,\tau,\kappa I_{2n}}\bar{u}$ and
$\psi_{\kappa,\lambda}(u):=\psi_\kappa(\lambda, u)$ be given by (\ref{e:HAction**}). Then
for all $\xi,\eta\in{L}^{2}([0, \tau];\R^{2n})$,
\begin{eqnarray*}
&&{\psi}''_{\kappa,\lambda}(\bar{w})[\xi,\eta]\\
&=&\int^\tau_0\Big[({A}_{M,\tau,\kappa}\xi)(t), \eta(t))_{\mathbb{R}^{2n}}
+([\nabla_z^2H_0(t, \bar{u}(t))+   \lambda\nabla_z^2\hat{H}(t, \bar{u}(t))-\kappa I_{2n}]^{-1}\xi(t),\eta(t))_{\mathbb{R}^{2n}})\Big]dt\\
&=&2q_{M, B_\lambda|A}(\xi,\eta),
\end{eqnarray*}
 where $B_\lambda(t)=\nabla^2_zH_0(t, \bar{u}(t))+   \lambda\nabla_z^2\hat{H}(t, \bar{u}(t))$,
 $A(t)=\kappa I_{2n}$, and $q_{M, B|A}$ is as in (\ref{e:quadratic}).
By (\ref{e:HPositive*}), for all $(\lambda, t)\in [\mu-1, \mu+1]\times [0,\tau]$ we have
\begin{equation}\label{e:HPositive**}
   c_1I_{2n}\le B_\lambda(t)-A(t)
      \le c_2I_{2n}.
   \end{equation}
Since $\gamma_\lambda=\Upsilon_{B_\lambda}$,
it follows from  (\ref{e:MorseIndex}) and (\ref{e:Nullity}) that
\begin{equation}\label{e:DongMorse}
\left.\begin{array}{ll}
m^-({\psi}_{\kappa,\lambda},\bar{w})=j_{\tau,M}(B_\lambda|A)\quad\hbox{and}\\
m^0({\psi}_{\kappa,\lambda},\bar{w})=\nu_{\tau,M}(B_\lambda|A)=
\nu_{\tau,M}(B_\lambda)=\nu_{\tau,M}(\gamma_\lambda).
\end{array}\right\}
\end{equation}

\underline{Firstly, let us assume $\nabla^2_z\hat{H}(t, \bar{u}(t))>0$ for all $t\in [0,\tau]$}.
 Then by (\ref{e:HPositive**}) we get
$$
B_{\lambda_2}>B_{\lambda_1}\ge A+ c_1I_{2n}\quad\hbox{for any}\quad \mu-1\le \lambda_1<\lambda_2\le\mu+1.
$$
Because of these we derive from (\ref{e:MorseRelation3})  that
\begin{eqnarray}\label{e:DongMorse1}
 j_{\tau,M}(B_{\lambda_2}|A)\ge j_{\tau,M}(B_{\lambda_1}|A)+\nu_{\tau,M}(B_{\lambda_1}|A)
\end{eqnarray}
for any $\mu-1\le \lambda_1<\lambda_2\le\mu+1$.
By (\ref{e:M-index}), both $j_{\tau,M}(B|A)$ and $\nu_{\tau,M}(B|A)$ are nonnegative integers. Hence
 (\ref{e:DongMorse1}) implies that there exist at most finitely many points
$\lambda\in [\mu-1, \mu+1]$ where $\nu_{\tau,M}(B_\lambda|A)=\nu_{\tau,M}(B_\lambda)=\nu_{\tau,M}(\Upsilon_{B_\lambda})
=\nu_{\tau,M}(\gamma_\lambda)\ne 0$.
It follows that $\{\lambda\in\R\,|\, \nu_{\tau,M}(\gamma_\lambda)>0\}$  is a discrete set in $\R$.
The first claim is proved in the present case.

In order to prove (\ref{e:case1}) let us
take $\rho\in (0, 1]$ so small that $\nu_{\tau,M}(B_\lambda|A)=\nu_{\tau,M}(B_\lambda)=0$ for $\lambda\in [\mu-\rho, \mu+\rho]\setminus\{\mu\}$.
Since  $\nabla_z^2\hat{H}(t, \bar{u}(t))>0\;\forall t\in [0,\tau]$, by (\ref{e:MorseRelation4}) we get
\begin{equation}\label{e:DongMorse2}
\left.\begin{array}{ll}
j_{\tau,M}(B_{\mu_1})\le j_{\tau,M}(B_{\mu_2})\le j_{\tau,M}(B_{\mu})&<j_{\tau,M}(B_{\mu})+\nu_{\tau,M}(B_\mu)\\
&\le j_{\tau,M}(B_{\lambda_1})\le j_{\tau,M}(B_{\lambda_2})
\end{array}\right\}
\end{equation}
for any $\mu-\rho\le\mu_1<\mu_2<\mu<\lambda_1<\lambda_2<\mu+\rho$.
Moreover, by (\ref{e:MorseIndex})  we can derive from
\cite[Proposition~2.3.3]{Ab} that
$$
B\mapsto j_{\tau,M}(B|A)\quad\hbox{and}\quad B\mapsto j_{\tau,M}(B|A)+\nu_{\tau,M}(B|A)
$$
are lower-semi continuous and
upper-semi continuous, respectively. These and (\ref{e:MorseIndex0}) imply that
$$
B\mapsto i_{\tau,M}(B)\quad\hbox{and}\quad B\mapsto i_{\tau,M}(B)+\nu_{\tau,M}(B)
$$
are lower-semi continuous and upper-semi continuous, respectively.
Using (\ref{e:DongMorse2}) again we deduce
\begin{eqnarray*}
&&j_{\tau,M}(B_{\mu_1})= j_{\tau,M}(B_{\mu})\;\;\forall\mu_1\in [\mu-\rho,\mu),\\
&&j_{\tau,M}(B_{\lambda_1})=j_{\tau,M}(B_{\lambda_1})+\nu_{\tau,M}(B_{\lambda_1})
=j_{\tau,M}(B_{\mu})+\nu_{\tau,M}(B_\mu)
\;\;\forall\lambda_1\in (\mu, \mu+\rho].
\end{eqnarray*}
By Theorem~\ref{th:twoIndex}, these lead to (\ref{e:case1}).

\underline{Similarly, if $\nabla_z^2\hat{H}(t, \bar{u}(t))<0\;\forall t\in [0,\tau]$},
we may obtain reversed inequalities
to those in (\ref{e:DongMorse1}) and (\ref{e:DongMorse2}),
 and therefore the discreteness of $\Sigma$.
Then as above we have
\begin{eqnarray*}
&&j_{\tau,M}(B_{\mu_1})= j_{\tau,M}(B_{\mu})\;\;\forall\mu_1\in (\mu, \mu+\rho],\\
&&j_{\tau,M}(B_{\lambda_1})=j_{\tau,M}(B_{\lambda_1})+\nu_{\tau,M}(B_{\lambda_1})
=j_{\tau,M}(B_{\mu})+\nu_{\tau,M}(B_\mu)
\;\;\forall\lambda_1\in [\mu-\rho,\mu)
\end{eqnarray*}
and so (\ref{e:case2}).

Finally, other  conclusions  follow from (\ref{e:case1})-(\ref{e:case2})
and Theorem~\ref{th:bif-suffict}.
\end{proof}

\subsection{Proofs of Theorems~\ref{th:bif-ness},
~\ref{th:bif-suffict} and Corollaries~\ref{cor:necess-suffi},~\ref{cor:bif-deform}
for a general symplectic matrix $M$}\label{sec:directM}

We first prove Theorems~\ref{th:bif-ness} and~\ref{th:bif-suffict}, and Corollary~\ref{cor:necess-suffi}, for an arbitrary symplectic matrix $M$. 
Via a simple transformation,  these results can be derived from their already-proved special cases with $M=I_{2n}$.
Next, we derive from Theorem~\ref{th:bif-suffict},  Corollary~\ref{cor:bif-deform}(I), which in turn leads to Corollary~\ref{cor:bif-deform}(II).

For a differentiable curve
$\gamma \in \mathcal{P}_\tau(2n) = \{\gamma \in C([0,\tau],\mathrm{Sp}(2n,\mathbb{R})) \mid \gamma(0) = I_{2n}\}$,
differentiating the identity $\gamma(t)^\top J_n \gamma(t) = J_n$ (which holds because $\gamma(t) \in \mathrm{Sp}(2n,\mathbb{R})$)
and using $J_n^\top = -J_n$, we derive that
the matrix $B^\gamma(t) := \gamma(t)^\top J_n \dot{\gamma}(t)$ is symmetric.
Moreover, since
$B^\gamma(0) = B^\gamma(\tau)$ if and only if $\gamma(0)^\top J_n \dot{\gamma}(0) = \gamma(\tau)^\top J_n \dot{\gamma}(\tau)$,
we see that  the condition $B^\gamma(0) = B^\gamma(\tau)$ is equivalent to $\dot{\gamma}(\tau) =  \gamma(\tau)\dot{\gamma}(0)$.

Since $M$ has a unique polar decomposition $M = \exp(N_1) \exp(N_2)$
with $N_1$ symmetric satisfying $N_1^T J + J N_1 = 0$
and $N_2$ skew-symmetric satisfying $N_2^T J + J N_2 = 0$,
\[
[0, \tau]\ni t\mapsto \gamma^M(t) := \exp(t N_1/\tau) \exp(t N_2/\tau)
\]
defines a smooth path in $\operatorname{Sp}(2n,\mathbb{R})$
joining $\gamma^M(0)=I_{2n}$ to $\gamma^M(\tau)=M$.

One can always choose a $C^1$ path $\bar{\gamma}^M: [0,\tau] \to \operatorname{Sp}(2n,\mathbb{R})$
that is homotopic to $\gamma^M$ with fixed endpoints in $\mathcal{P}_\tau(2n)$
and satisfies $\dot{\bar{\gamma}}^M(\tau) =\bar{\gamma}^M(\tau)\dot{\bar{\gamma}}^M(0) $
(see \cite[p.~973]{LiuTang15}). Therefore, for the indexes in (A.16) and (A.17), there hold
\begin{align}\label{e:2-P-index1C}
i_{\tau,M}(\bar\gamma^M)&=i_{\tau, M}(\gamma^M),\quad
\nu_{\tau,M}(\bar\gamma^M)=\nu_{\tau, M}(\gamma^M)=2n,\\
i_\tau^M(\bar\gamma)&=i_\tau^M(\gamma^M),\quad
\nu_\tau^M(\bar\gamma^M)=\nu_\tau^M(\gamma^M)=2n.
\label{e:2-P-index2C}
\end{align}

Let $\bar{B}_M(t):=B^{\bar\gamma^M}(t)= \bar{\gamma}^M(t)^\top J_n \dot{\bar{\gamma}}^M(t)$,
which is symmetric and satisfies $\bar{B}_M(0)=\bar{B}_M(\tau)$.
Let $H:\Lambda\times [0,\tau]\times\mathbb{R}^{2n}\to\mathbb{R}$ satisfy Assumption~\ref{ass:BasiAss1},
but with the requirement that $M$ is orthogonal omitted.
Define $\check{H}:\Lambda\times [0,\tau]\times\mathbb{R}^{2n}\to\mathbb{R}$ by
\begin{equation}\label{e:PerHboundary0C}
\check{H}(\lambda,t, z)=H(\lambda,t, \bar{\gamma}^M(t)z)+ \frac{1}{2}(\bar{B}_M(t)z,z)_{\mathbb{R}^{2n}}.
\end{equation}
It is easily seen that $\check{H}$ satisfies Assumption~\ref{ass:BasiAss1} with $M=I_{2n}$
and that $\check{H}(\lambda,t, \cdot)$ is even if and only if $H(\lambda,t, \cdot)$ is even.
A direct computation yields
\begin{align*}
\nabla_z\check{H}(\lambda,t, z)&=\bar{\gamma}^M(t)\nabla_zH(\lambda,t, \bar{\gamma}^M(t)z)+ \bar{B}_M(t)z,\\
\nabla^2_z\check{H}(\lambda,t, z)&=\bar{\gamma}^M(t)^\top\nabla_z^2H(\lambda,t, \bar{\gamma}^M(t)z)\bar{\gamma}^M(t)+ \bar{B}_M(t).
\end{align*}
From these we derive:
\begin{itemize}
\item[{\rm (i)}] A path $u:[0,\tau]\to\mathbb{R}^{2n}$ satisfies \eqref{e:Hboundary} if and only if
the transformed path $w(t):=\bar{\gamma}^M(t)^{-1}u(t)$ satisfies the Hamiltonian boundary value problem
\begin{equation}\label{e:PerHboundary1C}
\dot{w}(t)=J_n\nabla_z \check{H}(\lambda,t, w(t))\quad\text{and}\quad w(\tau)=w(0).
\end{equation}

\item[{\rm (ii)}] If $\gamma$ and $\check{\gamma}$ are the fundamental matrix solutions of
\begin{align*}
\dot{\eta}(t)&=J_n\nabla_z^2 {H}(\lambda,t, u(t))\eta,\\
\dot{\eta}(t)&=J_n\nabla_z^2 \check{H}(\lambda,t, w(t))\eta,
\end{align*}
respectively, then $\bar\gamma^M\check{\gamma}=\gamma$, and therefore by \cite[Lemmas~4.3 and 4.4]{Liu17+}
there hold
\begin{align}
\nu_\tau(\check{\gamma})&=\nu^M_\tau(\gamma),\label{e:PerHboundary2aC}\\
i_\tau(\check{\gamma})&=i^M_\tau(\gamma)-i^M_\tau(\bar\gamma^M)-n
=i^M_\tau(\gamma)-i^M_\tau(\gamma^M)-n.\label{e:PerHboundary2bC}
\end{align}
\end{itemize}

Let $\xi\in\mathcal{P}_\tau(2n)$ satisfy $\xi(\tau)=M^{-1}$.
By Proposition~\ref{prop:twoDef}
\begin{align*}
i_{\tau,M}(\gamma)=i^M_\tau(\gamma)+ \bigl[i_\tau(\xi)-\triangle(\xi)/\tau\bigr]\quad\text{and}\quad
i_{\tau,M}(\gamma^M)=i^M_\tau(\gamma^M)+ \bigl[i_\tau(\xi)-\triangle(\xi)/\tau\bigr].
\end{align*}
It follows from these and (\ref{e:PerHboundary2bC}) that
\begin{equation}\label{e:PerHboundary3C}
i_\tau(\check{\gamma})=i^M_\tau(\gamma)-i^M_\tau(\gamma^M)-n
=i_{\tau,M}(\gamma)-i_{\tau,M}(\gamma^M)-n.
\end{equation}

 For each $\lambda\in\Lambda$, define a path $w_\lambda:[0,\tau]\to\mathbb{R}^{2n}$
 by   $w_\lambda(t):=\bar{\gamma}^M(t)^{-1}u_\lambda(t)$.
 From Assumption~\ref{ass:BasiAss1} and the arguments below it,
we deduce that $\Lambda\times [0,\tau]\ni (\lambda,t)\mapsto w_\lambda(t)\in\mathbb{R}^{2n}$ is  continuous,
 and for any given $\mu\in\Lambda$ we have $\|w_\lambda-w_\mu\|_{C^1}\to 0$ as $\lambda\to\mu$.

By the above (i)-(ii), we have:
\begin{itemize}
\item[{\rm (iii)}]
 $w_\lambda$ satisfies (\ref{e:PerHboundary1C}) for any $\lambda\in\Lambda$,
and  $(\mu, u_\mu)$ is a bifurcation point along sequences of
 (\ref{e:Hboundary}) in $\Lambda\times C^0_{M}([0,\tau];\R^{2n})$
 with respect to the branch $\{(\lambda,u_\lambda)\,|\,\lambda\in\Lambda\}$ if and only if
 $(\mu, w_\mu)$ is a bifurcation point along sequences of
 (\ref{e:PerHboundary1C}) in $\Lambda\times C^0_{I_{2n}}([0,\tau];\R^{2n})$
 with respect to the branch $\{(\lambda, w_\lambda)\,|\,\lambda\in\Lambda\}$.
 \end{itemize}

Let  $\gamma_\lambda$ and $\check{\gamma}_\lambda$ be the fundamental matrix solutions of
\begin{align*}
\dot{\eta}(t)=J_n\nabla_z {H}^2(\lambda,t, u_\lambda(t))\eta\quad\text{and}\quad 
\dot{\eta}(t)=J_n\nabla_z \check{H}^2(\lambda,t, w_\lambda(t))\eta
\end{align*}
respectively. By (\ref{e:PerHboundary2aC}) and (\ref{e:PerHboundary3C}), we obtain
$\bar\gamma^M\check{\gamma}_\lambda=\gamma_\lambda$ and
\begin{align}\label{e:PerHboundary4C}
\nu_\tau(\check{\gamma}_\lambda)=\nu_{\tau, M}(\gamma_\lambda),\qquad
i_\tau(\check{\gamma}_\lambda)=i_{\tau,M}(\gamma_\lambda)-i_{\tau,M}(\gamma^M)-n.
\end{align}

\begin{proof}[\bf Proof of Theorem~\ref{th:bif-ness}(I) for a general symplectic matrix $M$]
By the above (iii),  $(\mu, w_\mu)$ is a bifurcation point along sequences of
 (\ref{e:PerHboundary1C}) in $\Lambda\times C^0_{I_{2n}}([0,\tau];\R^{2n})$
 with respect to the branch $\{(\lambda, w_\lambda)\,|\,\lambda\in\Lambda\}$.
 Then Theorem~1.4(I) with $M=I_{2n}$ implies $\nu_\tau(\check{\gamma}_\mu)$
  and hence $\nu^M_\tau(\gamma_\mu)\ne 0$ by (\ref{e:PerHboundary4C}).
 \end{proof}

\begin{proof}[\bf Proof of Theorem~\ref{th:bif-ness}(II) for a general symplectic matrix $M$]
Let $m=i_{\tau,M}(\gamma^M)+n$.
By (\ref{e:PerHboundary4C}), we have
\begin{align*}
&[i_{\tau}(\check{\gamma}_{\lambda_k^-}), i_{\tau}(\check{\gamma}_{\lambda_k^-})+\nu_{\tau}(\check{\gamma}_{\lambda_k^-})]\cap[i_{\tau}(\check{\gamma}_{\lambda_k^+}), i_{\tau}(\check{\gamma}_{\lambda_k^+})+\nu_{\tau}(\check{\gamma}_{\lambda_k^+})]\\
&=[i_{\tau,M}(\gamma_{\lambda_k^-})-m, i_{\tau,M}(\gamma_{\lambda_k^-})+\nu_{\tau,M}(\gamma_{\lambda_k^-})-m]\cap[i_{\tau,M}(\gamma_{\lambda_k^+})-m, i_{\tau,M}(\gamma_{\lambda_k^+})+\nu_{\tau,M}(\gamma_{\lambda_k^+})-m]\\
&=([i_{\tau,M}(\gamma_{\lambda_k^-}), i_{\tau,M}(\gamma_{\lambda_k^-})+\nu_{\tau,M}(\gamma_{\lambda_k^-})]-m)\cap([i_{\tau,M}(\gamma_{\lambda_k^+}), i_{\tau,M}(\gamma_{\lambda_k^+})+\nu_{\tau,M}(\gamma_{\lambda_k^+})]-m)=\emptyset
\end{align*}
   and either $\nu_{\tau}(\check{\gamma}_{\lambda_k^+})=0$ or $\nu_{\tau}(\check{\gamma}_{\lambda_k^-})=0$.
As above, the desired result directly follows from Theorem~1.4(II) with $M=I_{2n}$ and
  the above (iii).
 \end{proof}

\begin{proof}[\bf Proof of Theorem~\ref{th:bif-ness}(III) for a general symplectic matrix $M$]
As above, we have
\begin{align*}
&[i_{\tau}(\check{\gamma}_{\lambda^-}), i_{\tau,M}(\check{\gamma}_{\lambda^-})+\nu_{\tau}(\check{\gamma}_{\lambda^-})]\cap[i_{\tau}(\check{\gamma}_{\lambda^+}), i_{\tau}(\check{\gamma}_{\lambda^+})+\nu_{\tau}(\check{\gamma}_{\lambda^+})]\\
&=[i_{\tau,M}(\gamma_{\lambda^-})-m, i_{\tau,M}(\gamma_{\lambda^-})+\nu_{\tau,M}(\gamma_{\lambda^-})-m]\cap[i_{\tau,M}(\gamma_{\lambda^+})-m, i_{\tau,M}(\gamma_{\lambda^+})+\nu_{\tau,M}(\gamma_{\lambda^+})-m]\\
&=\emptyset
\end{align*}
 and hence the expected conclusion follows.
\end{proof}

\begin{proof}[\bf Proof of Theorem~\ref{th:bif-suffict} for a general symplectic matrix $M$]
By the first equality in (\ref{e:PerHboundary4C}),
 $\dim{\rm Ker}(\check{\gamma}_\mu(\tau)-I_{2n})=\dim{\rm Ker}(\gamma_\mu(\tau)-M)\ne 0$ and
  $\dim{\rm Ker}(\check{\gamma}_\lambda(\tau)-I_{2n})=\dim{\rm Ker}(\gamma_\lambda(\tau)-M)=0$
 for each $\lambda\in\Lambda\setminus\{\mu\}$ near $\mu$.
Let $m=i_{\tau,M}(\gamma^M)+n$, and let $\Lambda_1$ and $\Lambda_2$ be
two deleted half neighborhoods  of $\mu$ such that
 $i_{\tau,M}(\gamma_\lambda)=i_{\tau,M}(\gamma_\mu)$ for $\lambda\in\Lambda_1$,
 and $i_{\tau,M}(\gamma_\lambda)=i_{\tau,M}(\gamma_\mu)+ \nu_{\tau,M}(\gamma_\mu)$
for $\lambda\in\Lambda_2$.
Then  (\ref{e:PerHboundary4C}) implies that
\begin{align*}
i_{\tau}(\check{\gamma}_\lambda)&=i_{\tau,M}(\gamma_\lambda)-m=i_{\tau,M}(\gamma_\mu)-m=i_{\tau}(\check{\gamma}_\mu)
\;\text{for all $\lambda\in\Lambda_1$},\\
i_{\tau}(\check{\gamma}_\lambda)&=i_{\tau,M}(\gamma_\lambda)-m=i_{\tau,M}(\gamma_\mu)+\nu_{\tau,M}(\gamma_\mu)-m=
i_{\tau}(\check{\gamma}_\mu)+ \nu_{\tau}(\check{\gamma}_\mu)
\;\text{for all $\lambda\in\Lambda_2$}.
\end{align*}
Note also that $u_\lambda=0\;\Leftrightarrow\;w_\lambda=0$
and that $\check{H}(\lambda,t, \cdot)$ is even if and only if $H(\lambda,t, \cdot)$ is even.
As above, the desired result directly follows from Theorem~\ref{th:bif-suffict}, with $M=I_{2n}$ and
  the above (i).
\end{proof}

\begin{proof}[\bf Proof of Corollary~\ref{cor:necess-suffi} for a general symplectic matrix $M$]
Define
$$
H_0^\star(t, z)=H_0(t, \bar{\gamma}^M(t)z)+\frac{1}{2}(\bar{B}_M(t)z,z)_{\mathbb{R}^{2n}}\quad\hbox{and}\quad
\hat{H}^\star(t, z)=\hat{H}(t, \bar{\gamma}^M(t)z)
$$
for all $(t,z)\in [0, \tau]\times\mathbb{R}^{2n}$. By (\ref{e:PerHboundary0C}), we have
$\check{H}(\lambda,t, z)=H_0^\star(t, z)+ \lambda\hat{H}^\star(t, z)$.
 Define the transformed path $\bar{w}(t):=\bar{\gamma}^M(t)^{-1}\bar{u}(t)$. It satisfies
\begin{equation}\label{e:PerHboundary5C}
\dot{w}(t)=J_n\nabla_z {H}_0^\star(t, w(t))\quad\text{and}\quad w(\tau)=w(0).
\end{equation}
 Since $\nabla_z\hat{H}^\star(t, z)=\bar{\gamma}^M(t)^\top\nabla_z\hat{H}_0(t, \bar{\gamma}^M(t)z)$,
we have $\nabla_z\hat{H}^\star(t, \bar{w}(t))=\bar{\gamma}^M(t)^\top\nabla_z\hat{H}_0(t, \bar{u}(t))=0$ for all $t\in [0,\tau]$.
Moreover, from $\nabla^2_z\hat{H}^\star(t, z)=\bar{\gamma}^M(t)^\top\nabla_z^2\hat{H}(t, \bar{\gamma}^M(t)z)\bar{\gamma}^M(t)$,
it follows that for each $t\in [0, \tau]$,
\begin{align*}
\nabla^2_z\hat{H}^\star(t, \bar{w}(t))>0\;\Longleftrightarrow\;\nabla_z^2\hat{H}(t, \bar{u}(t))>0,\\[2pt]
\nabla^2_z\hat{H}^\star(t, \bar{w}(t))<0\;\Longleftrightarrow\;\nabla_z^2\hat{H}(t, \bar{u}(t))<0.
\end{align*}
Note that (\ref{e:PerHboundary4C}) implies $\Sigma=\{\lambda\in\mathbb{R}\mid \nu_{\tau,M}(\gamma_\lambda)>0\}=
\{\lambda\in\mathbb{R}\mid \nu_\tau(\check{\gamma}_\lambda)>0\}$ and that
 for each $\mu\in\Sigma$ and a small enough $\rho>0$ it holds that
\begin{equation*}
i_{\tau}(\check{\gamma}_\lambda)=
\begin{cases}
i_{\tau}(\check{\gamma}_\mu) & \forall\lambda\in[\mu-\rho,\mu),\\
i_{\tau}(\check{\gamma}_\mu)+\nu_{\tau}(\check{\gamma}_\mu) & \forall\lambda\in(\mu,\mu+\rho]
\end{cases}
\end{equation*}
if $\nabla^2_z\hat{H}^\star(t, \bar{w}(t))>0\;\forall t\in[0,\tau]$, and
\begin{equation*}
i_{\tau}(\check{\gamma}_\lambda)=
\begin{cases}
i_{\tau}(\check{\gamma}_\mu)+\nu_{\tau}(\check{\gamma}_\mu) & \forall\lambda\in[\mu-\rho,\mu),\\
i_{\tau}(\check{\gamma}_\mu) & \forall\lambda\in(\mu,\mu+\rho]
\end{cases}
\end{equation*}
if $\nabla^2_z\hat{H}^\star(t, \bar{w}(t))<0\;\forall t\in[0,\tau]$.
Since $\bar{u}=0\;\Leftrightarrow\;\bar{w}=0$
and ${H}_0^\star(t, \cdot)$ (resp. $\hat{H}^\star(t, \cdot)$) is even if and only if
${H}_0(t, \cdot)$ (resp. $\hat{H}(t, \cdot)$) is even.
As above, the desired result directly follows from Corollary~\ref{cor:necess-suffi}, with $M=I_{2n}$ and
  the above (i).
\end{proof}

\begin{proof}[\bf Proof of Corollary~\ref{cor:bif-deform}]
{\it Proof of} (I).
Define $F:(0, 1]\times{\R}^{2n}\to\R$ by
$$
F(\lambda, z)=\lambda H(\lambda,z).
$$
It satisfies Assumption~\ref{ass:BasiAss1} and $\nabla_z F(\lambda, \bar{u})=0$ for all $\lambda$.
Let
$$
B_\lambda=\nabla_z^2F(\lambda, \bar{u})=\lambda \nabla_z^2H(\bar{u})=\lambda B,
$$
and let  $\gamma_{\lambda}$ be the fundamental matrix solution of
$\dot{Z}=JB_\lambda(t)Z$ on $[0,\tau]$. It is easily checked that
$\gamma_{\lambda}(t)=\Upsilon_{B}(\lambda t)$
for all $\lambda\in (0, 1]$ and $t\in [0,\tau]$.
Since $B>0$, by (\ref{e:positive-negativeA}) we have
\begin{eqnarray*}
i_{\tau,M}(\Upsilon_{B})-i_{\tau,M}(\xi_{2n})-\dim{\rm Ker}(I_{2n}-M)
&=&\sum_{0< s<\tau}\nu_{s,M}(\Upsilon_{B}|_{[0,s]})\\
&=&\sum_{0<\lambda<1}\nu_{\lambda\tau,M}(\Upsilon_{B}|_{[0,\lambda\tau]})=
\sum_{0<\lambda<1}\nu_{\tau,M}(\gamma_{\lambda}).
\end{eqnarray*}
From this and the assumption $i_{\tau,M}(\Upsilon_{B})\ne i_{\tau,M}(\xi_{2n})+ \dim{\rm Ker}(I_{2n}-M)$
it follows that
$\{0<\lambda<1\,|\, \nu_{\tau,M}(\gamma_{\lambda})>0\}$ is nonempty and only
consists of  finitely many numbers $\lambda_1<\cdots <\lambda_l$.
Let $\tau_k=\lambda_k\tau$, $k=1,\cdots,l$. Then
\begin{eqnarray*}
\nu_{\tau_k,M}(\Upsilon_B|_{[0,\tau_k]}):=\dim{\rm Ker}(\Upsilon_{B}(\tau_k)-M)
=\dim{\rm Ker}(\gamma_{\lambda_k}(\tau)-M)=\nu_{\tau,M}(\gamma_{{\lambda_k}})\ne 0
\end{eqnarray*}
for $k=1,\cdots, l$.
Moreover, for each $\lambda\in (0, 1]$ we can derive from (\ref{e:positive-negativeA}) that
\begin{eqnarray*}
&&i_{\tau,M}(\gamma_\lambda)-i_{\tau,M}(\xi_{2n})-\dim{\rm Ker}(I_{2n}-M)
=j_{\tau,M}(B_\lambda|0)\\
&&=\sum_{0<t<\tau}\nu_{t,M}(\gamma_\lambda|_{[0,t]})=\sum_{0<t<\lambda\tau}
\dim{\rm Ker}(\Upsilon_{B}(t)-M).
\end{eqnarray*}
Hence for each $\lambda_k$, $k=1,\cdots,l$, we obtain
\begin{equation}\label{e:Deform-Equ.1}
\left.\begin{array}{ll}
&i_{\tau,M}(\gamma_\lambda)=i_{\tau,M}(\gamma_{\lambda_k})\quad\hbox{if $\lambda_{k-1}<\lambda\le \lambda_k$}, \\
&i_{\tau,M}(\gamma_\lambda)=i_{\tau,M}(\gamma_{\lambda_k})+ \nu_{\tau,M}(\gamma_{\lambda_k})\quad\hbox{if $\lambda_{k}<\lambda\le \lambda_{k+1}$}
\end{array}\right\}
\end{equation}
where $\lambda_0=0$ and $\lambda_{l+1}=1$. Note that
 $u:[0,\tau]\to\R^{2n}$ satisfies the  boundary value problem
\begin{equation}\label{e:Deform-Equ.2}
\dot{u}(t)=J\nabla_z F(\lambda, u(t))\;\forall t\in [0,\tau]\quad\hbox{and}\quad u(\tau)=Mu(0)
\end{equation}
if and only if $v:[0,\lambda\tau]\to \R^{2n},\;t\mapsto u(t/\lambda)$ satisfies
\begin{equation}\label{e:Deform-Equ.3}
\dot{v}(t)=J\nabla_z H(v(t))\;\forall t\in [0,\lambda\tau]\quad\hbox{and}\quad v(\lambda\tau)=Mv(0),
\end{equation}
and in this situation there holds
$$
\int^{\lambda\tau}_0\left[\frac{1}{2}(J\dot{v}(t),v(t))_{\mathbb{R}^{2n}}+ H({v}(t))\right]dt=\int^{\tau}_0\left[\frac{1}{2}(J\dot{u}(t),u(t))_{\mathbb{R}^{2n}}+ F(\lambda, {u}(t))\right]dt.
$$
Because of (\ref{e:Deform-Equ.1}), applying Theorem~\ref{th:bif-suffict} to $F$
leads to the desired results.

{\it Proof of} (II).
Define $\hat{H}:{\R}^{2n}\to\R$ by
$\hat{H}(z)=- H(z)$. It satisfies Assumption~\ref{ass:BasiAss1} and $\nabla \hat{H}(\bar{u})=0$ for all $(\lambda,t)$.
Note that $\hat{B}=\nabla^2\hat{H}(\bar{u})=- \nabla^2H(\bar{u})=-B$ is positive definite
for each $t$, and that $\Upsilon_{\hat{B}}(t)=\Upsilon_B(\tau-t)\Upsilon_B(\tau)^{-1}$
for all  $t\in [0,\tau]$. Applying the conclusions in (I) to $\hat{H}$ we get
 at least one and at most  finitely many numbers
in $(0, \tau)$, $\delta_1<\cdots<\delta_l$, such that
$$
\nu_{\delta_k,M}(\Upsilon_{\hat{B}}):=\dim{\rm Ker}(\Upsilon_{\hat{B}}(\delta_k)-M)\ne 0,\quad k=1,\cdots, l,
$$
and that (II) holds if $H$ and $\tau_k$ are replaced by $\hat{H}$ and $\delta_k$, $k=1,\cdots,l$.
Moreover, for $0<\rho<\tau$, $u:[0,\rho]\to\mathbb{R}^{2n}$ satisfies  the following boundary value problem
\begin{equation*}
\hbox{$\dot{u}(t)=J\nabla_z \hat{H}(u(t))\;\forall t\in [0,\rho]$ \quad and\quad $u(\rho)=Mu(0)$}
\end{equation*}
if and only if $w:[\tau-\rho^2/\tau, \tau]\to \mathbb{R}^{2n}$ given by $w(t)=u(\tau-\frac{\rho}{\tau} t)$
satisfies
\begin{equation*}
\hbox{$\dot{w}(t)=\frac{\tau}{\rho}J\nabla_z {H}(w(t))\;\forall t\in [\tau-\rho^2/\tau,\tau]$ \quad and\quad $Mw(\tau)=w(\tau-\rho^2/\tau)$};
\end{equation*}
and in this situation there holds
$$
\int^{\rho}_0\left[\frac{1}{2}(J\dot{u}(t),u(t))_{\mathbb{R}^{2n}}+ \hat{H}({u}(t))\right]dt=
\int^{\tau}_{\tau-\rho^2/\tau}\left[\frac{1}{2}(J\dot{w}(t), w(t))_{\mathbb{R}^{2n}}+
\frac{\tau}{\rho}H({w}(t))\right]dt.
$$
Using these and the conclusions in (I) it is easy for us to derive (II).
\end{proof}

\section{Proofs of Theorems~\ref{th:bif-per3},\ref{th:bif-ness-orbit},
\ref{th:bif-suffict1-orbit},\ref{th:bif-existence-orbit},\ref{th:bif-suffict-orbit} and Corollary~\ref{cor:bif-per5}}\label{sec:orbit}
\setcounter{equation}{0}

All proofs are completed in two subsections.
Before giving these we begin with  the following.\\

\noindent{\textsf{Variational spaces}}.
Consider the Hilbert spaces
\begin{eqnarray*}
\textsf{L}_{\tau,M}&=&\{v\in L^2_{\rm loc}(\R,\R^{2n})\,|\, v(t+\tau)=Mv(t)\;\hbox{for a.e.}\;t\in\R\},\\
\textsf{W}_{\tau,M}&=&\{v\in W^{1,2}_{\rm loc}(\R,\R^{2n})\,|\, v(t+\tau)=Mv(t)\;\hbox{for all}\;t\in\R\}
\end{eqnarray*}
equipped with inner products (\ref{e:innerP}) and (\ref{e:innerP2}) respectively.
There exists a natural (continuous) $\mathbb{R}$-action on $\textsf{L}_{\tau,M}$ and $\textsf{W}_{\tau,M}$:
\begin{equation}\label{e:R-action0}
(\theta, v)\mapsto \theta\cdot v
\end{equation}
given by $(\theta\cdot v)(t):=v(\theta+t)$ for any $t\in\mathbb{R}$.

Let $\mathcal{S}$ be one of the spaces $C^i_{M}([0,\tau];\R^{2n})$, ${W}^{1,2}_{M}([0,\tau];\R^{2n})$ and $L^{2}([0,\tau];\R^{2n})$.
For each $u\in\mathcal{S}$ we define $u^M:\R\to\R^{2n}$ by (\ref{e:extend}).
Clearly, if $u\in C^i_{M}([0,\tau];\R^{2n})$, then $u^M\in C^i(\R, \R^{2n})$ and satisfies $u^M(t+\tau)=Mu^M(t)\;\forall t\in\R$.
When $u\in L^{2}([0,\tau];\R^{2n})$ we have $u^M(t+\tau)=Mu^M(t)\;\forall t\in\R\setminus(\tau\Z)$, and so
$u^M(t+\tau)=Mu^M(t)$ for a.e. $t\in\R$. Thus there exists a natural (continuous) $\mathbb{R}$-action on $\mathcal{S}$,
\begin{equation}\label{e:R-action}
\mathbb{R}\times \mathcal{S}\to  \mathcal{S},\;(\theta, u)\mapsto (\theta\ast u)(\cdot):=(\theta\cdot u^M)|_{[0,\tau]}.
\end{equation}
Clearly, the actions in (\ref{e:R-action0}) act on $\mathcal{S}$
 by Banach isometries
if $M$ is also an orthogonal matrix, and
become $S^1$-actions, where $S^1=\mathbb{R}/(l\tau\mathbb{Z})$, if
  $M^l=I_{2n}$ for some integer $l\ge 1$. Moreover Hilbert space isomorphisms
$$
{W}^{1,2}_{M}([0,\tau];\R^{2n})\to \textsf{W}_{\tau,M}\quad\hbox{and}\quad L^{2}([0,\tau];\R^{2n})\to \textsf{L}_{\tau,M}
$$
given by $u\to u^M$
are equivariant with respect to the actions in (\ref{e:R-action0}) and (\ref{e:R-action}).
Because of these, \textsf{variational spaces will be chosen as ${W}^{1,2}_{M}([0,\tau];\R^{2n})$ and $L^{2}([0,\tau];\R^{2n})$
in the following proofs} though natural choices should be $\textsf{W}_{\tau,M}$ and $\textsf{L}_{\tau,M}$.

 Let ${\R}_u$ denote the isotropy group  at a point $u\in\mathcal{S}$
 of the action in (\ref{e:R-action}).
It is a subgroup of $(\mathbb{R}, +)$. If $\mathbb{R}_u$ contains a sequence $(\theta_k)\subset\mathbb{R}\setminus\{0\}$
converging to $0$,  then $\mathbb{R}_u$ must be dense in $\mathbb{R}$
and so is equal to $\mathbb{R}$ because the continuity of the map
$\theta\mapsto \theta\cdot u$ implies that $\mathbb{R}_u$ is closed. Therefore
 $\mathbb{R}_u$  must be equal to either $\R$, or $\{0\}$ or $\varrho\Z$ for some real $\varrho>0$.
It follows that $u$ is constant in the first case, injective in the second case
(if $\mathcal{S}$ is not $L^2$), and
has a period $\varrho=k\tau$ for some $k\in\N$ in the final case.
The second case has no effect on our arguments, since solutions are at least $C^1$, as demonstrated
by the following proposition.

\begin{proposition}\label{prop:criticalCase*}
Let $H:\mathbb{R}^{2n}\to\R$ be $C^{1,1}$ and  $M$-invariant.
Suppose that $u:[0,\tau]\to\R^{2n}$ satisfies $\dot{u}=J\nabla H(u)$ and $u(\tau)=Mu(0)$.
Then $u$ is $C^{1,1}$ and
\begin{enumerate}
\item[\rm (i)] if $\R_u=\R$, $u$ is equal to a constant vector $u_0$ in ${\rm Ker}(M-I_{2n})$;
 \item[\rm (ii)] if $\R_u=\{0\}$, $\theta\mapsto \theta\ast u$
  is an one-to-one $C^1$ immersion from $\mathbb{R}$ to $C^1_{M}([0,\tau];\R^{2n})$;
   \item[\rm (iii)] if $\R_u=\varrho\Z$ for some $\varrho>0$, i.e., $\varrho$ is a minimal period of $u$,
    the $\R$-orbit $\{\theta\ast u\,|\,\theta\in\R\}$
   is a $C^1$ embedded $S^1=\R/(\varrho\Z)$ in $C^1_{M}([0,\tau];\R^{2n})$.
\end{enumerate}
Moreover, if $M=I_{2n}$ then $\tau\mathbb{Z}\subset \mathbb{R}_u$ and therefore the case (ii) does not occur.
\end{proposition}

(ii) comes from the existence and uniqueness theorem for solutions of
$\dot{u}=J\nabla H(u)$ with initial values. (iii) uses the elementary topological fact that
a continuous one-to-one map from a compact space to a Hausdorff space is a homeomorphism.

\subsection{Proof of Theorem~\ref{th:bif-per3} and Corollary~\ref{cor:bif-per5}}\label{sec:orbit.2}

\begin{proof}[\bf Proof of Theorem~\ref{th:bif-per3}]
Define
$\hat{H}:\Lambda\times{\R}^{2n}\to\R$ by
$$
\hat{H}(\lambda,z)={H}(\lambda,z+v_\lambda).
$$
Clearly, it satisfies Assumption~\ref{ass:BasiAss2} with $v_\lambda\equiv 0$.
Therefore in what follows we always write $\hat{H}$ as $H$ and assume  $v_\lambda=0$ for all $\lambda\in\Lambda$.

When $\bar{v}$ is viewed as a constant value map from $[0,\tau]$ to $\R^{2n}$
we write it as $\bar{u}$, i.e.,  $\bar{u}=\bar{v}|_{[0,\tau]}$.
Then $M\bar{u}=\bar{u}$ and $\bar{u}$ is a fixed point for the $\R$-action in (\ref{e:R-action}).

  As in the first paragraph of Section~\ref{sec:HamBif}
 we can modify $H$ as $\tilde{H}(\lambda, z)=\chi(z)H(\lambda, z)$ so that
it is also $M$-invariant. (Indeed, $\hat{U}_0:=\sum^l_{j=0}M^jU_0$ is $M$-invariant and $Cl(\hat{U}_0)$
may be contained in $U$ by shrinking $U_0$. We only need to replace $\chi$ by $\hat{\chi}$,
where $\hat{\chi}(z):=\chi(\sum^l_{j=0}M^jz)$.)

 Follow the notations in proofs of Theorems~\ref{th:bif-ness},~\ref{th:bif-suffict}.
As pointed out at the beginning of the proof of Theorem~\ref{th:bif-suffict}
we  have actually checked in the proof of Theorem~\ref{th:bif-ness}(II)
 that $\mathcal{L}_\lambda(\cdot)=\psi_\kappa(\lambda, \bar{w}+\cdot)$ with
 $$
\bar{w}=-\Lambda_{M,\tau,\kappa I_{2n}}\bar{u}=-J\dot{\bar{u}}-\kappa \bar{u}=-\kappa\bar{u}
$$
(and hence $M\bar{w}=\bar{w}$),  $H=L^2([0,\tau];\R^{2n})$ and $X=C^0_M([0,\tau];\R^{2n})$
satisfy conditions in Theorem~\ref{th:A.11} (\cite[Theorem~3.6]{Lu10} or \cite[Theorem~4.6]{Lu8})
except for the condition (f).
Note that $\mathcal{L}_\lambda(\cdot)$ are also invariant for the $S^1$-action with $S^1=\mathbb{R}/(l\tau\mathbb{Z})$ given by
(\ref{e:R-action}).\\

\noindent{\bf Step 1} ({\it Prove the alternative of (i) or (ii)}).
Let us prove  that Theorem~\ref{th:A.12} (\cite[Theorem~3.7]{Lu10}
or \cite[Theorem~5.12]{Lu8}) can be used.
Because of the assumption (b),
it suffices to prove that
\begin{equation}\label{e:FixedPS}
\hbox{the fixed point set of the induced $S^1$-action on
$H_\mu^0:={\rm Ker}(B_\mu(\bar{w}))$ is $\{0\}$, }
\end{equation}
where $B_\mu(\bar{w})$ is given by (\ref{e:p-self-adjoint}).
Let $\xi\in H^0_\mu$ be a fixed point for the above $S^1$-action.
Then it is constant and satisfies $M\xi=\xi$.  Moreover,
$B_\mu(\bar{w})\xi= A_{M,\tau,\kappa}\xi+ \nabla^2_z(H_\kappa)^\ast(\mu; \bar{w})\xi=0$, that is,
$$
\xi=-\Lambda_{M,\tau,\kappa I_{2n}}(\nabla_z^2(H_\kappa)^\ast(\mu; \bar{w})\xi)=-\kappa\nabla_z^2(H_\kappa)^\ast(\mu; \bar{w})\xi
=-\kappa[\nabla^2_zH(\mu, \bar{u})-\kappa I_{2n}]^{-1}\xi,
$$
where the second equality comes from  (\ref{e:LambdaKA}),  and
the third equality is because
$$
\nabla_z^2(H_\kappa)^\ast(\mu; \bar{w})\nabla^2_z H_\kappa(\mu, \bar{u})=I_{2n}
$$
by the equality $\bar{w}=-J\dot{\bar{u}}-\kappa \bar{u}=\nabla_z H(\mu, \bar{u})-\kappa \bar{u}=\nabla_z H_\kappa (\mu, \bar{u})$ and
\cite[page 92, Proposition 10]{Ek90}.
Then the fixed point set of the induced $S^1$-action on $H^0_\mu$ is equal to
$$
\{\xi\in H^0_\mu\,|\,
\nabla^2_zH(\mu, \bar{u})\xi=0\}=
\{\xi\in \mathbb{R}^{2n}\,|\,
\nabla^2_zH(\mu, \bar{u})\xi=0\}.
$$
By the assumption (a), i.e., ${\rm Ker}(M-I_{2n})\cap{\rm Ker}(\nabla^2_zH(\mu, \bar{u}))=\{0\}$,
we deduce $\xi=0$. (\ref{e:FixedPS}) is proved.

Now by Theorem~\ref{th:A.12} (\cite[Theorem~3.7]{Lu10} or \cite[Theorem~5.12]{Lu8})
 one of the following alternatives occurs:
\begin{enumerate}
\item[(I)] $(\mu, \bar{w})$ is not an isolated solution  in  $\{\mu\}\times L^2([0,\tau];\R^{2n})$
 of $\nabla_w\psi_\kappa(\mu, w)=0$.

 \item[(II)] There exist left and right  neighborhoods $\Lambda^-$ and $\Lambda^+$ of $\mu$ in $\Lambda$
and integers $n^+, n^-\ge 0$, such that $n^++n^-\ge\frac{1}{2}\dim H^0_{\mu}$,
and that for $\lambda\in\Lambda^-\setminus\{\mu\}$ (resp. $\lambda\in\Lambda^+\setminus\{\mu\}$)
the functional $\psi_\kappa(\lambda, \cdot)$ has at least $n^-$ (resp. $n^+$) distinct critical
$S^1$-orbits disjoint with $\bar{w}$, which converge to  $\bar{w}$ in $C^0_M([0,\tau];\R^{2n})$  as $\lambda\to \mu$.
\end{enumerate}

\underline{{In the case of (I)}}, we have a sequence
 $(w_j)\subset L^2([0,\tau];\R^{2n})\setminus\{\bar{w}\}$
such that $\|w_j-\bar{w}\|_{2}\to 0$ as $j\to\infty$ and
that $\nabla_w\psi_\kappa(\mu, w_j)=0$ for each $j\in\N$.
Because $\bar{w}$ is a fixed point for the $S^1$-action,
and different $S^1$-orbits
are not intersecting, by passing to a subsequence we can assume that any two of $w_j$, $j=0,1,\cdots$,
do not belong the same $S^1$-orbit. The same claim also holds for $\bar{u}$ and
 $u_j:=-(\Lambda_{M,\tau,\kappa I_{2n}})^{-1}w_j\ne \bar{u}$, $j=1,2,\cdots$.
 Moreover, since  $\Lambda_{M,\tau,\kappa I_{2n}}: W^{1,2}_{M}([0,\tau];\R^{2n})\to L^2([0,\tau];\R^{2n})$
is a Banach space isomorphism,
$u_j\to \bar{u}$ in $W^{1,2}_M([0,\tau];\R^{2n})$, and hence $\|u_j-\bar{u}\|_{C^1}\to 0$ by Proposition~\ref{prop:threeBifu}(i).
Since $H(\mu,\cdot)$ is $M$-invariant, each $u_j$ can be extended into a solution $\bar{v}_j:=(u_j)^M$ of (\ref{e:PPer1})
with $\lambda=\mu$ via (\ref{e:extend}), which is not equal to $\bar{v}$. Clearly, any two of $\bar{v}_j$, $j=0,1,\cdots$,
are $\R$-distinct, and $(\bar{v}_j)$ converges to $\bar{v}$ on any compact interval $I\subset\R$ in $C^1$-topology.

\underline{In the case of (II)},
note firstly that $\dim H_\mu^0$ is equal to the number of maximal linearly independent solutions of
\begin{eqnarray}\label{e:LBV}
\dot{u}(t)=J\nabla_z^2H(\mu, \bar{u})u(t)\;\forall t\in [0,\tau]\quad\hbox{and}\quad u(\tau)=Mu(0).
\end{eqnarray}
Since $M^T\nabla_z^2H(\mu, Mz)M=\nabla_z^2H(\mu,z)$,
the extension via (\ref{e:extend}) establishes an isomorphism
between the solution space of (\ref{e:LBV}) and that of
\begin{eqnarray*}
\dot{v}(t)=J\nabla_z^2H(\mu, \bar{v})v(t)\quad\hbox{and}\quad v(t+\tau)=Mv(t),\;\forall t\in \R.
\end{eqnarray*}
Hence $\dim H_\mu^0$ is equal to the number of maximal linearly independent solutions of
(\ref{e:linear3*}) with $\lambda=\mu$.
For $\lambda\in\Lambda^-\setminus\{\mu\}$ (resp. $\lambda\in\Lambda^+\setminus\{\mu\}$) let
$S^1\cdot w_\lambda^i$,\; $i=1,\cdots,n^-$ (resp. $n^+$)
be  distinct critical $S^1$-orbits of $\psi_\kappa(\lambda, \cdot)$
disjoint with $\bar{w}$,  which converge to  $\bar{w}$ in $C^0_M([0,\tau];\R^{2n})$  as $\lambda\to \mu$.
Since  $\Lambda^c_{M,\tau,\kappa I_{2n}}: C^1_{M}([0,\tau];\R^{2n})\to C^0_M([0,\tau];\R^{2n})$
is a Banach space isomorphism,
$$
\hbox{$u^i_\lambda:=- (\Lambda^c_{M,\tau,\kappa I_{2n}})^{-1}(w_\lambda^i)\notin S^1\cdot \bar{u}=\{\bar{u}\}$,\quad
 $i=1,\cdots,n^-$ (resp. $n^+$)},
 $$
and $S^1\cdot u^i_\lambda\to \bar{u}$ in $C^1_M([0,\tau];\R^{2n})$ as $\lambda\to \mu$.
It follows that  $v^i_\lambda:=(u^i_{\lambda})^M+ v_\lambda$,
 $i=1,\cdots,n^-$ (resp. $n^+$),  are $\R$-distinct solutions of (\ref{e:PPer1})
 under the initial assumptions on $H$
 in Assumption~\ref{ass:BasiAss2}, and
 also $\R$-distinct from $v_\lambda$ because each $v_\lambda$
 is a fixed point for the $S^1$-action given by
 (\ref{e:R-action0}).\\

\noindent{\bf Step 2} ({\it Prove the part after ``Moreover''}).
Since any orbit $\mathcal{O}$ of the induced $S^1$-action on $H_\mu^0$
is either a point or an embedded circle,
our  assumption  $\dim H^0_\mu=\nu_{\tau,M}(\gamma_\mu)\ge 3$ shows that
Theorem~\ref{th:A.13} (or \cite[Theorem~3.10]{Lu10}) is available to the functionals
$\mathcal{L}_\lambda(\cdot)=\psi_\kappa(\lambda, \bar{w}+\cdot)$.
 Then we obtain either (I) or one of the following alternatives occurs:
\begin{description}
\item[(III)]  For every $\lambda\in\Lambda\setminus\{\mu\}$ near $\mu\in\Lambda$ there is a
 $S^1$-orbit $S^1\ast \bar{w}_\lambda\ne S^1\ast \bar{w}=\{\bar{w}\}$ near $\bar{w}\in  C^0_M([0,\tau];\R^{2n})$
 such that $\nabla_w\psi_\kappa(\lambda, \bar{w}_\lambda)=0$ and that
 $S^1\ast \bar{w}_\lambda\to \bar{w}$ in $C^0_M([0,\tau];\R^{2n})$  as $\lambda\to \mu$.

\item[(IV)] For any small $S^1$-invariant neighborhood $\mathcal{N}$ of
$\bar{w}$ in $C^0_M([0,\tau];\R^{2n})$
there is an one-sided deleted neighborhood $\Lambda^0$ of $\mu\in\Lambda$ such that
for any $\lambda\in\Lambda^0$, $\nabla_w\psi_\kappa(\lambda, \cdot)=0$
 has either infinitely many  $S^1$-orbits  of solutions in $\mathcal{N}$,
 $S^1\ast\bar{w}^j_\lambda$, $j=1,2,\cdots$,   or  at least two $S^1$-orbits of
 solutions in $\mathcal{N}$ with different energy,
 $S^1\ast \hat{w}^1_\lambda\ne\{\bar{w}\}$ and $S^1\ast \hat{w}^2_\lambda\ne\{\bar{w}\}$.
Moreover, these orbits converge to $\bar{w}$ in $C^0_M([0,\tau];\R^{2n})$  as $\lambda\to \mu$.
\end{description}

\underline{In the case of (III)},
let $\bar{u}_\lambda:=- \Lambda_{M,\tau,\kappa I_{2n}}^{-1}(\bar{w}_\lambda)$,
which is not in $S^1\cdot \bar{u}=\{\bar{u}\}$.
Then $\bar{v}_\lambda:=(\bar{u}_{\lambda})^M+ v_\lambda$
   is a solution of (\ref{e:PPer1}) under the initial assumptions on $H$
 in Assumption~\ref{ass:BasiAss2}, and  also $\R$-distinct from $v_\lambda$.

\underline{In the case of (IV)}, let $\hat{u}^i_\lambda:=- \Lambda_{M,\tau,\kappa I_{2n}}^{-1}(\hat{w}_\lambda)$, $i=1,2$,
and $\bar{u}^j_\lambda:=- \Lambda_{M,\tau,\kappa I_{2n}}^{-1}(\bar{w}^j_\lambda)$, $j=1,2,\cdots$. Then
$\hat{v}^i_\lambda:=(\hat{u}^i_{\lambda})^M+ v_\lambda$, $i=1,2$, and
$\bar{v}^j_\lambda:=(\bar{u}^j_{\lambda})^M+ v_\lambda$, $j=1,2,\cdots$, are the desired solutions
if $\mathcal{N}$ is small enough. In fact,
since  $\Lambda_{M,\tau,\kappa I_{2n}}^c: C^1_{M}([0,\tau];\R^{2n})\to C^0_M([0,\tau];\R^{2n})$
is a Banach space isomorphism and $\bar{w}=0$, we deduce that for $i=1,2$ and all $j\in\mathbb{N}$,
\begin{eqnarray*}
&&\|\hat{v}^i_\lambda|_{[0,\tau]}- v_\lambda|_{[0,\tau]}\|_{C^1}=\|\hat{u}^i_{\lambda}\|_{C^1}\le
\|(\Lambda_{M,\tau,\kappa I_{2n}}^c)^{-1}\|\cdot\|\hat{w}^i_\lambda\|_{C^0}=
\|(\Lambda_{M,\tau,\kappa I_{2n}}^c)^{-1}\|\cdot\|\hat{w}^i_\lambda-\bar{w}\|_{C^0},\\
&&\|\bar{v}^j_\lambda|_{[0,\tau]}- v_\lambda|_{[0,\tau]}\|_{C^1}=\|\bar{u}^j_{\lambda}\|_{C^1}\le
\|(\Lambda_{M,\tau,\kappa I_{2n}}^c)^{-1}\|\cdot\|\bar{w}^j_\lambda\|_{C^0}=
\|(\Lambda_{M,\tau,\kappa I_{2n}}^c)^{-1}\|\cdot\|\bar{w}^j_\lambda-\bar{w}\|_{C^0}.
\end{eqnarray*}
  Hence we only need to take the above neighborhood $\mathcal{N}$ as
  the open ball of radius
  $$\left(\|(\Lambda_{M,\tau,\kappa I_{2n}}^c)^{-1}\|/\varepsilon\right)^{-1}
  $$
  and center at $\bar{w}=0$ in $C^0_M([0,\tau];\R^{2n})$.
\end{proof}

\begin{proof}[\bf Proof of Corollary~\ref{cor:bif-per5}]
By Corollary~\ref{cor:bif-per4} with $H_0=0$ and $\hat{H}=H$,
$\gamma_\lambda(t)=\exp(\lambda tJ{H}''(v_0))$ and therefore
$$
\Sigma_1=\{\lambda\in\mathbb{R}\,|\,\nu_{1,M}(\gamma_\lambda)>0\}=\Gamma(H,v_0,M)\cup\{0\}
$$
is a discrete set in $\mathbb{R}$, which gives rise to the claim (A). In addition,
for each $\mu\in\Gamma(H, \bar{v}, M)$
the problem (\ref{e:linear3*}) with $H(\mu, \cdot)=\mu\hat{H}(\cdot)$ and $v_\mu=\bar{v}$ has no nonzero constant solutions because ${\rm Ker}(\hat{H}''(\bar{v}))=\{0\}$.

Suppose ${H}''(\bar{v})>0$ and $\Gamma(H, \bar{v}, M)\cap(0, \infty)\ne\emptyset$
(resp. ${H}''(\bar{v})<0$ and $\Gamma(H, \bar{v}, M)\cap(-\infty, 0)\ne\emptyset$).
For $\mu\in \Gamma(H, \bar{v}, M)\cap(0, \infty)$ (resp. $\Gamma(H, \bar{v}, M)\cap(-\infty, 0)$),
it follows from (\ref{e:positive-negative3}) (resp. (\ref{e:positive-negative4})) that
$i_{1,M}(\gamma_\lambda)=i_{1,M}(\gamma_\mu)$ for $\lambda\le\mu$ close to $\mu$ and that
$i_{1,M}(\gamma_\lambda)=i_{1,M}(\gamma_\mu)+ \nu_{1,M}(\gamma_\mu)$ for $\lambda>\mu$ close to $\mu$.
Hence for each $\mu\in\Gamma(H, \bar{v}, M)$  Corollary~\ref{cor:bif-per4}
implies that one of the following alternatives occurs:
 \begin{enumerate}
\item[(1)] The problem
 \begin{equation}\label{e:PPer3.5.2}
\dot{v}(t)=\mu J\nabla H(v(t))\quad\hbox{and}\quad v(t+1)=Mv(t),\;\forall t\in \R
\end{equation}
 has a sequence of $\R$-distinct solutions, $w_k$, $k=1,2,\cdots$,
which are $\R$-distinct from $\bar{v}$ and  converge to $\bar{v}$ on any compact interval $I\subset\R$ in $C^1$-topology.
\item[(2)] There exist left and right  neighborhoods $\Lambda^-$ and $\Lambda^+$ of $\mu$ in $\mathbb{R}\setminus\{\mu\}$
and integers $n^+, n^-\ge 0$, such that $n^++n^-\ge \nu_{1,M}(\gamma_\mu)/2$,
and for $\lambda\in\Lambda^-\setminus\{\mu\}$ (resp. $\lambda\in\Lambda^+\setminus\{\mu\}$),
the problem
  \begin{equation}\label{e:PPer3.5.3}
\dot{v}(t)=\lambda J\nabla H(v(t))\quad\hbox{and}\quad v(t+1)=Mv(t),\;\forall t\in \R
\end{equation}
 has at least $n^-$ (resp. $n^+$) $\R$-distinct solutions
  $w_\lambda^i$, $i=1,\cdots,n^-$ (resp. $n^+$),
which are $\R$-distinct from  $\bar{v}$ and converge to  $\bar{v}$ on any compact interval $I\subset\R$ in $C^1$-topology as $\lambda\to\mu$.
\end{enumerate}
Moreover, if $\dim{\rm Ker}(\exp(\mu{H}''(\bar{v}))-M)\ge 3$,
by the conclusions after ``Moreover'' in Theorem~\ref{th:bif-per3}
 with $H(\lambda, x)=\lambda\hat{H}(x)$ and $v_\lambda\equiv \bar{v}\;\forall\lambda$, we obtain that
 either (1) holds or one of the following alternatives occurs:
\begin{enumerate}
\item[\rm (3)]  For every $\lambda\in\mathbb{R}\setminus\{\mu\}$ near $\mu$,
(\ref{e:PPer3.5.3}) has a  solution $w_\lambda$, which is $\R$-distinct from $\bar{v}$ and converges to $\bar{v}$
on any compact interval $I\subset\R$ in $C^1$-topology as $\lambda\to\mu$.

\item[\rm (4)]
For a given $\epsilon>0$  there is an one-sided  neighborhood $\Lambda^0$ of $\mu$ in $\R\setminus\{\mu\}$ such that
for any $\lambda\in\Lambda^0\setminus\{\mu\}$, (\ref{e:PPer3.5.3}) with parameter value $\lambda$
has either infinitely many $\mathbb{R}$-distinct solutions $\bar{w}_\lambda^k$
such that each of them is $\R$-distinct from $\bar{v}$ and $\|\bar{w}_\lambda^k|_{[0,1]}-\bar{v}\|_{C^1}<\epsilon$, $k=1,2,\cdots$, or
at least two $\mathbb{R}$-distinct solutions $\hat{w}_\lambda^1$ and $\hat{w}_\lambda^2$
such that:
\begin{enumerate}
\item[\rm a)] each of them is $\R$-distinct from $\bar{v}$,
\item[\rm b)] $\|\hat{w}_\lambda^i|_{[0,1]}-\bar{v}\|_{C^1}<\epsilon$, $i=1,2$,
\item[\rm c)]
{\small
\begin{eqnarray*}
\int^{1}_0\left[\frac{1}{2}(J\dot{\hat{w}}_\lambda^1(t),\hat{w}^1_\lambda(t))_{\mathbb{R}^{2n}}+ \lambda H(\hat{w}_\lambda^1(t))\right]dt
\ne \int^{1}_0\left[\frac{1}{2}(J\dot{\hat{w}}^2_\lambda(t), \hat{w}_\lambda^2(t))_{\mathbb{R}^{2n}}+ \lambda H(\hat{w}_\lambda^2(t))\right]dt.
\end{eqnarray*}}
\end{enumerate}
\end{enumerate}

Define $v_k:\R\to\R^{2n},\;t\mapsto w_k(t/\mu)$, $k=1,2,\cdots$,
and $v_\lambda^i:\R\to\R^{2n},\;t\mapsto w_\lambda^i(t/\lambda)$, $i=1,\cdots,n^-$ (resp. $n^+$).
By (1) and (2) they satisfy (B.i) and (B.ii), respectively.
Similarly, let us define
$$
v_\lambda(t):=w_\lambda(t/\lambda),\quad \bar{v}_\lambda^k(t):=\bar{w}_\lambda^k(t),\quad
\hat{v}_\lambda^i(t):=\hat{w}_\lambda^i(t/\lambda),\;i=1,2.
$$
They satisfy (\ref{e:PPer3.5.1})  with parameter value $\lambda$.
It is easily computed that for $i=1,2$,
\begin{eqnarray*}
\int^{1}_0\left[\frac{1}{2}(J\dot{\hat{w}}_\lambda^i(t),\hat{w}^i_\lambda(t))_{\mathbb{R}^{2n}}+ \lambda H(\hat{w}_\lambda^i(t))\right]dt
=\int^{\lambda}_0\left[\frac{1}{2}(J\dot{\hat{v}}_\lambda^i(t),\hat{v}^i_\lambda(t))_{\mathbb{R}^{2n}}+  H(\hat{v}_\lambda^i(t))\right]dt.
\end{eqnarray*}
Clearly, $\|{w}_\lambda\|_{C^1([0, 1])}=\|{v}_\lambda\|_{C^0([0, \lambda])}+\lambda
\|\dot{v}_\lambda\|_{C^0([0, \lambda])}$ for $\lambda>0$.
When $\lambda<0$, since $M\in{\rm Sp}(2n,\mathbb{R})$ is an orthogonal symplectic matrix,
and ${v}_\lambda(t)=M{v}_\lambda(t-\lambda)$ for all $t\in\mathbb{R}$,
we deduce
$$
\|{v}_\lambda\|_{C^0([\lambda,0])}+|\lambda|
\|\dot{v}_\lambda\|_{C^0([\lambda,0])}=\|{v}_\lambda\|_{C^0([0, -\lambda])}+|\lambda|
\|\dot{v}_\lambda\|_{C^0([0, -\lambda])}.
$$
Hence there always holds
$$
\|{w}_\lambda\|_{C^1([0, 1])}=\|{v}_\lambda\|_{C^0([0, |\lambda|])}+|\lambda|
\|\dot{v}_\lambda\|_{C^0([0, |\lambda|])}.
$$
The same holds for $\hat{w}^i_\lambda$ and $\hat{v}^i_\lambda$ ($i=1,2$),
 as well as for  $\bar{v}_\lambda^k$ and $\bar{w}_\lambda^k$ ($k\in\mathbb{N}$).
Then, (B.iii) and (B.iv) follow from the estimates above, together 
with alternatives (3) and (4),
respectively.
\end{proof}

\subsection{Proofs of
Theorems~\ref{th:bif-ness-orbit},\ref{th:bif-suffict1-orbit},\ref{th:bif-existence-orbit},
\ref{th:bif-suffict-orbit}}\label{sec:orbit.3}

Since $\mathbb{R}$ is a non-compact Lie group and the $\mathbb{R}$-action in (\ref{e:R-action}) is not $C^1$,
we cannot directly use the theorems in \cite[Sect.5.2]{Lu8} and \cite{Lu10}.
But the proof ideas of those theorems may be used to complete our proofs.

If $u\in C^k_{M}([0,\tau];\R^{2n})$ with $k\ge 1$ is nonconstant,
since $u^M=(u^M)^{(0)},\cdots, (u^M)^{(k)}$ are uniformly continuous over any compact interval of $\R$ it is easy to prove that the map
$f:\R\to L^{2}([0,\tau];\R^{2n})$ given by $f(\theta)=\theta\ast u$ is $C^k$,
and
\begin{equation}\label{e:otbitmap}
f^{(i)}(\theta)=(u^M)^{(i)}(\theta+\cdot)|_{[0,\tau]},\quad i=1,\cdots,k.
\end{equation}

In what follows we assume that $\bar{v}:\R\to\R^{2n}$ is a nonconstant solution of (\ref{e:PPer1})  for each $\lambda\in\Lambda$.
Then $\bar{v}$ is $C^3$. As in Section~\ref{sec:orbit.2} we write $\bar{u}:=\bar{v}|_{[0,\tau]}$.
 It is nonconstant. By Proposition~\ref{prop:criticalCase*} and the arguments above this proposition we see:\\
$\bullet$ either the map
$\theta\mapsto\theta\ast \bar{u}$ is an one-to-one $C^2$ immersion from $\R$ to
$W^{1,2}_M([0,\tau];\R^{2n})$, (thus the restriction of it to any compact submanifold of $\R$ is
a $C^2$-embedding by an elementary result in general topology,)
and $\{\theta\ast \bar{u}\,|\,\theta\in\R\}$
is a $C^2$-immersed  submanifold in ${W}^{1,2}_{M}([0,\tau];\R^{2n})$ without self-intersection points;\\
$\bullet$ or $\{\theta\ast \bar{u}\,|\,\theta\in\R\}$ is a $C^2$-embedded circle.

Let $\bar{w}:=-\Lambda_{M,\tau,\kappa I_{2n}}(\bar{u})=-J\dot{\bar{u}}-\kappa \bar{u}$.
It belongs to $C^2_M([0,\tau];\R^{2n})$.
Since $\Lambda_{M,\tau,\kappa I_{2n}}$ is a Banach space isomorphism from ${W}^{1,2}_{M}([0,\tau];\R^{2n})$ to $L^2([0,\tau];\R^{2n})$,
and equivariant with respect to the action in (\ref{e:R-action}), i.e.,
$\Lambda_{M,\tau,\kappa I_{2n}}(\theta\ast u)=\theta\ast (\Lambda_{M,\tau,\kappa I_{2n}}(u))$ for any $(\theta, u)\in\R\times{W}^{1,2}_{M}([0,\tau];\R^{2n})$,
we have:

\begin{claim}\label{cl:5orbit}
$\mathcal{O}:=\{\theta\ast \bar{w}\,|\,\theta\in\R\}$
is either a $C^2$-immersed  submanifold in $L^2([0,\tau];\R^{2n})$ without self-intersection points
or a $C^2$-embedded circle in $L^2([0,\tau];\R^{2n})$. In particular,
by (\ref{e:otbitmap})
$$
0\ne \frac{d}{d\theta}(\theta\ast \bar{w})=(\dot{\bar{w}})^M(\theta+\cdot)|_{[0,\tau]}\quad\forall\theta\in\mathbb{R}
\quad\hbox{and so}\quad \dot{\bar{w}}=(\dot{\bar{w}})^M|_{[0,\tau]}\ne 0.
$$
\end{claim}

Let $H^\bot$ denote the orthogonal complementary of $\dot{\bar{w}}$ in  $H:=L^2([0,\tau];\R^{2n})$ and
$$
B_{H}^\bot(\bar{w},\varepsilon):=\{\bar{w}+ u\,|\, u\in{H}^\bot,\;
\|u\|_{2}<\varepsilon\},\quad\forall \varepsilon>0.
$$
Clearly, $B_{H}^\bot(\bar{w},\varepsilon)$ is a $C^\infty$ submanifold of $L^2([0,\tau];\R^{2n})$ containing $\bar{w}$, and
$T_{\bar{w}}B_{H}^\bot(\bar{w},\varepsilon)=H^\bot$
and $L^2([0,\tau];\R^{2n})=T_{\bar{w}}B_{H}^\bot(\bar{w},\varepsilon)\oplus(\R\dot{\bar{w}})$.

Note that the isotropy groups ${\R}_{\bar{w}}$ and ${\R}_{\dot{\bar{w}}}$ of the
$\mathbb{R}$-action defined by (\ref{e:R-action}) satisfy ${\R}_{\bar{w}}\subset{\R}_{\dot{\bar{w}}}$.
For $\theta\in {\R}_{\bar{w}}$
and $u\in H^\bot$,
$\|\theta\ast u\|_2=\|u\|_2$ and $(\dot{\bar{w}},\theta\ast u)_2=(\theta\ast\dot{\bar{w}},\theta\ast u)_2=(\dot{\bar{w}},\ast u)_2=0$
imply that $B_{H}^\bot(\bar{w},\varepsilon)$ is ${\R}_{\bar{w}}$-invariant.
Since $\Lambda_{M,\tau,\kappa I_{2n}}(\theta\ast \bar{u})=\theta\ast \bar{w}$ for any $\theta\in\R$,
${\R}_{\bar{w}}$ is equal to the isotropy group  ${\R}_{\bar{u}}$  of the $\R$-action at $\bar{u}\in {W}^{1,2}_{M}([0,\tau];\R^{2n})$.

\begin{proposition}\label{prop:orbi-neigh}
There exist $\varepsilon>0$ such that for any $0<\epsilon\le\varepsilon$,
 $\R\ast B_{H}^\bot(\bar{w},\epsilon)$  is a neighborhood of the orbit $\mathcal{O}=\R\ast \bar{w}$.
\end{proposition}

 Because the $\R$-action is not $C^1$  we cannot directly deduce  that the orbit  $\R\ast w$
 has a transversal intersection with  $B_{H}^\bot(\bar{w},\epsilon)$
for  every $w\in L^2([0,\tau];\R^{2n})$ sufficiently close to $\bar{w}$.
But Proposition~\ref{prop:orbi-neigh} can be proved with the following:

\begin{proposition}[\hbox{\cite[Proposition~3.5]{BetPS1}}]\label{prop:BetPS2}
 Let $\mathcal{M}$ be a finite-dimensional manifold, $N$ a (possibly infinite dimensional)
Banach manifold, $Q\subset N$ a Banach submanifold, and $A$ a topological
space. Assume that $\chi:A\times \mathcal{M}\to N$ is a continuous function such that there exists
$a_0\in A$ and $m_0\in \mathcal{M}$ with:
\begin{enumerate}
\item[\rm (a)]  $\chi(a_0,m_0)\in Q$;
\item[\rm (b)] $\chi(a_0,\cdot): \mathcal{M}\to N$ is of class $C^1$;
\item[\rm (c)] $\partial_2\chi(a_0,m_0)(T_{m_0}\mathcal{M})+ T_{\chi(a_0,m_0)}Q=T_{\chi(a_0,m_0)}N$.
\end{enumerate}
Then, for $a\in A$ near $a_0$, $\chi(a,\mathcal{M})\cap Q\ne\emptyset$.
\end{proposition}

\begin{proof}[\bf Proof of Proposition~\ref{prop:orbi-neigh}]
Applying Proposition~\ref{prop:BetPS2} to $A=N=L^2([0,\tau];\R^{2n})$, $\mathcal{M}=\R$, $Q=B_{H}^\bot(\bar{w},\epsilon)$,
$a_0=\bar{w}$, $m_0=0$ and
$\chi:A\times \mathcal{M}\to N:(w,\theta)\mapsto\theta\ast w$,
we get that
$(\R\ast w)\cap B_{H}^\bot(\bar{w},\epsilon)\ne\emptyset$ if $w\in L^2([0,\tau];\R^{2n})$ is close to $\bar{w}$.
\end{proof}

As in the first paragraph of Section~\ref{sec:HamBif}  we can require that \textsf{all modified $H(\lambda,\cdot)$
are also $M$-invariant}. (See the proof of Theorem~\ref{th:bif-per3} in Section~\ref{sec:orbit.2}.)
Following the notations in Section~\ref{sec:HamBif},
 we have a family of $C^1$ and twice G\^ateaux-differentiable functionals
on  $H=L^2([0,\tau];\R^{2n})$,
\begin{equation}\label{e:HAction***}
\psi_\kappa(\lambda, u)=\int^{\tau}_0\left[\frac{1}{2}(u(t), (A_{M,\tau,\kappa}u)(t))_{\mathbb{R}^{2n}}+ H_\kappa^\ast(\lambda; u(t))\right]dt,
\quad\lambda\in\Lambda.
\end{equation}
Moreover, each $(\psi_\kappa)_{\lambda}(\cdot):=\psi_\kappa(\lambda, \cdot)$ is also invariant for the $\R$-action  given by (\ref{e:R-action}), and
\begin{eqnarray}\label{e:p-gradient**}
&\nabla(\psi_\kappa)_{\lambda}(v)=A_{M,\tau,\kappa}v+ \nabla_zH_\kappa^\ast(\lambda; v(\cdot)),\\
&B_\lambda(v):=D_v\nabla(\psi_\kappa)_{\lambda}(v)=A_{M,\tau,\kappa}+ \nabla^2_zH_\kappa^\ast(\lambda; v(\cdot))
\in \mathscr{L}_s(L^{2}([0,\tau];\R^{2n})),\label{e:p-self-adjoint**}\\
&B_\lambda(\bar{w})\dot{\bar{w}}=0. \label{e:p-self-adjoint***}
\end{eqnarray}

Fortunately, we can directly  prove the following result, which naturally holds if the $\R$-action in (\ref{e:R-action}) is $C^1$. Define
$$
\mathcal{L}_\lambda(u)=\psi_{\kappa}(\lambda, \bar{w}+u),\quad \forall u\in B_{H}(0,\varepsilon).
$$
Suppose that  $u\in B_{H}^\bot(0,\varepsilon)$ is a critical point of  $\mathcal{L}_\lambda$.
It is clear that $u$ is also a critical point of the restriction of $\mathcal{L}_\lambda$ to $B_{H}^\bot(0,\varepsilon)$.
Conversely, we have:

\begin{proposition}\label{prop:solution}
There exists $\varepsilon>0$ such that for any $\lambda\in\Lambda$, if $u\in B_{H}^\bot(0,\varepsilon)$
is a critical point of the restriction of $\mathcal{L}_\lambda$ to $B_{H}^\bot(0,\varepsilon)$
then it is a critical point of $\mathcal{L}_\lambda$.
Moreover, $u$ is $C^1$ and $\|u\|_{C^1}\to 0$ as $\|u\|_2\to 0$.
\end{proposition}

\begin{proof}[\bf Proof]
Let $u\in B_{H}^\bot(0,\varepsilon)$ be a critical point of
the restriction of $\mathcal{L}_\lambda$ to $B_{H}^\bot(0,\varepsilon)$. Then
$\bar{w}+u$ is a critical point of the restriction of $(\psi_{\kappa})_{\lambda}$ to $B_{H}^\bot(\bar{w},\varepsilon)$,  i.e.,
 \begin{equation}\label{e:orbit-critical}
 d(\psi_{\kappa})_{\lambda}(\bar{w}+u)[\xi]=0\quad\forall \xi\in T_{\bar{w}+u}B_{H}^\bot(\bar{w},\varepsilon)=H^\bot.
 \end{equation}

{\bf Step 1} ({\it Prove that  $u$ is $C^1$}).
Since
$$
d(\psi_{\kappa})_{\lambda}(\bar{w}+u)[\xi]=(A_{M,\tau,\kappa}(\bar{w}+u)+ \nabla_zH_\kappa^\ast(\lambda; \bar{w}+u), \xi)_2=0
$$
for all $\xi\in H^\bot$,
 we have  $\mu(\lambda, u)\in\R$ such that
\begin{equation}\label{e:gradient}
 \nabla(\psi_{\kappa})_{\lambda}(\bar{w}+u)=A_{M,\tau,\kappa}(\bar{w}+u)+ \nabla_zH_\kappa^\ast(\lambda; \bar{w}+u)=\mu(\lambda, u)\dot{\bar{w}}.
\end{equation}
  Note that $\dot{\bar{w}}$ is $C^1$ and that  $A_{M,\tau,\kappa}(\bar{w}+u)\in W^{1,2}_{M}([0,\tau];\R^{2n})$ by (\ref{e:Lambda-inverse}).
  The map
 $$
\Gamma:[0,\tau]\times \mathbb{R}^{2n}\to \R^{2n},\;(t,  z)\mapsto \nabla_z H_\kappa^\ast(\lambda; z)-\mu(\lambda, u)\dot{\bar{w}}(t)+
(A_{M,\tau,\kappa}(\bar{w}+u))(t)
 $$
is continuous, and $C^1$ in $z$. But $D_z\Gamma(t,z)=\nabla^2_zH_\kappa^\ast(\lambda; z)$ is invertible.
 Therefore applying the implicit function theorem to $\Gamma$
 we derive from (\ref{e:gradient}) that $\bar{w}+u$ and therefore $u$ is $C^0$.
 By (\ref{e:inver}),  $A_{M,\tau,\kappa}(\bar{w}+u)\in C^1_{M}([0,\tau];\R^{2n})$
 and hence
 $$
 t\mapsto\mu(\lambda, u)\dot{\bar{w}}(t)-(A_{M,\tau,\kappa}(\bar{w}+u))(t)
 $$
 is $C^1$. Then $\Gamma$ is $C^1$. Using the implicit function theorem to $\Gamma$ again
 we obtain that $u$ is $C^1$.

{\bf Step 2} ({\it For any $\epsilon>0$, prove that there exists $\delta>0$
such that $\|u\|_{2}<\delta$ implies $|\mu(\lambda, u)|<\epsilon$
for all $\lambda\in\Lambda$}).
 Since $\nabla(\psi_{\kappa})_{\lambda}(\bar{w})=0\;\forall\lambda\in\Lambda$,
  by (\ref{e:gradient}) we have
\begin{eqnarray*}
\mu(\lambda, u)(\dot{\bar{w}}, v)_2&=&(\nabla(\psi_{\kappa})_{\lambda}(\bar{w}+u), v)_2-(\nabla(\psi_{\kappa})_{\lambda}(\bar{w}), v)_2\nonumber\\
 &=&(A_{M,\tau,\kappa}(u), v)_2+ (\nabla_zH_\kappa^\ast(\lambda; \bar{w}+u), v)_2-(\nabla_z H_\kappa^\ast(\lambda;\bar{w}), v)_2\nonumber\\
 &=&\int^\tau_0\int^1_0(\nabla^2_zH_\kappa^\ast(\lambda,t;\bar{w}(t)+su(t))u(t), v(t))_{\mathbb{R}^{2n}}dtds\nonumber\\
 &&+ (A_{M,\tau,\kappa}(u), v)_2,\quad\forall v\in H.
\end{eqnarray*}
It follows from this and (\ref{e:HPositive1}) that
$$
|\mu(\lambda, u)|\cdot\|\dot{\bar{w}}\|_2\le \|A_{M,\tau,\kappa}(u)\|_2+ \frac{1}{c_1}\|u\|_2.
$$
 Then the claim follows because  $\|\dot{\bar{w}}\|_2>0$ and
 $A_{M,\tau,\kappa}\in\mathscr{L}_s(L^{2}([0,\tau];\R^{2n}))$ by (\ref{e:Lambda-inverse}).

{\bf Step 3} ({\it For any $\epsilon>0$, prove there exists $\varepsilon>0$
such that $\|u\|_{2}\le\varepsilon$
implies $\|u\|_{C^1}<\epsilon$}).
Since $\nabla(\psi_{\kappa})_{\lambda}(\bar{w})=0\;\forall\lambda\in\Lambda$,
by (\ref{e:gradient}) we have
\begin{equation}\label{e:gradient+}
A_{M,\tau,\kappa}(u)+ \nabla_zH_\kappa^\ast(\lambda; \bar{w}+u)-\nabla_z H_\kappa^\ast({\lambda};\bar{w})=\mu(\lambda, u)\dot{\bar{w}}.
\end{equation}
Note that this also holds in the sense of point-wise
because $\bar{w}$ is $C^2$, $u$ is $C^1$ and
\begin{equation}\label{e:gradient++}
A_{M,\tau,\kappa}(u)(t)=e^{\kappa tJ}\mathfrak{J}^{-1}\int^\tau_0e^{-\kappa tJ}Ju(t)dt-
e^{\kappa tJ}\int^t_0e^{-\kappa sJ}Ju(s)ds
\end{equation}
is $C^2$ by (\ref{e:Lambda-inverse})--(\ref{e:inver*}).
It follows that there exists a constant $C_\kappa>0$ such that
\begin{equation}\label{e:gradient+++}
|A_{M,\tau,\kappa}(u)(t)|_{\mathbb{R}^{2n}}\le C_\kappa\|u\|_2,\quad\forall t\in [0,\tau].
\end{equation}
Clearly,  (\ref{e:HPositive1}), the mean value theorem of integrals and (\ref{e:gradient+})  lead to
\begin{eqnarray*}
\frac{1}{c_2}|u(t)|^2_{\mathbb{R}^{2n}}&\le&\int^1_0(\nabla^2_zH_\kappa^\ast(\lambda;\bar{w}(t)+su(t))u(t), u(t))_{\mathbb{R}^{2n}}ds\\
&=&\big(\nabla_zH_\kappa^\ast({\lambda};\bar{w}(t)+u(t))-\nabla_zH_\kappa^\ast(\lambda; \bar{w}(t)), u(t)\big)_{\mathbb{R}^{2n}}\\
&=&\big(\mu(\lambda, u)\dot{\bar{w}}(t)-A_{M,\tau,\kappa}(u)(t), u(t)\big)_{\mathbb{R}^{2n}}\\
&\le&|\mu(\lambda, u)|_{\mathbb{R}^{2n}}\cdot|\dot{\bar{w}}(t)|_{\mathbb{R}^{2n}}\cdot|u(t)|_{\mathbb{R}^{2n}}+
|A_{M,\tau,\kappa}(u)(t)|_{\mathbb{R}^{2n}}\cdot |u(t)|_{\mathbb{R}^{2n}}
\end{eqnarray*}
and so
\begin{eqnarray*}
\frac{1}{c_2}|u(t)|_{\mathbb{R}^{2n}}
\le|\mu(\lambda, u)|_{\mathbb{R}^{2n}}\cdot|\dot{\bar{w}}(t)|_{\mathbb{R}^{2n}}+C_\kappa\|u\|_2
\quad\forall t\in [0,\tau]
\end{eqnarray*}
by (\ref{e:gradient+++}).
Therefore, since $\dot{\bar{w}}$ is $C^1$, for any $\epsilon'>0$,
from this and the result in Step 2  we derive that
there exists $\varepsilon>0$ such that
\begin{equation}\label{e:gradient4}
\|u\|_{2}\le\varepsilon\quad \Longrightarrow\quad
\|u\|_{C^0}<\epsilon'.
\end{equation}

Taking the derivatives on two sides of equation in (\ref{e:gradient+}) with respect to $t$ we obtain
\begin{eqnarray*}
&&\frac{d}{dt}A_{M,\tau,\kappa}(u)(t)+ \nabla^2_zH_\kappa^\ast(\lambda;\bar{w}(t)+u(t))\dot{u}(t)\\
&&+ \nabla^2_zH_\kappa^\ast(\lambda;\bar{w}(t)+u(t))\dot{\bar{w}}(t)
-\nabla^2_zH_\kappa^\ast(\lambda;\bar{w}(t))\dot{\bar{w}}(t)=\mu(\lambda, u)\ddot{\bar{w}}(t).
\end{eqnarray*}
It follows from  (\ref{e:HPositive1}) and this that
\begin{eqnarray}\label{e:gradient5}
\frac{1}{c_2}|\dot{u}(t)|_{\mathbb{R}^{2n}}&\le& |\nabla^2_zH_\kappa^\ast(\lambda;\bar{w}(t)+u(t))\dot{u}(t)|_{\mathbb{R}^{2n}}\nonumber\\
&\le&\left\|\nabla^2_zH_\kappa^\ast(\lambda;\bar{w}(t)+u(t))-
\nabla^2_zH_\kappa^\ast(\lambda;\bar{w}(t))\right\|_{\mathbb{R}^{2n\times 2n}}|\dot{\bar{w}}(t)|
\nonumber\\
&&+ \left|\frac{d}{dt}A_{M,\tau,\kappa}(u)(t)\right|_{\mathbb{R}^{2n}}+ |\mu(\lambda, u)|\cdot|\ddot{\bar{w}}(t)|
\end{eqnarray}
When $\|u\|_{C^0}\to 0$, (\ref{e:gradient++}) and Step 2 lead to
$$
\max_t\left\|\frac{d}{dt}A_{M,\tau,\kappa}(u)(t)\right\|_{\mathbb{R}^{2n}}\to 0\quad\hbox{and}\quad
 |\mu(\lambda, u)|\cdot\|\ddot{\bar{w}}\|_{C^0}\to 0,
 $$
respectively. Moreover, for a given compact neighborhood $\mathscr{U}$ of $\bar{w}([0,\tau])$ in $\mathbb{R}^{2n}$,
since $\nabla^2_zH_\kappa^\ast(\lambda;z)$ is uniformly continuous in $\Lambda\times \mathscr{U}$,
 we deduce that
$$
\max_t\left\|\nabla^2_zH_\kappa^\ast(\lambda;\bar{w}(t)+u(t))-
\nabla^2_zH_\kappa^\ast(\lambda;\bar{w}(t))\right\|_{\mathbb{R}^{2n\times 2n}}\to 0
$$
uniformly with respect to $\lambda\in\Lambda$ as $\|u\|_{C^0}\to 0$.
From these and (\ref{e:gradient5}) we deduce  that $\|\dot{u}\|_{C^0}\to 0$ as $\|u\|_{C^0}\to 0$.
The desired claim follows from this and (\ref{e:gradient4}).

{\bf Step 4} ({\it Prove $d(\psi_{\kappa})_\lambda(\bar{w}+u)=0$}).
Note that $\|\dot{\bar{w}}\|_{C^0}>0$.
By Step 3 we get $\varepsilon>0$ such that $\|u\|_{2}\le\varepsilon$ implies $\|u\|_{C^1}<\|\dot{\bar{w}}\|_{C^0}$ and
$\|u\|_{C^1}<\|\dot{\bar{w}}\|_{2}/\sqrt{\tau}$. Then $w:=\bar{w}+u$ satisfy
\begin{eqnarray*}
\|\dot{w}\|_{C^0}&\ge& \|\dot{\bar{w}}\|_{C^0}-\|\dot{u}\|_{C^0}>0\quad\hbox{and}\\
(\dot{\bar{w}}, \dot{w})_{2}&=&(\dot{\bar{w}}, \dot{\bar{w}})_{2}+ (\dot{\bar{w}}, \dot{u})_{2}\\
&\ge& \|\dot{\bar{w}}\|^2_{2}-\|\dot{\bar{w}}\|_{2}\|\dot{u}\|_{2}\\
&\ge&\|\dot{\bar{w}}\|^2_{2}-\sqrt{\tau}\|\dot{\bar{w}}\|_{2}\|{u}\|_{C^1}>0.
\end{eqnarray*}
The former shows that $w$ is nonconstant, and therefore that
the map $\theta\mapsto\theta\ast w$ is a $C^1$ immersion from $\R$ to
$H=L^2([0,\tau];\R^{2n})$. Then for $\nu>0$ small enough,   $\Delta:=[-\nu, \nu]\ast w$ is
a $C^1$ embedded submanifold and $T_w\Delta=\mathbb{R}\dot{w}$. Because $(\dot{\bar{w}}, \dot{w})_{2}>0$,
the orthogonal decomposition ${H}=(\R\dot{\bar{w}})\oplus T_{\bar{w}}B_{H}^\bot(\bar{w},\varepsilon)$ implies
  a direct sum decomposition of Banach spaces
  \begin{equation}\label{e:directsum}
  {H}=T_{w}\Delta\dot{+} T_{w}B_{{H}^\bot}(\bar{w},\varepsilon).
  \end{equation}
Since  $(\psi_{\kappa})_\lambda$ is  invariant for the $\R$-action  given by (\ref{e:R-action}),
  $d(\psi_{\kappa})_\lambda(w)[\xi]=0\;\forall \xi\in T_{w}\Delta$.
From this, (\ref{e:orbit-critical}) and (\ref{e:directsum}) it follows that $d(\psi_{\kappa})_\lambda(w)=0$.
\end{proof}

Consider the following  Banach space and functional
\begin{eqnarray}\label{e:orth-space}
&&X^\bot:=\{u\in C^0_M([0,\tau];\R^{2n})\,|\,(\dot{\bar{w}}, u)_2=0\},\\
&&\mathcal{L}^\bot_\lambda:H^\bot\to\R,\;u\mapsto\mathcal{L}_\lambda(u)=\psi_\kappa(\lambda, \bar{w}+ u).\label{e:orth-funct}
\end{eqnarray}
They are invariant under the isotropy group $\mathbb{R}_{\bar{w}}$, and $d\mathcal{L}^\bot_\lambda(0)=0\;\forall\lambda$.
Moreover, by shrinking $\varepsilon>0$ if necessary, Proposition~\ref{prop:solution} shows that for $u\in B_{H^\bot}(0,\varepsilon)$,
\begin{equation}\label{e:Bi.3.11.2}
d\mathcal{L}^\bot_\lambda(u)=0\;
\quad\Longleftrightarrow\quad d\mathcal{L}_\lambda(u)=
d(\psi_{\kappa})_{\lambda}(\bar{w}+u)=0.
\end{equation}
Denote by $\Pi:H\to H^\bot$ the orthogonal projection. Then $\Pi(u)=u-\frac{(u,\dot{\bar{w}})_2}{\|\dot{\bar{w}}\|_2^2}\dot{\bar{w}}$ for $u\in H$, and
\begin{equation}\label{e:gradient2}
\nabla\mathcal{L}^\bot_\lambda(u)=\nabla(\psi_{\kappa})_{\lambda}(\bar{w}+u)-
\frac{(\nabla(\psi_{\kappa})_{\lambda}(\bar{w}+u),\dot{\bar{w}})_2}{\|\dot{\bar{w}}\|_2^2}\dot{\bar{w}}
\quad \forall u\in H^\bot,
\end{equation}
where $\nabla(\psi_{\kappa})_{\lambda}(\bar{w}+u)=A_{M,\tau,\kappa}(\bar{w}+u)+ \nabla_zH_\kappa^\ast(\lambda, \bar{w}(\cdot)+ v(\cdot))$ by (\ref{e:p-gradient}).
Since $\dot{\bar{w}}$ is $C^1$, it easily follows from (\ref{e:gradient2})
 that $\nabla\mathcal{L}^\bot_\lambda(u)\in X^\bot$ for any $u\in X^\bot$, and
\begin{equation}\label{e:gradient3}
\mathfrak{A}_\lambda:X^\bot\to X^\bot,\;u\mapsto\nabla\mathcal{L}^\bot_\lambda(u)
\end{equation}
is $C^1$. Actually, we have more: for any $(\lambda, x), (\lambda_0, x_0)\in\Lambda\times X^\bot$,
noting that both $\|\cdot\|_X$ and $\|\cdot\|_{X^\bot}$
are $\|\cdot\|_{C^0}$, and
\begin{eqnarray*}
\mathfrak{A}_\lambda(x)-\mathfrak{A}_{\lambda_0}(x_0)&=&
\nabla(\psi_{\kappa})_{\lambda}(\bar{w}+x)-\nabla(\psi_{\kappa})_{\lambda_0}(\bar{w}+x_0)\\
&+&\frac{(\nabla(\psi_{\kappa})_{\lambda}(\bar{w}+x)-
\nabla(\psi_{\kappa})_{\lambda_0}(\bar{w}+x_0),\dot{\bar{w}})_2}{(\|\dot{\bar{w}}\|_2)^2}\dot{\bar{w}}
\end{eqnarray*}
we deduce
\begin{eqnarray*}\label{e:gradient3+}
\|\mathfrak{A}_\lambda(x)-\mathfrak{A}_{\lambda_0}(x_0)\|_{X^\bot}&\le&
\|\nabla(\psi_{\kappa})_{\lambda}(\bar{w}+x)-\nabla(\psi_{\kappa})_{\lambda_0}(\bar{w}+x_0)\|_X\\
&+&\sqrt{\tau}\|\nabla(\psi_{\kappa})_{\lambda}(\bar{w}+x)-
\nabla(\psi_{\kappa})_{\lambda_0}(\bar{w}+x_0)\|_X\|\dot{\bar{w}}\|_{C^0}.
\end{eqnarray*}
By Step 2 of the proof of Theorem~\ref{th:bif-ness}(II),
$\Lambda\times X\ni (\lambda,v)\mapsto \nabla(\psi_{\kappa})_{\lambda}(v)=\nabla_v\psi_\kappa(\lambda, v)\in X$ is continuous.
This and (\ref{e:gradient3}) lead to

\begin{claim}\label{cl:contBot}
$\Lambda\times X^\bot\ni (\lambda, x)\mapsto \mathfrak{A}_\lambda(x)\in X^\bot$ is continuous.
\end{claim}

From (\ref{e:p-gradient}) and (\ref{e:p-self-adjoint}) we may deduce that
$\nabla\mathcal{L}^\bot_\lambda$ has a G\^ateaux derivative at $u\in H^\bot$,
$\mathfrak{B}_\lambda(u)\in\mathscr{L}_s(H^\bot)$, given by
\begin{equation}\label{e:gradient44}
\mathfrak{B}_\lambda(u)v=B_\lambda(u+\bar{w})v-
\frac{(B_\lambda(u+\bar{w})v,\dot{\bar{w}})_2}{(\|\dot{\bar{w}}\|_2)^2}\dot{\bar{w}}
\quad \forall v\in H^\bot,
\end{equation}
where $B_\lambda(u+\bar{w})=A_{M,\tau,\kappa}+ \nabla^2_zH_\kappa^\ast({\lambda};u(\cdot)+ \bar{w}(\cdot))$
by (\ref{e:p-self-adjoint}). Since $P_\lambda(v)=\nabla^2_zH_\kappa^\ast({\lambda};v(\cdot))$
(resp. $Q_\lambda(v)=A_{M,\tau,\kappa}$) is a positive definite (resp. compact self-adjoint) operator on $H$,
\begin{equation}\label{e:gradient55}
\mathfrak{P}_\lambda(u):=\Pi\circ P_\lambda(u+\bar{w})|_{H^\bot}\quad\hbox{(resp.\quad
$\mathfrak{Q}_\lambda(u):=\Pi\circ Q_\lambda(u+\bar{w})|_{H^\bot}$)}
\end{equation}
is a positive definite (resp. compact self-adjoint) operator on $H^\bot$. Moreover, it is clear that
$\mathfrak{B}_\lambda(u)=\mathfrak{P}_\lambda(u)+\mathfrak{Q}_\lambda(u)$.

\begin{proposition}\label{prop:hypA.5-bot}
$(H^\bot, X^\bot, \mathcal{L}^\bot_\lambda, \mathfrak{A}_\lambda, \mathfrak{B}_\lambda)$
satisfies \cite[Hypothesis~1.1]{Lu8} (\cite[Hypothesis~3.1]{Lu10}) and \cite[Hypothesis~1.3]{Lu8} (\cite[Hypothesis~3.2]{Lu10}).
\end{proposition}
\begin{proof}[\bf Proof]
It remains  to prove that
$(H^\bot, X^\bot, \mathcal{L}^\bot_\lambda, \mathfrak{A}_\lambda, \mathfrak{B}_\lambda)$
satisfies (D1) in \cite[Hypothesis~1.1]{Lu8} (\cite[Hypothesis~3.1]{Lu10}) and (C) in
\cite[Hypothesis~1.3]{Lu8} (\cite[Hypothesis~3.2]{Lu10}), that is,
\begin{equation}\label{e:D-D1}
\{u\in H^\bot\,|\,\mathfrak{B}_\lambda(0)u=su,\, s\le 0\}\subset X^\bot\quad\hbox{and}\quad
\{u\in H^\bot\,|\,\mathfrak{B}_\lambda(0)u\in X^\bot\}\subset X^\bot.
\end{equation}
 In fact, if  $u\in H^\bot$ and $\varrho\le 0$ satisfy $\mathfrak{B}_\lambda(0)u=\varrho u$,
 then $B_\lambda(\bar{w})u=\varrho u$ by (\ref{e:gradient44}) and (\ref{e:p-self-adjoint***}).
 The first inclusion in (\ref{e:C-D-D1}) leads to $u\in X$ and so $u\in X^\bot$.
 Similarly, if $u\in H^\bot$ is such that $v:=\mathfrak{B}_\lambda(0)u\in X^\bot$,
 (\ref{e:gradient44}) and (\ref{e:p-self-adjoint***}) lead to $B_\lambda(\bar{w})u=v$, and so the second claim
by the second inclusion in (\ref{e:C-D-D1}).
\end{proof}

Note that (\ref{e:p-self-adjoint***}) implies $\R\dot{\bar{w}}\subset{\rm Ker}(B_\lambda(\bar{w}))$. By (\ref{e:L-MorseIndex}) and (\ref{e:L-Nullity}) we get
 \begin{eqnarray}
 m^-(\mathcal{L}^\bot_\lambda, 0)&=&  m^-(({\psi}_{\kappa})_{\lambda}, \bar{w})=i_{\tau,M}(\gamma_\lambda)- i_{\tau,M}(\Upsilon_{\kappa I_{2n}})-\nu_{\tau,M}(\Upsilon_{\kappa I_{2n}}),\label{e:MorseIndex++}\\
  m^0(\mathcal{L}^\bot_\lambda, 0)&=& m^0(({\psi}_{\kappa})_{\lambda}, \bar{w})-1=\dim{\rm Ker}(\gamma_\lambda(\tau)-M)-1.\label{e:Nullity++}
\end{eqnarray}

\begin{proof}[\bf Proof of Theorem~\ref{th:bif-ness-orbit}]
In the proof of Theorem~\ref{th:bif-ness} we proved that
$$
\{\mathcal{F}_\lambda(\cdot):=(\psi_\kappa)_{\lambda}(\bar{w}+\cdot)\,|\,\lambda\in\Lambda\}
$$
satisfies the conditions of
Theorem~\ref{th:A.10} (\cite[Theorem~3.1]{Lu8})  with $H=X=L^2([0,\tau];\R^{2n})$.
 Thus if $\lambda_k\to\mu$ and $(x_k)\subset H^\bot$ converges to $0$
 it follows from (\ref{e:gradient55}) that
 \begin{eqnarray*}
\|\mathfrak{P}_{\lambda_k}(x_k)h-\mathfrak{P}_{\mu}(0)h\|&=&\|\Pi\circ P_{\lambda_k}(x_k+\bar{w})h-
\Pi\circ P_{\mu}(\bar{w})h\|\\
&\le &\|P_{\lambda_k}(x_k+\bar{w})h- P_{\mu}(\bar{w})h\|\to 0,\quad\forall h\in H^\bot.
 \end{eqnarray*}
 Similarly, we have $\| \mathfrak{Q}_{\lambda_k}(0)-\mathfrak{Q}_{\mu}(0)\|_{\mathscr{L}(H^\bot)}\le
\| Q_{\lambda_k}(\bar{w})-Q_{\mu}(\bar{w})\|_{\mathscr{L}(H^\bot)}\to 0$.
Moreover, for $x\in H^\bot$ near $0$ it is  easily seen that
 \begin{eqnarray*}
 &&(\mathfrak{P}_{\lambda}(x)h, h)_H=(P_{\lambda}(x+\bar{w})h, h)_H,\quad\forall h\in H^\bot,\\
 &&\| \mathfrak{Q}_{\lambda}(x)-\mathfrak{Q}_{\lambda}(0)\|_{\mathscr{L}(H^\bot)}
 \le\|Q_{\lambda}(x+\bar{w})- Q_{\lambda}(\bar{w})\|_{\mathscr{L}(H)}.
 \end{eqnarray*}
These imply that the conditions (i)-(iv) in \cite[Theorem~3.1]{Lu8} (cf. Theorem~\ref{th:A.10})
can be satisfied with $\mathfrak{P}_{\lambda}$ and  $\mathfrak{Q}_{\lambda}$ near $0\in H^\bot$.
Hence Theorem~\ref{th:A.10} (\cite[Theorem~3.1]{Lu8}) is applicable to
$\mathcal{F}_\lambda:=\mathcal{L}^\bot_\lambda$ near $0\in H^\bot$.

Let $(\lambda_k, v_k)$ be as in Theorem~\ref{th:bif-ness-orbit}. Then $\lambda_k\to\mu$, $\|v_k|_{[0,\tau]}- \bar{v}|_{[0,\tau]}\|_{C^0}\to 0$
 and each $w_k:=-\Lambda_{M,\tau,\kappa I_{2n}}(v_k|_{[0,\tau]})$ is a critical point of $(\psi_{\kappa})_{\lambda_k}$ on $L^2([0,\tau];\R^{2n})$.
By Proposition~\ref{prop:threeBifu}, $\|v_k|_{[0,\tau]}- \bar{v}|_{[0,\tau]}\|_{C^1}\to 0$
 and hence $\|w_k-\bar{w}\|_{C^0}\to 0$.
Take $0<\epsilon_m<\varepsilon$ such that $\epsilon_m\to 0$.
Since each $\R\ast B_{{H}^\bot}(\bar{w},\epsilon_m)$  is a neighborhood of the orbit $\mathcal{O}=\R\ast \bar{w}$ by Proposition~\ref{prop:orbi-neigh},
we have a subsequence $(w_{k_m})$ of $(w_k)$ such that $w_{k_m}\in\mathbb{R}\ast  B_{{H}^\bot}(\bar{w},\epsilon_m)$ for all $m\in\mathbb{N}$,
 that is, there exists a $\theta_m\in\R$ such that $u_m:=\theta_m\ast w_{k_m}\in B_{{H}^\bot}(\bar{w},\epsilon_m)\subset B_{{H}^\bot}(\bar{w},\varepsilon)$ for each $m\in\mathbb{N}$.
Note that $\|u_m-\bar{w}\|_2<\epsilon_m\to 0$ and that each $u_m$ is a critical point of $\psi_{\kappa}(\lambda_{k_m},\cdot)$ on
$L^2([0,\tau];\R^{2n})$.
The latter implies that each $u_m-\bar{w}$ is also a critical point of  $\mathcal{L}_{\lambda_{k_m}}^\bot$.
Since all $v_k$ are $\R$-distinct, any two of all $w_k$ belong to different $\R$-orbits. It follows that all $u_m-\bar{w}$ are distinct each other.
These show that $(\mu, 0)$ is a bifurcation point of $\nabla\mathcal{L}^\bot_\lambda(u)=0$ in $\Lambda\times H^\bot$.
By Theorem~\ref{th:A.10} (\cite[Theorem~3.1]{Lu8}) we get $m^0(\mathcal{L}^\bot_{\mu}, 0)>0$.
From (\ref{e:Nullity++}) it follows that
 $\dim{\rm Ker}(\gamma_\mu(\tau)-M)>1$.
\end{proof}

\begin{proof}[\bf Proof of Theorem~\ref{th:bif-suffict1-orbit}]
Follow the notations above. In the proof of Theorem~\ref{th:bif-ness}(II) we  proved that
$\{\mathcal{L}_\lambda(\cdot)=\psi_\kappa(\lambda, \bar{w}+\cdot)\,|\,
\lambda\in \Lambda\}$  satisfies the conditions of Theorem~\ref{th:A.9}
on  $H=L^2([0,\tau];\R^{2n})$ and $X=C^0_M([0,\tau];\R^{2n})$.
By Claim~\ref{cl:contBot}, Proposition~\ref{prop:hypA.5-bot} and
the proof of Theorem~\ref{th:bif-ness-orbit},
the conditions of Theorem~\ref{th:A.9} are also satisfied for
 $\{(H^\bot, X^\bot, \mathcal{L}^\bot_\lambda, \mathfrak{A}_\lambda, \mathfrak{B}_\lambda)\,|\,\lambda\in\Lambda\}$.

Note that using (\ref{e:MorseIndex++})-(\ref{e:Nullity++}) we can translate the assumptions (a)-(c) of Theorem~\ref{th:bif-suffict1-orbit}
as follows:
\begin{description}
\item[(${\rm a}^\prime$)]  The orbit ${\cal O}:=\R\ast \bar{w}$ is an embedded circle
(i.e., $\mathbb{R}_{\bar{w}}$ is an infinite cyclic subgroup of $\R$ with generator $p>0$).
\item[(${\rm b}^\prime$)]
$m^0(\mathcal{L}^\bot_{\mu}, 0)\ge 1$,  for each $k\in\mathbb{N}$, $\lambda_k^-\ne\lambda_k^+$,
$$
[m^-(\mathcal{L}^\bot_{\lambda_k^-}, 0), m^-(\mathcal{L}^\bot_{\lambda_k^-}, 0)+m^0(\mathcal{L}^\bot_{\lambda_k^-}, 0)]\cap[m^-(\mathcal{L}^\bot_{\lambda_k^+}, 0), m^-(\mathcal{L}^\bot_{\lambda_k^+}, 0)+
m^0(\mathcal{L}^\bot_{\lambda_k^+}, 0)]=\emptyset,
$$
 and either $m^0(\mathcal{L}^\bot_{\lambda_k^+}, 0)=0$ or $m^0(\mathcal{L}^\bot_{\lambda_k^-}, 0)=0$.
 \item[(${\rm c}^\prime$)] For any critical point $w$ of $\psi_\kappa(\mu, \cdot)$ which sits in $X$,
 if there exists a sequence $(s_k)$ of reals such that  $s_k\cdot w$
  converges to  $\bar{w}$ in $X$, then $w$ is periodic, and so $\mathbb{R}\ast w$ is closed.
  (Clearly, this holds if $M^l=id_{\mathbb{R}^{2n}}$ for some $l\in\mathbb{N}$.)
 \end{description}

By (${\rm b}^\prime$) we may use Theorem~\ref{th:A.9} to get $m^0(\mathcal{L}^\bot_{\mu}, 0)>0$
and a sequence $\{(\lambda_k, u_k)\}_{k\ge 1}$  in $\hat\Lambda\times H^\bot$
   converging to $(\mu, 0)$, where $\hat{\Lambda}:=\{\mu,\lambda^+_k, \lambda^-_k\,|\,k\in\mathbb{N}\}$,
   such that each $u_k$  is a nonzero solution of $\nabla\mathcal{L}^\bot_{\lambda_k}(u)=0$, $k=1,2,\cdots$.
 Clearly, we can assume $u_k\in B_{H^\bot}(0,\varepsilon)\;\forall k\in\mathbb{N}$.
 By Proposition~\ref{prop:solution}
 \begin{equation}\label{e:Pro4.5}
 \hbox{$w_k:=\bar{w}+u_k\in B_{H}^\bot(\bar{w},\varepsilon)$ is $C^1$, satisfies $d(\psi_{\kappa})_{\lambda_k}(w_k)=0$ and $\|w_k-\bar{w}\|_{C^1}\to 0$.}
 \end{equation}

 From now on we \textsf{use (${\rm a}^\prime$), i.e.,  the orbit $\mathcal{O}=\R\ast \bar{w}$  is a $C^1$ embedded circle}.
 Since it  transversely intersects with  $B_{H}^\bot(\bar{w},\varepsilon)$  at $\bar{w}$,
 by shrinking $\varepsilon>0$ we can assume
 $$
 \mathcal{O}\cap B_{H}^\bot(\bar{w},\varepsilon)=\{\bar{w}\}.
 $$
 It follows from this and (\ref{e:Pro4.5}) that there exists  $k_0>0$ such that  for each $k>k_0$,
 $\mathbb{R}\ast w_k$ is an immersed $C^1$ submanifold which
 transversely intersects  $B_{H}^\bot(\bar{w},\varepsilon)$ at $w_k$. Hence $\mathbb{R}\ast w_k\ne\mathcal{O}$ for any $k>k_0$. (Otherwise, $\mathcal{O}$ and $B_{H}^\bot(\bar{w},\varepsilon)$
 have at least two distinct intersection points $w_k$ and
 $\bar{w}$ in $B_{H}^\bot(\bar{w},\varepsilon)$.)

Finally, we conclude that $\{\mathbb{R}\ast w_k\,|\, k\in\mathbb{N}\}$ is an infinite set.
 (Thus $(w_k)$ has a subsequence which only consists of $\R$-distinct elements,
  also denoted by $(w_{k})$ for brevity.
 For each $k=1,2,\cdots$,  let $v_k$ be defined by $-(\Lambda_{M,\tau,\kappa I_{2n}})^{-1}w_{k}$ via
         (\ref{e:extend}). Then $\{(\lambda_k, v_k)\}_{k\ge 1}$ is the required one. The proof is completed.)
 Otherwise,  passing to a subsequence we may assume that all $w_k$ belong to an $\R$-orbit, i.e.,
  $w_k=s_k\ast \hat{w}$ for some $s_k\in\mathbb{R}$,
 where $\hat{w}\in X$ satisfies  $d(\psi_{\kappa})_{\lambda_k}(\hat{w})=d(\psi_{\kappa})_{\lambda_k}({w}_k)=0$ for all $k\in\mathbb{N}$.
Since $\Lambda\times X\ni (\lambda,v)\mapsto \nabla_v\psi_\kappa(\lambda, v)\in X$ is continuous, and $\lambda_k\to \mu$,
we obtain $d(\psi_{\kappa})_\mu(\hat{w})=0$.
 Clearly, $\bar{w}$ sits in the intersection of $\mathcal{O}$
 and the closure of $\mathbb{R}\ast \hat{w}$. By the assumption
 (${\rm c}^\prime$), $\mathbb{R}\ast \hat{w}$ is closed and
 so equal to $\mathcal{O}$, which contradicts the claim that
 $\mathcal{O}\ne\mathbb{R}\ast w_k=\mathbb{R}\ast \hat{w}$ for any $k>k_0$
  in the last paragraph.
\end{proof}

\begin{proof}[\bf Proof of Theorem~\ref{th:bif-existence-orbit}]
The first two paragraphs in the proof of Theorem~\ref{th:bif-suffict1-orbit} are still valid
provided we replace $\Lambda$ and $\hat\Lambda$ by $\alpha([0,1])$, and
(${\rm b}^\prime$) by
\begin{enumerate}
\item[(${\rm d}^\prime$)]
$[m^-(\mathcal{L}^\bot_{\lambda^-}, 0), m^-(\mathcal{L}^\bot_{\lambda^-}, 0)
+m^0(\mathcal{L}^\bot_{\lambda^-}, 0)]\cap[m^-(\mathcal{L}^\bot_{\lambda^+}, 0), m^-(\mathcal{L}^\bot_{\lambda^+}, 0)+m^0(\mathcal{L}^\bot_{\lambda^+}, 0)]=\emptyset$,
and either $m^0(\mathcal{L}^\bot_{\lambda^+}, 0)=0$ or $m^0(\mathcal{L}^\bot_{\lambda^-}, 0)=0$.
  \end{enumerate}
By Theorem~\ref{th:A.9+} there exists a sequence $\{(t_k, u_k)\}_{k\ge 1}$  in $[0,1]\times H^\bot$
   converging to $(\bar{t}, 0)$, such that each $u_k$
  is a nonzero solution of $\nabla\mathcal{L}^\bot_{\alpha(t_k)}(u)=0$, $k=1,2,\cdots$.
Moreover, $\alpha(\bar{t})\ne \lambda^+$ (resp. $\mu\ne \lambda^-$)
 if $m^0(\mathcal{L}^\bot_{\lambda^+}, 0)=0$ (resp. $m^0(\mathcal{L}^\bot_{\lambda^-}, 0)=0$).
 As in the proof of Theorem~\ref{th:bif-suffict1-orbit}
we may assume $u_k\in B_{H^\bot}(0,\varepsilon)\;\forall k\in\mathbb{N}$ and have
 (\ref{e:Pro4.5}). Letting $w_k:=\bar{w}+u_k$ for each $k\in\mathbb{N}$
  and repeating other arguments in the proof of Theorem~\ref{th:bif-suffict1-orbit}
   the required results can be obtained.
 \end{proof}

\begin{proof}[\bf Proof of Theorem~\ref{th:bif-suffict-orbit}]
Part (I) follows from Theorem~\ref{th:bif-suffict1-orbit} directly.
In order to prove parts (II) and (III),
following the first paragraph of the proof of Theorem~\ref{th:bif-suffict1-orbit} and
  comparing conditions in Theorem~\ref{th:A.9}
 it is easily seen  that $(\mathcal{L}^\bot_\lambda, H^\bot, X^\bot)$
 with $\lambda\in \Lambda$  satisfy
the conditions in Theorem~\ref{th:A.9}
with $\lambda^\ast=\mu$.
As in the proof of Theorem~\ref{th:bif-suffict} we may use
 (\ref{e:MorseIndex++})-(\ref{e:Nullity++}) and the assumptions about Morse indexes in Theorem~\ref{th:bif-suffict-orbit}  to check that
 $m^0(\mathcal{L}^\bot_\mu, 0)=\nu_{\tau}(\gamma_\mu)-1>0$, and
 $m^-(\mathcal{L}^\bot_\lambda, 0)$ takes values $m^-(\mathcal{L}^\bot_\mu, 0)=i_{\tau}(\gamma_\mu)$ and $
 m^-(\mathcal{L}^\bot_\mu, 0)+ m^0(\mathcal{L}^\bot_\mu, 0)=i_{\tau}(\gamma_\mu)+ \nu_{\tau}(\gamma_\mu)-1$
  as $\lambda\in\mathbb{R}$ varies in  two deleted half neighborhoods  of $\mu$.

Since $M=I_{2n}$, the $\mathbb{R}$-action on $H$ is actually a $S^1=\mathbb{R}/(\tau\mathbb{Z})$ action,
$$
[\theta]\diamond x=(\theta+k\tau)\ast x,\quad\forall x\in H,\;[\theta]=\theta+\tau\mathbb{Z}\in S^1,\;k\in\mathbb{Z}.
$$
This is a continuous action on $H$ (resp. $X$) via Hilbert (resp. Banach) isometry isomorphisms.

\begin{claim}\label{cl:isotryGroup}
The isotropy group of the above $S^1$-action at $\bar{w}$ is
$$
S^1_{\bar{w}}=\left\{\left[\frac{l\tau}{p}\right]\,\Bigg|\, l=0,\cdots,p-1\right\}\cong \mathbb{Z}_p:=\mathbb{Z}/p\mathbb{Z},
$$
and $[\theta]\diamond\dot{\bar{w}}=\dot{\bar{w}}$ for any $[\theta]\in S^1_{\bar{w}}$.
\end{claim}
\begin{proof}[\bf Proof]
Recall that $\bar{u}=\bar{v}|_{[0,\tau]}$ and $\bar{w}=-\Lambda_{I_{2n},\tau,\kappa I_{2n}}\bar{u}$.
Suppose $[\theta]\diamond \bar{w}=\bar{w}$, i.e.,  $(\theta+k\tau)\ast \bar{w}=\bar{w}$.
Note that $\Lambda_{I_{2n},\tau,\kappa I_{2n}}$ is the equivariant isomorphism.
Hence  $(\theta+k\tau)\ast \bar{u}=\bar{u}$. That is,
$$
(\bar{u})^{I_{2n}}(\theta+k\tau+t)=\bar{u}(t),\;\forall 0\le t\le\tau.
$$
Since $(\bar{u})^{I_{2n}}=\bar{v}$ and $\bar{v}$ has the minimal period $\tau/p$, this means that
$\bar{v}(\theta+ t)=\bar{v}(t)\;\forall 0\le t\le\tau$ and so $\bar{v}(\theta+ t)=\bar{v}(t)\;\forall t\in\mathbb{R}$. It follows that
$$
\theta\in \bigcup^{p-1}_{l=0}\left(\mathbb{Z}\tau+ \frac{l\tau}{p}\mathbb{Z}\right).
$$
Moreover, it is easily checked that  such a $\theta$ satisfies $[\theta]\diamond \bar{w}=\bar{w}$.
The first claim follows.

Since $\bar{u}$ is $C^2$, and $\bar{w}=-\Lambda_{M,\tau,\kappa I_{2n}}(\bar{u})=-J\dot{\bar{u}}-\kappa \bar{u}$,
we deduce $\dot{\bar{w}}=-\Lambda_{M,\tau,\kappa I_{2n}}(\dot{\bar{u}})$.
Note that $\tau/p$ is also a period of $\dot{\bar{v}}$ and $\dot{\bar{u}}=\dot{\bar{v}}|_{[0,\tau]}$.
It must hold that $[\theta]\diamond \dot{\bar{w}}=\dot{\bar{w}}$ for any $[\theta]\in S^1_{\bar{w}}$.
\end{proof}

Recall that $H^\bot$ is the orthogonal complementary of $\dot{\bar{w}}$ in  $H=L^2([0,\tau];\R^{2n})$ and
$X^\bot=X\cap H^\bot$. By the second conclusion in Claim~\ref{cl:isotryGroup}
 spaces $H^\bot, X^\bot$ and each functional $\mathcal{L}^\bot_{\lambda}$ are $S^1_{\bar{w}}$-invariant.

Since $\{[\theta]\diamond \bar{u}\,|\,[\theta]\in S^1\}=\{\theta\ast \bar{u}\,|\,\theta\in\R\}$ is a $C^2$-embedded circle,
by Claim~\ref{cl:5orbit} the orbit $\mathcal{O}:=\{[\theta]\diamond \bar{w}\,|\,[\theta]\in S^1\}$
is  a $C^2$-embedded circle in $L^2([0,\tau];\R^{2n})$ and
$T_{\bar{w}}\mathcal{O}=\mathbb{R}\dot{\bar{w}}$.
Let  $\pi:N{\cal O}\to{\cal O}$ be the normal bundle
of ${\cal O}$ in $H$. It is a $C^{1}$
Hilbert vector bundle over ${\cal O}$ and its fibre $N{\cal O}_{\bar{w}}$ at $\bar{w}$ is equal to $H^\bot$.
Let $N{\cal O}(\varepsilon):=\{(w,x)\in N{\cal O}\,|\,\|x\|_{2}<\varepsilon\}$ and define
\begin{equation*}
{\rm EXP}:N{\cal O}(\varepsilon)\to H,\;(w,x)\mapsto w+x.
\end{equation*}
Clearly, ${\rm EXP}$ is equivariant, i.e.,
$[\theta]\diamond({\rm EXP}(w,x))={\rm EXP}([\theta]\diamond w, [\theta]\diamond x)$ for
all $[\theta]\in S^1$. By shrinking   $\varepsilon>0$,  ${\rm EXP}$ gives rise to a $C^{1}$
diffeomorphism  $\digamma$ from $N{\cal O}(\varepsilon)$ to
an open neighborhood ${\cal N}({\cal O}, \varepsilon)$ of ${\cal O}$ in $H$.
Let $x,y\in N{\cal O}(\varepsilon)_{\bar{w}}=B_{H}^\bot(\bar{w},\varepsilon)$ such that
$$
[\theta]\diamond(\bar{w}+x)=\bar{w}+y\quad\hbox{for some}\; [\theta]\in S^1=\mathbb{R}/(\tau\mathbb{Z}).
$$
Then $[\theta]\diamond \bar{w}+ [\theta]\diamond x=\bar{w}+y$, i.e.,
${\rm EXP}([\theta]\diamond \bar{w}, [\theta]\diamond x)={\rm EXP}(\bar{w}, y)$.
It follows that $[\theta]\diamond \bar{w}=\bar{w}$ and $[\theta]\diamond x=y$.
The former shows  $[\theta]\in S^1_{\bar{w}}$, and the latter implies that $x$ and $y$
sit in the same $S^1_{\bar{w}}$-orbit. Therefore we have proved:

\begin{claim}\label{cl:distinctOrbit}
If $x, y\in N{\cal O}(\varepsilon)_{\bar{w}}=B_{H}^\bot(\bar{w},\varepsilon)$ belong to
distinct $S^1_{\bar{w}}$-orbits, then $S^1\diamond(\bar{w}+x)\ne S^1\diamond(\bar{w}+y)$.
\end{claim}

\begin{claim}\label{cl:fixedSet}
The fixed point set of $S^1_{\bar{w}}$-action on
${\rm Ker}(\mathfrak{B}_\lambda(0))$,
$$
\{x\in {\rm Ker}(\mathfrak{B}_\lambda(0))\,|\, [\theta]\diamond x=x\;\forall [\theta]\in S^1_{\bar{w}}\},
$$
is a linear subspace of dimension $\nu_{\tau/p}(\gamma_\mu)-1$.
\end{claim}
\begin{proof}[\bf Proof]
 Since $B_\lambda(\bar{w})\dot{\bar{w}}=0$, by (\ref{e:gradient44})
we have
$$
{\rm Ker}(\mathfrak{B}_\lambda(0))=\{x\in H^\bot\,|\, B_\lambda(\bar{w})x=0\}.
$$
Note that $[\tau/p]$ is a generator of $S^1_{\bar{w}}$. $x\in {\rm Ker}(\mathfrak{B}_\lambda(0))$
is a fixed point of $S^1_{\bar{w}}$-action if and only if $[\tau/p]\diamond x=x$, i.e., $(\tau/p)\ast x=x$.
By (\ref{e:p-gradient**}) and (\ref{e:p-self-adjoint**}),
\begin{eqnarray*}
&&(\Lambda_{I_{2n},\tau,{\kappa I_{2n}}})^{-1}\bar{w}+ \nabla(H_\kappa^\ast)_\lambda(\bar{w}(\cdot))=A_{I_{2n},\tau,\kappa}\bar{w}+ \nabla(H_\kappa^\ast)_\lambda(\bar{w}(\cdot))=0,\\
&&(\Lambda_{I_{2n},\tau,{\kappa I_{2n}}})^{-1}x+(H_\kappa^\ast)''_{\lambda}(\bar{w}(\cdot))x
=A_{I_{2n},\tau,\kappa}x+ (H_\kappa^\ast)''_{\lambda}(\bar{w}(\cdot))x=0.
\end{eqnarray*}
Let $\bar{u}=-(\Lambda_{I_{2n},\tau,{\kappa I_{2n}}})^{-1}\bar{w}$. Then $\bar{u}=\nabla(H_\kappa^\ast)_\lambda(\bar{w}(\cdot))$
and so $\bar{w}=\nabla(H_\kappa)_\lambda(\bar{u}(\cdot))$ and
$$
I_{2n}=(H_\kappa^\ast)''_{\lambda}(\bar{w}(\cdot))(H_\kappa)''_{\lambda}(\bar{u}(\cdot)).
$$
It follows from these that $y:=(\Lambda_{I_{2n},\tau,{\kappa I_{2n}}})^{-1}x$ satisfies
$$
\dot{y}(t)=JH_\lambda''(\bar{u}(t))y(t),\quad 0\le t\le\tau.
$$
Moreover, $(\tau/p)\ast y=y$ implies $y^{I_{2n}}$ has a period $\tau/p$.
Hence the desired claim follows.
\end{proof}

\noindent{\bf Proof of (II)}. Since $m^0(\mathcal{L}^\bot_\mu, 0)=\nu_{\tau}(\gamma_\mu)-1>1$,
Theorem~\ref{th:A.11} (\cite[Theorem~3.6]{Lu10}) implies that one of the following alternatives occurs:
\begin{description}
\item[(${\rm i}'$)] $(\mu,0)$ is not an isolated solution  in  $\{\mu\}\times H^\bot$ of the equation
$\mathfrak{A}_\lambda(x)=0$.

\item[(${\rm ii}'$)]  For every $\lambda\in\Lambda$ near $\mu$ there is a nontrivial solution $x_\lambda$ of
$\mathfrak{A}_\lambda(x)=0$ in $X^\bot$,
which  converges to $0$ in $X^\bot$ as $\lambda\to \mu$.

\item[(${\rm iii}'$)] For any  neighborhood $\mathcal{N}$ of $0$ in $X^\bot$, which may be assumed to be $S^1_{\bar{w}}$-invariant,
there is a one-sided  neighborhood $\Lambda^0$ of $\mu$ in $\Lambda$ such that
for any $\lambda\in\Lambda^0\setminus\{\mu\}$, $\mathfrak{A}_\lambda(x)=0$
has at least two nontrivial solutions $x_\lambda^1$ and $x_\lambda^2$ in $\mathcal{N}$ with different energy
if $\mathfrak{A}_\lambda(x)=0$ has only finitely many nontrivial solutions in $\mathcal{N}$.
\end{description}

By (${\rm i}'$) we have a sequence $(x_k)\subset H^\bot\setminus\{0\}$
converging to $0$ in $H^\bot$ such that $\nabla\mathcal{L}^\bot_\mu(x_k)=0$ for each $k$.
Since $S^1_{\bar{w}}$ is a finite group, by passing to a subsequence (if necessary)
 we can assume that all $S^1_{\bar{w}}\diamond x_k$ are pairwise distinct.
By Proposition~\ref{prop:solution} and Claim~\ref{cl:distinctOrbit}, all $S^1\diamond(\bar{w}+x_k)=\mathbb{R}\ast(\bar{w}+x_k)$
are distinct critical orbits of $\psi_\kappa(\mu, \cdot)$ and $\bar{w}+x_k\to \bar{w}$ in $H$.
For each $k=1,2,\cdots$, set
$$
u_k:=-(\Lambda_{I_{2n},\tau,\kappa I_{2n}})^{-1}(\bar{w}+x_k)\quad\hbox{and}\quad v_k:=(u_k)^{I_{2n}}.
$$
Then $\|u_k-\bar{v}|_{[0,\tau]}\|_{1,2}\to 0$, and therefore $\|u_k-\bar{v}|_{[0,\tau]}\|_{C^1}\to 0$ by Proposition~\ref{prop:threeBifu}.
Hence $(v_k)$ satisfy (II.1).

For every $\lambda\in\Lambda\setminus\{0\}$, let $x_\lambda$ be as in (${\rm ii}'$).
Set $u_\lambda:=-(\Lambda_{I_{2n},\tau,\kappa I_{2n}})^{-1}(\bar{w}+x_\lambda)$ and $v_\lambda:=(u_\lambda)^{I_{2n}}$.
As above these $v_\lambda$ satisfy (II.2).

By (${\rm iii}'$), for any $\lambda\in\Lambda^0\setminus\{\mu\}$, $\mathfrak{A}_\lambda(x)=0$ has \\
$\bullet$ either  infinitely many distinct nontrivial  solutions in
$\mathcal{N}$, and therefore infinitely many distinct nontrivial $S^1_{\bar{w}}$-orbit of solutions in
$\mathcal{N}$ (because of the finiteness of  $S^1_{\bar{w}}$), $S^1_{\bar{w}}\diamond \bar{x}_\lambda^k$, $k=1,2,\cdots$, \\
$\bullet$ or at least two nontrivial $S^1_{\bar{w}}$-orbit of solutions in $\mathcal{N}$,
 $S^1_{\bar{w}}\diamond x^1_\lambda$ and $S^1_{\bar{w}}\diamond x^2_\lambda$, such that $\psi_{\kappa}(\lambda,x^1_\lambda)\ne \psi_{\kappa}(\lambda, x^2_\lambda)$.

As above we define  $\bar{u}^k_\lambda:=-(\Lambda_{I_{2n},\tau,\kappa I_{2n}})^{-1}(\bar{w}+\bar{x}^k_\lambda)$
 and $\bar{v}^k_\lambda:=(\bar{u}^k_\lambda)^{I_{2n}}$, $k=1,2,\cdots$, and
 $u^i_\lambda:=-(\Lambda_{I_{2n},\tau,\kappa I_{2n}})^{-1}(\bar{w}+x^i_\lambda)$ and $v^i_\lambda:=(u^i_\lambda)^{I_{2n}}$, $i=1,2$.
Then $\bar{v}^k_\lambda$ and $v^i_\lambda$ satisfy (II.3).\\

\noindent{\bf Proof of (III)}.
Since $\nu_{\tau/p}(\gamma_\mu)=1$, by Claim~\ref{cl:fixedSet}
the fixed point set of $S^1_{\bar{w}}$-action on ${\rm Ker}(\mathfrak{B}_\lambda(0)$
is equal to $\{0\}$. Using Theorem~\ref{th:A.12} (\cite[Theorem~3.7]{Lu10}) and \cite[Remark~3.9]{Lu10} we obtain:

$\bullet$ If $p=2$,  one of (${\rm i}^\prime$) and the following (${\rm iv}^\prime$) occurs:
  \begin{enumerate}
\item[(${\rm iv}^\prime$)] There exist left and right  neighborhoods $\Lambda^-$ and $\Lambda^+$ of $\mu$ in $\Lambda$
and integers $n^+, n^-\ge 0$, such that $n^++n^-\ge \nu_{\tau}(\gamma_\mu)$,
and for $\lambda\in\Lambda^-\setminus\{\mu\}$ (resp. $\lambda\in\Lambda^+\setminus\{\mu\}$),
$\mathfrak{A}_\lambda(x)=0$   has at least $n^-$ (resp. $n^+$)
nontrivial $S^1_{\bar{w}}$-orbit of solutions of in $X^\bot$,
 $S^1_{\bar{w}}\diamond x_\lambda^i$, $i=1,\cdots,n^-$ (resp. $n^+$),
which  converge to  $0$ in $X^\bot$ as $\lambda\to \mu$.
\end{enumerate}

$\bullet$ If $p>2$ is a prime,  one of (${\rm i}^\prime$) and the following (${\rm v}^\prime$)  occurs:
 \begin{enumerate}
\item[(${\rm v}^\prime$)] The conclusions obtained by
replacing $\nu_{\tau}(\gamma_\mu)$ with $\nu_{\tau}(\gamma_\mu)/2$ in
 (${\rm iv}^\prime$).
\end{enumerate}
 As above, the desired conclusions follow from these and Claim~\ref{cl:distinctOrbit}.
 \end{proof}

\section{Proofs of Theorems~\ref{th:bif-nessbrake},\ref{th:bif-suffictbrake},\ref{th:bif-per3brake} and
Corollaries~\ref{cor:bif-per2brake},\ref{cor:bif-per4brake},\ref{cor:bif-per5brake}
}\label{sec:brake}\setcounter{equation}{0}

Let $\textsf{L}_\tau$ be the Hilbert subspace of $L^{2}(S_\tau; \R^{2n})$ as in (\ref{e:Lbrake}).
Below (\ref{e:Lbrake}) we have showed that the operator $\Lambda_{I_{2n},\tau, 0}$ on $L^2(S_\tau; \mathbb{R}^{2n})$
defined by (\ref{e:LambdaKA})  with domain $W^{1,2}(S_\tau; \mathbb{R}^{2n})$
restricts to a closed linear and self-adjoint operator
$\tilde{\Lambda}_{I_{2n},\tau, 0}$
on $\textsf{L}_\tau$ with domain $\textsf{W}_\tau$, and
$$
\sigma(\tilde{\Lambda}_{I_{2n},\tau, 0})=\frac{2\pi}{\tau}\mathbb{Z}
$$
consists of only eigenvalues, and each eigenvalues has multiplicity $n$. Then
for each real $\kappa\notin \frac{2\pi}{\tau}\mathbb{Z}$,
\begin{equation}\label{e:isomorphism}
\tilde{\Lambda}_{I_{2n},\tau, \kappa I_{2n}}: \textsf{W}_\tau\to \textsf{L}_\tau,\;v\mapsto\tilde{\Lambda}_{I_{2n},\tau, 0}v+ \kappa v=
J\frac{d}{dt}v+\kappa v
\end{equation}
is a Banach space isomorphism.

Solutions of (\ref{e:Pbrake}) are critical points of the functional
$$
\Phi_\lambda(v)=\int^\tau_0\left[\frac{1}{2}(J\dot{v}(t),v(t))_{\mathbb{R}^{2n}}+
 H(\lambda, t,{v}(t))\right]dt
$$
on the Hilbert space $\textsf{W}_\tau$.

Note that the function $\bar{H}:\Lambda\times\R\times{\R}^{2n}\to\R$ by
$$
\bar{H}(\lambda,t,z)={H}(\lambda,t,z+u_\lambda(t))-(z, \nabla_z{H}(\lambda,t, u_\lambda(t)))_{\mathbb{R}^{2n}}
$$
 satisfies Assumption~\ref{ass:brake},
 and for each $\lambda\in\Lambda$, $\bar{v}\equiv 0\in\mathbb{R}^{2n}$ satisfies
 \begin{equation}\label{e:Pbrake.1}
\dot{v}(t)=J\nabla_z \bar{H}(\lambda,t, v(t)),\quad v(t+\tau)=v(t)\quad\hbox{and}\quad v(-t)=Nv(t)\;\forall t\in\R.
\end{equation}

As done above the proof of Theorem~\ref{th:bif-ness}(I) in Section~\ref{sec:HamBif},
replacing $\bar{H}$ by $\tilde{H}$
\textsf{in what follows we can assume that $H$
satisfies inequalities in (\ref{e:modifyH}) and (\ref{e:modifyH1})
for all $(\lambda,t, z)\in \Lambda\times \mathbb{R}\times \mathbb{R}^{2n}$, and
  only need to prove Theorems~\ref{th:bif-nessbrake}--\ref{th:bif-suffictbrake}
in the case where $v_\lambda\equiv 0$ for all} $\lambda\in\Lambda$.

Then we may choose $\kappa\notin \frac{2\pi}{\tau}\mathbb{Z}$,
 $c_i>0$, $i=1,2,3$   such that each
 $$
 H_\kappa(\lambda, t, z):=H(\lambda, t, z)- \frac{\kappa}{2}|z|^2
 $$
 satisfies  (ii)-(iii) in the proof of Theorem~\ref{th:bif-ness}(I).
Consider the dual action  $\Psi_\kappa(\lambda,\cdot):
\textsf{W}_\tau\to\R$ defined by (\ref{e:HAction*}).
 We obtain a family of $C^1$ and twice G\^ateaux-differentiable  functionals
$$
\psi_\kappa(\lambda, \cdot)=
\Psi_\kappa(\lambda, \cdot)\circ(-\tilde{\Lambda}_{I_{2n},\tau,\kappa I_{2n}})^{-1}:
\textsf{L}_\tau\to\R
$$
given by
\begin{equation}\label{e:HActionbrake}
\psi_\kappa(\lambda, v)=\int^{\tau}_0\left[\frac{1}{2}\left(v(t), ((\tilde{\Lambda}_{I_{2n},\tau, \kappa I_{2n}})^{-1}v)(t)\right)_{\mathbb{R}^{2n}}
+ H_\kappa^\ast(\lambda, t, v(t))\right]dt,
\end{equation}
such that $v\in\textsf{W}_\tau$ is a critical point of $\Psi_\kappa(\lambda, \cdot)$ if and only if
$w:=-\tilde{\Lambda}_{I_{2n},\tau, \kappa I_{2n}}v$ is a critical point of $\psi_\kappa(\lambda, \cdot)$.
In particular, $(\mu, 0)\in\Lambda\times \textsf{W}_\tau$ is a bifurcation point of
the problem  (\ref{e:Pbrake}) if and only if  $(\mu, 0)\in\Lambda\times \textsf{L}_\tau$
 is a bifurcation point of $\nabla_w\psi_\kappa(\lambda, w)=0$.

Let $\psi_{\kappa,\lambda}(v)=\psi_\kappa(\lambda, v)$.
As in Remark~\ref{rm:MWTh.2.5} we can compute
\begin{eqnarray*}
({\psi}_{\kappa,\lambda})''(0)[\xi,\eta]&=&\int^\tau_0\left[\left((\tilde{\Lambda}_{I_{2n},\tau,\kappa I_{2n}})^{-1}\xi)(t), \eta(t)\right)_{\mathbb{R}^{2n}}+
\left(\nabla^2_zH_{\kappa}(\lambda,t;0)\xi(t),\eta(t)\right)_{\mathbb{R}^{2n}}\right]dt\\
&=&\int^\tau_0\left[\left((\tilde{\Lambda}_{I_{2n},\tau,\kappa I_{2n}})^{-1}\xi)(t)+
[\nabla^2_zH_{\kappa}(\lambda,t, 0)]^{-1}\xi(t),\eta(t)\right)_{\mathbb{R}^{2n}}\right]dt\\
&=&\textsf{Q}_{B_\lambda,\kappa I_{2n}}(\xi,\eta),\quad\forall \xi,\eta\in\textsf{{L}}_\tau,
\end{eqnarray*}
where $\textsf{Q}_{\textsf{B}_\lambda,\kappa I_{2n}}$ is given by (\ref{e:Bquadratic}) with $\textsf{B}_\lambda(t)=\nabla^2_zH(\lambda,t, 0)$.
By Theorem~\ref{th:brakeIndex} we obtain
 \begin{eqnarray}
   m^-({\psi}_{\kappa,\lambda}, 0)&=&\mu_{1,\tau}(\gamma_\lambda)-n[\kappa\tau/(2\pi)],\label{e:psi-MorseIndexbrake}\\
  m^0({\psi}_{\kappa,\lambda}, 0)&=&\nu_{1,\tau}(\gamma_\lambda).\label{e:psi-Nullitybrake}
 \end{eqnarray}

Let $\iota:\textsf{W}_\tau\to \textsf{L}_\tau$ be the inclusion. Then
\begin{equation}\label{e:Lambda-inverseBrake}
\tilde{A}_{I_{2n},\tau,\kappa}:\textsf{L}_\tau\to \textsf{L}_\tau,\;v\mapsto  \iota\circ(\tilde{\Lambda}_{I_{2n},\tau,\kappa I_{2n}})^{-1}v,
\end{equation}
 is a compact self-adjoint operator, and
 $\psi_\kappa(\lambda,\cdot)$ has the gradient
\begin{equation}\label{e:p-gradientBrake}
\nabla_v\psi_\kappa(\lambda, v)=\tilde{A}_{I_{2n},\tau,\kappa}v+ \nabla_zH_\kappa^\ast(\lambda, \cdot; v(\cdot)).
\end{equation}
Moreover,  $\textsf{L}_\tau\ni v\mapsto \nabla_v\psi_\kappa(\lambda, v)\in \textsf{L}_\tau$
has a G\^ateaux derivative
\begin{equation}\label{e:p-self-adjointBrake}
\textsf{B}_\lambda(v):=D_v\nabla_v\psi_\kappa(\lambda, v)=\tilde{A}_{I_{2n},\tau,\kappa}+ \nabla^2_zH_\kappa^\ast(\lambda, \cdot; v(\cdot))
\in \mathscr{L}_s(\textsf{L}_\tau).
\end{equation}
Let $P_\lambda(v)=\nabla^2_zH_\kappa^\ast(\lambda, \cdot; v(\cdot))$
and $Q_\lambda(v)=\tilde{A}_{I_{2n},\tau,\kappa}$.

\begin{proof}[\bf Proof of Theorem~\ref{th:bif-nessbrake}(I)]
As in the proof of Theorem~\ref{th:bif-ness}
we can show that Theorem~\ref{th:A.10} (\cite[Theorem~3.1]{Lu8})  is applicable to
$H=X=\textsf{L}_\tau$ and $\mathcal{F}_\lambda(\cdot)=\psi_\kappa(\lambda, \cdot)$.
\end{proof}

Consider Banach subspaces of $C^{k}(S_\tau; \R^{2n})$, $\textsf{C}^k_\tau=C^{k}(S_\tau; \R^{2n})\cap\textsf{L}_\tau$,
$k=0,1,\cdots$.

\begin{proof}[\bf Proof of Theorem~\ref{th:bif-nessbrake}(II)]
Note that $\tilde{\Lambda}_{I_{2n},\tau, \kappa I_{2n}}$ gives rise to a Banach space isomorphism from
 $\textsf{C}^1_\tau$ to $\textsf{C}^0_\tau$, denoted by $\tilde{\Lambda}^c_{I_{2n},\tau, \kappa I_{2n}}$.
 As in the proof of Theorem~\ref{th:bif-ness}(II), using (\ref{e:psi-MorseIndexbrake}) and (\ref{e:psi-Nullitybrake})
we can  prove that $\mathcal{L}_\lambda(\cdot)=\psi_\kappa(\lambda, \cdot)$
with $\lambda\in \Lambda$  satisfies the conditions of Theorem~\ref{th:A.9}
with $\lambda^\ast=\mu$ and $H=\textsf{L}_\tau$ and $X=\textsf{C}^0_\tau$.
\end{proof}

\begin{proof}[\bf Proof of Theorem~\ref{th:bif-nessbrake}(III)]
Following the notations above,  using (\ref{e:psi-MorseIndexbrake}) and (\ref{e:psi-Nullitybrake})
we can  prove as in the proof of Theorem~\ref{th:bif-ness}(III) that $\mathcal{L}_\lambda(\cdot)=\psi_\kappa(\lambda, \cdot)$
with $\lambda\in \Lambda$  satisfies the conditions of Theorem~\ref{th:A.9+}
with $\lambda^\ast=\mu$ and $H=\textsf{L}_\tau$ and $X=\textsf{C}^0_\tau$.
\end{proof}

\begin{proof}[\bf Proof of Theorem~\ref{th:bif-suffictbrake}]
As in the proof of Theorem~\ref{th:bif-suffict} the first part follows from
Theorem~\ref{th:A.11} (\cite[Theorem~3.6]{Lu10} or \cite[Theorem~4.6]{Lu8}),
and others are obtained by
Theorem~\ref{th:A.12} (\cite[Theorem~3.7]{Lu10}) and \cite[Remark~3.9]{Lu10}.
\end{proof}

\begin{proof}[\bf Proof of Corollary~\ref{cor:bif-per2brake}]
As in the proof of Corollary~\ref{cor:necess-suffi} the desired conclusions  can be completed
using Theorems~\ref{th:brakeIndex},~\ref{th:brakeIndexMono} and the above arguments.
\end{proof}

\begin{proof}[\bf Proof of Theorem~\ref{th:bif-per3brake}]
{\it Step 1}.  The Lie group $G=\mathbb{Z}_2=\{[0],[1]\}$
 orthogonally  acts on $\textsf{L}_\tau$ and $\textsf{W}_\tau$  by
$[0]\cdot u=u$ and $[1]\cdot u=N\cdot u$, where $(N\cdot u)(t)=N(u(t))$.
Clearly, this action induces an isometric action on each space $\textsf{C}_\tau^k$.
The set of fixed points of this action, ${\rm Fix}_G$, is equal to
$\{u\in \textsf{L}_\tau\,|\, u(-t)=u(t), {\rm a.e.}\, t\in\mathbb{R}\}$.
Clearly, each functional $\psi_\kappa(\lambda, \cdot)$ defined by
(\ref{e:HActionbrake}) is invariant under this $G$-action.
As in the proof of Theorem~\ref{th:bif-per3} we can obtain that
the functional $\mathcal{L}_\lambda(\cdot)=\psi_\kappa(\lambda, \bar{w}+\cdot)$ with
$\bar{w}:=-\tilde{\Lambda}_{I_{2n},\tau, \kappa I_{2n}}\bar{v}$
satisfies conditions in Theorem~\ref{th:A.11} (\cite[Theorem~3.6]{Lu10} or \cite[Theorem~4.6]{Lu8})
with $H=\textsf{L}_\tau$ and $X=\textsf{C}_\tau^0$.

Since $\bar{v}$ is even, $\bar{v}(t)=\bar{v}(-t)=N\bar{v}(t)$ for all $t$, that is, $\bar{v}\in {\rm Fix}_G$.
Then $\bar{w}$ belongs to ${\rm Fix}_G$ and so
$\mathcal{L}_\lambda$ is $G$-invariant.
The assumption (b) implies that the fixed point set of the induced $G$-action on $H_\mu^0$ is $\{0\}$.
Combining the assumption (a) we can use Theorem~\ref{th:A.12} (\cite[Theorem~3.7]{Lu10} or \cite[Theorem~5.12]{Lu8}) to derive
that one of the claims (i) and (ii) in Theorem~\ref{th:bif-per3brake} holds.

{\it Step 2}. The Lie group
$G=\mathbb{Z}_2\times\mathbb{Z}_2=\{([0],[0]), ([0], [1]), ([1], [0]), ([1],[1])\}$
 orthogonally  acts on $\textsf{L}_\tau$ and $\textsf{W}_\tau$  by
$$
([0], [0])\cdot u=u, \quad ([0], [1])\cdot u=N\cdot u,\quad([1], [0])\cdot u=-u,\quad([1], [1])\cdot u=-(N\cdot u),
$$
and the fixed point set of the induced $G$-action on $H_\mu^0$ is $\{0\}$.
Clearly, each functional $\mathcal{L}_\lambda(\cdot)$ is invariant for this action, using
Theorem~\ref{th:A.12} (\cite[Theorem~3.7]{Lu10} or \cite[Theorem~5.12]{Lu8})
and \cite[Remark~3.9]{Lu10} (\cite[Remark~5.14]{Lu8})
 we may directly obtain either (iii) or (iv).
\end{proof}

\begin{proof}[\bf Proof of Corollary~\ref{cor:bif-per4brake}]
Suppose $\nu_{1,\tau}(\gamma_\mu)>0$.
As in the proof of Corollary~\ref{cor:necess-suffi},
by modifying values of $H_0$ and $\hat{H}$ outside a neighborhood $U$ of $0\in\mathbb{R}^{2n}$
we can choose $\kappa\in\mathbb{R}\setminus\frac{2\pi}{\tau}\mathbb{Z}$ and
 $c_i>0$, $i=1,2,3$   such that
$H_\kappa(\lambda, z):=H_0(z)+\lambda\hat{H}(z)- \frac{\kappa}{2}|z|^2$
satisfies
\begin{eqnarray*}
  && c_1I_{2n}\le (H_0)''(z)+ \lambda\hat{H}''(z)-\kappa I_{2n}\le c_2I_{2n}\quad\hbox{and}\\
  &&c_1|z|^2-c_3\le H_\kappa(\lambda, z)\le c_2|z|^2+c_3
   \end{eqnarray*}
for all $(\lambda, t, z)\in [\mu-1, \mu+1]\times [0,\tau]\times\R^{2n}$.
Let $\psi_{\kappa,\lambda}(u):=\psi_\kappa(\lambda, u)$ be given by (\ref{e:HActionbrake}). Then
$$
{\psi}''_{\kappa,\lambda}(0)[\xi,\eta]=\textsf{Q}_{B_\lambda,\kappa I_{2n}}(\xi,\eta)\quad\hbox{for all $\xi,\eta\in\textsf{{L}}_\tau$,}
$$
where $\textsf{Q}_{B_\lambda,\kappa I_{2n}}$ is given by (\ref{e:Bquadratic}) with
 $B_\lambda$ constant equal to $(H_0)''(0)+ \lambda\hat{H}''(0)$. Therefore
$$
m^-({\psi}_{\kappa,\lambda},0)=m^-(\textsf{Q}_{B_\lambda,\kappa I_{2n}})\quad\hbox{and}\quad
 m^0({\psi}_{\kappa,\lambda},0)=m^0(\textsf{Q}_{B_\lambda,\kappa I_{2n}}).
 $$

If $\hat{H}''(0)>0$ (resp. $\hat{H}''(0)<0$), by Theorem~\ref{th:brakeIndexMono} we deduce that
\begin{eqnarray*}\label{e:DongMorse1brake}
&&m^-(\textsf{Q}_{B_{\lambda_2},\kappa I_{2n}})\ge m^-(\textsf{Q}_{B_{\lambda_1},\kappa I_{2n}})+
 m^0(\textsf{Q}_{B_{\lambda_1},\kappa I_{2n}})
\quad\\
&&\hbox{(resp. $m^-(\textsf{Q}_{B_{\lambda_1},\kappa I_{2n}})\ge
m^-(\textsf{Q}_{B_{\lambda_2},\kappa I_{2n}})+  m^0(\textsf{Q}_{B_{\lambda_2},\kappa I_{2n}})$)}
\end{eqnarray*}
for any $\mu-1\le \lambda_1<\lambda_2\le\mu+1$. These imply that
$\{\lambda\in [\mu-1,\mu+1]\,|\,m^0(\textsf{Q}_{B_{\lambda},\kappa I_{2n}})>0\}$ is a finite set.
It follows from this and Theorem~\ref{th:brakeIndex}
that $\{\lambda\in\R\,|\, \nu_{1,\tau}(\gamma_\lambda)>0\}$  is a discrete set in $\R$.
The first claim is proved.

As in the proof of Corollary~\ref{cor:necess-suffi} we can use
Theorem~\ref{th:brakeIndexMono} to derive that
$$
m^-(\textsf{Q}_{B_\lambda,\kappa I_{2n}})=
\begin{cases}
 m^-(\textsf{Q}_{B_\mu,\kappa I_{2n}})\;&\forall\lambda\in [\mu-\rho,\mu),\\
m^-(\textsf{Q}_{B_\mu,\kappa I_{2n}})+
m^0(\textsf{Q}_{B_\mu,\kappa I_{2n}})\;&\forall\lambda\in (\mu, \mu+\rho]
\end{cases}
$$
if $\hat{H}''(0)>0$ and $\rho>0$ is small enough, and
$$
m^-(\textsf{Q}_{B_\lambda,\kappa I_{2n}})=
\begin{cases}
 m^-(\textsf{Q}_{B_\mu,\kappa I_{2n}})+
m^0(\textsf{Q}_{B_\mu,\kappa I_{2n}})\;&\forall\lambda\in [\mu-\rho,\mu),\\
m^-(\textsf{Q}_{B_\mu,\kappa I_{2n}})\;&\forall\lambda\in (\mu, \mu+\rho]
\end{cases}
$$
if $\hat{H}''(0)<0$ and $\rho>0$ is small enough.
These and Theorem~\ref{th:brakeIndex} lead to (\ref{e:case1brake}) and (\ref{e:case2brake}).
Other conclusions follow from Theorem~\ref{th:bif-per3brake}.
\end{proof}

\begin{proof}[\bf Proof of Corollary~\ref{cor:bif-per5brake}]
The first claim follows from Corollary~\ref{cor:bif-per4brake}.
Suppose  ${H}''(0)>0$ and $\mu\in\Delta(H)\cap(0, \infty)$.
Let $\gamma_\lambda(t)=\exp(\lambda tJ{H}''(0))$.
It follows from (\ref{e:case1brake}) that
$\mu_{1,1}(\gamma_\lambda)=\mu_{1,1}(\gamma_\mu)$ for $\lambda\le\mu$ close to $\mu$ and that
$\mu_{1,1}(\gamma_\lambda)=\mu_{1,1}(\gamma_\mu)+ \nu_{1,1}(\gamma_\mu)$ for $\lambda>\mu$ close to $\mu$.
By the third conclusion of Corollary~\ref{cor:bif-per4brake} with $H_0=0$ and $\hat{H}=H$ we obtain the desired claims.
The final result may follow from the fourth conclusion of
Corollary~\ref{cor:bif-per4brake} with $H_0=0$ and $\hat{H}=H$.
\end{proof}

\section{Proofs of Theorems~\ref{th:bif-nessHam}, \ref{th:bif-suffHam} and \ref{th:two-point}}\label{sec:twoLagr}\setcounter{equation}{0}

\subsection{Proofs of Theorems~\ref{th:bif-nessHam}, \ref{th:bif-suffHam}}\label{sec:twoLagr-1}

As in (\ref{e:ModifiedH}) we define
$\bar{H}(\lambda,t, z)={H}(\lambda,t, z+u_\lambda(t))-(z, \nabla_z{H}(\lambda,t, u_\lambda(t)))_{\mathbb{R}^{2n}}$
for $(\lambda,t, z)\in \Lambda\times [0, \tau]\times{\R}^{2n}$.
Then $\bar{H}$ still satisfies Assumption~\ref{ass:TwoPoint}, and it holds that
 \begin{equation}\label{e:TwoPoint3}
 \left.\begin{array}{ll}
\nabla_z\bar{H}(\lambda,t, z)=\nabla_z H(\lambda,t, z+ u_\lambda(t))-\nabla_z{H}(\lambda,t, u_\lambda(t)),\\
 \nabla^2_z\bar{H}(\lambda,t, z)=\nabla^2_zH(\lambda,t, z+ u_\lambda(t)).
 \end{array}\right\}
\end{equation}
It is easy to see that $w:[0,\tau]\to\mathbb{R}^{2n}$ satisfies (\ref{e:TwoPoint1}) if and only if $v:=w-u_\lambda$
satisfies  the Hamiltonian boundary value problem
  \begin{equation}\label{e:TwoPoint4}
\left.\begin{array}{ll}
\dot{u}(t)=J\nabla_z \bar{H}(\lambda,t, u(t))\;\forall t\in [0, \tau],\\ 
u(0)\in L,\quad u(\tau)\in L'.
\end{array}\right\}
\end{equation}
In particular,  $v\equiv 0\in\mathbb{R}^{2n}$ satisfies (\ref{e:TwoPoint4}) for each $\lambda\in\Lambda$.
For the vertical Lagrangian subspace $L_0:=\{0\}\times\mathbb{R}^n\subset\mathbb{R}^{2n}$,
we may choose an orthogonal symplectic matrix $O\in{\rm Sp}(2n)$ such that $OL_0=L$.
Define $L''=O^{-1}L'$ and
\begin{equation*}
K:\Lambda\times [0,\tau]\times{\R}^{2n}\to\R,\;(\lambda,t, z)\mapsto \bar{H}(\lambda,t, Oz).
\end{equation*}
Then $K$ also satisfies Assumption~\ref{ass:TwoPoint}, and
\begin{equation}\label{e:TwoPoint5+}
\nabla_z K(\lambda,t, z)=O^{-1}\nabla_z \bar{H}(\lambda,t, Oz),\quad \nabla^2_zK(\lambda,t, z)=O^{-1}\nabla^2_z\bar{H}(\lambda,t,Oz)O.
\end{equation}
It follows that $v:[0,\tau]\to\mathbb{R}^{2n}$ satisfies (\ref{e:TwoPoint4}) if and only if $w(t):=O^{-1}v(t)$
 satisfies  the Hamiltonian boundary value problem
  \begin{equation}\label{e:TwoPoint6}
\left.\begin{array}{ll}
\dot{u}(t)=J\nabla_z K(\lambda,t, u(t))\;\forall t\in [0, \tau],\\
u(0)\in L_0,\quad u(\tau)\in L''.
\end{array}\right\}
\end{equation}
In particular,  $w\equiv 0\in\mathbb{R}^{2n}$ satisfies (\ref{e:TwoPoint6}) for each $\lambda\in\Lambda$.

 Let $\alpha_\lambda:[0, \tau]\to{\rm Sp}(2n)$ be the fundamental matrix solution of
$\dot{u}(t)=J\nabla^2_zK(\lambda,t, 0)u(t)$. Since
  (\ref{e:TwoPoint3}) and (\ref{e:TwoPoint5+}) imply $\nabla^2_zK(\lambda,t, 0)=O^{-1}\nabla^2_z\bar{H}(\lambda,t, 0)O=
O^{-1}\nabla^2_zH(\lambda,t, u_\lambda(t))O$, we have $\alpha_\lambda(t)=O^{-1}\gamma_\lambda(t)O$ and hence
 \begin{equation}\label{e:TwoPoint8}
 i^{L'}_L(\gamma_\lambda)=i^{L''}_{L_0}(\alpha_\lambda)
 \quad\hbox{and}\quad \nu^{L'}_L(\gamma_\lambda)=\nu^{L''}_{L_0}(\alpha_\lambda)
 =\dim(\gamma_\lambda(\tau)L\cap L')
\end{equation}
by \cite[Definition~2.4]{LiuWL11}.
Note that $v:[0, \tau]\to\mathbb{R}^{2n}$ satisfies (\ref{e:TwoPoint6}) if and only if
$u(t):=Ov(t)+ u_\lambda(t)$ satisfies (\ref{e:TwoPoint1}).
Therefore {\bf from now on we can assume}:
\begin{equation}\label{e:TwoPoint9}
\hbox{\textsf{in Assumption~\ref{ass:TwoPoint}, $L=L_0$, $u_\lambda=0$, $w_\lambda=0=w_\lambda'$ for all $\lambda\in\Lambda$}.}
\end{equation}

Under assumptions in (\ref{e:TwoPoint9}), we only consider solutions of
(\ref{e:TwoPoint1}) near $0\in \mathbb{R}^{2n}$. Because of Remark~\ref{rm:effective},
 $\Lambda$ may be assumed to be compact and sequentially compact.
Therefore by modifying $H$ outside a large ball in what follows we may assume that
\begin{equation}\label{e:TwoPoint12}
 \|\nabla^2_zH(\lambda,t, x)\|_{\mathbb{R}^{2n\times 2n}}\le C(H)\quad\forall (\lambda,t,x)
\end{equation}
and so
\begin{equation}\label{e:TwoPoint12+}
\|\nabla_zH(\lambda,t, x)\|_{\mathbb{R}^{2n}}\le C(H)\|x\|_{\mathbb{R}^{2n}}+C(H)'\quad\forall (\lambda,t,x),
\end{equation}
where $C(H)$ and $C(H)'$ are positive constants.
Because of these, using \cite[Proposition C.1]{Lu9} we derive that the maps
\begin{equation}\label{e:TwoPoint12++}
 L^2([0,\tau];\R^{2n})\to  L^2([0,\tau];\R^{2n}),\; u\mapsto \nabla_zH(\lambda,\cdot; u(\cdot))
\end{equation}
have an uniform bound on any bounded subset and are uniform continuous at any $\bar{u}\in L^2([0,\tau];\R^{2n})$
with respect to $\lambda\in\Lambda$.
Let $u_{\lambda_i}\in L^2([0,\tau],\mathbb{R}^{2n})$, $i=1,2$,  satisfy  $\|u_{\lambda_2}-u_{\lambda_1}\|_2\to 0$ as $\lambda_2\to\lambda_1$.
Then
\begin{eqnarray*}
&&\left(\int^\tau_0|\nabla_zH(\lambda_2, t, u_{\lambda_2}(t))-\nabla_zH(\lambda_1, t, u_{\lambda_1}(t))|^2dt\right)^{1/2}\\
&\le&\left(\int^\tau_0|\nabla_zH(\lambda_2, t, u_{\lambda_2}(t))-\nabla_zH(\lambda_2, t, u_{\lambda_1}(t))|^2dt\right)^{1/2}\\
&&+\left(\int^\tau_0|\nabla_zH(\lambda_2, t, u_{\lambda_1}(t))-\nabla_zH(\lambda_1, t, u_{\lambda_1}(t))|^2dt\right)^{1/2}.
\end{eqnarray*}
Clearly, as $\lambda_2\to\lambda_1$ the first term of the right side converges to zero because of
 the uniform continuity of the maps in (\ref{e:TwoPoint12++}) at any $u_{\lambda_1}\in L^2([0,\tau];\R^{2n})$
with respect to $\lambda\in\Lambda$. From (\ref{e:TwoPoint12+}) and Lebesgue dominated convergence theorem
it follows that the second term of the right side converges to zero as $\lambda_2\to\lambda_1$.
Hence we obtain the first claim of the following.

\begin{lemma}\label{lem:solutionRegu}
The map
\begin{equation}\label{e:TwoPoint12+++}
\Lambda\times L^2([0,\tau];\R^{2n})\to  L^2([0,\tau];\R^{2n}),\;(\lambda, u)\mapsto
 \nabla_zH(\lambda,\cdot, u(\cdot))
\end{equation}
is continuous. Therefore if $u_{\lambda_i}\in C^1([0,\tau],\mathbb{R}^{2n})$, $i=1,2$, satisfy $\dot{u}_{\lambda_i}=J\nabla_zH(\lambda_i,t, u_{\lambda_i}(t))$,
and $\|u_{\lambda_2}-u_{\lambda_1}\|_2\to 0$ as $\lambda_2\to\lambda_1$, then
$u_{\lambda_2}\to u_{\lambda_1}$ in $C^1([0,\tau],\mathbb{R}^{2n})$ as $\lambda_2\to\lambda_1$.
\end{lemma}
\begin{proof}[\bf Proof]
By the first claim, $\|\dot{u}_{\lambda_2}-\dot{u}_{\lambda_1}\|_2\to 0$ as $\lambda_2\to\lambda_1$.
Then $\|u_{\lambda_2}-u_{\lambda_1}\|_{1,2}\to 0$ and so $\|u_{\lambda_2}-u_{\lambda_1}\|_{C^0}\to 0$.
By Assumption~\ref{ass:TwoPoint}, (\ref{e:TwoPoint3}) and (\ref{e:TwoPoint5+}) imply that
$$
\Lambda\times [0,\tau]\times \mathbb{R}^{2n}\ni(\lambda,t, z)\to \nabla_zH(\lambda, t, z)\in \mathbb{R}^{2n}
$$
is continuous. It follows from $\dot{u}_{\lambda_i}=J\nabla_zH(\lambda_i,t, u_{\lambda_i}(t))$ that
 $\|\dot{u}_{\lambda_2}-\dot{u}_{\lambda_1}\|_{C^0}\to 0$ as $\lambda_2\to\lambda_1$.
\end{proof}

Following \cite[\S4]{LiuWL11} the problem is reduced to one in finitely dimensional spaces.
 For each $s\geq 0$ consider a Hilbert space
  \begin{eqnarray*}
H^s_{L_0}=\left\{x\in L^2([0,\tau],\mathbb{R}^{2n})\,\Bigg| \begin{array}{ll}
x\stackrel{L^2}{=}\sum_{m\in\mathbb{Z}}-J\exp\left(\frac{m\pi tJ}{\tau}\right)a_m,\;
a_m\in \mathbb{R}^n\oplus\{0\}, \\
\sum_{m\in\mathbb{Z}}(1+|m|^{2s})|a_m|^2<\infty
\end{array}
\right\}
\end{eqnarray*}
with the following inner product
\begin{equation}\label{innerproduct}
\langle x,y\rangle_{H^s_{L_0}} ={\tau}\langle x_0,y_0\rangle_{\mathbb{R}^{2n}}+\sum_{k\neq 0} |k|^{2s}(x_k, y_k)_{\mathbb{R}^{2n}},\quad x, y\in H^s_{L_0}.
\end{equation}
Denote the associated norm by $\|\cdot\|_{H^s_{L_0}}$.
Then $H^0_{L_0}$ and $H^1_{L_0}$ are Hilbert subspaces of $L^2([0,\tau],\mathbb{R}^{2n})$ and $W^{1,2}([0,\tau],\mathbb{R}^{2n})$, respectively.
Hence we also write $\langle x,y\rangle_{H^s_{L_0}}$ as
$\langle x,y\rangle_{2}$ (resp. $\langle x,y\rangle_{1,2}$) for $s=0$ (resp. $s=1$), and
$\|\cdot\|_{H^s_{L_0}}$ as
$\|\cdot\|_2$ (resp. $\|\cdot\|_{1,2}$) for $s=0$ (resp. $s=1$) below.

Take an orthogonal symplectic matrix $P\in{\rm Sp}(2n)$ such that $PL_0=L'$.
We have a matrix $M$ such that $P=e^M$.
Consider Hilbert subspaces of $L^2([0,\tau],\mathbb{R}^{2n})$ and $W^{1,2}([0,\tau],\mathbb{R}^{2n})$
\begin{eqnarray*}
&&\mathcal{W}_0=\left\{x\in L^2([0,\tau],\mathbb{R}^{2n})\,\Big|\, [0,\tau]\ni t\mapsto \exp\left(-\frac{t}{\tau}M\right)x(t)\;\hbox{belongs to}\;H^0_{L_0}\right\},\\
&&\mathcal{W}_1=\left\{x\in L^2([0,\tau],\mathbb{R}^{2n})\,\Big|\, [0,\tau]\ni t\mapsto \exp\left(-\frac{t}{\tau}M\right)x(t)\;\hbox{belongs to}\;H^1_{L_0}\right\}
\end{eqnarray*}
and  a unbounded self-adjoint operator $\mathbb{A}$ in $\mathcal{W}_0$ with domain $D(\mathbb{A})=\mathcal{W}_1$, which is given by
$$
\langle \mathbb{A}x,y\rangle_{2}=\int^\tau_0(-J\dot{x}(t)+JMx(t),y(t))_{\mathbb{R}^{2n}}dt.
$$
The range of $\mathbb{A}$ is closed, the resolution of $\mathbb{A}$ is compact and the spectrum of
 $\mathbb{A}$ only consists of eigenvalues, more precisely $\sigma(\mathbb{A})=\frac{\pi}{\tau}\mathbb{Z}$.
Note that ${\rm Ker}(\mathbb{A})=\{[0,\tau]\ni t\mapsto e^{\frac{t}{\tau}M}a\,|\, a\in\mathbb{R}^n\oplus\{0\}\}$.

For each $\lambda\in\Lambda$, define functionals $\mathcal{H}_\lambda:\mathcal{W}_0\to\mathbb{R}$ and $f_\lambda: (\mathcal{W}_1, \|\cdot\|_{2})\to \mathbb{R}$  by
\begin{eqnarray}\label{e:SR1}
&&\mathcal{H}_\lambda(x)=\int^\tau_0H(\lambda, t,x(t))dt+ \frac{1}{2}\int^\tau_0(JMx(t),x(t))_{\mathbb{R}^{2n}}dt,\\
&&f_\lambda(x)=\frac{1}{2}\bigl(\mathbb{A}x, x\bigr)_{2}-\mathcal{H}_\lambda(x).\label{e:SR2}
\end{eqnarray}
(\ref{e:TwoPoint12}) and (\ref{e:TwoPoint12+}) imply that
both functionals are  of class $C^1$, and that  the $L^2$-gradient $\nabla^{L^2}\mathcal{H}_\lambda$ of $\mathcal{H}_\lambda$ defined by
\begin{eqnarray*}
(\nabla^{L^2}\mathcal{H}_\lambda(x), y)_2=\int^\tau_0(\nabla_zH(\lambda, t,x(t)), y(t))_{\mathbb{R}^{2n}}dt+ \int^\tau_0(JMx(t),y(t))_{\mathbb{R}^{2n}}dt\quad\forall x,y\in\mathcal{W}_0
\end{eqnarray*}
  is G\^ateaux differentiable. It is easily checked that the G\^ateaux derivative $D\nabla^{L^2}\mathcal{H}_\lambda(x)\in\mathcal{L}_s(\mathcal{W}_0)$ at $x\in\mathcal{W}_0$ is given by
\begin{eqnarray}\label{e:SR5}
(D\nabla^{L^2}\mathcal{H}_\lambda(x)y, z)_2=\int^\tau_0(\nabla^2_zH(\lambda, t,x(t))y(t), z(t))_{\mathbb{R}^{2n}}dt+ \int^\tau_0(JMy(t), z(t))_{\mathbb{R}^{2n}}dt
\end{eqnarray}
for all $y,z\in\mathcal{W}_0$. From (\ref{e:TwoPoint12}) and (\ref{e:SR5}) we derive that $\|D\nabla^{L^2}\mathcal{H}_\lambda(x)\|_{\mathcal{L}_s(\mathcal{W}_0)}\le C(H)+\|M\|_{\mathbb{R}^{2n\times 2n}}$.
Moreover  the $L^2$-gradient of $f_\lambda$ is given by
\begin{eqnarray}\label{e:SR4}
\nabla^{L^2}f_\lambda(x)=\mathbb{A}x-\nabla_zH(\lambda, \cdot,x(\cdot))-JMx,\quad\forall x\in\mathcal{W}_1,
\end{eqnarray}
and  $x\in (\mathcal{W}_1, \|\cdot\|_2)$ is a critical point of $f_\lambda$ if and only if
 it satisfies (\ref{e:TwoPoint1}) with the assumptions in (\ref{e:TwoPoint9}).

 Let $\Pi_0:\mathcal{W}_0\to{\rm Ker}(\mathbb{A})$
be the orthogonal projection. Define $\underline{\mathbb{A}}=\mathbb{A}+\Pi_0$, and denote by $E_s$ the spectral resolution of
$\underline{\mathbb{A}}$. Choose $\beta\in\mathbb{R}\setminus\sigma(\underline{\mathbb{A}})$ such that $\beta>2(C(H)+\|M\|_{\mathbb{R}^{2n\times 2n}}+1)$, and define
\begin{eqnarray*}
&&P=\int_{-\beta}^{\beta}dE_{s},\qquad P^{+}=\int_{\beta}^{+\infty}dE_{s},\qquad P^{-}=\int_{-\infty}^{-\beta}dE_{s},\\
&& \mathcal{W}_0^+:=P^+(\mathcal{W}_0),\quad \mathcal{W}_0^-:=P^-(\mathcal{W}_0),\quad Z:=P(\mathcal{W}_0).
\end{eqnarray*}
Then $\mathcal{W}_0=\mathcal{W}_0^+\oplus \mathcal{W}_0^-\oplus Z$, and $\dim Z<\infty$. Bounded self-adjoint linear operators on $\mathcal{W}_0$,
 \begin{eqnarray*}
S^+=\int_\beta^{\infty}s^{-1/2}dE_s,\quad
S^-=\int^{-\beta}_{-\infty}(-s)^{-1/2}dE_s,\quad
R=\int_{-\beta}^\beta|s|^{-1/2}dE_s,
\end{eqnarray*}
 have ranges $\mathcal{W}_0^+$, $\mathcal{W}_0^-$ and $Z$, respectively,  are pairwise commuting, and $S^+|_{\mathcal{W}_0^+}$,
 $S^-|_{\mathcal{W}_0^-}$ and $R|_Z$ are injective.

According to  Amann \cite{Am},  Amann-Zehnder \cite{AmZe}, Chang \cite{Ch}, Long \cite{Long97}
and Liu-Wang-Lin \cite[Theorems~4.1,~4.2]{LiuWL11} we can use Lemma~\ref{lem:solutionRegu} to obtain:

\begin{theorem}\label{th:AmZe}
There exist $(x(\cdot,\cdot), y(\cdot,\cdot))\in C(\Lambda\times Z, \mathcal{W}_0^+\times \mathcal{W}_0^-)$ satisfying the following:
\begin{description}
\item[(i)] The continuous map $\mathfrak{u}(\cdot,\cdot):\Lambda\times Z\to \mathcal{W}_0$ defined by
$$
\mathfrak{u}(\lambda,z)=S^+x(\lambda, R^{-1}z)+S^-y(\lambda, R^{-1}z)+  z
$$
takes values in $\mathcal{W}_1$, each $\mathfrak{u}(\lambda,\cdot):Z\to \mathcal{W}_0$ is a $C^1$ embedding, and
$d_z\mathfrak{u}(\lambda, z)$ is continuous in $(\lambda,z)$.

\item[(ii)] For each $\lambda\in\Lambda$, the functional $a_\lambda: Z\to\mathbb{R}$ defined by
\begin{eqnarray}\label{e:reducedAction}
a_\lambda(z)=f_\lambda(\mathfrak{u}(\lambda,z))&=&\frac{1}{2}\bigl(\underline{\mathbb{A}}\mathfrak{u}(\lambda,z),\mathfrak{u}(\lambda,z)\bigr)_{2}-
\mathcal{H}_\lambda(\mathfrak{u}(\lambda,z))
\end{eqnarray}
is $C^2$, and
\begin{eqnarray}\label{e:one-diff}
&&a'_\lambda(z)=\underline{\mathbb{A}}z-P\mathcal{H}'_\lambda(\mathfrak{u}(\lambda,z))=
\underline{\mathbb{A}}\mathfrak{u}(\lambda,z)-
\mathcal{H}'_\lambda(\mathfrak{u}(\lambda,z)),\\
&&a''_\lambda(z)=\underline{\mathbb{A}}|_{{Z}}-Pd\mathcal{H}'_\lambda(\mathfrak{u}(\lambda, z))d_z\mathfrak{u}(\lambda,z)=[\underline{\mathbb{A}}-
d\mathcal{H}'_\lambda(\mathfrak{u}(\lambda,z))]d_z\mathfrak{u}(\lambda,z).\label{e:two-diff}
\end{eqnarray}
(It follows that
\begin{eqnarray}\label{e:SR9}
\Lambda\times{Z}\ni (\lambda,z)\mapsto a'_\lambda(z)\in{Z}\quad\hbox{and}\quad
\Lambda\times{Z}\ni (\lambda,z)\mapsto a''_\lambda(z)\in\mathscr{L}({Z})
\end{eqnarray}
are continuous. Note that (i) implies $\Lambda\times{Z}\ni (\lambda,z)\mapsto a_\lambda(z)\in\mathbb{R}$ to be continuous.)
\item[(iii)] For each $\lambda\in\Lambda$, the map $Z\ni z\mapsto \mathfrak{u}(\lambda, z)\in \mathcal{W}_0$ gives a one-to-one correspondence
between the  critical points of $a_\lambda$ and those of $f_\lambda$.
(The condition (\ref{e:TwoPoint9}) implies that $\mathfrak{u}(\lambda, 0)=0$ and $0\in Z$ is
a critical point of $a_\lambda$ for any $\lambda\in\Lambda$.)
\item[(iv)] If $z\in Z$ is a critical point of $a_\lambda$ and
 $\beta_{\lambda,z}:[0, \tau]\to{\rm Sp}(2n)$ is the fundamental matrix solution of
 $\dot{v}(t)=J\nabla^2_zH(\lambda, t, \mathfrak{u}(\lambda, z)(t))v(t)$, then
 \begin{equation*}
 m^0(a_\lambda,z)=\nu_{L_0}^{L'}(\beta_{\lambda,z})\quad\hbox{and}\quad m^-(a_\lambda,z)=i_{L_0}^{L'}(\beta_{\lambda,z}).
 \end{equation*}
 In particular, $m^0(a_\lambda, 0)=\nu_{L_0}^{L'}(\gamma_\lambda)=\dim(\gamma_\lambda(\tau)L_0\cap L')$ and $m^-(a_\lambda, 0)=i_{L_0}^{L'}(\gamma_\lambda)$.
\end{description}
\end{theorem}
 \begin{proof}[\bf Proof of Theorem~\ref{th:bif-nessHam}]
 \noindent{\bf (I)}.
Under the assumptions in (\ref{e:TwoPoint9}) let $(\mu, 0)$  is a bifurcation point along sequences of the problem (\ref{e:TwoPoint1})
with respect to the {trivial branch} $\{(\lambda, 0)\,|\,\lambda\in\Lambda\}$.
That is,  there exists a sequence $(\lambda_k)\subset\Lambda$ converging to $\mu$ and
solutions $v_k\ne 0$ of (\ref{e:TwoPoint1}) with $\lambda=\lambda_k$
such that $v_k\to 0$ in $C^1([0, \tau], \mathbb{R}^{2n})$ by Proposition~\ref{prop:threeBifu}.
Then $v_k$  is a critical point of $f_{\lambda_k}$ in
$(\mathcal{W}_1, \|\cdot\|_2)$ and $\|v_k\|_2\to 0$ as $k\to\infty$.
Therefore for each $k\in\mathbb{N}$ by Theorem~\ref{th:AmZe}(iii) we have an unique $z_k\in Z$ such that
$$
v_k=\mathfrak{u}(\lambda_k,z_k)=S^+x(\lambda_k, R^{-1}z_k)+S^-y(\lambda_k, R^{-1}z_k)+  z_k.
$$
This implies that $\|z_k\|_2\le\|v_k\|_2\to 0$. Clearly, $z_k\ne 0$ by (i) and (iii) in  Theorem~\ref{th:AmZe}.
These show that $(\mu,0)\in\Lambda\times Z$  is a bifurcation point along sequences of $a'_\lambda(z)=0$
with respect to the {trivial branch} $\{(\lambda, 0)\,|\,\lambda\in\Lambda\}\subset\Lambda\times Z$.
Suppose  $\nu^{L'}_{L_0}(\gamma_\mu)=\dim(\gamma_\mu(\tau)L_0\cap L')=0$.
Theorem~\ref{th:AmZe}(iv) gives rise to $m^0(a_\mu, 0)=0$, that is, $a_\mu''(0):Z\to Z$ is invertible.
By Theorem~\ref{th:AmZe}(iii) we may use the implicit function theorem for
$\Lambda\times{Z}\ni (\lambda,z)\mapsto a'_\lambda(z)\in{Z}$ near $(\mu,0)$
to yield $z_k=0$ for $k$ large enough,
in contradiction with $z_k\ne 0$.

\noindent{\bf (II)}. As above we can derive from the assumptions and
Theorem~\ref{th:AmZe}(iv) that for each $k\in\mathbb{N}$,
 \begin{eqnarray*}
&&[m^-(a_{\lambda_k^-}, 0), m^-(a_{\lambda_k^-}, 0)+m^0(a_{\lambda_k^-}, 0)]\cap [m^-(a_{\lambda_k^+}, 0), m^-(a_{\lambda_k^+}, 0)+m^0(a_{\lambda_k^+}, 0)]\\
&=&[i^{L'}_L(\gamma_{\lambda_k^-}), i^{L'}_L(\gamma_{\lambda_k^-})+\nu^{L'}_L(\gamma_{\lambda_k^-})]\cap[i^{L'}_L(\gamma_{\lambda_k^+}), i^{L'}_L(\gamma_{\lambda_k^+})+\nu^{L'}_L(\gamma_{\lambda_k^+})]=\emptyset
\end{eqnarray*}
 and either $m^0(a_{\lambda_k^+}, 0)=\nu^{L'}_L(\gamma_{\lambda_k^+})=0$ or $m^0(a_{\lambda_k^-}, 0)=\nu^{L'}_L(\gamma_{\lambda_k^-})=0$.
 By Theorem~\ref{th:A.9} there exists a sequence $\{(\lambda_k, z_k)\}_{k\ge 1}$
 in $\hat\Lambda\times Z\setminus\{(\mu,0)\}$ converging to $(\mu,0)$ such that $z_k\ne 0$ and $a'_{\lambda_k}(z_k)=0$ for each $k\in\mathbb{N}$.
 Let $v_k=\mathfrak{u}(\lambda_k,z_k)=S^+x(\lambda_k, R^{-1}z_k)+S^-y(\lambda_k, R^{-1}z_k)+  z_k$ for each $k\in\mathbb{N}$.
Then by Theorem~\ref{th:AmZe} we obtain that $v_k\ne 0$, $v_k\to 0$ in $(\mathcal{W}_1, \|\cdot\|_2)$ and $f'_{\lambda_k}(v_k)=0$ for each $k\in\mathbb{N}$.
  The expected result is proved.

 \noindent{\bf (III)}. As above we may obtain
$$
[m^-(a_{\lambda^-}, 0), m^-(a_{\lambda^-}, 0)+m^0(a_{\lambda^-}, 0)]\cap [m^-(a_{\lambda^+}, 0), m^-(a_{\lambda^+}, 0)+m^0(a_{\lambda^+}, 0)]=\emptyset
$$
 and either $m^0(a_{\lambda^+}, 0)=\nu^{L'}_L(\gamma_{\lambda^+})=0$ or $m^0(a_{\lambda^-}, 0)=\nu^{L'}_L(\gamma_{\lambda^-})=0$.
 Then by Theorem~\ref{th:A.9+} we obtain a $\mu\in \Lambda$ such that
  $(\mu, 0)$ is a bifurcation point along sequences of $a_\lambda(z)=0$
   in $\Lambda\times Z$. Moreover, if $m^0(a_{\lambda^+}, 0)=0$ (resp. $m^0(a_{\lambda^-}, 0)=0$),
   then $\mu\ne \lambda^+$ (resp. $\mu\ne \lambda^-$).
 Finally,  the desired result follows from these as above.
\end{proof}

\begin{proof}[\bf Proof of Theorem~\ref{th:bif-suffHam}]
Under the assumptions of Theorem~\ref{th:bif-suffHam} with additional conditions in (\ref{e:TwoPoint9}),
following the notations in the proof of Theorem~\ref{th:bif-nessHam} we obtain that
  $m^0(a_{\mu}, 0)\ne 0$
   and  $m^0(a_{\lambda}, 0)=0$  for each $\lambda\in\Lambda\setminus\{\mu\}$ near $\mu$, and that
  $m^-(a_{\lambda}, 0)$ take, respectively, values $m^-(a_{\mu},0)$ and
  $m^-(a_{\mu},0)+ m^0(a_{\mu},0)$
 as $\lambda\in\Lambda$ varies in two deleted half neighborhoods  of $\mu$.
 Therefore \cite[Theorem~3.6]{Lu10} (or \cite[Theorem~4.10]{Lu8}) is applicable to $\{a_\lambda\,|\,\lambda\in\Lambda\}$.
 The corresponding conclusions are easily translated into those of
 Theorem~\ref{th:bif-suffHam} with additional conditions in (\ref{e:TwoPoint9}).
\end{proof}

\subsection{Proof of Theorem~\ref{th:two-point}}\label{sec:twoLagr-2}

Define $G:(0,\tau]\times [0, 1]\times\mathbb{R}^{2n}\to\mathbb{R}$ by
$$
G(\lambda,t, z)=\lambda H\big(z+ (1-t)u(0)+tu(\lambda)\big)+\big(z+ (1-t)u(0)+tu(\lambda),J(u(\lambda)-u(0))\big)_{\mathbb{R}^{2n}}.
$$
Then $\nabla_zG(\lambda,t, z)=\lambda\nabla H\bigr(z+ (1-t)u(0)+tu(\lambda)\bigl)+ J(u(\lambda)-u(0))$ and
$\nabla_z^2G(\lambda,t, z)=\lambda\nabla^2 H\bigr(z+ (1-t)u(0)+tu(\lambda)\bigl)$.
For $\lambda\in (0,\tau]$ it is easily proved that $z:[0,\lambda]\to\mathbb{R}^{2n}$ satisfies
\begin{equation}\label{e:TwoPoint52}
\left.\begin{array}{ll}
\dot{z}(t)=J\nabla_z H(z(t))\;\forall t\in [0, \lambda],\\
z(0)\in u(0)+L,\quad z(\lambda)\in u(\lambda)+ L
\end{array}\right\}
\end{equation}
if and only if $w:[0,1]\to\mathbb{R}^{2n}$ defined by $w(t):=z(\lambda t)-(1-t)u(0)-tu(\lambda)$ satisfies
\begin{equation}\label{e:TwoPoint53}
\left.\begin{array}{ll}
\dot{w}(t)=J\nabla_z G(\lambda,t, w(t))\;\forall t\in [0, 1],\\
w(0)\in L,\quad w(1)\in L.
\end{array}\right\}
\end{equation}
In particular, since $u|_{[0,\lambda]}$ satisfies (\ref{e:TwoPoint52}), the path
\begin{equation}\label{e:TwoPoint530}
w_\lambda:[0,1]\to \mathbb{R}^{2n},\; t\mapsto u(\lambda t)-(1-t)u(0)-tu(\lambda)
\end{equation}
 satisfies (\ref{e:TwoPoint53}). Let $B(t)=\nabla^2 H(u(t))$ for $t\in [0, \tau]$, and let
  $\gamma:[0, \tau]\to{\rm Sp}(2n)$  be the fundamental matrix solution of
$\dot{v}(t)=JB(t)v(t)$ on $[0,\tau]$. Then it is easily checked that
 \begin{equation}\label{e:TwoPoint53+}
B_\lambda(t):=\nabla^2_z G(\lambda,t, w_\lambda(t))=\lambda\nabla^2 H(u(\lambda t))=\lambda B(\lambda t)
\quad\forall t\in [0, 1],
\end{equation}
and $[0, 1]\ni t\mapsto \gamma_\lambda(t):=\gamma(\lambda t)$
 is the fundamental matrix solution of $\dot{v}(t)=JB_\lambda(t)v(t)$ on $[0, 1]$.

 We shall obtain conclusions of Theorem~\ref{th:two-point} by applying  Theorem~\ref{th:bif-nessHam}(I) and  Theorem~\ref{th:bif-suffHam} to (\ref{e:TwoPoint53}).
 Note that $(i^{L}_L, \nu^{L}_L)=(i_L, \nu_L)$.

 {\it Proof of (A).}
  Suppose that $\mu\in (0, \tau]$ is a bifurcation instant for $(H, u, L)$.
By the above definition  there exists a sequence $(\tau_k)\subset (0,\tau]$ converging to $\mu$
such that for each $k\in\mathbb{N}$ the boundary value problem
(\ref{e:TwoPointMu}) has a solution $v_k\ne u|_{[0,\tau_k]}$ such that $\|v_k-u\|_{C^1([0,\tau_k],\mathbb{R}^{2n})}\to 0$ as $k\to\infty$. Then
 $$
 w^k:[0,1]\to\mathbb{R}^{2n},\;t\mapsto v_k(\tau_k t)-(1-t)u(0)-tu(\tau_k)
 $$
 is not equal to $w_{\tau_k}$, and satisfies (\ref{e:TwoPoint53}) and
 $$
 \|w^k-w_{\tau_k}\|_{C^1([0,1],\mathbb{R}^{2n})}\le (1+\tau)\|v_k-u\|_{C^1([0,\tau_k],\mathbb{R}^{2n})} \to 0
 $$
  as $k\to\infty$. Since both $u$ and $\dot{u}$ are uniformly continuous on $[0,\tau]$, it is not hard to prove that
  $\|w_{\tau_k}-w_\mu\|_{C^1([0,1],\mathbb{R}^{2n})}\to 0$ as $k\to\infty$. Hence $\|w^k-w_{\mu}\|_{C^1([0,1],\mathbb{R}^{2n})}\to 0$.
  These show that  $(\mu, w_\mu)$  is a bifurcation point of (\ref{e:TwoPoint53}).
 By Theorem~\ref{th:bif-nessHam}(I) we deduce
 $$
 \nu_L(\gamma|_{[0,\mu]})=\dim (\gamma(\mu)L)\cap L=\nu_L(\gamma_\mu)>0
 $$
 and therefore Theorem~\ref{th:two-point}(A).

  {\it Proof of (B).}
 If all $B_{22}(t)$ (resp. all $B_{11}(t)$) are positive definite, it was claimed in \cite[Lemma~5.1]{Liu07} that
 \begin{eqnarray}\label{e:TwoPoint53+1}
 &&i_{L_0}(\gamma)=\sum_{0<\lambda<\tau}\nu_{L_0}(\gamma_\lambda)=\sum_{0<\lambda<\tau}\dim ((\gamma(\lambda)L_0)\cap L_0)\\
 &&\hbox{(resp.}\quad i_{L_1}(\gamma)=\sum_{0<\lambda<\tau}\nu_{L_1}(\gamma_\lambda)=\sum_{0<\lambda<\tau}
 \dim ((\gamma(\lambda)L_1)\cap L_1)\quad\hbox{)}.\label{e:TwoPoint53+2}
 \end{eqnarray}
 Clearly, the positive definiteness  of $B(t)=\left(\begin{array}{cc}
             B_{11}(t) & B_{12}(t) \\
            B_{21}(t) & B_{22}(t) \\
           \end{array}\right)$
 implies that both $B_{22}(t)$ and $B_{11}(t)$ are positive definite.
 Choose an orthogonal symplectic matrix  $O\in {\rm Sp}(2n,\mathbb{R})$ such that $OL_0=L$.
 Then $\Upsilon(t):=O^{-1}\gamma(t)O$ is the fundamental matrix solution of
$\dot{v}(t)=JO^TB(t)Ov(t)$ on $[0,\tau]$. Note that all $O^TB(t)O$ are also positive definite. By (\ref{e:L-index}) and
(\ref{e:TwoPoint53+1}) we obtain
 \begin{equation}\label{e:TwoPoint53+3}
i_L(\gamma)=i_{L_0}(\Upsilon)=\sum_{0<\lambda<\tau}\dim ((\Upsilon(\lambda)L_0)\cap L_0)=\sum_{0<\lambda<\tau}\dim ((\gamma(\lambda)L)\cap L)
\end{equation}
 because $\dim ((\Upsilon(\lambda)L_0)\cap L_0)=\dim ((O^{-1}\gamma(\lambda)OL_0)\cap L_0)=\dim (\gamma(\lambda)L)\cap L)$.
 This implies the first claim in  Theorem~\ref{th:two-point}(B).

Let $\mu\in (0, \tau)$ be such that $\dim ((\gamma(\mu)L)\cap L)>0$.
Then there exists a small $\varepsilon>0$ such that
$\dim ((\gamma(\lambda)L)\cap L)=0$ for all $\lambda\in [\mu-\varepsilon, \mu+\varepsilon]\setminus\{\mu\}$.
Note that (\ref{e:TwoPoint53+3}) implies
 $$
 i_L(\gamma_\lambda)=\sum_{0<s<1}\dim ((\gamma_\lambda(s)L)\cap L)=
 \sum_{0<\theta<\lambda}\dim ((\gamma(\theta)L)\cap L)
 $$
and so
\begin{eqnarray*}
i_L(\gamma_\lambda)=\left\{
\begin{array}{ll}
 i_L(\gamma_{\mu})\;&\forall\lambda\in (\mu-\varepsilon,\mu),\\
 i_L(\gamma_{\mu})+ \nu_L(\gamma_{\mu}) \;&\forall\lambda\in (\mu, \mu+\varepsilon).
\end{array}\right.
\end{eqnarray*}
We can apply  Theorem~\ref{th:bif-suffHam} to (\ref{e:TwoPoint53}).
Then the corresponding conclusions (i), (ii) and (iii) of Theorem~\ref{th:bif-suffHam} about
(\ref{e:TwoPoint53}) are equivalent  to (B.1), (B.ii) and (B.iii), respectively.

With (\ref{e:TwoPoint53+1}) and (\ref{e:TwoPoint53+2}),
 the final conclusions can be derived  as above. \hfill$\Box$\vspace{2mm}

%
%
%
%
%
%
%
%
%

\vspace{5mm}

\noindent{\bf Acknowledgements}\quad I would like to thank Professors Yujun Dong and Chungen Liu for the  helpful discussion, and Professor Yiming Long for helping me to correct an error in the first version.
I am very grateful to the anonymous referee for his/her many valuable comments and suggestions regarding the mathematical content and the writing, which have greatly improved the revised version.\vspace{2mm}

\noindent{\bf Research ethics}\quad Not applicable.\\
\noindent{\bf Informed consent}\quad 	Not applicable.\\
\noindent{\bf Author contributions}\quad
The author has accepted responsibility for the entire content of this manuscript and approved its submission.\\
\noindent{\bf Use of LLM, AI and MLT}\quad 	None declared.\\
\noindent{\bf Conflict of interest}\quad 
The author states no conflict of interest.\\
\noindent{\bf Research funding}\quad 
This work was supported by National Natural Science Foundation of China (Grant No. 11271044 and 12371108).\\
\noindent{\bf Data Availability}\quad We do not analyse or generate any datasets, because our work proceeds within a theoretical
and mathematical approach.\vspace{2mm}

%
%
%

%


\appendix

\section{Maslov-type index and Morse index}\label{app:Index}

Let ${\rm Sp}(2n,\mathbb{R})^0 = \{ M \in {\rm Sp}( 2 n )\,|\, D(M)= 0 \}$, where
$D(M)=( - 1 ) ^ { n - 1 } \operatorname { d e t } \left( M - I _ { 2 n }\right)$.
The complementary set ${\rm Sp}( 2 n,\mathbb{R}) ^ { * } := {\rm Sp}( 2 n,\mathbb{R} ) \backslash{\rm Sp}( 2 n,\mathbb{R}) ^ { 0 }$
 has exactly two path-connected components ${\rm Sp}( 2 n,\mathbb{R}) ^ {\pm} = \{ M \in {\rm Sp}( 2 n,\mathbb{R})\,|\, \pm D(M) < 0 \}$,
which contain, respectively,
$$
M _ { n } ^ { + } = {\rm blockdiag}(2I_n, \frac{1}{2}I_n)\quad\hbox{and}\quad
M _ { n } ^ { - } = {\rm blockdiag}(-2, 2I_{n-1}, -\frac{1}{2}, \frac{1}{2}I_{n-1}).
$$
Every loop in these two components is contractible in ${\rm Sp}(2n)$
(\cite[Lemma~1.7]{CoZe} and \cite[Lemma~3.2]{SaZe2}).

Let $\mathcal{P}_\tau(2n) =\{\gamma\in C([0,\tau],{\rm Sp}(2n,\R))\,|\, \gamma(0)=I_{2n}\}$
and let ${\cal P}^\ast_\tau(2n)$ consist of $\gamma\in\mathcal{P}_\tau(2n)$ such that
$\gamma(\tau)\in{\rm Sp}(2n,\mathbb{R})^\ast$, $\tau>0$.
If $\gamma_1,\gamma_2\in\mathcal{P}_\tau(2n)$ satisfy $\gamma_1(\tau)=\gamma_2(0)$, define
$\gamma_2\ast\gamma_1:[0,\tau]\to{\rm Sp}(2n,\mathbb{R})$ by
$\gamma_2\ast\gamma_1(t) = \gamma_1(2t)$ for $t\in [0,\tau/2]$, and $\gamma_2\ast\gamma_1(t) =\gamma_2(2t-\tau)$ for
$t\in[\tau/2,\tau]$.

Recall that each $S\in{\rm Sp}(2n,\mathbb{R})$ has a unique decomposition $PU$, where $P=\sqrt{SS^T}$ and
$U=\left(\begin{array}{cc}
             U_1 & -U_2 \\
             U_2 & U_1 \\
           \end{array}
         \right)$
such that $\mathfrak{u}(S):=U_1+\sqrt{-1}U_2$ is a unitary matrix, i.e., $\mathfrak{u}(S)\in U(n,\C)$.
Therefore each path $\gamma\in C([a, b], {\rm Sp}(2n,\mathbb{R}))$ corresponds to a unique continuous path
$\mathfrak{u}_\gamma:[a,b]\to U(n,\C)$ given by $\mathfrak{u}_\gamma(t):=\mathfrak{u}(\gamma(t))$.
By the lifting criterion in the cover space theory
there exists a continuous real function $\triangle_\gamma$ on $[a, b]$
such that $\det \mathfrak{u}_\gamma(t)=\exp(\sqrt{-1}\triangle_\gamma(t))$ for $t\in[a,b]$, and
$\triangle(\gamma):=\triangle_\gamma(b)-\triangle_\gamma(a)$ is uniquely determined by $\gamma$.

 Conley and Zehnder \cite{CoZe} (for $n\ge 2$)
 and Long and  Zehnder \cite{LongZeh}  (for $n=1$) assign an integer
 $i_{\rm CZ}(\gamma)=\triangle({\beta\ast\gamma})/\pi$
 to each $\gamma\in\mathcal{P}^\ast_\tau(2n)$,
 the so-called \textsf{Conley-Zehnder index}  of $\gamma$,
  where $\beta$ is a path in ${\rm Sp}(2n)^\ast$ connecting $\gamma(\tau)$ to $M^+_n$ or $M^-_n$.
   An alternative exposition was presented in  \cite{SaZe2}.

 There exist two ways to extend the Conley-Zehnder index to paths in $\mathcal{P}_\tau(2n)\setminus\mathcal{P}^\ast_\tau(2n)$.
  Long \cite{Long97, Long02}  defined the \textsf{Maslov-type index}
  of $\gamma\in\mathcal{P}_\tau(2n)$ to be a pair of integers $(i_\tau(\gamma),\nu_\tau(\gamma))$,
\begin{eqnarray*}
 {i}_\tau(\gamma)=\inf\{i_{\rm CZ}(\beta)\,|\, \beta\in{\cal
P}^\ast_\tau(2n)\;\hbox{is sufficiently $C^0$ close to $\gamma$
in}\;{\cal P}_\tau(2n)\}
\end{eqnarray*}
and $\nu_\tau(\gamma):=\dim{\rm Ker}(\gamma(\tau)-I_{2n})$.
It was proved in \cite{Long97} that ${i}_\tau$ has the following \textsf{homotopy invariance}: Two paths $\gamma_0$ and $\gamma_1$ in $\mathcal{P}_\tau(2n)$
have the same Maslov-type indexes $i_\tau(\gamma_0)=i_\tau(\gamma_1)$
if there is a map $\delta\in C([0,1]\times[0,\tau], {\rm Sp}(2n,\R))$ such that $\delta(0,\cdot)=\gamma_0$, $\delta(1,\cdot)=\gamma_1$,
$\delta(s,0)=I_{2n}$ and $\nu_\tau(\delta(s,\cdot))$ is constant for $0\le s\le 1$.

For $a<b$ and any path $\gamma\in C([a,b],{\rm Sp}(2n,\R))$, choose
$\beta\in{\cal P}_1(2n)$ with $\beta(1)=\gamma(a)$, and define
$\phi\in{\cal P}_1(2n)$ by $\phi(t)=\beta(2t)$ for $0\le t\le 1/2$,
and $\phi(t)=\gamma(a+(2t-1)(b-a))$ for $1/2\le t\le 1$.
 Long  showed in \cite{Long97} that the difference $i_1(\phi)-i_1(\beta)$ only
depends on $\gamma$, and called
\begin{equation}\label{e:Long.1}
i(\gamma, [a,b]):=i_1(\phi)-i_1(\beta)
\end{equation}
the \textsf{ Maslov-type index of $\gamma$}. Clearly, $i(\gamma, [0,1])=i_1(\gamma)$ for any $\gamma\in{\cal P}_1(2n)$.

Let us introduce another extension  of $i_{\rm CZ}$ due to Robbin-Salamon \cite{RoSa93}.
Let $(F,\Omega)$ be the symplectic space
$(\mathbb{R}^{2n}\oplus\mathbb{R}^{2n}, (-\omega_0)\oplus\omega_0)$
and let $\mathscr{L}(F,\Omega)$ denote the manifold of Lagrangian subspaces of $(F,\Omega)$.
For $M\in{\rm Sp}(2n,\R)$, both
\begin{center}
$W:=\{(x^T,x^T)^T\in\R^{4n}\,|\,x\in\R^{2n}\}$\quad  and\quad
${\rm Gr}(M):=\{(x^T,(Mx)^T)^T\in\R^{4n}\,|\,x\in\R^{2n}\}$
\end{center}
belong to $\mathscr{L}(F,\Omega)$.
The Robbin-Salamon index $\mu^{\rm RS}$ defined in \cite{RoSa93} assigns a half integer
$\mu^{\rm RS}(\Lambda,\Lambda')$ to every
pair of Lagrangian paths $\Lambda,\Lambda':[a,b]\to\mathscr{L}(F,\Omega)$.
 The \textsf{Conley-Zehnder index}  of $\gamma\in C([a,b],{\rm Sp}(2n,\R))$ was defined in \cite[Remark~5.35]{RoSa95}  by
\begin{eqnarray}\label{e:RobbinSaCZ}
\mu_{\rm CZ}(\gamma)=\mu^{\rm RS}({\rm Gr}\left(\gamma), W\right).
\end{eqnarray}
By \cite[Remark~5.4]{RoSa93} it holds that
\begin{eqnarray}\label{e:RobbinSaCZ+}
\mu_{\rm CZ}(\gamma)=i_{\rm CZ}(\gamma),\quad\forall \gamma\in\mathcal{P}^\ast_\tau(2n).
\end{eqnarray}

There exists a precise relation between $i(\gamma, [a,b])$ and $\mu_{\rm CZ}(\gamma)$
for each $\gamma\in C([a,b],{\rm Sp}(2n,\R))$.
 Recall that the Cappell-Lee-Miller index $\mu^{\rm CLM}_F$ characterized by properties
I-VI of \cite[pp. 127-128]{CLM}  assigns an integer $\mu^{\rm CLM}_F(\Lambda,\Lambda')$ to every
pair of Lagrangian paths $\Lambda,\Lambda':[a,b]\to\mathscr{L}(F,\Omega)$.
 It was proved in \cite[Corollary 2.1]{LongZhu00} that
\begin{eqnarray}\label{e:Long.2}
i_\tau(\gamma)=\mu^{\rm CLM}_F(W, {\rm Gr}(\gamma), [0,\tau])-n,\quad\forall \gamma\in\mathcal{P}_\tau(2n).
\end{eqnarray}
Then  for any $\gamma\in C([a,b],{\rm Sp}(2n,\R))$,
 (\ref{e:Long.1}), (\ref{e:Long.2}) and the path additivity of $\mu^{\rm CLM}_F$  lead to
\begin{eqnarray}\label{e:Long.3}
i(\gamma, [a,b])=\mu^{\rm CLM}_F(W, {\rm Gr}\left(\gamma),
[a,b]\right)
\end{eqnarray}
and therefore
\begin{eqnarray}\label{e:RobbinSa2}
i(\gamma, [a,b])
=\mu_{\rm CZ}(\gamma)-\frac{1}{2}(\dim
{\rm Ker}(I_{2n}-\gamma(b))-\dim
{\rm Ker}(I_{2n}-\gamma(a)))
\end{eqnarray}
by (\ref{e:RobbinSaCZ}) and \cite[Theorem~3.1]{LongZhu00}. In particular,
\begin{eqnarray}\label{e:RobbinSa3}
i_\tau(\gamma)=\mu_{\rm CZ}(\gamma)-\frac{1}{2}\dim
{\rm Ker}(I_{2n}-\gamma(\tau)),\quad\forall\gamma\in\mathcal{P}_\tau(2n).
\end{eqnarray}
For $\phi$ and $\beta$ in (\ref{e:Long.1}), since $\mu_{\rm CZ}$
 is additive with respect to the catenation of paths
we have also
\begin{eqnarray}\label{e:RobbinSa1}
\mu_{\rm CZ}(\gamma)=\mu_{\rm CZ}(\phi)-\mu_{\rm CZ}(\beta).
\end{eqnarray}
Both indexes $i$ and $\mu_{\rm CZ}$  also satisfy
vanishing, product and the following properties:
\begin{description}
\item[(Homotopy)] If a continuous path $\ell:[0,1]\to C([a,b],{\rm Sp}(2n,\R))$ is such that
\begin{eqnarray}\label{e:homotopy1}
s\mapsto\dim{\rm Ker}(I_{2n}-\ell(s)(a))\quad\hbox{and}\quad
s\mapsto\dim{\rm Ker}(I_{2n}-\ell(s)(b))
\end{eqnarray}
are constant functions on $[0,1]$, then
\begin{eqnarray}\label{e:RobbinSa4}
&&i(\ell(s), [a,b])=i(\ell(0), [a,b])\quad\forall s\in [0,1],\\
&&\mu_{\rm CZ}(\ell(s))=\mu_{\rm CZ}(\ell(0))\quad\forall s\in [0,1].\label{e:RobbinSa5}
\end{eqnarray}
\item[(Naturality)] For any $\phi, \gamma\in C([a,b],{\rm Sp}(2n,\R))$ it holds that
\begin{eqnarray}\label{e:RobbinSa6}
i(\phi\gamma\phi^{-1}, [a,b])=i(\gamma, [a,b])\quad\hbox{and}\quad
\mu_{\rm CZ}(\phi\gamma\phi^{-1})=\mu_{\rm CZ}(\gamma).
\end{eqnarray}
\end{description}

In fact, we have continuous maps $[0,1]\times[0,1]\to{\rm Sp}(2n,\R), (s,t)\mapsto\beta_s(t)$ given by
$$
\beta_s(t)=\beta_0\left(\frac{t}{1-s}\right)\;\hbox{for}\;0\le t\le 1-s,\qquad
\beta_s(t)=\ell\left(t-1+s\right)(a)\;\hbox{for}\;1-s\le t\le 1,
$$
and $[0,1]\times[0,1]\to{\rm Sp}(2n,\R), (s,t)\mapsto\phi_s(t)$ defined by
$$
\phi_s(t)=\beta_s\left(2t\right)\;\hbox{for}\;0\le t\le 1/2,\qquad
\phi_s(t)=\ell(s)\left(a+(2t-1)(b-a)\right)\;\hbox{for}\;1/2\le t\le 1.
$$
Then $\beta_s(0)=\phi_s(0)=I_{2n}$ for all $s\in [0,1]$, and $\nu_1(\beta_s)=\dim{\rm Ker}(I_{2n}-\ell(s)(a))$ and
$\nu_1(\phi_s)=\dim{\rm Ker}(I_{2n}-\ell(s)(b))$ are constant functions on $[0,1]$.
By (\ref{e:Long.1}), $i(\beta_s, [0,1])=i_1(\beta_s)$ for all $s\in [0,1]$.
The  homotopy invariance of ${i}_1$ leads to (\ref{e:RobbinSa4}), and therefore (\ref{e:RobbinSa5}) by (\ref{e:RobbinSa2}) and (\ref{e:homotopy1}).

In order to prove (\ref{e:RobbinSa6}), we consider a continuous path $\ell:[0,1]\to C([a,b],{\rm Sp}(2n,\R))$ given by $\ell(s)(t)=\phi(st)\gamma(t)(\phi(st))^{-1}$.
Then (\ref{e:RobbinSa5}) implies $\mu_{\rm CZ}(\phi\gamma\phi^{-1})=\mu_{\rm CZ}(\phi(0)\gamma(\phi(0))^{-1})$. Note that
\begin{eqnarray*}
{\rm Gr}(\phi(0)\gamma(\phi(0))^{-1})(t)=\{(x^T,(\phi(0)\gamma(t)(\phi(0))^{-1}x)^T)^T\in\R^{4n}\,|\,x\in\R^{2n}\}
=\Xi{\rm Gr}(\gamma)(t),
\end{eqnarray*}
where $\Xi:(F,\Omega)\to (F,\Omega)$ is the symplectic isomorphism given by
$\Xi(x\oplus y)=(\phi(0)x)\oplus(\phi(0)y)$. Since $\Xi W=W$,
by (\ref{e:RobbinSaCZ}) and the naturality property of $\mu^{\rm RS}$ (\cite[\S5]{RoSa95}) we obtain
\begin{eqnarray*}
\mu_{\rm CZ}(\phi(0)\gamma(\phi(0))^{-1})=\mu^{\rm RS}(\Xi{\rm Gr}(\gamma), W)=\mu^{\rm RS}({\rm Gr}(\gamma), W)
=\mu_{\rm CZ}(\gamma).
\end{eqnarray*}
The first equality in (\ref{e:RobbinSa6}) may follow from this and (\ref{e:RobbinSa2}).

Let $U_1=\{0\}\times\R^n$ and $U_2=\R^n\times\{0\}$.
For $\gamma\in{\cal P}_\tau(2n)$  with
$\gamma(\tau/2)={\scriptscriptstyle \left(\begin{array}{cc}
A& B\\
C& D\end{array}\right)}$, where $A, B, C, D\in\R^{n\times n}$,
Long, Zhang and Zhu \cite{LoZZ} defined
\begin{eqnarray}\label{e:2.3M}
&&\mu_{k, \tau}(\gamma)=\mu^{\rm CLM}_{\R^{2n}}\left(U_k,\gamma U_k,
[0, \tau/2]\right),\quad k=1,2,\\
&&\nu_{1,\tau}(\gamma)=\dim{\rm Ker}(B)\quad{\rm and}\quad \nu_{2,\tau}(\gamma)=\dim{\rm Ker}(C).\label{e:2.4M}
\end{eqnarray}

Dong \cite{Do06} and Liu \cite{Liu06} extended
the Maslov-type index $(i_\tau(\gamma),\nu_\tau(\gamma))$ of $\gamma\in\mathcal{P}_\tau(2n)$
 to the case relative to a given symplectic matrix $M\in{\rm Sp}(2n,\mathbb{R}))$
via two different methods; they denoted these two indexes by
$$
(i_{\tau, M}(\gamma),\nu_{\tau,M}(\gamma))\quad\hbox{and}\quad (i_{\tau}^M(\gamma),\nu_{\tau}^M(\gamma))
$$
respectively. (The former was written as  $(i_M(\gamma),\nu_M(\gamma))$ in \cite{Do06}).
 They are more suitable and convenient for dual variational methods and saddle point reduction ones, respectively.
 If $M=I_{2n}$, by \cite{Do06} and \cite[Definition~2.5 and Definition~2.6]{Liu06} there holds
\begin{equation}\label{e:dongLiuIndex+}
(i_{\tau,M}(\gamma),\nu_{\tau,M}(\gamma))=(i_{\tau}^M(\gamma),\nu_{\tau}^M(\gamma))
=(i_\tau(\gamma),\nu_\tau(\gamma)).
\end{equation}
Let  $[a]$  denote the largest integer no more than $a\in\R$, and let
 $\xi$ be any element in $\mathcal{P}_\tau(2n)$ satisfying $\xi(\tau)=M^{-1}$,
 Dong  defined
\begin{equation}\label{e:dongIndex}
\begin{array}{ll}
&\nu_{\tau,M}(\gamma)=\dim{\rm Ker}(\gamma(\tau)-M)\quad\hbox{and}\\
&i_{\tau,M}(\gamma)=[i_\tau((\gamma M^{-1})\ast\xi)-\triangle({\xi})/\pi]
\end{array}
\end{equation}
(\cite[Definitions~2.1, 2.2]{Do06}), and
 Liu (\cite[Definition~2.7 and Remark~2.8]{LiuTang15}) defined
\begin{equation}\label{e:LiuTang-index}
\begin{array}{ll}
&\nu_{\tau}^M(\gamma)=\dim{\rm Ker}(\gamma(\tau)-M)\quad\hbox{and}\\
&i^{M}_\tau(\gamma)=i_\tau((M^{-1}\gamma)\ast\xi)-i_\tau({\xi})
\end{array}
 \end{equation}
 if $M\ne I_{2n}$,  and $(i_{\tau}^M(\gamma),\nu_{\tau}^M(\gamma))=(i_\tau(\gamma),\nu_\tau(\gamma))$
  if $M= I_{2n}$.
 (Note that $i_\tau((M^{-1}\gamma)\ast\xi)$ may be replaced by $i_\tau((\gamma M^{-1})\ast\xi))$,
 see the proof of (\ref{e:LiuTang-index+}) below.
It was shown in \cite[Remark~2.8]{LiuTang15} that when $M=I_{2n}$
the right side of the second equality in (\ref{e:LiuTang-index}) is $i_\tau(\gamma)+n$.)
A direct observation leads to a precise relation between $i_{M,\tau}$ and $i^{M}_\tau$ as follows.

\begin{proposition}\label{prop:twoDef}
 For any $(M,\gamma)\in ({\rm Sp}(2n,\R)\setminus\{I_{2n}\})\times\mathcal{P}_\tau(2n)$
 and  any $\xi\in\mathcal{P}_\tau(2n)$ satisfying $\xi(\tau)=M^{-1}$
 it holds that
  \begin{equation}\label{e:indexrelation}
 i_{\tau,M}(\gamma)=[i^{M}_\tau(\gamma)+i_\tau(\xi)-\triangle({\xi})/\pi]=i^{M}_\tau(\gamma)+ [i_\tau(\xi)-\triangle({\xi})/\pi].
\end{equation}
 {\rm (}$[i_\tau(\xi)-\triangle({\xi})/\pi]$
only depends on $M$.{\rm )}
   \end{proposition}
   \begin{proof}[\bf Proof]
Clearly,  $M^{-1}\xi M\in\mathcal{P}_\tau(2n)$ and $(M^{-1}\xi M)(\tau)=M^{-1}$
 for any $\xi\in\mathcal{P}_\tau(2n)$ satisfying $\xi(\tau)=M^{-1}$.
By (\ref{e:LiuTang-index}) and  \cite[Corollary~6.5]{Long97}
\begin{eqnarray}\label{e:LiuTang-index+}
i^{M}_\tau(\gamma)&=&i_\tau\left((M^{-1}\gamma)\ast(M^{-1}\xi M)\right)-i_\tau(M^{-1}{\xi}M)\nonumber\\
&=&i_\tau\left(M^{-1}((\gamma M^{-1})\ast\xi))M\right)-i_\tau({\xi})\nonumber\\
&=&i_\tau\left((\gamma M^{-1})\ast\xi)\right)-i_\tau({\xi}).
 \end{eqnarray}
Substituting  this in the second equality in (\ref{e:dongIndex}) we obtain (\ref{e:indexrelation}).
\end{proof}

We notice that $i_{\tau,M}$ and $i^M_\tau$ may be different. For the symplectic path
 \begin{equation}\label{e:StanSymPath}
\xi_{2n,\tau}: [0, \tau]\to{\rm Sp}(2n,\mathbb{R}),\;t\mapsto I_{2n}.
\end{equation}
A direct computation yields
$i_{\tau}^{-J_n}(\xi_{2n,\tau})=0$ and $i_{\tau, -J_n}(\xi_{2n,\tau})=[n/2]$.

For any two $m\times m$ real symmetric matrices $A_1$ and $A_2$, we write $A_1\le A_2$
(resp. $A_1<A_2$)  if $A_2-A_1$ is positive semi-definite (resp. positive definite). For any
 $B_1,B_2\in L^\infty([0,\tau];\mathscr{L}_s(\mathbb{R}^{2n}))$
 we define $B_1\le B_2$ (resp. $B_1<B_2$) if $B_1(t)\le B_2(t)$ for a.e. $t\in [0,\tau]$
 (resp. if $B_1\le B_2$ and $B_1(t)<B_2(t)$ on a subset of $[0, \tau]$ with positive measure).

For $B, B_1,B_2\in L^\infty([0,\tau];\mathscr{L}_s(\mathbb{R}^{2n}))$ with $B_1<B_2$,
$\nu_{\tau,M}(B)$ and $I_{\tau,M}(B_1,B_2)$
were defined in \cite[Definitions~3.1, 3.2]{Do06} by
\begin{equation}\label{e:relativeMorse}
\left.\begin{array}{ll}
\nu_{\tau,M}(B):=\nu_M(\Upsilon_B)\quad\hbox{and}\\
I_{\tau,M}(B_1,B_2):=\sum_{s\in [0,1)}\nu_{\tau,M}((1-s)B_1+ sB_2).
\end{array}\right\}
\end{equation}
where $\Upsilon_B:[0,\tau]\to {\rm Sp}(2n,\mathbb{R})$ is
the fundamental matrix solution of the linear system
$\dot{Z}(t)=JB(t)Z(t)$ with $\Upsilon_B(0)=I_{2n}$.

For any $B_1,B_2\in L^\infty([0,\tau];\mathscr{L}_s(\mathbb{R}^{2n}))$,
and $s\in\R$ such that $sI_{2n}>B_1$ and $sI_{2n}>B_2$,
\begin{equation}\label{e:relativeMorse1}
I_{\tau,M}(B_1,B_2):=I_{\tau,M}(B_1, sI_{2n})- I_{\tau,M}(B_2, sI_{2n})
\end{equation}
is was called the \textsf{relative Morse index} between $B_1$ and $B_2$ (\cite[Definition~3.5]{Do06}).
The \textsf{$M$-index} of $B\in L^\infty([0,\tau];\mathscr{L}_s(\mathbb{R}^{2n}))$
was defined in \cite[Definition~3.8]{Do06} by
\begin{equation}\label{e:relativeMorse2}
j_{\tau,M}(B):=i_{\tau,M}(\xi_{2n})+ I_{\tau,M}(0, B),
\end{equation}
where $\xi_{2n}$ is as in (\ref{e:StanSymPath}).
 Then $j_{\tau,M}(0)=i_{\tau,M}(\xi_{2n})$ (\cite[Remark~3.9]{Do06}).

\begin{theorem}[\hbox{\cite[Theorem~4.1]{Do06}}]\label{th:twoIndex}
 $j_{\tau,M}(B)=i_{\tau,M}(\Upsilon_B)$ for any $B\in L^\infty([0,\tau];\mathscr{L}_s(\mathbb{R}^{2n}))$.
\end{theorem}

For $B\in L^\infty([0,\tau];\mathscr{L}_s(\mathbb{R}^{2n}))$, let  $A\in L^\infty([0,\tau];\mathscr{L}_s(\mathbb{R}^{2n}))$
satisfy $B-A\ge\varepsilon I_{2n}$ for $\varepsilon>0$ small enough, and let $C(t)=(B(t)-A(t))^{-1}$.
 Consider the quadratic form
\begin{equation}\label{e:quadratic}
q_{M, B|A}(u,v)=\frac{1}{2}\int^\tau_0\Bigl(((\hat{\Lambda}_{M,\tau, A})^{-1}u)(t), v(t))_{\mathbb{R}^{2n}}
+ (C(t)u(t), v(t))_{\mathbb{R}^{2n}}\Bigr) \, dt
\end{equation}
on the Hilbert space $\tilde{L}^{2}_{M, A}([0, \tau];\R^{2n})=R(\Lambda_{M,\tau, A})$ (cf. \S~2.1).
By \cite[Th.7.1]{He51}, $\tilde{L}^{2}_{M, A}([0, \tau];\R^{2n})$
 is expressible in a unique manner as the direct sum of three linear subspaces
 $$E_M^-(B|A),\quad E_M^0(B|A),\quad E_M^+(B|A)$$
  having the following properties:
  \begin{enumerate}
\item[\rm (a)] $E_M^-(B|A)$, $E_M^0(B|A)$, $E_M^+(B|A)$ are mutually $L^2$-orthogonal and
$q_{M, B|A}$-orthogonal;
\item[\rm (b)] $q_{M, B|A}$ is negative on $E_M^-(B|A)$, zero on $E_M^0(B|A)$, and positive on $E_M^+(B|A)$.
\end{enumerate}
Moreover, since $2q_{M, B|A}(u,v)=((\hat{\Lambda}_{M,\tau, A})^{-1}u, v)_{L^2}+ (\bar{C}u,v)_{L^2}$
for a compact self-adjoint operator $\bar{C}:\tilde{L}^{2}_{M, A}([0, \tau];\R^{2n})\to \tilde{L}^{2}_{M, A}([0, \tau];\R^{2n})$ defined by
$$
(\bar{C}u,v)_{L^2}=\int^\tau_0(C(t)u(t), v(t))_{\mathbb{R}^{2n}}\, dt,
$$
 namely
$q_{M, B|A}$ is a Legendre form by \cite[Theorem~11.6]{He51},
the subspaces $E_M^-(B|A)$ and $E_M^0(B|A)$ have finite dimensions (cf.
\cite[Theorem~11.3]{He51} or Corollary on page 311 of \cite{He61}).
Dong \cite{Do06} also proved these  directly   and called
\begin{equation}\label{e:M-index}
j_{\tau,M}(B|A):=\dim E_M^-(B|A)\quad\hbox{and}\quad \nu_{\tau,M}(B|A):=\dim E_M^0(B|A)
\end{equation}
\textsf{$M$-index} and \textsf{$M$-nullity of $B$ with respect to $A$}, respectively.
Let $m^-(q_{M, B|A})$ and $m^0(q_{M, B|A})$ denote the Morse index and the nullity
of $q_{M, B|A}$ on $\tilde{L}^{2}_{M, A}([0, \tau];\R^{2n})$, respectively.
Then
\begin{equation}\label{e:Nullity}
m^0(q_{M, B|A})=\nu_{\tau,M}(B|A)=\nu_{\tau,M}(B)=\nu_{\tau,M}(\Upsilon_B),
\end{equation}
where the second equality comes from \cite[Proposition~3.14]{Do06}.
By \cite[Remark~3.24]{Do06} and Theorem~\ref{th:twoIndex} we have also
\begin{equation}\label{e:MorseIndex0}
\left.\begin{array}{ll}
j_{\tau,M}(B|A)&=j_{\tau,M}(B)- j_{\tau,M}(A)-\nu_{\tau,M}(A)\\
&=i_{\tau,M}(\Upsilon_B)- i_{\tau,M}(\Upsilon_A)- \nu_{\tau,M}(\Upsilon_A).
\end{array}\right\}
\end{equation}
It follows from these, and  \cite[Theorem~9.1]{He51} or \cite[Lemma~3.16]{Do06} that
\begin{equation}\label{e:MorseIndex}
 m^-(q_{M, B|A})=j_{\tau,M}(B|A)=i_{\tau,M}(\Upsilon_B)- i_{\tau,M}(\Upsilon_A)- \nu_{\tau,M}(\Upsilon_A).
\end{equation}
For $B_1,B_2\in L^\infty([0,\tau];\mathscr{L}_s(\mathbb{R}^{2n}))$ with $B_2>B_1>A+\varepsilon I_{2n}$
for a small enough $\varepsilon>0$, it was proved in
\cite[Lemma~3.18, Proposition~3.25]{Do06} that
\begin{eqnarray}\label{e:MorseRelation3}
 j_{\tau,M}(B_2|A)&\ge& j_{\tau,M}(B_1|A)+\nu_{\tau,M}(B_1|A),\\
 j_{\tau,M}(B_2)&\ge& j_{\tau,M}(B_1)+\nu_{\tau,M}(B_1).\label{e:MorseRelation4}
\end{eqnarray}
The former is also a special case of \cite[Theorem~9.4]{He51}.

\begin{remark}\label{rm:Liu}
{\rm  When $M=I_{2n}$ and $A=\ell I_{2n}$ with $\ell\in\mathbb{R}\setminus 2\pi\mathbb{Z}$,
based on  the arguments in \cite{GirMa} and the saddle point reduction (cf. \cite[\S6.1, Theorem~1]{Long02})
Liu  \cite[Theorem~2.1]{Liu05}  proved the following precise versions of (\ref{e:MorseIndex0}) and (\ref{e:Nullity}) (in terms of the above notations)
 \begin{eqnarray}\label{e:LiuC1}
&&j_{1, I_{2n}}(B|\ell I_{2n})=i_1(\Upsilon_B)+ n+ 2n\left[-\frac{\ell}{2\pi}\right],\\
&&\nu_{1, I_{2n}}(B|\ell I_{2n})=\nu_1(\Upsilon_B).\label{e:LiuC2}
\end{eqnarray}}
\end{remark}

For $B\in L^\infty([0,\tau];\mathscr{L}_s(\mathbb{R}^{2n}))$ and $0\le t\le\tau$,
let
$$
\nu_{t,M}(\Upsilon_B|_{[0,t]})=\dim{\rm Ker}(\Upsilon_B(t)-M)
$$
and so  $\nu_{0,M}(\Upsilon_B|_{[0,0]})=\dim{\rm Ker}(I_{2n}-M)$.
As a generalization of  \cite[Chapter I, \S4, Theorem~6]{Ek90}
and \cite[Proposition~15.3]{Long02}, 
it was proved in \cite[Lemma~4.2]{Do06} that
\begin{equation}\label{e:positive-negativeA}
i_{\tau,M}(\Upsilon_B)=i_{\tau,M}(\xi_{2n})+\sum_{0\le t<\tau}\nu_{t,M}(\Upsilon_B|_{[0,t]})
\end{equation}
if $B$ is positive definite, and
\begin{equation}\label{e:positive-negativeB}
i_{\tau,M}(\Upsilon_B)=i_{\tau,M}(\xi_{2n})
-\sum_{0<t\le\tau}\nu_{t,M}(\Upsilon_B|_{[0,t]})
\end{equation}
if $B$ is negative definite. Here $\xi_{2n}$ is as in (\ref{e:StanSymPath}).

\begin{remark}\label{rm:DongLiu}
{\rm
 The sum $\sum_{0<t\le\tau}\nu_{t,M}(\Upsilon_B|_{[0,t]})$ in (\ref{e:positive-negativeB})
 was  written as $\sum_{0<t<\tau}\nu_{t,M}(\Upsilon_B|_{[0,t]})$ in \cite[(4.2)]{Do06}.
Actually, it was only asserted in \cite{Do06} that \cite[(4.2)]{Do06} can be obtained
in a similar way to the proof of (\ref{e:positive-negativeA}) \cite[(4.1)]{Do06}.
Since Proposition~\ref{prop:twoDef} implies that
$i_{\tau,M}(\gamma)-i_{\tau,M}(\xi_{2n,\tau})=i^{M}_\tau(\gamma)-i^{M}_\tau(\xi_{2n,\tau})$, and
$$
i_{\tau}^M(\xi_{2n,\tau})-i_{\tau}^M(\Upsilon_B)=
\sum_{0<t\le\tau}\nu_{t,M}(\Upsilon_B|_{[0,t]})
$$
by \cite[(4.8)]{Liu17+}, (\ref{e:positive-negativeB})  follows.
}
\end{remark}

From (\ref{e:positive-negativeA}) and  (\ref{e:positive-negativeB}) we derive:

\begin{proposition}\label{prop:Index}
Let $A\in\mathscr{L}_s(\mathbb{R}^{2n})$ be either positive definite or negative definite,
and $\Upsilon_{\lambda A}(t)=\exp(\lambda tJA)$ for $\lambda\in\mathbb{R}\setminus\{0\}$.
 {\rm (}Clearly, $\Upsilon_{\lambda A}(t)=\Upsilon_A(\lambda t)$.{\rm )} Then
\begin{eqnarray*}
\Gamma(A):=\{t\in\mathbb{R}\setminus\{0\}\,|\, \dim{\rm Ker}(\exp(tJA)-M)>0\}
\end{eqnarray*}
is a discrete set in $\mathbb{R}$, and the following holds.
{\rm (}$\lambda_0=0$ and so $\Upsilon_{\lambda_0 A}(t)=I_{2n}$ for all $t\in \mathbb{R}$.{\rm )}
  \begin{enumerate}
\item[\rm (i)] Let $A>0$. If $\Gamma(A)\cap (0, \infty)=\{\lambda_1<\cdots<\lambda_k<\cdots<\lambda_\kappa\}$ with $\kappa\in\mathbb{N}\cup\{\infty\}$,
then
\begin{equation}\label{e:positive-negative3}
i_{1,M}(\Upsilon_{\lambda A})
=\left\{\begin{array}{ll}
i_{1,M}(\Upsilon_{\lambda_k A})\quad&\hbox{$\lambda_{k-1}<\lambda\le\lambda_{k}$ with $k\ge 1$},\\
i_{1,M}(\Upsilon_{\lambda_k A})+\nu_{1,M}(\Upsilon_{\lambda_k A})
\quad&\hbox{if $\lambda_k<\lambda\le\lambda_{k+1}$ with $k\ge 0$},\\
i_{1,M}(\Upsilon_{\lambda_{\kappa}A})+\nu_{1,M}(\Upsilon_{\lambda_{\kappa}A})
\quad&\hbox{if $\kappa\in\mathbb{N}$ and $\lambda>\lambda_{\kappa}$};
\end{array}\right.
\end{equation}
and if $\Gamma(A)\cap (-\infty, 0)=\{\lambda_{-\kappa}\cdots<\lambda_{-k}<\cdots<\lambda_{-1}\}$ with
$\kappa\in\mathbb{N}\cup\{\infty\}$,
then
$$
i_{1,M}(\Upsilon_{\lambda A})
=\left\{\begin{array}{ll}
i_{1,M}(\Upsilon_{\lambda_{-k}A})+ \nu_{1,M}(\Upsilon_{\lambda_{-k}A})
\quad&\hbox{if $\lambda_{-k}<\lambda<\lambda_{-k+1}$ with $k\ge 1$},\\
i_{1,M}(\Upsilon_{\lambda_{-k}A})
\quad&\hbox{if $\lambda_{-k-1}<\lambda\le\lambda_{-k}$ with $k\ge 1$},\\
i_{1,M}(\Upsilon_{\lambda_{-\kappa}A})
\quad&\hbox{if $\kappa\in\mathbb{N}$ and $\lambda\le\lambda_{-\kappa}$}.
\end{array}\right.
$$
\item[\rm (ii)] Let $A<0$. If $\Gamma(A)\cap (0, \infty)=
\{\lambda_1<\cdots<\lambda_k<\cdots<\lambda_\kappa\}$
 with $\kappa\in\mathbb{N}\cup\{\infty\}$, then
$$
i_{1,M}(\Upsilon_{\lambda A})
=\left\{\begin{array}{ll}
i_{1,M}(\Upsilon_{\lambda_{k}A})+\nu_{1,M}(\Upsilon_{\lambda_{k}A})
\quad&\hbox{if $\lambda_{k-1}<\lambda<\lambda_{k}$ with $k\ge 1$},\\
i_{1,M}(\Upsilon_{\lambda_{k}A})\quad&\hbox{if $\lambda_k\le\lambda<\lambda_{k+1}$ with $k\ge 1$},\\
i_{1,M}(\Upsilon_{\lambda_{\kappa}A})
\quad&\hbox{if $\kappa\in\mathbb{N}$ and $\lambda\ge\lambda_{\kappa}$};
\end{array}\right.
$$
and if $\Gamma(A)\cap (-\infty, 0)=\{\lambda_{-\kappa}\cdots<\lambda_{-k}<\cdots<\lambda_{-1}\}$ with
 $\kappa\in\mathbb{N}\cup\{\infty\}$,
then
\begin{equation}\label{e:positive-negative4}
i_{1,M}(\Upsilon_{\lambda A})
=\left\{\begin{array}{ll}
i_{1,M}(\Upsilon_{\lambda_{-k}A})\quad&\hbox{if $\lambda_{-k}\le\lambda<\lambda_{-k+1}$ with $k\ge 1$},\\
i_{1,M}(\Upsilon_{\lambda_{-k}A})+\nu_{1,M}(\Upsilon_{\lambda_{-k}A})
\quad&\hbox{if $\lambda_{-k-1}\le\lambda<\lambda_{-k}$ and $k\ge 1$},\\
i_{1,M}(\Upsilon_{\lambda_{-\kappa}A})+\nu_{1,M}(\Upsilon_{\lambda_{-\kappa}A})
\quad&\hbox{if $\kappa\in\mathbb{N}$ and $\lambda<\lambda_{-\kappa}$}.
\end{array}\right.
\end{equation}
 \end{enumerate}
\end{proposition}
\begin{proof}[\bf Proof]
By (\ref{e:positive-negativeA}) and  (\ref{e:positive-negativeB}) we derive
\begin{eqnarray*}
&&i_{1,M}(\Upsilon_{\lambda A})-i_{1,M}(\Upsilon_{\lambda_0 A})
=\left\{\begin{array}{ll}
&\sum_{0\le t<\lambda}\dim{\rm Ker}(\exp(tJA)-M)\quad\hbox{if $A>0$},\vspace{1mm}\\
&-\sum_{0<t\le\lambda}\dim{\rm Ker}(\exp(tJA)-M)\quad\hbox{if $A<0$}
\end{array}\right.
\end{eqnarray*}
for $\lambda>0$, and
\begin{eqnarray*}
i_{1,M}(\Upsilon_{\lambda A})-i_{1,M}(\Upsilon_{\lambda_0 A})
=\left\{\begin{array}{ll}
&\sum_{\lambda<t\le 0}\dim{\rm Ker}(\exp(tJA)-M)\quad\hbox{if $A<0$},\vspace{1mm}\\
&-\sum_{\lambda\le t<0}\dim{\rm Ker}(\exp(tJA)-M)\quad\hbox{if $A>0$}
\end{array}\right.
\end{eqnarray*}
for $\lambda<0$. These imply that $\sum_{-N<t<N}\dim{\rm Ker}(\exp(tJA)-M)<\infty$ for any $N>0$.
The first claim is proved.

{\bf (i)}[$A>0$].\quad Let $\Gamma(A)\cap (0, \infty)=\{\lambda_1<\cdots<\lambda_k<\cdots<\lambda_\kappa\}$
with $\kappa\in\mathbb{N}\cup\{\infty\}$.
Recall that $\nu_{t,M}(\Upsilon_{\lambda A}|_{[0,t]})=\dim{\rm Ker}(\exp(t\lambda A)-M)$.
 By (\ref{e:positive-negativeA}) we have
 $$
i_{1,M}(\Upsilon_{\lambda A})=i_{1,M}(\Upsilon_{\lambda_0 A})
+\sum_{0\le t<\lambda}\dim{\rm Ker}(\exp(tJA)-M)=i_{1,M}(\Upsilon_{\lambda_k A})
$$
for $\lambda_{k-1}<\lambda\le\lambda_{k}$ with $k\ge 1$, and
\begin{eqnarray*}
i_{1,M}(\Upsilon_{\lambda A})&=&i_{1,M}(\Upsilon_{\lambda_0 A})
+\sum_{0\le t<\lambda}\dim{\rm Ker}(\exp(tJA)-M)\\
&=&i_{1,M}(\Upsilon_{\lambda_0 A})+\sum_{0\le t<\lambda_k}\dim{\rm Ker}(\exp(tJA)-M)
+\dim{\rm Ker}(\exp(\lambda_kJA)-M)\\
&=&i_{1,M}(\Upsilon_{\lambda_kA})+ \nu_{1,M}(\Upsilon_{\lambda_kA})
\end{eqnarray*}
for $\lambda_k<\lambda\le\lambda_{k+1}$ with $k\ge 0$.
If $\kappa\in\mathbb{N}$, then the last two equalities hold true for $k=\kappa$ and
any $\lambda>\lambda_{\kappa}$.

Next, let $\Gamma(A)\cap (-\infty, 0)=\{\lambda_{-\kappa}\cdots<\lambda_{-k}<\cdots<\lambda_{-1}\}$ with $\kappa\in\mathbb{N}\cup\{\infty\}$.
By (\ref{e:positive-negativeB}),
\begin{eqnarray*}
i_{1,M}(\Upsilon_{\lambda A})
&=&i_{1,M}(\Upsilon_{\lambda_0 A})-\sum_{\lambda\le t<0}\dim{\rm Ker}(\exp(t JA)-M)\\
&=&i_{1,M}(\Upsilon_{\lambda_0 A})-\sum_{\lambda_{-k}\le t<0}\dim{\rm Ker}(\exp(t JA)-M)+\nu_{1,M}(\Upsilon_{\lambda_{-k}A})\\
&=&i_{1,M}(\Upsilon_{\lambda_{-k}A})+ \nu_{1,M}(\Upsilon_{\lambda_{-k}A})
\end{eqnarray*}
 for $\lambda_{-k}<\lambda<\lambda_{-k+1}$ with $k\ge 1$, and
\begin{eqnarray*}
i_{1,M}(\Upsilon_{\lambda A})
&=&i_{1,M}(\Upsilon_{\lambda_0 A})-\sum_{\lambda\le t<0}\dim{\rm Ker}(\exp(t JA)-M)\\
&=&i_{1,M}(\Upsilon_{\lambda_0 A})-\sum_{\lambda_{-k}\le t<0}\dim{\rm Ker}(\exp(t JA)-M)=i_{1,M}(\Upsilon_{\lambda_{-k}A})
\end{eqnarray*}
for $\lambda_{-k-1}<\lambda\le\lambda_{-k}$ with $k\ge 1$.
Clearly, when $\kappa\in\mathbb{N}$, the last two equalities also hold true for $k=\kappa$ and
any $\lambda\le\lambda_{-\kappa}$.

{\bf (ii)}[$A<0$].\quad
Assume that $\Gamma(A)\cap (0, \infty)=\{\lambda_1<\cdots<\lambda_k<\cdots<\lambda_\kappa\}$ with $\kappa\in\mathbb{N}\cup\{\infty\}$.
By (\ref{e:positive-negativeB}) we have
\begin{eqnarray*}
i_{1,M}(\Upsilon_{\lambda A})&=&i_{1,M}(\Upsilon_{\lambda_0 A})-\sum_{0<t\le\lambda}\dim{\rm Ker}(\exp(tJA)-M)
\\
&=&i_{1,M}(\Upsilon_{\lambda_0 A})-\sum_{0< t\le\lambda_k}\dim{\rm Ker}(\exp(tJA)-M)+\nu_{1,M}(\Upsilon_{\lambda_kA})
\\
&=&i_{1,M}(\Upsilon_{\lambda_kA})+\nu_{1,M}(\Upsilon_{\lambda_kA})
\end{eqnarray*}
for $\lambda_{k-1}<\lambda<\lambda_{k}$ with $k\ge 1$, and
\begin{eqnarray*}
i_{1,M}(\Upsilon_{\lambda A})&=&i_{1,M}(\Upsilon_{\lambda_0 A})
-\sum_{0< t\le\lambda}\dim{\rm Ker}(\exp(tJA)-M)\\
&=&i_{1,M}(\Upsilon_{\lambda_0 A})-\sum_{0< t\le\lambda_k}\dim{\rm Ker}(\exp(tJA)-M)
=i_{1,M}(\Upsilon_{\lambda_kA})
\end{eqnarray*}
for $\lambda_k\le\lambda<\lambda_{k+1}$ with $k\ge 1$.
If $\kappa\in\mathbb{N}$,  the last two lines also hold true for $k=\kappa$ and
any $\lambda\ge\lambda_{\kappa}$.

When $\Gamma(A)\cap (-\infty, 0)=\{\lambda_{-\kappa}\cdots<\lambda_{-k}<\cdots<\lambda_{-1}\}$ with $\kappa\in\mathbb{N}\cup\{\infty\}$,
by (\ref{e:positive-negativeA}) we obtain
\begin{eqnarray*}
i_{1,M}(\Upsilon_{\lambda A})
=i_{1,M}(\Upsilon_{\lambda_0 A})+\sum_{\lambda< t\le 0}\dim{\rm Ker}(\exp(tJA)-M)
=i_{1,M}(\Upsilon_{\lambda_{-k} A})
\end{eqnarray*}
for $\lambda_{-k}\le\lambda<\lambda_{-k+1}$ with $k\ge 1$, and
\begin{eqnarray*}
i_{1,M}(\Upsilon_{\lambda A})
&=&i_{1,M}(\Upsilon_{\lambda_0 A})+\sum_{\lambda< t\le 0}\dim{\rm Ker}(\exp(tJA)-M)\\
&=&i_{1,M}(\Upsilon_{\lambda_{-k} A})+ \nu_{1,M}(\Upsilon_{\lambda_{-k}A})
\end{eqnarray*}
for $\lambda_{-k-1}\le\lambda<\lambda_{-k}$ with $k\ge 1$.
Clearly, when $\kappa\in\mathbb{N}$, the last two lines also hold true for $k=\kappa$ and
any $\lambda<\lambda_{-\kappa}$.
\end{proof}

Let $N$ be as in Assumption~\ref{ass:brake},
and let $\mu_{k, \tau}$ and $\nu_{k,\tau}$, $k=1,2$, be as in (\ref{e:2.3M}) and (\ref{e:2.4M}).

\begin{assumption}\label{ass:brakeB}
{\rm \begin{enumerate}
\item[(B1)] Let $B\in C(\R, \R^{2n\times 2n})$
be a path of symmetric matrix which is $\tau$-periodic in time $t$,
i.e., $B(t+\tau)=B(t)$ for any $t\in\R$.
\item[(B2)] $N^TB(-t)N=B(t)$, that is, if  $B(t)=\left(\begin{array}{cc}B_{11}(t) & B_{12}(t)\\
B_{21}(t)& B_{22}(t)\end{array}\right)$, then  $B_{11},
B_{22}:\;\R\to\R^{n\times n}$ are \verb"even" at $t=0$ and
$\tau/2$, and $B_{12}, B_{21}:\;\R\to\R^{n\times n}$ are
\verb"odd" at $t=0$ and $\tau/2$.
\end{enumerate} }
\end{assumption}

Note that the operator $\Lambda_{I_{2n},\tau, 0}$ defined by (\ref{e:LambdaKA}) has the spectrum $\frac{2\pi}{\tau}\mathbb{Z}$.
Under the assumption (B1) we can choose $\kappa\notin \frac{2\pi}{\tau}\mathbb{Z}$ such that
 $B(t)-\kappa I_{2n}\ge\varepsilon I_{2n}$ for all $t$ and  small enough $\varepsilon>0$.

Consider the Hilbert subspaces of $L^{2}(S_\tau; \R^{2n})$ and $W^{1,2}(S_\tau;\R^{2n})$,
\begin{eqnarray}\label{e:Lbrake}
&&\textsf{L}_\tau=\left\{\,z\in L^2(S_\tau;
\R^{2n})\,|\,z(-t)=Nz(t)\,{\rm a.e.}t\in\R\right\},\\
&&\textsf{W}_\tau=\left\{\,z\in W^{1,2}(S_\tau;\R^{2n})\,|\,z(-t)=Nz(t)\,{\rm a.e.}\,
t\in\R\right\}.\label{e:Wbrake}
\end{eqnarray}
Then $\textsf{W}_\tau=W^{1,2}(S_\tau; \R^{2n})\cap \textsf{L}_\tau$.
Clearly, the operator $\Lambda_{I_{2n},\tau, 0}$ maps $\textsf{W}_\tau$ into $\textsf{L}_\tau$.
Thus it restricts to a closed linear and self-adjoint operator on $\textsf{L}_\tau$ with domain $\textsf{W}_\tau$,
denoted by $\tilde{\Lambda}_{I_{2n},\tau, 0}$ for clearness. Note that the spectrum of $\tilde{\Lambda}_{I_{2n},\tau, 0}$
 is also $\frac{2\pi}{\tau}\mathbb{Z}$.
Choose $\kappa\notin \frac{2\pi}{\tau}\mathbb{Z}$ as above. Then the operator
$$
\tilde{\Lambda}_{I_{2n},\tau, \kappa I_{2n}}: \textsf{W}_\tau\subset \textsf{L}_\tau\to \textsf{L}_\tau,\;u\mapsto \tilde{\Lambda}_{I_{2n},\tau, 0}-\kappa u
$$
 is invertible and the inverse $(\tilde{\Lambda}_{I_{2n},\tau, \kappa I_{2n}})^{-1}:
\textsf{L}_\tau\to \textsf{L}_\tau$ is compact and self-adjoint.
It follows under the assumptions (B1) and (B2) that the quadratic form on $\textsf{L}_\tau$,
\begin{equation}\label{e:Bquadratic}
\textsf{Q}_{B,\kappa I_{2n}}(u,v)=\frac{1}{2}\int^\tau_0\Bigl(((\tilde{\Lambda}_{I_{2n},\tau, \kappa I_{2n}})^{-1}u)(t), v(t))_{\mathbb{R}^{2n}}
+ (C(t)u(t), v(t))_{\mathbb{R}^{2n}}\Bigr) \, dt
\end{equation}
with $C(t)=(B(t)-\kappa I_{2n})^{-1}$,
  is a Legendre form, and hence has finite  Morse index $m^-(\textsf{Q}_{B,\kappa I_{2n}})$ and nullity $m^0(\textsf{Q}_{B,\kappa I_{2n}})$.
Using \cite[Theorem~5.1]{LoZZ} and suitably modifying the proof of  \cite[Theorem~2.1]{Liu05}
we can get the corresponding results with  (\ref{e:LiuC1})-(\ref{e:LiuC2}).

\begin{theorem}\label{th:brakeIndex}
\begin{eqnarray}\label{e:brakeIndex1}
&&m^-(\textsf{Q}_{B,\kappa I_{2n}})=\mu_{1,\tau}(\Upsilon_B)-n\left[\frac{\kappa\tau}{2\pi}\right],\\
&&m^0(\textsf{Q}_{B,\kappa I_{2n}})=\nu_{1,\tau}(\Upsilon_B),\label{e:brakeIndex2}
\end{eqnarray}
\end{theorem}

If $u\in\textsf{L}_\tau$ belongs to the zero space of $\textsf{Q}_{B,\kappa I_{2n}}$, i.e.,
$\textsf{Q}_{B,\kappa I_{2n}}(u,v)=0\;\forall v\in\textsf{L}_\tau$, it easily follows that
$w:=(\tilde{\Lambda}_{I_{2n},\tau, \kappa I_{2n}})^{-1}u\in \textsf{W}_\tau$
satisfies $J\dot{w}(t)+B(t)w(t)=0$
for a.e. $t\in\mathbb{R}$. Thus $u\ne 0$ implies that $w(t)\ne 0$ for a.e. $t\in\mathbb{R}$.
Repeating the proof of \cite[Lemma~3.18]{Do06} we may obtain:

\begin{theorem}\label{th:brakeIndexMono}
 Let $B_1, B_2\in C(\R, \R^{2n\times 2n})$ satisfy Assumption~\ref{ass:brakeB}. Assume $B_2>B_1$. Then
for each $\kappa\notin \frac{2\pi}{\tau}\mathbb{Z}$ such that
 $B_i(t)-\kappa I_{2n}\ge\varepsilon I_{2n}$ for all $t$ and  small enough $\varepsilon>0$, $i=1,2$, it holds that
\begin{eqnarray*}
m^-(\textsf{Q}_{B_2,\kappa I_{2n}})\ge m^-(\textsf{Q}_{B_1,\kappa I_{2n}})+  m^0(\textsf{Q}_{B_1,\kappa I_{2n}})
\end{eqnarray*}
or equivalently $\mu_{1,\tau}(\Upsilon_{B_2})\ge \mu_{1,\tau}(\Upsilon_{B_1})+  \nu_{1,\tau}(\Upsilon_{B_1})$.
\end{theorem}

Using the last two theorems, it is straightforward to give corresponding results with
\cite[\S4]{Liu05} and \cite[\S5]{Do06} for the system (\ref{e:Pbrake}) under suitable
conditions on $H$.
They will be provided elsewhere.

\section{Proof of Proposition~\ref{prop:threeBifu}}\label{app:threeBifu}

Since the operator $\Lambda_{M,\tau}\equiv \Lambda_{M,\tau, 0}$ defined by (\ref{e:LambdaKA})
is a self-adjoint unbounded linear operator in $L^2([0,\tau],\mathbb{R}^{2n})$  with domain
${\rm dom}(\Lambda_{M,\tau})=W^{1,2}_{M}([0,\tau]; \mathbb{R}^{2n})$,
by \cite[Remarks~2.2,2.3]{JinLu1916}, all eigenvalues of $\Lambda_{M,\tau}$ have form
\begin{equation*}
\cdots\leq \lambda_{-k}\le\cdots\leq \lambda_{-1}<0<\lambda_1\leq\cdots\leq \lambda_k\leq \cdots
\end{equation*}
 and there exists a unit orthogonal  basis of $L^2([0,\tau],\mathbb{R}^{2n})$,
 $\{e_j\,|\,\pm j\in\mathbb{N}\}\cup\{e_0^i\}_{i=1}^q$,
such that  $|e_j(t)|=|e_0^i(t)|\equiv 1\;\forall t$,
${\rm Ker}(\Lambda_{M,\tau})={\rm Span}(\{e_0^i\}_{i=1}^q)$ and that each $e_j$
is an eigenvector corresponding to $\lambda_j$, $j=\pm 1,\pm 2,\cdots$.

For $s\geq 0$ we say that $x\in L^2([0,\tau],\mathbb{R}^{2n})$ belongs to $E^s_M$ if and only if
$x=\sum_{k\in\mathbb{Z}}x_ke_k$ satisfies
\begin{equation*}
\sum_{k\in\mathbb{Z}}|\lambda_k|^{2s}x_k^2<\infty.
\end{equation*}
It is easy to prove that $E^s_M$ is a Hilbert space with respect to the inner product
\begin{equation*}
\langle x,y\rangle_{E^s_M} =\langle x_0,y_0\rangle_{\mathbb{R}^{2n}}+
\sum_{k\neq 0} |\lambda_k|^{2s}x_ky_k,\quad x, y\in E^s_M.
\end{equation*}
 Denote the associated norm by $\|\cdot\|_{E^s_M}$.
Note that $E^0_M=L^2([0,\tau],\mathbb{R}^{2n})$ and $\|\cdot\|_{E^0_M}=\|\cdot\|_{L^2}$.
For $t>s\ge 0$, the inclusion $E^t_M\hookrightarrow E^s_M$ is compact (\cite[Proposition~2.5]{JinLu1916}).

\begin{proof}[\bf Proof of Proposition~\ref{prop:threeBifu}]
{\it Proof of} (i).
 Let $(\lambda_k, \bar{u}_k)$ be a sequence of solutions of the problem
(\ref{e:Hboundary}) converging to $(\mu, u_\mu)$ in $\Lambda\times C^{0}([0,\tau];\R^{2n})$.
It suffices to prove that
$$
|\dot{\bar{u}}_k(t)-\dot{u}_\mu(t)|=|J\nabla_zH(\lambda_k,t, \bar{u}_k(t))-J\nabla_zH(\mu,t, u_\mu(t))|=
|\nabla_zH(\lambda_k,t, \bar{u}_k(t))-\nabla_zH(\mu,t, u_\mu(t))|
$$
uniformly converges to zero on $[0,\tau]$. Otherwise, there exists $\varepsilon>0$ and subsequences
$(k_i)$ and $(t_i)\subset [0, \tau]$ such that
\begin{equation}\label{e:threeBifu0}
|\nabla_zH(\lambda_{k_i}, t_i, \bar{u}_{k_i}(t_i))-\nabla_zH(\mu,t_i, u_\mu(t_i))|\ge\varepsilon,\quad\forall i=1,2,\cdots.
\end{equation}
Passing to a subsequence (if necessary) we can assume $t_i\to t_0\in [0, \tau]$. Note that
\begin{eqnarray*}
|\bar{u}_{k_i}(t_i)-u_\mu(t_0)|&\le&|\bar{u}_{k_i}(t_i)-u_\mu(t_i)|+|u_\mu(t_i)-u_\mu(t_0)|\\
&\le&\|\bar{u}_{k_i}-u_\mu\|_{C^0}+|u_\mu(t_i)-u_\mu(t_0)|\to 0.
\end{eqnarray*}
Since $\nabla_zH(\lambda,t,z)$ is continuous in  $(\lambda, t, z)\in\Lambda\times [0,\tau]\times\mathbb{R}^{2n}$ by Assumption~\ref{ass:BasiAss1},
letting $i\to\infty$ in (\ref{e:threeBifu0}) we obtain
$0=|\nabla_zH(\mu, t_0, u_\mu(t_0))-\nabla_zH(\mu,t_0, u_\mu(t_0))|\ge \varepsilon$, which is a contradiction.

{\it Proof of} (ii).
 Let $(\lambda_k, \bar{u}_k)$ be a sequence of solutions of the problem
(\ref{e:Hboundary}) converging to $(\mu, u_\mu)$ in $\Lambda\times \mathbb{E}$.
By \cite[Proposition~2.10]{JinLu1916}, all $\bar{u}_k$ and $u_\mu$ are $C^1$.

Firstly, we assume $M\ne I_{2n}$. By \cite[Proposition~2.5]{JinLu1916},
we have a compact inclusion $\mathbb{E}\to L^{2}([0,\tau];\R^{2n})$ and therefore,
up to pass to a subsequence,
\begin{equation}\label{e:threeBifu}
\|{\bar{u}}_k-{u}_\mu\|_{L^r}\to 0
\end{equation}
because $1<r\le 2$. Note that
\begin{eqnarray}\label{e:threeBifu1}
\|\dot{\bar{u}}_k-\dot{u}_\mu\|_{L^1}&=&
\int^\tau_0|\nabla_zH(\lambda_k,t, \bar{u}_k(t))-\nabla_zH(\mu,t, u_\mu(t))|dt\nonumber\\
&\le&\int^\tau_0|\nabla_zH(\lambda_k,t, {u}_\mu(t))-\nabla_zH(\mu,t, u_\mu(t))|dt\nonumber\\
&&+\int^\tau_0|\nabla_zH(\lambda_k,t, \bar{u}_k(t))-\nabla_zH(\lambda_k,t, u_\mu(t))|dt.
\end{eqnarray}
Since $\nabla_zH$  is continuous, by (\ref{e:growth})
it follows from Lebesgue dominated convergence theorem that
\begin{equation}\label{e:threeBifu2}
\lim_{k\to\infty}\int^\tau_0|\nabla_zH(\lambda_k,t, {u}_\mu(t))-\nabla_zH(\mu,t, u_\mu(t))|dt=0.
\end{equation}
 Note that \cite[Proposition C.1]{Lu9} is still true for $p\ge 1$ and that $\Lambda_0:=\{\mu, \lambda_k\,|\, k\in\mathbb{N}\}$
is sequentially compact. From (\ref{e:growth}) we deduce that  maps
\begin{equation*}
 L^r([0,\tau];\R^{2n})\to L^1([0,\tau];\R^{2n}),\; u\mapsto \nabla_zH(\lambda, \cdot, u(\cdot))
\end{equation*}
are uniform continuous at ${u}_\mu$ with respect to $\lambda\in\Lambda_0$. This result and (\ref{e:threeBifu}) imply
$$
\lim_{k\to\infty}\int^\tau_0|\nabla_zH(\lambda_k,t, \bar{u}_k(t))-\nabla_zH(\lambda_k,t, u_\mu(t))|dt=0.
$$
It follows from this, (\ref{e:threeBifu2}) and (\ref{e:threeBifu1}) that
$\|\dot{\bar{u}}_k-\dot{u}_\mu\|_{L^1}\to 0$ and so $\|{\bar{u}}_k-{u}_\mu\|_{1,1}\to 0$
by (\ref{e:threeBifu}). Then
\begin{equation}\label{e:threeBifu3}
\|{\bar{u}}_k-{u}_\mu\|_{C^0}\le \max\{\tau^{-1}, 1\}\|{\bar{u}}_k-{u}_\mu\|_{1,1}\to 0.
\end{equation}
By the proof of (i), $\bar{u}_k\to u_\mu$ in $C^1([0,\tau];\R^{2n})$.

Next, we consider the case $M=I_{2n}$. Then there exists a stronger result than \cite[Proposition~2.5]{JinLu1916}:
 $\mathbb{E}=H^{1/2}(S_\tau;\R^{2n})$ is compactly embedded in $L^{s}(S_\tau;\R^{2n})$ for each $s\in [1, \infty)$
(cf. \cite[Proposition~6.6]{Rab86} or Theorem~3 in \cite[Appendix~A.3]{HoZe94}). Therefore
(\ref{e:threeBifu}) is still true for the present $r$. Repeating the above arguments we may obtain the desired conclusions.
\end{proof}

\begin{question}
{\rm  When $M\ne I_{2n}$, can $\mathbb{E}$ be embedded in $L^{s}(S_\tau;\R^{2n})$ for each $s\in (2, \infty)$?
If this is affirmative, then the condition ``$1<r\le 2$'' may be replaced by ``$1<r<\infty$''.
}
\end{question}

\section{Generalizations and corrections for related results in \cite{Lu3, Lu8, Lu10} }\label{app:Th3.5}

We begin with the following remark about \cite[Theorems~A.1,A.2]{Lu3}.

\begin{remark}\label{rm:splitting}
{\rm In \cite[Theorems~A.1,A.2]{Lu3} we assumed that the topological space $\Lambda$ is compact.
Actually, ``compact'' can be replaced by ``first countable and sequentially compact''.
Here we say the topological space $\Lambda$ to be {\it sequentially compact} if every
sequence $(\lambda_n)$ of points in it has a subsequence converging to some $\lambda\in\Lambda$.
Recall that compact topological spaces satisfying the first axiom of countability  must be sequentially compact.
In what follows let us point out how these different conditions of compactness  are used in the proofs of \cite[Theorems~A.1,A.2]{Lu3}.
\begin{enumerate}
\item[\rm (i)] For the proof of ``$\phi$ is continuous'' in  \cite[line 20, page 2980]{Lu3} using nets we can complete it without any requirements.
If $\Lambda$ is first countable we can easily prove this claim with sequences. (See (v) below.)

\item[\rm (ii)] For the proof of ``Step 1'' in  \cite[page 2980]{Lu3}, from the proof of Claim~A.3 in the proof of \cite[Theorem~A.2]{Lu3}
we see that the sequential compactness of $\Lambda$ was actually used. When $\Lambda$ is only compact we may work with nets instead of sequences
to complete the desired proof.

\item[\rm (iii)] Our proof of ``Step 2'' in  \cite[pages 2980-2981]{Lu3} actually used the assumption that
$\Lambda$ is first countable and sequentially compact. If $\Lambda$ is only compact the expected proof can be completed with nets.

\item[\rm (iv)] The proof of ``(A.2)'' in  \cite[page 2981]{Lu3} actually used the sequential compactness of
$\Lambda$. For compact $\Lambda$ a complete proof can be given with nets.

\item[\rm (v)] As in (i), for the proof of ``$\phi$ is continuous'' in Step 7 in  \cite[page 2983]{Lu3}
we need to use nets. It suffices to use sequences for first countable $\Lambda$.

These five remarks are effective for the proof of \cite[Theorem~A.2]{Lu3}.
\end{enumerate}
In summary, the topological space $\Lambda$ in \cite[Theorems~A.1,A.2]{Lu3} should be either {\it compact} or {\it first countable and sequentially compact}.
Because of these we have also:
\begin{enumerate}
\item[\rm (vi)]  The topological space $\Lambda$ in \cite[Theorem~A.3]{Lu8} and \cite[Theorem~3.3]{Lu10}
should be assumed to be {\it first countable and sequentially compact} because first
countable compact spaces are sequentially compact and our proof for (A.11) in the proof of \cite[Theorem~A.3]{Lu8} actually used the fact that $\Lambda$ has a countable neighborhood basis at $\lambda^\ast$.
\end{enumerate}
 }
\end{remark}

Let $\Lambda$ be a topological space, and let $X,Y$ be Banach spaces.
For continuous maps $F:\Lambda\times X\to Y$ and $\Lambda\ni\lambda\mapsto x_\lambda\in X$
satisfying $F(\lambda,x_\lambda)=0$ for all $\lambda\in\Lambda$, recall that
$(\lambda^\ast, x_{\lambda^\ast})\in\Lambda\times X$ is called a \textsf{bifurcation point} of
$F(\lambda,x)=0$ in $\Lambda\times X$ with respect to the trivial branch $\{(\lambda,x_\lambda)\,|\,\lambda\in\Lambda\}$
if every neighborhood $U$ of $(\lambda^\ast, x_{\lambda^\ast})$ contains a point $(\lambda, y_\lambda)\ne (\lambda,x_\lambda)$
satisfying $F(\lambda,y_\lambda)=0$. If $\Lambda$  is first countable,
this definition is equivalent to the following:

\begin{definition}\label{def:seqbifur}
{\rm Under the assumptions above, $(\lambda^\ast, x_{\lambda^\ast})\in\Lambda\times X$ is said to be
a \textsf{bifurcation point along sequences} of
$F(\lambda,x)=0$ in $\Lambda\times X$ with respect to the trivial branch $\{(\lambda,x_\lambda)\,|\,\lambda\in\Lambda\}$
if there exists a sequence $\{(\lambda_n, y_n)\}_{n\ge 1}$ of solutions of $F(\lambda,x)=0$ in $\Lambda\times X$ converging to $(\lambda^\ast, x_{\lambda^\ast})$
in $\Lambda\times X$ such that $y_n\ne x_{\lambda_n}$ for all $n\in\mathbb{N}$.}
\end{definition}

Clearly, a bifurcation point along sequences must be a bifurcation point.

\begin{hypothesis}[{\cite[Hypothesis~1.3]{Lu8}}]\label{hyp:A.5}
{\rm Let $H$ be a Hilbert space with inner product $(\cdot,\cdot)_H$
and the induced norm $\|\cdot\|$, and let $X$ be a Banach space with
norm $\|\cdot\|_X$, such that $X\subset H$ is dense in $H$ and $\|x\|\le\|x\|_X\;\forall x\in X$.
For an open neighborhood $U$ of $0$ in $H$, $U\cap X$
is also an open neighborhood of $0$ in $X$, denoted by $U^X$.
 Let $\mathcal{L}:U\to\mathbb{R}$ be a continuous functional  satisfying  the following
conditions:
\begin{description}
\item[(F1)] $\mathcal{L}$ is continuously directionally
differentiable and $D\mathcal{L}(0)=0$.
\item[(F2)] There exists a continuous and continuously directionally differentiable
 map $A: U^X\to X$, which is also  strictly Fr\'{e}chet differentiable
 at $0$,   such that
$D\mathcal{ L}(x)[u]=(A(x), u)_H$ for all $x\in U\cap X$ and $u\in X$.
\item[(F3)] There exists a map $B: U\cap X\to \mathscr{L}_s(H)$ such that
$(DA(x)[u], v)_H=(B(x)u, v)_H$ for all $x\in U\cap X$  and
$u, v\in X$. (So $B(x)$ induces an element in $\mathscr{L}(X)$, denoted by $B(x)|_X$,
and $B(x)|_X=DA(x)\in\mathscr{L}(X),\;\forall x\in U\cap X$.)
\item[(C)]   $\{u\in H\,|\, B(0)(u)\in X\}\subset X$,
in particular  ${\rm Ker}(B(0))\subset X$.
\item[(D)] $B$ satisfies the same conditions as in \cite[Hypothesis~1.1]{Lu8},
that is,
\begin{description}
\item[(D1)]  $\{u\in H\,|\, B(0)u=\mu u,\;\mu\le 0\}\subset X$,
\end{description}
and $B$ has a decomposition $B=P+Q$, where for each $x\in U\cap X$,  $P(x)\in\mathscr{L}_s(H)$ is
 positive definite and
$Q(x)\in\mathscr{L}_s(H)$ is compact, such that the maps $P$ and $Q$  also  satisfy the following
properties:
\begin{description}
\item[(D2)] For any sequence $(x_k)\subset
U\cap X$ with $\|x_k\|\to 0$, it holds that $\|P(x_k)u-P(0)u\|\to 0$ for any $u\in H$.
\item[(D3)] The  map $Q: U\cap X\to \mathscr{L}_s(H)$ is continuous at $0$ with respect to the topology
on $H$.
\item[(D4)] For any sequence $(x_k)\subset U\cap X$ with $\|x_k\|\to 0$, there exist
 constants $C_0>0$ and $k_0\in\N$ such that
$(P(x_k)u, u)_H\ge C_0\|u\|^2$ for all $u\in H$ and for all $k\ge k_0$.
\end{description}
\end{description}}
\end{hypothesis}

We say an isolated critical point $p$ of a $C^1$-functional $f$ on a Banach manifold $\mathcal{M}$
to be \textsf{homologically visible} if there exists a nonzero critical
 group $C_m(f,p;\mathbf{K})$ for some
 Abelian group $\mathbf{K}$.

In view of Remark~\ref{rm:splitting} we may revise and refine \cite[Theorem~3.5]{Lu10} as follows.

\begin{theorem}\label{th:A.9}
Let $H$, $X$ and $U$ be as in Hypothesis~\ref{hyp:A.5},
 and let $\Lambda$ be a first countable and sequentially compact
  topological space.
Let $\mathcal{L}_\lambda\in C^1(U, \mathbb{R})$, $\lambda\in\Lambda$, be a continuous family of functionals
    satisfying $\nabla\mathcal{L}_\lambda(0)=0$ for all $\lambda\in\Lambda$.
 For each $\lambda\in\Lambda$, assume that
there exist maps $A_\lambda\in C^1(U^X, X)$ and $B_\lambda:U\cap X\to\mathscr{L}_s(H)$
such that
\begin{enumerate}
\item[\rm a)] $\Lambda\times U^X\ni (\lambda, x)\to A_\lambda(x)\in X$ is continuous;
\item[\rm b)]  for all $x\in U\cap X$  and $u, v\in X$,
 \begin{equation}\label{e:LAB}
 D\mathcal{L}_\lambda(x)[u]=(A_\lambda(x), u)_H\quad\hbox{and}\quad
(DA_\lambda(x)[u], v)_H=(B_\lambda(x)u, v)_H;
\end{equation}
\item[\rm c)]  $B_\lambda$ has a decomposition
$B_\lambda=P_\lambda+Q_\lambda$, where for each $x\in U\cap X$,
 $P_\lambda(x)\in\mathscr{L}_s(H)$ is  positive definitive and
$Q_\lambda(x)\in\mathscr{L}_s(H)$ is compact.
\end{enumerate}
   Suppose also that for some $\lambda^\ast\in\Lambda$, $P_\lambda$ and $Q_\lambda$ satisfy the following conditions:
   \begin{enumerate}
\item[\rm (i)]  For each $h\in H$, it holds that $\|P_{\lambda}(x)h-P_{\lambda^\ast}(0)h\|\to 0$
as $x\in U\cap X$ approaches to $0$ in $H$ and $\lambda\in\Lambda$ converges to $\lambda^\ast$.

 \item[\rm (ii)]  For some small $\delta>0$, there exists  $c_0>0$ such that
$$
(P_\lambda(x)u, u)\ge c_0\|u\|^2\quad\forall u\in H,\;\forall x\in
\bar{B}_H(0,\delta)\cap X,\quad\forall\lambda\in \Lambda.
$$
 \item[\rm (iii)]  $Q_\lambda: U\cap X\to \mathscr{L}_s(H)$ is uniformly
  continuous at $0$  with respect to $\lambda\in \Lambda$.
  \item[\rm (iv)]  If $\lambda\in \Lambda$ converges to $\lambda^\ast$ then
  $\|Q_{\lambda}(0)-Q_{\lambda^\ast}(0)\|\to 0$.
   \item[\rm (v)] Each tuple $(\mathcal{L}_{\lambda}, H, X, U, A_{\lambda}, B_{\lambda}=P_{\lambda}+ Q_{\lambda})$,
   $\lambda\in\Lambda$, satisfies Hypothesis~\ref{hyp:A.5}.
   \end{enumerate}
  Then there holds:
    \begin{description}
\item[(A)] If $(\lambda^\ast, 0)$ is not a bifurcation point  of $\nabla\mathcal{L}_\lambda(x)=0$
   in $\Lambda\times U$, (which implies that $0$ is an isolated critical point
   of $\mathcal{L}_{\lambda}$ for each $\lambda$ near $\lambda^\ast$),
    $\Lambda$ is  sequentially compact,
and ${\rm Ker}(B_{\lambda^\ast}(0))\ne\{0\}$, then
   $C_\ast(\mathcal{L}_{\lambda}, 0;{\bf K})\cong
   C_\ast(\mathcal{L}_{\lambda^\ast}, 0;{\bf K})$ for any
    $\lambda$  in a small neighborhood of $\lambda^\ast\in\Lambda$ and for any Abelian group ${\bf K}$.
      ({\rm Note}: From the proof it is easily seen that
   (v) need only to be satisfied for $\lambda^\ast$.)
    In addition, the condition ``$C_\ast(\mathcal{L}_{\lambda}, 0;{\bf K})\cong C_\ast(\mathcal{L}_{\lambda^\ast}, 0;{\bf K})$''
  may be changed into ``$C_\ast(\mathcal{L}_{\lambda}|_{U^X}, 0;{\bf K})\cong C_\ast(\mathcal{L}_{\lambda^\ast}|_{U^X}, 0;{\bf K})$''
    provided that for each $\lambda$ near
    $\lambda^\ast$, $\mathcal{L}_\lambda|_{U^X}\in C^2(U^X, \mathbb{R})$
  and $B_\lambda$ is continuous as a map from $U^X$ to $\mathscr{L}_s(H)$.

  \item[(B)] Suppose  that there exist two sequences in  $\Lambda$ converging to $\lambda^\ast$, $(\lambda_k^-)$ and
$(\lambda_k^+)$,  such that for each $k\in\mathbb{N}$, $[\mu_{\lambda^+_k}, \mu_{\lambda^+_k}+\nu_{\lambda^+_k}]\cap[\mu_{\lambda^-_k}, \mu_{\lambda^-_k}+\nu_{\lambda^-_k}]=\emptyset$ and at least one of the following conditions is satisfied:
 \begin{enumerate}
\item[\rm (B.1)] There exists $\lambda\in\{\lambda^+_k, \lambda^-_k\}$ such that $0$  is  either
a nonisolated or a homologically visible critical point of
$\mathcal{L}_{\lambda}$.
\item[\rm (B.2)] $\mathcal{L}_\lambda|_{U^X}\in C^2(U^X, \mathbb{R})$ and $B_\lambda$ is continuous
 as a map from $U^X$ to $\mathscr{L}_s(H)$ for each $\lambda\in\{\lambda^+_k, \lambda^-_k\}$, and there exists $\lambda\in\{\lambda^+_k, \lambda^-_k\}$ such that $0$  is either a nonisolated or a homological visible critical point of $\mathcal{L}_{\lambda}|_{U^X}$.
\item[\rm (B.3)] For each $k$, either $\nu_{\lambda^+_k}=0$ or $\nu_{\lambda^-_k}=0$.
 \end{enumerate}
  (Here $\mu_{\lambda}=\dim H^-_\lambda$ and $\nu_{\lambda}=\dim H^0_\lambda$ are the dimensions of
the negative spectral subspace $H^-_\lambda$ and the kernel $H^0_\lambda$ of $B_\lambda(0)$, respectively.)
  Then $\nu_{\lambda^\ast}>0$ and there exists a sequence $\{(\lambda_k, x_k)\}_{k\ge 1}$  in $\hat\Lambda\times U$
   converging to $(\lambda^\ast, 0)$, where $\hat{\Lambda}:=\{\lambda^\ast,\lambda^+_k, \lambda^-_k\,|\,k\in\mathbb{N}\}$, such that each $x_k$
  is a nonzero solution of $\nabla\mathcal{L}_{\lambda_k}(x)=0$, $k=1,2,\cdots$.
       In particular,
    $(\lambda^\ast, 0)$ is a bifurcation point  of $\nabla\mathcal{L}_\lambda(x)=0$
   in $\hat\Lambda\times U$ (and so in $\Lambda\times U$).
 \end{description}
\end{theorem}
That $(\lambda^\ast, 0)$ is not a bifurcation point  of $\nabla\mathcal{L}_\lambda(x)=0$
   in $\Lambda\times U$ also implies that $(\lambda^\ast, 0)$ is not a bifurcation point  of $D(\mathcal{L}_{\lambda}|_{U^X})(x)=0$
   in $\Lambda\times U^X$. The latter claim shows that
 $0$ is also an isolated critical point of $\mathcal{L}_{\lambda}|_{U^X}$
 for each $\lambda$ near $\lambda^\ast$
and thus $C_\ast(\mathcal{L}_{\lambda}|_{U^X}, 0;{\bf K})$ is well-defined.
The condition ``$C_\ast(\mathcal{L}_{\lambda}|_{U^X}, 0;{\bf K})\cong C_\ast(\mathcal{L}_{\lambda^\ast}|_{U^X}, 0;{\bf K})$''
is easily checked in applications.

\begin{proof}[\bf Proof]
{\it Proof of} (A).
Since ${\rm Ker}(B_{\lambda^\ast}(0))\ne\{0\}$ and $\Lambda$ is first countable and sequentially compact,
by Remark~\ref{rm:splitting}(vi) we may use \cite[Theorem~3.3]{Lu10} or \cite[Theorem~A.3]{Lu8} to find a neighborhood
 $\Lambda_0$ of $\lambda^\ast$ in $\Lambda$, $\epsilon>0$, a (unique) $C^0$ map
$\psi:\Lambda_0\times B_{H^0_{\lambda^\ast}}(0,\epsilon)\to X^\pm_{\lambda^\ast}$
which is $C^1$ in the second variable and
satisfies $\psi(\lambda, 0)=0\;\forall\lambda\in \Lambda_0$ and
\begin{equation}\label{e:Spli.2.1.2}
 P^\pm_{\lambda^\ast}A_\lambda(z+ \psi(\lambda,z))=0\quad\forall (\lambda,z)\in \Lambda_0
 \times B_{H^0_{\lambda^\ast}}(0,\epsilon),
 \end{equation}
and a homeomorphism $\Phi_\lambda$ from $B_{H^0_{\lambda^\ast}}(0,\epsilon)\oplus
B_{H^+_{\lambda^\ast}}(0, \epsilon)\oplus B_{H^-_{\lambda^\ast}}(0, \epsilon)$ onto
an open neighborhood of $0$ in $H$ satisfying $\Phi_\lambda(0)=0$ for each $\lambda\in\Lambda_0$,
 such that for each $\lambda\in \Lambda_0$,
\begin{eqnarray}\label{e:Spli.2.2}
&&\mathcal{L}_{\lambda}\circ\Phi_{\lambda}(z, u^++ u^-)=\|u^+\|^2-\|u^-\|^2+ \mathcal{
L}_{{\lambda}}(z+ \psi({\lambda}, z))\\
&& \quad\quad \forall (z, u^+ + u^-)\in  B_{H^0_{\lambda^\ast}}(0,\epsilon)\times
\left(B_{H^+_{\lambda^\ast}}(0, \epsilon) + B_{H^-_{\lambda^\ast}}(0, \epsilon)\right).\nonumber
\end{eqnarray}
 Moreover, the functional
\begin{equation}\label{e:Spli.2.3}
\mathcal{L}_{\lambda}^\circ: B_{H^0_{\lambda^\ast}}(0,\epsilon)\to \mathbb{R},\;
z\mapsto\mathcal{L}_{\lambda}(z+ \psi({\lambda}, z))
\end{equation}
 is of class $C^{2}$, and
 \begin{eqnarray}\label{e:Spli.2.4}
&& d\mathcal{L}^\circ_\lambda(z)[\zeta]=\bigl(A_\lambda(z+ \psi(\lambda, z)), \zeta\bigr)_H,\quad\forall
(z,\zeta)\in B_{H^0_{\lambda^\ast}}(0,\epsilon)\times H^0_{\lambda^\ast}.
   \end{eqnarray}
 By  (\ref{e:Spli.2.1.2}) and (\ref{e:Spli.2.4}),   for each $\lambda\in\Lambda_0$,
the map $z\mapsto z+ \psi({\lambda}, z))$ induces an one-to-one correspondence
 between the critical points of  $\mathcal{L}_{\lambda}^\circ$ near $0\in H^0_{\lambda^\ast}$
and zeros of $A_{\lambda}$ near $0\in X$.

Since $(\lambda^\ast, 0)\in\Lambda\times U$
 is not a bifurcation point of $\nabla\mathcal{L}_\lambda(x)=0$ in $\Lambda\times U$,
 by shrinking $\Lambda_0$ towards $\lambda^\ast$ (in $\Lambda$) we can find $\delta>0$ such that
$\nabla\mathcal{L}_\lambda(x)=0$  has only trivial solutions in $\Lambda_0\times B_H(0,\delta)\subset\Lambda_0\times U$
 (and therefore $A_\lambda(x)=0$ has only trivial solutions in $\Lambda_0\times B_X(0,\delta)\subset\Lambda_0\times B_H(0,\delta)$).
 It follows from this and the sentence below (\ref{e:Spli.2.4}) that after shrinking $\epsilon>0$ (if necessary),
 \begin{eqnarray}\label{e:Spli.2.4.0}
\hbox{for each $\lambda\in \Lambda_0$ the functional $\mathcal{L}^\circ_\lambda$ has a unique
critical point $0$ in ${B}_{H^0_{\lambda^\ast}}(0, \epsilon)$.}
\end{eqnarray}
Since $\psi$ is continuous, by the assumption c), (\ref{e:Spli.2.3}) and (\ref{e:Spli.2.4}), the maps
\begin{eqnarray*}
&&\Lambda_0\times {B}_{H^0_{\lambda^\ast}}(0, \epsilon)\ni (\lambda, z)\mapsto\mathcal{L}_{\lambda}^\circ(z)\in\R\quad\hbox{and}\\
&&\Lambda_0\times {B}_{H^0_{\lambda^\ast}}(0, \epsilon)\ni (\lambda, z)\mapsto
d\mathcal{L}^\circ_\lambda(z)\in H_{\lambda^\ast}^0=X_{\lambda^\ast}^0
\end{eqnarray*}
 are  continuous. Let us show that
 \begin{eqnarray}\label{e:Spli.2.4.1}
&&\lim_{\lambda\to\lambda^\ast}\sup\{|\mathcal{L}_{\lambda}^\circ(z)-
\mathcal{L}_{\lambda^\ast}^\circ(z)|:\,z\in\bar{B}_{H^0_{\lambda^\ast}}(0, \epsilon/2)\}=0,\\
&&\lim_{\lambda\to\lambda^\ast}\sup\{\|d\mathcal{L}_{\lambda}^\circ(z)-
d\mathcal{L}_{\lambda^\ast}^\circ(z)\|:\,z\in\bar{B}_{H^0_{\lambda^\ast}}(0, \epsilon/2)\}=0.
\label{e:Spli.2.4.2}
\end{eqnarray}
Clearly, (\ref{e:Spli.2.4.2}) implies (\ref{e:Spli.2.4.1}) because $\mathcal{L}_{\lambda}^\circ(0)-
\mathcal{L}_{\lambda^\ast}^\circ(0)=\mathcal{L}_{\lambda}(0)-\mathcal{L}_{\lambda^\ast}(0)\to 0$ as $\lambda\to\lambda^\ast$.
(Recall that $\Lambda$ has a countable neighborhood basis at $\lambda^\ast$). By contradiction (\ref{e:Spli.2.4.2}) may follow
from the sequential compactness of  $\bar{B}_{H^0_{\lambda^\ast}}(0, \epsilon/2)$.
For any two different points in $\bar{B}_{H^0_{\lambda^\ast}}(0, \epsilon/2)$, $z_1,z_2$, it follows from the mean value theorem that
 \begin{eqnarray*}
\frac{|(\mathcal{L}_{\lambda}^\circ-\mathcal{L}_{\lambda^\ast}^\circ)(z_1)-
(\mathcal{L}_{\lambda}^\circ-\mathcal{L}_{\lambda^\ast}^\circ)(z_2)|}{|z_1-z_2|}
\le\int^1_0\|d\mathcal{L}_{\lambda}^\circ(tz_1+(1-t)z_2)-
d\mathcal{L}_{\lambda^\ast}^\circ(tz_1+(1-t)z_2)\|dt.
\end{eqnarray*}
This and (\ref{e:Spli.2.4.2}) lead to
\begin{eqnarray*}
\|\mathcal{L}_{\lambda}^\circ-\mathcal{L}_{\lambda^\ast}^\circ\|_{\rm Lip}:=
\sup\left\{\frac{|(\mathcal{L}_{\lambda}^\circ-\mathcal{L}_{\lambda^\ast}^\circ)(z_1)-
(\mathcal{L}_{\lambda}^\circ-\mathcal{L}_{\lambda^\ast}^\circ)(z_2)|}{|z_1-z_2|}\,\Bigg|\,
z_1,z_2\in\bar{B}_{H^0_{\lambda^\ast}}(0, \epsilon/2),\;z_1\ne z_2\right\}\to 0
\end{eqnarray*}
as $\lambda\to\lambda^\ast$. By this and (\ref{e:Spli.2.4.0})-(\ref{e:Spli.2.4.1})  we derive from  \cite[Theorem 5.1]{CorH} that
 critical groups
 \begin{eqnarray}\label{e:criticalgroup}
 C_\ast(\mathcal{L}^\circ_\lambda, 0;{\bf K})=C_\ast(\mathcal{L}^\circ_{\lambda^\ast}, 0;{\bf K})\quad\forall \lambda\in \Lambda_0
\end{eqnarray}
 for any Abelian group ${\bf K}$ by shrinking $\Lambda_0$ towards $\lambda^\ast$ (if necessary).
Moreover $0\in U$ is a unique critical point of $\mathcal{L}_\lambda$ in $B_H(0,\delta)\subset U$.
Using  (\ref{e:Spli.2.4.0}) and \cite[Corollary~A.5]{Lu8} we obtain
\begin{eqnarray*}
C_q(\mathcal{L}_{{\lambda}}, 0;{\bf K})=C_{q-\mu_{\lambda^\ast}}(\mathcal{L}^\circ_{{\lambda}}, 0;{\bf K}),\quad\forall
(\lambda,q)\in\Lambda_0\times(\mathbb{N}\cup\{0\}).
\end{eqnarray*}
This and (\ref{e:criticalgroup}) lead to
$ C_\ast(\mathcal{L}_\lambda, 0;{\bf K})=C_\ast(\mathcal{L}_{\lambda^\ast}, 0;{\bf K})$
for all $\lambda\in \Lambda_0$.

Under the assumptions in part ``In addition'' of (A), if ${\rm Ker}(B_\lambda(0))\ne\{0\}$ (resp. ${\rm Ker}(B_\lambda(0))=\{0\}$)
by Jiang \cite{JM} there exists a splitting lemma (resp. Morse lemma) for $\mathcal{L}_{\lambda}|_{U^X}$
near $0\in U^X$. (The case of ${\rm Ker}(B_\lambda(0))=\{0\}$ can be easily obtained from the proof therein.)
By \cite[Corollary~2.8]{JM} we get $C_\ast(\mathcal{L}_{\lambda}|_{U^X}, 0;{\bf K})\cong
   C_\ast(\mathcal{L}_{\lambda}, 0;{\bf K})$ for each $\lambda$ near $\lambda^\ast$.
   The required conclusion follows from the last one.

{\it Proof of} (B). Firstly, we prove $\nu_{\lambda^\ast}>0$.
 Otherwise,  ${\rm Ker}(B_{\lambda^\ast}(0))=\{0\}$.
 By Theorem~\ref{th:A.15}, there exists a sequentially compact neighborhood
$\tilde{\Lambda}$ of $\lambda^\ast$ in $\Lambda$ such that,
 for each $\lambda\in \tilde{\Lambda}$,  $0$ is an isolated critical point of $\mathcal{L}_\lambda$.
Moreover,  for any Abelian group ${\bf K}$, the critical groups are given by
\begin{eqnarray}\label{e:criticalgroup*}
 C_q(\mathcal{L}_\lambda, 0;{\bf K})=\delta^q_{\mu_{\lambda^\ast}}{\bf K}
 \quad\forall (\lambda, q)\in \tilde\Lambda\times\mathbb{Z},
\end{eqnarray}
 where $\delta^q_p=1$ if $p=q$, and $\delta^q_p=0$ if $p\ne q$.

 Choose a large $k$ so that $\lambda_k^+$ and $\lambda_k^-$ belong to $\tilde\Lambda$.
    Because of the assumption (v) in Theorem~\ref{th:A.9} it follows from \cite[(2.7)]{Lu3}
    and the shifting theorem \cite[Corollary~2.6]{Lu3} that
    for any Abelian group ${\bf K}$ and $\star=+,-$,
\begin{eqnarray}\label{e:criticalgroup**}
&&C_q(\mathcal{L}_{{\lambda_k^\star}}, 0;{\bf K})=\delta^q_{\mu_{\lambda^\star_k}}{\bf K}\;\;\forall q\in\mathbb{N}\cup\{0\}\;\hbox{if}\;\nu_{\lambda_k^\star}=0,\\
&&C_q(\mathcal{L}_{{\lambda_k^\star}}, 0;{\bf K})=0\;\;\forall q\notin[\mu_{\lambda_k^\star}, \mu_{\lambda_k^\star}+\nu_{\lambda_k^\star}]
\;\hbox{if}\;\nu_{\lambda_k^\star}>0,\label{e:criticalgroup***}
\end{eqnarray}
respectively. (See \cite[Corollary~5.1]{Ch} for the latter). From (\ref{e:criticalgroup*}), (\ref{e:criticalgroup**}) and
(\ref{e:criticalgroup***}) we deduce
$$
\mu_{\lambda^\ast}\in
[\mu_{\lambda_k^-}, \mu_{\lambda_k^-}+\nu_{\lambda_k^-}]\cap[\mu_{\lambda_k^+}, \mu_{\lambda_k^+}+\nu_{\lambda_k^+}],
$$
which contradicts the first condition in (B). Hence $\nu_{\lambda^\ast}>0$.

\textsf{By contradiction, suppose that  $(\lambda^\ast, 0)\in\hat\Lambda\times U$
 is not a bifurcation point of $\nabla\mathcal{L}_\lambda(x)=0$ in $\hat\Lambda\times U$.}
Here $\hat{\Lambda}=\{\lambda^\ast,\lambda^+_k, \lambda^-_k\,|\,k\in\mathbb{N}\}$.
Define $\hat\Lambda_0=\hat{\Lambda}\cap\Lambda_0$, where $\Lambda_0$ is as in the proof
(A). Then for each $\lambda\in \hat\Lambda_0$,
\begin{description}
\item[($\spadesuit$)] the functional $\mathcal{L}_\lambda$ has an isolated critical point $0$ in $U$,
\item[($\clubsuit$)] the functional $\mathcal{L}^\circ_\lambda$ has a unique
critical point $0$ in ${B}_{H^0_{\lambda^\ast}}(0, \epsilon)$.
\end{description}
Clearly,  the first paragraph in the proof (A) is still valid
after $\Lambda$ and $\Lambda_0$ are replaced by $\hat\Lambda$ and $\hat\Lambda_0$, respectively.
Repeating the arguments in the second paragraph in Step 1, for any Abelian group ${\bf K}$ we get
 $$
 C_q(\mathcal{L}_\lambda, 0;{\bf K})=C_q(\mathcal{L}_{\lambda^\ast}, 0;{\bf K})
 $$
  for all $(\lambda, q)\in \hat\Lambda_0\times(\mathbb{N}\cup\{0\})$,
 and hence
\begin{eqnarray}\label{e:criticalgroup****}
C_q(\mathcal{L}_{{\lambda_k^+}}, 0;{\bf K})=C_q(\mathcal{L}_{{\lambda_k^-}}, 0;{\bf K}),\;\;\forall q\in\mathbb{N}\cup\{0\}
\end{eqnarray}
if $k$ is so large that $\lambda_k^+,\lambda_k^-\in\hat\Lambda_0$.
 For such a large $k$, we can obtain a contradiction under any of the assumptions (B.1), (B.2) and (B.3) as follows.

 {\bf Case (B.1)}: Because of ($\spadesuit$), $0$  is a homologically visible critical point of
$\mathcal{L}_{{\lambda_k^+}}$ and $\mathcal{L}_{{\lambda_k^-}}$.
Combining with (\ref{e:criticalgroup****}) we derive
$C_m(\mathcal{L}_{{\lambda_k^+}}, 0;{\bf K})=C_m(\mathcal{L}_{{\lambda_k^-}}, 0;{\bf K})\ne 0$
for some Abelian group ${\bf K}$ and some $m\in\mathbb{N}\cup\{0\}$.
As above this and (\ref{e:criticalgroup**})-(\ref{e:criticalgroup***}) yield
a contradiction to the first condition in (B).

{\bf Case (B.2)}:
Since ($\spadesuit$) implies that $0\in U^X$ is an isolated critical
 point $\mathcal{L}_{\lambda}|_{U^X}$ for any $\lambda\in\{\lambda^+_k, \lambda^-_k\}$,
$0$  is a homological visible critical point of either
$\mathcal{L}_{{\lambda_k^+}}|_{U^X}$ or $\mathcal{L}_{{\lambda_k^-}}|_{U^X}$.
 It follows from our assumptions that
  either $C_m(\mathcal{L}_{{\lambda_k^+}}|_{U^X}, 0;{\bf K})\ne 0$
 for some Abelian group ${\bf K}$ and some $m\in\mathbb{N}\cup\{0\}$,
 or $C_n(\mathcal{L}_{{\lambda_k^-}}|_{U^X}, 0;{\bf K}')\ne 0$
 for some Abelian group ${\bf K}'$ and some $n\in\mathbb{N}\cup\{0\}$.
As in the final paragraph in the proof (A), for any Abelian group ${\bf G}$  we have also $C_\ast(\mathcal{L}_{\lambda}|_{U^X}, 0;{\bf G})\cong
   C_\ast(\mathcal{L}_{\lambda}, 0;{\bf G})$ for $\lambda=\lambda^+_k, \lambda^-_k$.
These and (\ref{e:criticalgroup****}) lead to either
\begin{eqnarray*}
C_m(\mathcal{L}_{{\lambda_k^+}}, 0;{\bf K})=C_m(\mathcal{L}_{{\lambda_k^-}}, 0;{\bf K})\ne 0\quad\hbox{or}\quad
C_n(\mathcal{L}_{{\lambda_k^+}}, 0;{\bf K}')=C_n(\mathcal{L}_{{\lambda_k^-}}, 0;{\bf K}')\ne 0.
\end{eqnarray*}
As above, from these and (\ref{e:criticalgroup**})-(\ref{e:criticalgroup***}) we derive
a contradiction to the first condition in (B).

{\bf Case (B.3)}: This case can also be divided into three cases.

\underline{{\it Case 1}: $\nu_{\lambda_k^+}=0$ and $\nu_{\lambda_k^-}=0$}.
 Because of ($\spadesuit$), we may use (\ref{e:criticalgroup**}) and
(\ref{e:criticalgroup****}) to derive
$\mu_{\lambda^+_k}=\mu_{\lambda^\ast}=\mu_{\lambda^-_k}$, which contradicts
the first condition in (B).

 \underline{{\it Case 2}: $\nu_{\lambda_k^+}=0$ and $\nu_{\lambda_k^-}>0$}.  Because ($\spadesuit$),
 (\ref{e:criticalgroup**}) with $\star=+$, (\ref{e:criticalgroup****})  and
 (\ref{e:criticalgroup***}) with $\star=-$  show that $\mu_{\lambda^+_k}\in[\mu_{\lambda_k^-}, \mu_{\lambda_k^-}+\nu_{\lambda_k^-}]$.
 This  contradicts the first condition in (B).

  \underline{{\it Case 3}: $\nu_{\lambda_k^+}>0$ and $\nu_{\lambda_k^-}=0$}.
  This case may be proved as in Case 2.
\end{proof}

 The following bifurcation theorem may be viewed as a global version of Theorem~\ref{th:A.9}.

\begin{theorem}\label{th:A.9+}
Under the assumptions a)-c) and (v) of Theorem~\ref{th:A.9} \textsf{without requirement of the first countability for} $\Lambda$, suppose that the  topological
space $\Lambda$ is path connected  and the following conditions are satisfied.
\begin{enumerate}
\item[\rm d)]  $P_\lambda$ and $Q_\lambda$ satisfy the conditions
(i)-(iv) of Theorem~\ref{th:A.9} for any $\lambda^\ast\in \Lambda$;
\item[\rm e)] There exist two  points $\lambda^+, \lambda^-\in\Lambda$ such that
one of the following conditions is satisfied:
 \begin{enumerate}
 \item[\rm (e.1)] Either $0$  is not an isolated critical point of $\mathcal{L}_{\lambda^+}$,
 or $0$ is not an isolated critical point of $\mathcal{L}_{\lambda^-}$,
 or $0$  is an isolated critical point of $\mathcal{L}_{\lambda^+}$ and $\mathcal{L}_{\lambda^-}$ and
  $C_m(\mathcal{L}_{\lambda^+}, 0;{\bf K})$ and $C_m(\mathcal{L}_{\lambda^-}, 0;{\bf K})$ are not isomorphic for some Abelian group ${\bf K}$
and some $m\in\mathbb{Z}$.  Moreover, ``$C_m(\mathcal{L}_{\lambda^+}, 0;{\bf K})$ and $C_m(\mathcal{L}_{\lambda^-}, 0;{\bf K})$''
in the third case  may be changed into ``$C_\ast(\mathcal{L}_{\lambda^+}|_{U^X}, 0;{\bf K})$ and $C_\ast(\mathcal{L}_{\lambda^-}|_{U^X}, 0;{\bf K})$'' provided  also  that $\mathcal{L}_\lambda|_{U^X}\in C^2(U^X, \mathbb{R})$
  and $B_\lambda$ is continuous as a map from $U^X$ to $\mathscr{L}_s(H)$
  for each $\lambda\in\{\lambda^+, \lambda^-\}$.

\item[\rm (e.2)] $[\mu_{\lambda^+}, \mu_{\lambda^+}+\nu_{\lambda^+}]\cap[\mu_{\lambda^-}, \mu_{\lambda^-}+\nu_{\lambda^-}]=\emptyset$,
and  there exists $\lambda\in\{\lambda^+, \lambda^-\}$ such that $0$  is  either a nonisolated or a homologically visible critical point of
$\mathcal{L}_{\lambda}$.

\item[\rm (e.3)] $[\mu_{\lambda^+}, \mu_{\lambda^+}+\nu_{\lambda^+}]\cap[\mu_{\lambda^-}, \mu_{\lambda^-}+\nu_{\lambda^-}]=\emptyset$, $\mathcal{L}_\lambda|_{U^X}\in C^2(U^X, \mathbb{R})$ and $B_\lambda$ is continuous
 as a map from $U^X$ to $\mathscr{L}_s(H)$ for each $\lambda\in\{\lambda^+, \lambda^-\}$, and
 there exists $\lambda\in\{\lambda^+, \lambda^-\}$ such that
  $0$  is either a nonisolated or a homologically visible critical point of   $\mathcal{L}_{\lambda}|_{U^X}$.

  \item[\rm (e.4)] $[\mu_{\lambda^+}, \mu_{\lambda^+}+\nu_{\lambda^+}]\cap[\mu_{\lambda^-}, \mu_{\lambda^-}+\nu_{\lambda^-}]=\emptyset$,
and either $\nu_{\lambda^+}=0$ or $\nu_{\lambda^-}=0$.
 \end{enumerate}
 \end{enumerate}
Then for any path $\alpha:[0,1]\to\Lambda$ connecting $\lambda^+$ to $\lambda^-$
there exists $\bar{t}\in [0, 1]$
such that $(\bar{t}, 0)\in [0, 1]\times U$  is a bifurcation point
of $\nabla\mathcal{L}_{\alpha(t)}(x)=0$ in $[0, 1]\times U$,
and therefore that $(\alpha(\bar{t}), 0)\in \Lambda\times U$  is
 a bifurcation point along sequences of $\nabla\mathcal{L}_{\lambda}(x)=0$ in $\Lambda\times U$.

  Moreover,
  if as a map from $U\cap X$ to $H$, $A_{\lambda^+}$ (resp. $A_{\lambda^-}$)
   has a G\^ateaux derivative $B_{\lambda^+}(x)$ (resp. $B_{\lambda^-}(x)$) at every point $x\in U\cap X$
   and  $\nu_{\lambda^+}=0$ (resp. $\nu_{\lambda^-}=0$), then $\alpha(\bar{t})\ne \lambda^+$ (resp. $\alpha(\bar{t})\ne \lambda^-$).
\end{theorem}

In e.1), even if $0$ is an isolated critical point of $\mathcal{L}_{\lambda^+}$ and $\mathcal{L}_{\lambda^-}$
it is possible that either $(\lambda^+,0)$ or $(\lambda^-,0)$ is a bifurcation point.
The second part in e.1) and e.4) are, sometimes, more conveniently used in applications, see \cite{Lu12}.

\begin{proof}[\bf Proof]
Clearly,  it suffices to prove the case that $\Lambda=[0,1]$ and $\lambda^+=0, \lambda^-=1$.
\textsf{By contradiction, suppose that $[0,1]\times U$
contains no bifurcation point of $\nabla\mathcal{L}_{\lambda}(x)=0$ in $[0,1]\times U$.}
Then there exists a small $\varepsilon>0$  such that
 \begin{eqnarray}\label{e:Spli.2.4.4}
\hbox{for each $\lambda\in [0, 1]$, $\nabla\mathcal{L}_{\lambda}(x)=0$ has a unique
solution $0$ in ${B}_{H}(0, \varepsilon)\subset U$.}
\end{eqnarray}
(In particular, this implies that $0$ is an isolated critical point of $\mathcal{L}_{\lambda}$ for each $\lambda\in [0,1]$.)
We conclude that for any Abelian group ${\bf K}$ and any $\lambda, \lambda'\in [0,1]$,
\begin{eqnarray}\label{e:Spli.2.4.5}
C_q(\mathcal{L}_{{\lambda}}, 0;{\bf K})=C_q(\mathcal{L}_{{\lambda'}}, 0;{\bf K}),\;\;\forall q\in\mathbb{N}\cup\{0\}.
\end{eqnarray}
In fact, for any given point $\lambda_0\in [0, 1]$, if ${\rm Ker}(B_{\lambda_0}(0))\ne\{0\}$,
since $[0, 1]$ is first countable and sequentially compact,
by the proof of (A) in the proof of Theorem~\ref{th:A.9} there exists a connected open neighborhood
$\mathcal{N}(\lambda_0)$ of $\lambda_0$ in $[0, 1]$ such that
$ C_\ast(\mathcal{L}_\lambda, 0;{\bf K})=C_\ast(\mathcal{L}_{\lambda_0}, 0;{\bf K})$
for any Abelian group ${\bf K}$
and for all $\lambda\in \mathcal{N}(\lambda_0)$.

If ${\rm Ker}(B_{\lambda_0}(0))=\{0\}$, replacing $\hat{\Lambda}$ by $[0, 1]$ in the proof of
(\ref{e:criticalgroup*}) we obtain  a connected open neighborhood
$\mathcal{N}(\lambda_0)$ of $\lambda_0$ in $[0, 1]$ such that
$C_\ast(\mathcal{L}_\lambda, 0;{\bf K})=\delta^q_{\mu_{\lambda_0}}{\bf K}$
for any Abelian group ${\bf K}$ and all $(q,\lambda)\in \mathbb{Z}\times\mathcal{N}(\lambda_0)$.

In summary, for any Abelian group
${\bf K}$, $[0,1]\ni\lambda\mapsto C_\ast(\mathcal{L}_\lambda, 0;{\bf K})$ is locally constant, and
therefore (\ref{e:Spli.2.4.5}) holds.
 In particular, this yields for any Abelian group ${\bf K}$
\begin{eqnarray}\label{e:Spli.2.4.6}
C_q(\mathcal{L}_{{\lambda^+}}, 0;{\bf K})=C_q(\mathcal{L}_{{\lambda^-}}, 0;{\bf K}),\;\;\forall q\in\mathbb{N}\cup\{0\},
\end{eqnarray}
Let us derive contradictions in each case.

{\it Step 1}[\textsf{Case (e.1)}].
By (\ref{e:Spli.2.4.4})  the third case in e.1) must occur. It contradicts (\ref{e:Spli.2.4.6}).
For the part of ``Moreover'', since $\mathcal{L}_\lambda|_{U^X}\in C^2(U^X, \mathbb{R})$
  and $B_\lambda$ is continuous as a map from $U^X$ to $\mathscr{L}_s(H)$
  for each $\lambda\in\{\lambda^+, \lambda^-\}$,
  as in the final paragraph in the proof of (B)
  in the proof of Theorem~\ref{th:A.9} we may derive from
  \cite[Corollary~2.8]{JM} that
  $C_\ast(\mathcal{L}_{\lambda}|_{U^X}, 0;{\bf K})\cong
   C_\ast(\mathcal{L}_{\lambda}, 0;{\bf K})$ for $\lambda=\lambda^+, \lambda^-$.
  This and (\ref{e:Spli.2.4.6}) lead to
  $C_\ast(\mathcal{L}_{{\lambda^+}}|_{U^X}, 0;{\bf K})=C_\ast(\mathcal{L}_{{\lambda^-}}|_{U^X}, 0;{\bf K})$,
which contradicts the new assumption.

{\it Step 2}[\textsf{Case (e.2)}].
By (\ref{e:Spli.2.4.4}), $0$  is an isolated critical point of $\mathcal{L}_{\lambda^+}$ and $\mathcal{L}_{\lambda^-}$,
and $0$  is a homologically visible critical point of either $\mathcal{L}_{\lambda^+}$ or $\mathcal{L}_{\lambda^-}$.
 As in arguments below (\ref{e:criticalgroup*}), for any Abelian group ${\bf K}$ we have also for $\star=+,-$,
\begin{eqnarray}\label{e:Spli.2.4.7}
&&C_q(\mathcal{L}_{{\lambda^\star}}, 0;{\bf K})=\delta^q_{\mu_{\lambda^\star}}{\bf K}\;\;\forall q\in\mathbb{N}\cup\{0\}\;\hbox{if}\;\nu_{\lambda^\star}=0,\\
&&C_q(\mathcal{L}_{{\lambda^\star}}, 0;{\bf K})=0\;\;\forall q\notin[\mu_{\lambda^\star}, \mu_{\lambda^\star}+\nu_{\lambda^\star}]
\;\hbox{if}\;\nu_{\lambda^\star}>0,\label{e:Spli.2.4.8}
\end{eqnarray}
Suppose that $0$  is a homologically visible critical point of $\mathcal{L}_{\lambda^+}$.
(The proof of another case is similar.) Then
 $C_m(\mathcal{L}_{\lambda^+}, 0;{\bf K})\ne 0$  for some Abelian group ${\bf K}$
and some $m\in\mathbb{Z}$, and therefore
 $m\in[\mu_{\lambda^+}, \mu_{\lambda^+}+\nu_{\lambda^+}]\cap[\mu_{\lambda^-}, \mu_{\lambda^-}+\nu_{\lambda^-}]$
 by (\ref{e:Spli.2.4.6}) and (\ref{e:Spli.2.4.8}). A contradiction is yielded.

{\it Step 3}[\textsf{Case (e.3)}].
Since ${B}_{H}(0, \varepsilon)^X$ is a neighborhood of $0$ in $U^X$,
by (\ref{e:Spli.2.4.4}), $0$  is also an isolated critical point of   $\mathcal{L}_{\lambda}|_{U^X}$ for each $\lambda\in\{\lambda^+, \lambda^-\}$. Then the third case in e.3) must occur, and so
either $C_m(\mathcal{L}_{{\lambda^+}}|_{U^X}, 0;{\bf K})\ne 0$
for some Abelian group ${\bf K}$ and some $m\in\mathbb{N}\cup\{0\}$ or
$C_n(\mathcal{L}_{{\lambda^-}}|_{U^X}, 0;{\bf K}')\ne 0$ for some Abelian group ${\bf K}'$ and some $n\in\mathbb{N}\cup\{0\}$.
Since for any Abelian group ${\bf G}$ we may use
  \cite[Corollary~2.8]{JM} to obtain   $C_\ast(\mathcal{L}_{\lambda}|_{U^X}, 0;{\bf G})\cong
   C_\ast(\mathcal{L}_{\lambda}, 0;{\bf G})$ for $\lambda=\lambda^+, \lambda^-$,
   as in the proof of the case (B.2) in Step 2 of the proof of Theorem~\ref{th:A.9} using these and (\ref{e:Spli.2.4.6}) we obtain
   either $C_m(\mathcal{L}_{{\lambda^+}}, 0;{\bf K})=C_m(\mathcal{L}_{{\lambda^-}}, 0;{\bf K})\ne 0$ or $
C_n(\mathcal{L}_{{\lambda^+}}, 0;{\bf K}')=C_n(\mathcal{L}_{{\lambda^-}}, 0;{\bf K}')\ne 0$.
These and  (\ref{e:Spli.2.4.7})-(\ref{e:Spli.2.4.8}) lead to a contradiction to the first condition in (e.3).

{\it Step 4}[\textsf{Case (e.4)}].
Because of (\ref{e:Spli.2.4.4}), (\ref{e:Spli.2.4.6}) and
 (\ref{e:Spli.2.4.7})-(\ref{e:Spli.2.4.8}) can be used. \\
$\bullet$ If $\nu_{\lambda^+}=0$ and $\nu_{\lambda^-}=0$, then (\ref{e:Spli.2.4.6}) and (\ref{e:Spli.2.4.7}) lead to
$\mu_{\lambda^+}=\mu_{\lambda^-}$, and hence a contradiction. \\
$\bullet$ If $\nu_{\lambda^+}=0$ and $\nu_{\lambda^-}>0$, then (\ref{e:Spli.2.4.6}) and
 (\ref{e:Spli.2.4.7}) with $\star=+$ imply $C_{\mu_{\lambda^+}}(\mathcal{L}_{{\lambda^-}}, 0;{\bf K})={\bf K}$
 and therefore $\mu_{\lambda^+}\in[\mu_{\lambda^-}, \mu_{\lambda^-}+\nu_{\lambda^-}]$ by
 (\ref{e:Spli.2.4.8}) with $\star=-$, which is a contradiction. \\
$\bullet$  If $\nu_{\lambda^+}>0$ and $\nu_{\lambda^-}=0$, as in the second case we obtain
$\mu_{\lambda^-}\in[\mu_{\lambda^+}, \mu_{\lambda^+}+\nu_{\lambda^+}]$ and so a contradiction as above.

{\it Step 5}[\textsf{Proof of the final claim}].
 Let $(\lambda^\ast, 0)\in[0,1]\times U$ be a bifurcation point of $\nabla\mathcal{L}_{\lambda}(x)=0$ in $[0,1]\times U$,
which is also a bifurcation point along sequences because $[0,1]$ is first countable.
By contradiction suppose that $\lambda^\ast=\lambda^+$ (resp. $\lambda^\ast=\lambda^-$)
in the case of $\nu_{\lambda^+}=0$ (resp. $\nu_{\lambda^-}=0$).
The additional assumption on $A_{\lambda^+}$ (resp. $A_{\lambda^-}$) implies
that $(U,X, \mathcal{L}_{\lambda^+}, A_{\lambda^+}, B_{\lambda^+})$
[resp. $(U,X, \mathcal{L}_{\lambda^-}, A_{\lambda^-}, B_{\lambda^-})$]
also satisfies \cite[Hypothesis~1.1]{Lu10}. Then by \cite[Theorem~3.1]{Lu10} (see also Theorem~\ref{th:A.10})
we obtain $\nu_{\lambda^+}=\nu_{\lambda^\ast}>0$ (resp. $\nu_{\lambda^-}=\nu_{\lambda^\ast}>0$), a contradiction.
\end{proof}

Because of remarks above,  in most of results in \cite{Lu8,Lu9,Lu10}
we either require the topological space $\Lambda$ to be first countable or
change `` bifurcation point'' into ``bifurcation point along sequences''.
For example, \cite[Theorems~3.1,3.2]{Lu8} should be revised as follows.

\begin{theorem}\label{th:A.10}
In \cite[Theorems~3.1,3.2]{Lu8}, `` bifurcation point'' should be changed into ``bifurcation point along sequences''.
{\rm (Note that $0$ is not required to be an isolated critical point of $\mathcal{F}_{\lambda^\ast}$
from the proofs therein.)}
\end{theorem}

\begin{theorem}\label{th:A.11}
The conclusions in \cite[Theorems~3.6]{Lu10}  should be revised as follows:\\
``Then  $(\lambda^\ast,0)\in\Lambda\times U$ is a bifurcation point  for the equation
\begin{equation*}
D\mathcal{L}_\lambda(u)=0,
\quad (\lambda,u)\in \Lambda\times U.
\end{equation*}
More precisely, one of the following alternatives occurs:
\begin{description}
\item[(i)] $(\lambda^\ast,0)$ is not an isolated solution  of the equation
$D\mathcal{L}_{\lambda^\ast}(u)=0$ in $\{\lambda^\ast\}\times U$.

\item[(ii)]  For every $\lambda\in\Lambda$ near $\lambda^\ast$ there is a nontrivial solution $u_\lambda$ of
$A_\lambda(u)=0$ in $U^X$, which  converges to $0$ in $X$ as $\lambda\to\lambda^\ast$.

\item[(iii)] For any given neighborhood $W$ of $0$ in $X$ there is an one-sided  neighborhood $\Lambda^\ast$ of $\lambda^\ast$ such that
for any $\lambda\in\Lambda^\ast\setminus\{\lambda^\ast\}$, $A_\lambda(u)=0$ has  at least two  nontrivial solutions in $W$,
which can also be required to correspond to distinct critical values
provided that  $\nu_{\lambda^\ast}>1$ and $A_\lambda(u)=0$ has only finitely many nontrivial solutions in $W$.''
\end{description}
In the conclusions of \cite[Theorems~4.6]{Lu8}, (i) is changed into
\begin{quotation}
\noindent ``{\bf (i)} $(\lambda^\ast,0)$ is not an isolated solution of the equation
$D\mathcal{L}_{\lambda^\ast}(u)=0$ in $\{\lambda^\ast\}\times U$.''
\end{quotation}
and the sentence
\begin{quotation}
\noindent ``
Therefore $(\lambda^\ast,0)\in\Lambda\times U^X$ is a bifurcation point  for the equation
(4.15); in particular $(\lambda^\ast,0)\in\Lambda\times U$ is a bifurcation point  for the equation
$$
D\mathcal{L}_\lambda(u)=0,
\quad (\lambda,u)\in \Lambda\times U.\eqno(4.16)
$$''
\end{quotation}
should be replaced by
\begin{quotation}
\noindent ``
Therefore,  $(\lambda^\ast,0)\in\Lambda\times U$ is a bifurcation point  for the equation
$$
D\mathcal{L}_\lambda(u)=0,
\quad (\lambda,u)\in \Lambda\times U.\eqno(4.16)
$$''
\end{quotation}
\end{theorem}

\begin{proof}[\bf Proof]
By the assumption f) in \cite[Theorems~3.6]{Lu10} we obtain  the first claim by the conclusion (B) in Theorem~\ref{th:A.9}.
Suppose that (i) is not true. Then $0\in U$ is an isolated critical point  of
$\mathcal{L}_{\lambda^\ast}$ in $U$. (Of course, this implies that $0\in U^X$ is an isolated zero of
$A_{\lambda^\ast}(u)=0$ in $U^X$.)
By the arguments in the first paragraph in Step 2 in the proof of Theorem~\ref{th:A.9}, $0\in U$
is also an isolated critical point of $\mathcal{L}_\lambda$ for each $\lambda\in[\lambda^\ast-\delta,\lambda^\ast+\delta]$.
The other arguments in the proofs of  \cite[Theorems~4.6]{Lu8} and \cite[Theorems~3.6]{Lu10}
are valid.
\end{proof}

\begin{theorem}\label{th:A.12}
Statement (i) of both \cite[Theorems~3.7]{Lu10} and \cite[Theorems~5.12]{Lu8}
 requires the following correction:
\begin{quotation}
\noindent ``{\bf (i)} $(\lambda^\ast,0)$ is not an isolated solution of the equation
$D\mathcal{L}_{\lambda^\ast}(u)=0$ in $\{\lambda^\ast\}\times U$.''
\end{quotation}
The following sentences should be revised:
\begin{itemize}
\item[$\bullet$] In \cite[Theorems~3.7]{Lu10}:  ``In particular, $(\lambda^\ast,0)\in\Lambda\times U^X$ is a bifurcation point of (3.23).''
\item[$\bullet$] In \cite[Theorems~5.12]{Lu8}: ``In particular, $(\lambda^\ast,0)\in\Lambda\times U^X$ is a bifurcation point of (4.15).''
\end{itemize}
They are to be changed to the:
\begin{quotation}
\noindent ``$(\lambda^\ast,0)\in\Lambda\times U$ is a bifurcation point of  the equation
$D\mathcal{L}_\lambda(u)=0$ in $\Lambda\times U$.''
\end{quotation}
\end{theorem}

\cite[Theorems~3.10]{Lu10} should be revised as follows:
\begin{theorem}\label{th:A.13}
In \cite[Theorems~3.7]{Lu10}, if the assumption ``the fixed point set of the induced $G$-action on $H^0_{\lambda^\ast}$ is $\{0\}$''
is removed, then
$(\lambda^\ast,0)\in\Lambda\times U$ is a bifurcation point  for the equation
\begin{equation*}
D\mathcal{L}_\lambda(u)=0,
\quad (\lambda,u)\in \Lambda\times U.
\end{equation*}
More precisely, one of the following alternatives occurs:
\begin{description}
\item[(i)]  $(\lambda^\ast,0)$ is not an isolated solution  of the equation
$D\mathcal{L}_{\lambda^\ast}(u)=0$ in $\{\lambda^\ast\}\times U$.

\item[(ii)]  For every $\lambda\in\Lambda\setminus\{\lambda^\ast\}$
near $\lambda^\ast$ there is a nontrivial $G$-orbit of solutions of
     \begin{equation}\label{e:Bif.2.2.1*}
A_\lambda(u)=0,\quad (\lambda, u)\in\Lambda\times U^X,
\end{equation}
which  converges to $0$ in $X$ as $\lambda\to\lambda^\ast$.

\item[(iii)] For any given $G$-invariant neighborhood $\mathcal{N}$ of $0$ in $X$
there is an one-sided neighborhood $\Lambda^0$ of $\lambda^\ast$ such that
for any $\lambda\in\Lambda^0\setminus\{\lambda^\ast\}$,
(\ref{e:Bif.2.2.1*})  has at least two nontrivial $G$-orbit of solutions in $\mathcal{N}$
 provided that the Euler-Poincar\'e characteristic of any nontrivial orbit
 near $0$ of the induced $G$-action on
 $H^0_{\lambda^\ast}$ is not equal to $1-(-1)^{\nu_{\lambda^\ast}}$, where $\nu_{\lambda^\ast}=\dim H^0_{\lambda^\ast}$
 is the nullity of $\mathcal{L}_{\lambda^\ast}$ at $0$. Moreover,
 for $\lambda\in\Lambda^0\setminus\{\lambda^\ast\}$, if (\ref{e:Bif.2.2.1*})
 has only finitely many  $G$-orbit of solutions in $\mathcal{N}$, then
it has at least two nontrivial $G$-orbit of solutions in $\mathcal{N}$ with different energy
 provided that  $\nu_{\lambda^\ast}>1$ and
 any nontrivial orbit $\mathcal{O}$ near $0$ of the induced $G$-action on $H^0_{\lambda^\ast}$ satisfies one of the following conditions:
  \begin{description}
\item[iii-1)]  $\dim\mathcal{O}=0$ or  $1\le\dim\mathcal{O}\le \nu_{\lambda^\ast}-2$.
\item[iii-2)]  $1\le \dim \mathcal{O}_1=\nu_{\lambda^\ast}-1$, either $\mathcal{O}$ is non-connected or
$\mathcal{O}$ is connected and $H_r(\mathcal{O}, \mathbb{Z}_2)\ne H_r(S^{\nu_{\lambda^\ast}-1}, \mathbb{Z}_2)$ for some $0\le r\le \nu_{\lambda^\ast}-1$.
\end{description}
\end{description}
\end{theorem}

\begin{remark}\label{rm:Final}
{\rm
Using theorems in \cite[Appendix~B]{Lu7} we can also prove corresponding results with the theorems above
as supplements to \cite[\S6]{Lu7}. They and related applications will be given elsewhere.
}
\end{remark}

The following theorem is a parameterized version of
the general Morse lemma.

\begin{theorem}\label{th:A.15}
Under the assumptions of Theorem~\ref{th:A.9},
 where we only require $\Lambda$ to be a topological space,
suppose  ${\rm Ker}(B_{\lambda^\ast}(0))=\{0\}$.
 Let  $H^+_{\lambda^\ast}$ and $H^-_{\lambda^\ast}$ be
  the spectral subspaces of $B_{\lambda^\ast}(0)$
  corresponding to the positive and negative parts of the spectrum, respectively.
Then there exist a neighborhood $\Lambda_0$ of $\lambda^\ast$ in $\Lambda$
and a positive number $\varepsilon$ such that $B_{H^+_{\lambda^\ast}}(0,\varepsilon)\oplus B_{H^-_{\lambda^\ast}}(0,\varepsilon)\subset U$, and
 \begin{eqnarray}\label{e:add-inequality}
 &&\big(D\mathcal{L}_{\lambda}(u^+ + u^-_2)-  D\mathcal{L}_{\lambda}(u^+ + u^-_1)\big)[u^-_2-u^-_1]\le
-\mathfrak{a}_1\|u^-_2-u^-_1\|^2,\nonumber\\
&&D\mathcal{L}_{\lambda}(u^++u^-)[u^+-u^-]\ge  \mathfrak{a}_2(\|u^+\|^2+ \|u^-\|^2)
\end{eqnarray}
for all $(\lambda, u^+, u^-, u^-_1, u^-_2)\in \Lambda_0\times B_{H^+_{\lambda^\ast}}(0,\varepsilon)\times
(B_{H^-_{\lambda^\ast}}(0,\varepsilon))^3$,
 where  $\mathfrak{a}_1$ and $\mathfrak{a}_2$ are positive constants.
   As a direct consequence of the second inequality, for each $\lambda\in \Lambda_0$,
  $0$ is the unique critical point of $\mathcal{L}_\lambda$ in
 $B_{H^+_{\lambda^\ast}}(0,\epsilon)\oplus B_{H^-_{\lambda^\ast}}(0, \epsilon)$.
  ({\bf Note}: Until now, it has been sufficient to only require that ${\Lambda}$ is a topological space.)

 Moreover, suppose that $\Lambda$ is either compact or a first countable and sequentially compact topological space, so that we can take $\Lambda_0$ to be a space of the same kind.
Then  there exists a number $0<\epsilon\le\varepsilon$, an open neighborhood $\mathcal{W}$ of $\Lambda_0 \times \{0\}$ in $\Lambda_0 \times H$, and a homeomorphism
\[
\phi: \Lambda_0 \times \bigl( B_{H^+_{\lambda^\ast}}(0, \epsilon) + B_{H^-_{\lambda^\ast}}(0, \epsilon) \bigr) \longrightarrow \mathcal{W} \subset\Lambda_0\times U
\]
satisfying $\phi(\lambda, 0) = (\lambda, 0)$ for all $\lambda \in \Lambda_0$, and such that:
\begin{itemize}
\item[\rm (i)] For all $(\lambda, u^+, u^-) \in \Lambda_0 \times B_{H^+_{\lambda^\ast}}(0, \epsilon) \times B_{H^-_{\lambda^\ast}}(0, \epsilon)$,
\begin{equation}\label{e:add-equality}
\mathcal{L}_\lambda\bigl( \phi(\lambda, u^+ + u^-) \bigr) = \|u^+\|^2 - \|u^-\|^2
\end{equation}
\item[\rm (ii)] Each $\phi(\lambda,\cdot)$ is a homeomorphism
 from $B_{H^+_{\lambda^\ast}}(0, \epsilon)+B_{H^-_{\lambda^\ast}}(0, \epsilon)$ onto an open neighborhood of $0$ in $H$.
\end{itemize}
Consequently,  for any Abelian group ${\bf K}$, critical groups
\begin{equation*}
 C_q(\mathcal{L}_\lambda, 0;{\bf K})=\delta^q_{\mu_{\lambda^\ast}}{\bf K}
 \quad\forall (\lambda, q)\in \Lambda_0\times\mathbb{Z},
\end{equation*}
 where $\delta^q_p=1$ if $p=q$, and $\delta^q_p=0$ if $p\ne q$.
\end{theorem}

\begin{proof}[\bf Proof]
{\it Proof of} (\ref{e:add-inequality}).
We continue to use the notation of the proof of Theorem~\ref{th:A.9}.
By the condition (v), $X^-_{\lambda^\ast}=H^-_{\lambda^\ast}$ is finite-dimensional.
Let $e_1,\cdots,e_m$ be an unit orthogonal basis of $H^-_{\lambda^\ast}$.
Then there exists a
constant $C_1>0$ such that
$$
\Bigl(\sum^m_{i=1}|t_i|^2\Bigr)^{1/2}\le C_1\|u\|,\quad\forall
u=\sum^m_{i=1}t_ie_i\in  H^-_{\lambda^\ast}.
$$
Hence for any $u=\sum^m_{i=1}t_ie_i\in  H^-_{\lambda^\ast}$ and $v\in H$ we have
\begin{eqnarray}\label{e:C.22}
&&|(B_\lambda(x)u, v)_H- (B_\lambda(0)u, v)_H | \nonumber\\
&\le &\sum^m_{i=1}|t_i|\|P_\lambda(x)e_i- P_\lambda(0)e_i\|\cdot\|v\|+
\sum^m_{i=1}|t_i|\|Q_\lambda(x)-Q_\lambda(0)\|\cdot\|v\|\nonumber\\
&\le &\left(\sum^m_{i=1}\|P_\lambda(x)e_i-
P_\lambda(0)e_i\|^2\right)^{1/2}\left(\sum^m_{i=1}|t_i|^2\right)^{1/2}\|v\|\nonumber\\
&&\quad +
\sqrt{m}\left(\sum^m_{i=1}|t_i|^2\right)^{1/2}\|Q_\lambda(x)-Q_\lambda(0)\|\cdot\|v\|\nonumber\\
&\le &\hspace{-3mm}
\left[C_4
\left(\sum^m_{i=1}\|P_\lambda(x)e_i-P_\lambda(\theta)e_i\|^2\right)^{1/2}
+ C_4\sqrt{m}\|Q_\lambda(x)-Q_\lambda(0)\| \right]\|u\|\|v\|\nonumber\\
&=&\omega_\lambda(x)\|u\|\|v\|,
\end{eqnarray}
 where
$$
\omega_\lambda(x)=\hspace{-2mm}\left[C_4 \left(\sum^m_{i=1}\|P_\lambda(x)e_i-
P_\lambda(0)e_i\|^2\right)^{1/2} + C_4\sqrt{m}\|Q_\lambda(x)-Q_\lambda(0)\|
\right].
$$
Clearly, the assumptions (i) and (iii)  in Theorem~\ref{th:A.9} imply
\begin{align} \label{e:C.23}
\text{$\omega_\lambda(x)\to 0$
as $x\in U\cap X$ approaches to $0$ in $H$ and $\lambda\in\Lambda$ converges to $\lambda^\ast$.}
\end{align}
Note that $(B_\lambda(0)u, v)_H=0$ holds for any $u\in H^+_{\lambda^\ast}$ and $v\in H^0_{\lambda^\ast}\oplus H^-_{\lambda^\ast}$. By (\ref{e:C.22}) we have
\begin{align} \label{e:C.24}
|(B_\lambda(x)u,v)_H|\le \omega_\lambda(x)\|u\|\|v\|,\quad
\forall x\in U\cap X,\;
\forall u\in H^+_{\lambda^\ast}, \;\forall
v\in H^0_{\lambda^\ast}\oplus H^-_{\lambda^\ast}.
\end{align}
Note that   there exists $a_0>0$ such that
\begin{eqnarray}\label{e:C.25}
&&(B_{\lambda^\ast}(0)u, u)_H\ge 2a_0\|u\|^2\quad\forall u\in H^+_{\lambda^\ast},\\
&&(B_{\lambda^\ast}(0)v, v)_H\le -2a_0\|u\|^2\quad\forall v\in H^-_{\lambda^\ast}.\label{e:C.26}
\end{eqnarray}
By (\ref{e:C.23}) we can take a neighborhood $V\subset U$ of $0\in H$
 and a neighborhood ${\Lambda}_0$ of $\lambda^\ast$ in ${\Lambda}$ such that
 $\omega_\lambda(x)<a_0$ for all $x\in V\cap X$ and $\lambda\in{\Lambda}_0$.
   It follows from this, (\ref{e:C.22}) and (\ref{e:C.26}) that
\begin{eqnarray}\label{e:C.27}
(B_\lambda(x)v,v)_H\le -2a_0\|v\|^2+ \omega_\lambda(x)\|v\|^2\le -a_0\|v\|^2\quad\forall v\in H^-_{\lambda^\ast}
\end{eqnarray}
for all $x\in V\cap X$ and $\lambda\in{\Lambda}_0$.
In the proof of \cite[Theorem~A.3)]{Lu8},
by shrinking $V$ and ${\Lambda}_0$ if necessary,
we proved that there exists  $a_1>0$ such that
for all $x\in V\cap X$ and $\lambda\in{\Lambda}_0$,
\begin{eqnarray}\label{e:C.28}
(B_\lambda(x)u,u)_H\ge a_1\|u\|^2\quad\forall u\in H^+_{\lambda^\ast}.
\end{eqnarray}

Take $\varepsilon>0$ so small that
$B_{H^+_{\lambda^\ast}}(0,\varepsilon)\oplus B_{H^-_{\lambda^\ast}}(0,\varepsilon)\subset V$.
Then $B_{H^+_{\lambda^\ast}}(0,\varepsilon)\oplus B_{X^-_{\lambda^\ast}}(0,\varepsilon)\subset V^X$
since $X^+_{\lambda^\ast}\stackrel{\Delta}{=}H^+_{\lambda^\ast}\cap X=H^+_{\lambda^\ast}$
and $\|x\|\le\|x\|_X$ for all $x\in X$.

Let us prove the first inequaity in (\ref{e:add-inequality}).
For $u^+\in \bar B_H(\theta,\varepsilon)\cap X^+$ and $u^-_1, u^-_2\in\bar
B_{H^-}(\theta,\varepsilon)$, by  (\ref{e:LAB}) we have
\begin{eqnarray*}
&&[D\mathcal{L}_\lambda(u^+ + u^-_2)-D\mathcal{L}_\lambda(u^++ u^-_1)](u^-_2-u^-_1)\\
&=&(A_\lambda( u^++u^-_2), u^-_2-u^-_1)_H- (A_\lambda( u^++u^-_1), u^-_2-u^-_1)_H.
\end{eqnarray*}
Since $A_\lambda$ is continuously directional differentiable, so is the
function
$$
u\mapsto (A_\lambda(u^++u), u^-_2-u^-_1)_H.
$$
By the mean value theorem we have $t\in (0, 1)$ such that
\begin{eqnarray*}
&&(A_\lambda( u^++u^-_2), u^-_2-u^-_1)_H
- (A_\lambda( u^++u^-_1), u^-_2-u^-_1)_H\\
&=&\left(DA_\lambda( u^++ u^-_1+ t(u^-_2-u^-_1))(u^-_2-u^-_1),
u^-_2-u^-_1\right)_H\\
&\stackrel{(\ref{e:LAB})}{=}&\left(B_\lambda( u^++ u^-_1+
t(u^-_2-u^-_1))(u^-_2-u^-_1),
u^-_2-u^-_1\right)_H\\
&\le& -a_0\|u^-_2-u^-_1\|^2
\end{eqnarray*}
by (\ref{e:C.27}). Hence
\begin{eqnarray*}
[D\mathcal{L}_\lambda(u^+ + u^-_2)-D\mathcal{L}_\lambda(u^++ u^-_1)](u^-_2-u^-_1)\le
-a_0\|u^-_2-u^-_1\|^2.
\end{eqnarray*}
 Since $\bar
B_H(0,\varepsilon)\cap X^+$ is dense in $\bar
B_H(0,\varepsilon)\cap H^+$, for all  $u^+\in \bar B_H(0,\varepsilon)\cap H^+$ and $u^-_i\in\bar
B_{H}(0,\varepsilon)\cap H^-$, $i=1, 2$, we get
\begin{equation}\label{e:C.29}
 [D\mathcal{L}_\lambda(u^+ + u^-_2)-D\mathcal{L}_\lambda(u^++ u^-_1)](u^-_2-u^-_1)\le
-a_0\|u^-_2-u^-_1\|^2.
\end{equation}
The first inequaity in (\ref{e:add-inequality}) is thus proved.

To prove the second inequaity in (\ref{e:add-inequality}),
let  $u^+\in \bar B_H(0,\varepsilon)\cap X^+$ and $u^-\in\bar
B_{H^-}(0,\varepsilon)$. Since $D\mathcal{L}_\lambda(0)=0$, by the mean value
theorem,
for some $t\in (0, 1)$ we have
\begin{eqnarray*}
&&D\mathcal{L}_\lambda(u^++u^-)(u^+-u^-)\\
&=&D\mathcal{L}_\lambda(u^++u^-)(u^+-u^-)- D\mathcal{L}_\lambda(0)(u^+-u^-)\\
&=&(A_\lambda( u^++u^-), u^+-u^-)_H-(A_\lambda(0), u^+-u^-)_H\\
&=&\left(B_\lambda( t(u^++u^-))(u^++u^-), u^+-u^-\right)_H\\
&=&\left(B_\lambda( t(u^++u^-))u^+, u^+\right)_H\\
&-&\left(B_\lambda(t(u^++u^-))u^-, u^-\right)_H\\
&\ge & a_1\|u^+\|^2+ a_0\|u^-\|^2
\end{eqnarray*}
by (\ref{e:C.27}) and (\ref{e:C.28}).  As above, this inequality also
holds for all $u^+\in \bar B_{H^+}(0,\varepsilon)$ because
$\bar B_H(0,\varepsilon)\cap X^+$ is dense in $\bar
B_H(0,\varepsilon)\cap H^+$ and the mapping $u^+\mapsto D\mathcal{L}_\lambda(u^++u^-)(u^+-u^-)$
is continuous (on $H^+$).
This completes the proof of the second inequality in (\ref{e:add-inequality}).
({\bf Note}: Until now, we have only needed ${\Lambda}$ to be a topological space.)

{\it Proof of the ``Moreover'' statement.}
Two inequalities in (\ref{e:add-inequality}) imply,
respectively, that conditions (ii) and (iii) of \cite[Theorem~A.1]{Lu3}
are satisfied.
By the second inequality in (\ref{e:add-inequality}) we deduce that
for all $(\lambda, u^+)\in \hat\Lambda_0\times (B_{H^+_{\lambda^\ast}}(0,\varepsilon)\setminus\{0\})$,
$$
D\mathcal{L}_{\lambda}(u^+)[u^+]\ge  \mathfrak{a}_2(\|u^+\|^2)
> p(\|u^+\|),
$$
where $p:(0, \varepsilon]\to (0, \infty)$ is a non-decreasing
function given by $p(t)=\frac{\mathfrak{a}_2}{2}t^2$.
Consequently, condition (iv) in \cite[Theorem~A.1]{Lu3} is satisfied.
Then, from Remark~\ref{rm:splitting} and \cite[Theorem~A.1]{Lu3}, it follows that (i) and (ii) hold.
\end{proof}

%
%

%


\renewcommand{\refname}{REFERENCES}

\medskip

\begin{tabular}{l}
 School of Mathematical Sciences, Beijing Normal University\\
 Laboratory of Mathematics and Complex Systems, Ministry of Education\\
 Beijing 100875, The People's Republic of China\\
 E-mail address: gclu@bnu.edu.cn\\
\end{tabular}

\end{document}